\documentclass[a4paper,11pt,reqno,noindent]{amsart}
\usepackage[centertags]{amsmath}
\usepackage{amsfonts,amssymb,amsthm,dsfont,cases,amscd,esint,enumerate, stmaryrd}
\usepackage[T1]{fontenc}
\usepackage[english]{babel}
\usepackage[applemac]{inputenc}
\usepackage{newlfont}
\usepackage{color}
\usepackage[body={17cm,21.5cm}, top = 1in, bottom = 1.5in, centering]{geometry} 
\usepackage{fancyhdr}
\usepackage{enumerate}

\theoremstyle{plain}
\newtheorem{theor}{Theorem}[section]
\newtheorem{lem}[theor]{Lemma}
\newtheorem{prop}[theor]{Proposition}
\newtheorem{cor}[theor]{Corollary}
\theoremstyle{definition}

\newtheorem{rem}[theor]{Remark}
\newtheorem{defin0}[theor]{Definition}

\mathchardef\emptyset="001F
\numberwithin{equation}{section}

\newcommand{\Dd}{\mathbb D}
\newcommand{\Xd}{\mathbb X}
\newcommand{\N}{\mathbb N}

\newcommand{\R}{\mathbb R}
\newcommand{\Z}{\mathbb Z}

\newcommand{\T}{\mathbb T}
\newcommand{\e}{\varepsilon}

\newcommand{\Pc}{\mathcal{P}}
\newcommand{\Lc}{\mathcal{L}}

\newcommand{\Dm}{\mathbb{D}}

\newcommand{\Sc}{\mathcal S}

\newcommand{\calL}{\mathcal L}

\newcommand{\calP}{\mathcal P}
\newcommand{\calW}{\mathcal W}
\newcommand{\calA}{\mathcal{A}}

\newcommand{\Nc}{\mathcal N}

\newcommand{\loc}{{\operatorname{loc}}}
\newcommand{\Id}{\operatorname{Id}}

\newcommand{\E}{\mathbb{E}}
\newcommand{\Var}{\operatorname{Var}}

\newcommand{\Cov}{\operatorname{Cov}}

\newcommand{\Tr}{\operatorname{tr}}

\newcommand{\supess}{\operatorname{sup\,ess}}
\newcommand{\Ld}{\operatorname{L}}

\newcommand{\Div}{{\operatorname{div}}}
\newcommand{\Sym}{{\operatorname{sym}}}
\newcommand{\3}{\operatorname{|\hspace{-0.4mm}|\hspace{-0.4mm}|}}
\newcommand{\step}[1]{\noindent \textit{Step} #1.}

\newcommand{\Pm}{\mathbb{P}}
\newcommand{\pr}[1]{\mathbb{P}\left[ #1 \right]}
\newcommand{\expec}[1]{\mathbb{E}\left[ #1 \right]}
\newcommand{\expecm}[1]{\mathbb{E}\big[ #1 \big]}
\newcommand{\expecM}[1]{\mathbb{E}\bigg[ #1 \bigg]}
\newcommand{\var}[1]{\mathrm{Var}\left[#1\right]}

\newcommand{\dW}[2]{\operatorname{d}_{\operatorname{W}}\left({#1};{#2}\right)}
\newcommand{\dK}[2]{\operatorname{d}_{\operatorname{K}}\left({#1};{#2}\right)}
\newcommand{\Uc}{\mathcal{U}}
\newcommand{\Vc}{\mathcal{V}}
\newcommand{\Lind}[1]{\frac{\delta {#1}}{\delta \mu}}
\newcommand{\Lindk}[2]{\frac{\delta^{#2} {#1}}{\delta \mu^{#2}}}
\newcommand{\om}{\omega}
\newcommand{\ddr}{\mathrm{d}}

\newcommand{\dd}{\color{black}}

\usepackage[colorlinks,citecolor=black,urlcolor=black]{hyperref}
\usepackage{tikz}
\usetikzlibrary{fit}
\usetikzlibrary{shapes.geometric}
\usetikzlibrary{calc}

\newcommand{\vertiii}[1]{{\left\vert\kern-0.25ex\left\vert\kern-0.25ex\left\vert #1 
    \right\vert\kern-0.25ex\right\vert\kern-0.25ex\right\vert}}
\newcommand{\Vertiii}[1]{{\Big\vert\kern-0.25ex\Big\vert\kern-0.25ex\Big\vert #1 
    \Big\vert\kern-0.25ex\Big\vert\kern-0.25ex\Big\vert}}
\newcommand{\Intb}[2]{\hyperref[hyp:int_b]{Int-b-$({#1},{#2})$}}
\newcommand{\Lipb}[2]{\hyperref[hyp:lip_b]{Lip-b-$({#1},{#2})$}}
\newcommand{\TLip}[1]{\hyperref[hyp:tlip]{TLip-$\Phi$-$({#1})$}}
\newcommand{\TReg}[2]{\hyperref[hyp:treg]{TReg-$\Phi$-$({#1},{#2})$}}
\newcommand{\TInt}[2]{\hyperref[hyp:tint]{TInt-$\Phi$-$({#1},{#2})$}}
\newcommand{\Ergp}[4]{\hyperref[hyp:erg0]{Erg-$({#1},{#2},{#4})$}}
\newcommand{\Erg}[1]{\hyperref[hyp:erg]{(Erg-${#1}$)}}

\makeatletter
\newcommand{\@fatnormbar}{\vrule\@width 1.8\p@}
\newcommand{\fatnorm}{\@ifstar\@xfatnorm\@fatnorm}
\newcommand{\@xfatnorm}[1]{%
  \left.\kern-\nulldelimiterspace
  \@fatnormbar
  #1
  \@fatnormbar
  \right.\kern-\nulldelimiterspace
}
\newcommand\@fatnorm[2][]{%
  \mathopen{\vphantom{#1|}\mkern2mu\@fatnormbar\mkern2mu}
  #2
  \mathclose{\vphantom{#1|}\mkern2mu\@fatnormbar\mkern2mu}
}
\makeatother

\makeatletter
\newcommand\RSloop{\@ifnextchar\bgroup\RSloopa\RSloopb}
\makeatother
\newcommand\RSloopa[1]{\bgroup\RSloop#1\relax\egroup\RSloop}
\newcommand\RSloopb[1]%
  {\ifx\relax#1%
   \else
     \ifcsname RS:#1\endcsname
       \csname RS:#1\endcsname
     \else
       \GenericError{(RS)}{RS Error: operator #1 undefined}{}{}%
     \fi
   \expandafter\RSloop
   \fi
  }
\newcommand\X{0}
\newcommand\RS[1]%
  { \begin{tikzpicture}
     [blue,
      line cap=round,
      baseline = -.7ex
     ]
   \coordinate(\X) at (0,0);
   \node (a) at (\X){};
   \RSloop{#1}\relax
   \end{tikzpicture}
  }
\newcommand\RSdef[1]{\expandafter\def\csname RS:#1\endcsname}
\newlength\RSu
\RSdef{p}{ \draw[fill = blue] (\X) circle [radius = .05]; }%% draw point at current position
\RSdef{U}{ \path (\X) -- +(90:\RSu) coordinate (\X I); \edef\X{\X I}(\X)} %move current position up and re-centers
\RSdef{C}{\draw[blue, rounded corners]  ($(\X) + (-.2,.2)$) rectangle ($(.2,-.2)$);} %% rectangle from current to bottom
\RSdef{L}{ \path (\X) -- +(180:\RSu) coordinate(\X I);\edef\X{\X I}(\X)}
\RSdef{g}{ \node[blue] (a) at (\X) {G};}
\RSdef{P}{\node[blue] (b) at (\X) {$\Phi_{\bar \tau}$};}
\RSdef{2}{\node[blue, below] (d) at ($(\X) + (.3,-0.05)$) {$_2$};}
\RSdef{h}{\node[blue, below] (c) at ($(\X) + (.2,-.1)$) {\color{black} $_m$};}
\RSdef{e}{\draw[blue] (\X) -- +(90:\RSu) coordinate (\X I); \edef\X{\X I}(\X)} %% move up and draw the edge 
\RSdef{S}{\begin{scope}[shift = {(0, -.5)}] \end{scope}} %% shift the whole graph down by one unit.\\
\RSdef{t}{\draw (\X) -- ($(\X) + (0,.2)$);}
\RSdef{b}{\draw[blue, rounded corners] ($(\X) + (-.2,-.2)$) rectangle ($(\X) + (.2,.2)$);}
\RSdef{F}{\begin{scope}[scale = 1.3]
\draw[blue, rounded corners] ($(\X) + (-.2,.2)$) rectangle ($(0,0) + (.2,-.2)$);
\end{scope}} %% draws a big box around the whole plot.
\RSdef{H}{\draw[blue, rounded corners] ($(\X) + (-.2,.2)$) rectangle ($(\X) + (.2,-.5)$);} %% to draw a box around an edge 
\RSdef{s}{\draw ($(\X) + (0,.2)$) -- ($(\X) + (0,.6)$);}
\RSdef{G}{\draw[blue, rounded corners] ($(\X) + (-.4,-.4)$) rectangle ($(.4,.4)$);}
\RSdef{B}{\draw [blue, rounded corners] (-.3,-.3) rectangle (.3,.3);}
\RSdef{D}{\draw [blue] (-.4,-.4) rectangle (.9,.4); }
\RSdef{q}{ \draw[fill = blue] (.5,0) circle [radius = .05]; }

\RSdef{f}{\draw[blue] (.3, 0) -- (.5,0);}

\RSdef{r}{\draw (\X) -- ($(\X) + (.5,0)$);}
\usepackage[colorlinks,citecolor=black,urlcolor=black]{hyperref}

\title[Uniform-in-time estimates on corrections to mean field]{Uniform-in-time estimates on corrections to mean field\\for interacting Brownian particles}
\author[A. Bernou]{Armand Bernou}
\address[Armand Bernou]{Universit\'e de Lyon, Universit\'e Claude Bernard Lyon 1, Laboratoire de Sciences Actuarielle et Financi\`ere, 50 Avenue Tony Garnier, F-69007 Lyon, France}
\email{armand.bernou@univ-lyon1.fr}
\author[M. Duerinckx]{Mitia Duerinckx}
\address[Mitia Duerinckx]{Universit\'e Libre de Bruxelles, D\'epartement de Math\'ematique, 1050~Brussels, Belgium}
\email{mitia.duerinckx@ulb.be}

\begin{document}

\begin{abstract}
We consider a system of $N$ classical Brownian particles interacting via a smooth long-range potential in the mean-field regime, and we analyze the propagation of chaos in form of uniform-in-time estimates with optimal $N$-dependence on many-particle correlation functions. Our results cover both the kinetic Langevin setting and the corresponding overdamped Brownian dynamics. The approach is mainly based on so-called Lions expansions, which we combine with new diagrammatic tools to capture many-particle cancellations, as well as with fine ergodic estimates on the linearized mean-field equation, and with discrete stochastic calculus with respect to initial data. In the process, we derive some new ergodic estimates for the linearized Vlasov--Fokker--Planck kinetic equation that are of independent interest. Our analysis also leads to a uniform-in-time quantitative central limit theorem and to concentration estimates for the empirical measure associated with the particle dynamics.
\end{abstract}


%82C31  	Stochastic methods (Fokker-Planck, Langevin, etc.) applied to problems in time-dependent statistical mechanics

\maketitle

\tableofcontents

\section{Introduction}

\subsection{General overview}
We consider the Langevin dynamics for a system of $N$ Brownian particles with mean-field interactions, moving in a confining potential in $\R^d$, $d \ge 1$, as described by the following system of coupled SDEs: for $1 \le i \le N,$
\begin{align}\label{eq:Langevin}
\left\{\begin{array}{ll}
\ddr X^{i,N}_t=V^{i,N}_t\ddr t,&\\[2mm]
\ddr V^{i,N}_t =-\tfrac\kappa N\sum_{j=1}^N\nabla W(X^{i,N}_t-X_t^{j,N})\,\ddr t-\tfrac\beta2 V^{i,N}_t\ddr t - \nabla A(X^{i,N}_t)\,\ddr t+\ddr B_t^{i}, &\quad t \ge 0, \\[2mm]
(X^{i,N}_t,V^{i,N}_t)|_{t=0} = (X^{i,N}_\circ,V^{i,N}_\circ),
\end{array}\right.
\end{align}
where $\{Z^{i,N}:=(X^{i,N},V^{i,N})\}_{1\le i\le N}$ is the set of particle positions and velocities in the phase space $\Dd^d:=\R^d\times\R^d$, where $W:\R^d\to\R$ is a long-range interaction potential, where $A$ is a uniformly convex confining potential, where~$\{B^i\}_i$ are i.i.d.\@ $d$-dimensional Brownian motions, and where $\kappa,\beta>0$ are given constants.
We assume that the interaction potential $W$ satisfies the action-reaction condition~$W(x)=W(-x)$,
and that it is smooth, $W\in C^\infty_b(\R^d)$.
Regarding the confining potential~$A$, we choose it to be quadratic for simplicity,
\begin{equation}\label{eq:confinement-A}
A(x)\,:=\,\tfrac12a|x|^2,\qquad\text{for some $a>0$},
\end{equation}
although this is not essential for our results; see Section~\ref{rem:gener} below.
Next to this Langevin dynamics, we also consider its overdamped limit, that is, the following inertialess Brownian dynamics: for~$1 \le i \le N$,
\begin{align}\label{eq:Brownian}
\left\{\begin{array}{ll}
\ddr Y^{i,N}_t = -\tfrac\kappa N\sum_{j=1}^N\nabla W(Y^{i,N}_t-Y_t^{j,N})\,\ddr t - \nabla A (Y^{i,N}_t) \, \ddr t + \ddr B_t^{i}, \qquad t \ge 0,\\[2mm]
Y^{i,N}_t|_{t=0} = Y^{i,N}_\circ,
\end{array}\right.
\end{align}
where $\{Y^{i,N}\}_{1\le i\le N}$ is now the corresponding set of particle positions in $\R^d$.
For presentation purposes in this introduction, we restrict to the more delicate setting of the Langevin dynamics~\eqref{eq:Langevin}, but we emphasize that all our results hold in both cases.

In the regime of a large number $N\gg1$ of particles, we turn to a statistical description of the system and consider the evolution of a random ensemble of particles. In terms of a probability density~$F^N$ on the $N$-particle phase space $(\Dd^d)^N$, the Langevin dynamics~\eqref{eq:Langevin} is equivalent to the Liouville equation
\begin{multline}\label{eq:Liouville}
\partial_tF^N+\sum_{i=1}^Nv_i\cdot\nabla_{x_i}F^N \,=\,\tfrac12\sum_{i=1}^N\Div_{v_i}((\nabla_{v_i}+\beta v_i)F^N)\\[-3mm]
+\tfrac\kappa N\sum_{i,j=1}^N\nabla W(x_i-x_j)\cdot\nabla_{v_i}F^N
+\sum_{i=1}^N \nabla_{x_i} A \cdot \nabla_{v_i} F^N.
\end{multline}
Particles are assumed to be exchangeable, which amounts to the symmetry of $F^N$ in its $N$ variables $z_i=(x_i,v_i)\in\Dd^d$, $1\le i\le N$. More precisely, we assume for simplicity that particles are initially chaotic, meaning that the initial data $\{Z_\circ^{i,N}:=(X^{i,N}_\circ,V_\circ^{i,N})\}_{1\le i\le N}$ are i.i.d.\@ with some common phase-space density $\mu_\circ\in\Pc(\Dd^d)$: in other words, $F^N$ is initially tensorized,
\begin{equation}\label{eq:chaos-in}
F_t^N|_{t=0}=\mu_\circ^{\otimes N}.
\end{equation}
Throughout, we assume that the initial law $\mu_\circ$ has some stretched exponential moments:
\begin{equation}\label{eq:mom-mu0}
\int_{\Dd^d}e^{|z|^\theta}\mu_\circ(\ddr z)<\infty,\qquad\text{for some $\theta>0$}.
\end{equation}
In the large-$N$ limit, we aim at an averaged description of the system and we focus on the evolution of a finite number of ``typical'' particles as described by the marginals of $F^N$,
\[F^{m,N}_t(z_1,\ldots,z_n)\,:=\,\int_{(\Dd^d)^{N-m}}F^N_t(z_1,\ldots,z_N)\,\ddr z_{m+1}\ldots \ddr z_N,\qquad 1\le m\le N.\]
In view of Boltzmann's chaos assumption, correlations between particles are expected to be negligible to leading order, hence the chaotic behavior of initial data would remain approximately satisfied: this is the so-called propagation of chaos,
\begin{equation}\label{eq:prop-chaos}
F^{m,N}_t- (F^{1,N}_t)^{\otimes m}\,\to\,0,\qquad\text{as $N\uparrow\infty$},
\end{equation}
for any fixed $m\ge1$ and $t\ge0$.
If this holds, it automatically implies the validity of the mean-field limit
\[F^{m,N}_t\,\to\, \mu_t^{\otimes m},\qquad\text{as~$N\uparrow\infty$},\]
where $\mu_t$ is the solution of the Vlasov--Fokker--Planck mean-field equation
\begin{equation}\label{eq:VFP}
\left\{\begin{array}{l}
\partial_t\mu+v\cdot\nabla_x\mu=\tfrac12\Div_v((\nabla_v+\beta v)\mu)+ (\nabla A + \kappa\nabla W\ast \mu)\cdot \nabla_v\mu,\qquad t\ge0,\\[1mm]
\mu|_{t=0}=\mu_\circ,
\end{array}\right.
\end{equation}
with the short-hand notation $\nabla W\ast \mu(x):=\int_{\Dd^d}\nabla W(x-y)\,\mu(y,v)\, \ddr y \, \ddr v$. This topic has been extensively investigated since the 1990s, starting in particular with~\cite{Gaertner_1988,Sznitman_1991}; see e.g.~\cite{Jabin-Wang-17,Chaintron-Diez-21} for a review.

On the formal level, corrections to the propagation of chaos and to the mean-field limit are naturally unravelled by means of the BBGKY approach, which goes back to the work of Bogolyubov~\cite{Bogolyubov-46}. This starts by noting that the Liouville equation~\eqref{eq:Liouville} is equivalent to the following hierarchy of coupled equations for marginals: for $1\le m\le N$,
\begin{multline}\label{eq:BBGKY0}
\partial_tF^{m,N}+\sum_{i=1}^mv_i\cdot\nabla_{x_i}F^{m,N}\,=\,\tfrac12\sum_{i=1}^m\Div_{v_i}((\nabla_{v_i}+\beta v_i)F^{m,N})\\
+ \sum_{i=1}^m \nabla_{x_i} A \cdot \nabla_{v_i} F^{m,N}
+\kappa\tfrac{N-m}N\sum_{i=1}^m\int_{\Dd^d}\nabla W(x_i-x_*)\cdot\nabla_{v_i}F^{m+1,N}(\cdot,z_*)\,\ddr x_* \, \ddr v_*\\[-3mm]
+\tfrac\kappa N\sum_{i,j=1}^m\nabla W(x_i-x_j)\cdot\nabla_{v_i}F^{m,N},
\end{multline}
with the convention $F^{m,N}\equiv0$ for $m>N$. 
In each of those equations, the last right-hand side term is precisely the one that disrupts the chaotic structure: it creates correlations between initially independent particles, hence leads to deviations from the mean-field approximation. This term is formally of order~$O(m^2N^{-1})$, and we are actually led to conjecture the following error estimate for the propagation of chaos,
\begin{equation}\label{eq:size-chaos}
F^{m,N}_t-(F^{1,N}_t)^{\otimes m}\,=\,O(mN^{-1}).
\end{equation}
This was first established by Lacker~\cite{Lacker-21}, see also~\cite{BJS-22}, based on a rigorous BBGKY analysis, and it is referred to as estimating the {\it size of chaos}, cf.~\cite{BenArous-Zeitouni-99,PPS-19}.
In case of non-Brownian interacting particles, the problem is more difficult as no rigorous BBGKY analysis is available: the rate $O(N^{-1})$ for fixed $m$ was obtained in~\cite{MD-21} by means of different, non-hierarchical techniques. In the sequel, we systematically leave aside all issues related to the sharp dependence on $m$ in such estimates and only focus on optimal convergence rates in $N$ for fixed $m$.

A variant of the above estimates
is given by so-called {\it weak propagation of chaos} estimates: for any sufficiently well-behaved functional $\Phi$ defined on the space  $\Pc(\Dd^d)$ of probability measures on $\Dd^d$, one expects
\begin{equation}\label{eq:size-chaos-bis}
\expec{\Phi(\mu_t^N)}-\Phi(\mu_t)\,=\,O(N^{-1}),
\end{equation}
in terms of the empirical measure
\begin{equation}\label{eq:empirical-meas}
\textstyle\mu_t^N:=\tfrac1N\sum_{i=1}^N\delta_{Z_t^{i,N}}~~\in~\Pc(\Dd^d),
\end{equation}
where we recall that the limit~$\mu_t$ is the solution of the mean-field equation~\eqref{eq:VFP} and where the expectation~$\E$ is taken with respect to both the initial data and the Brownian forces.
Such an estimate is essentially equivalent to~\eqref{eq:size-chaos} (leaving aside the dependence on $m$ and $\Phi$),
and we refer to~\cite{Kolokoltsov-10,Mischler_2012,Mischler_2013,Bencheikh_2019,Chassagneux_2019} for results in that direction.
Note that the rate $O(N^{-1})$ in~\eqref{eq:size-chaos-bis} is only expected for $\Phi$ smooth enough. For the specific choice $\Phi=\calW_2(\cdot,\mu_t)$, for instance, the question amounts to estimating the expectation of the $2$-Wasserstein distance between $\mu_t^N$ and $\mu_t$: this is referred to as {\it strong propagation of chaos} and is known to lead only to a weaker convergence rate $\expec{\calW_2(\mu_t^N,\mu_t)}=O(N^{-1/2})$ in link with random fluctuations of the empirical measure; see~\cite{BRTV-98,Malrieu-03,Bolley-Guillin-Villani-07,Cattiaux-Guillin-Malrieu-08,Fournier_15,Lacker-21}.

In recent years, there has been an increasing interest in {\it uniform-in-time versions} of the above chaos estimates~\eqref{eq:size-chaos} or~\eqref{eq:size-chaos-bis}.
This happens to be an important question both in theory and for practical applications: it amounts to describing the long-time behavior of particle systems uniformly in the limit~$N\uparrow\infty$, thus showing in particular the proximity of corresponding equilibria.
This is naturally related to the long-time behavior of the mean-field equation~\eqref{eq:VFP}, which has itself been 
an intense topic of research for more than two decades. While the mean-field equilibrium is not unique in general, cf.~\cite{Duong_2016}, the long-time convergence of the mean-field density has been established
under several types of assumptions guaranteeing uniqueness~\cite{Villani_2009,Bolley_2010,Herau_2016, Monmarche_2017,Guillin_2022,Bayraktar_2022};
see also~\cite{BRV-98,Carrillo_2003,Bolley_2012} for the Brownian dynamics.
Uniform-in-time propagation of chaos, however, is a more subtle question and is indeed not ensured by the uniqueness of the mean-field equilibrium~\cite{Malrieu-03,BGP-14}.
In case of small interaction $\kappa\ll1$, uniform-in-time weak chaos estimates with optimal rate~$O(N^{-1})$ were first obtained by Delarue and Tse~\cite{Delarue_Tse_21} for the Brownian dynamics, and the uniform-in-time version of~\eqref{eq:size-chaos} was actually obtained recently in~\cite{Lacker_LeFlem_2022} both for the Langevin and Brownian dynamics.
We also refer to~\cite{Malrieu-03,Bolley-Guillin-Villani-07,Cattiaux-Guillin-Malrieu-08,Bolley_2010,DEGZ-20,Salem-20,Guillin_2022,Chen_2024} for corresponding uniform-in-time {\it strong} chaos estimates with weaker convergence rates.
Much recent work has focused on weakening the assumptions for the validity of uniform-in-time chaos estimates: we refer in particular to~\cite[Sections~3.6 and~4]{Delarue_Tse_21} and references therein for the relaxation of the smallness condition $\kappa\ll1$, and we refer to~\cite{Guillin_2021b,Rosenzweig-Serfaty-21,GLBM_2022,Rosenzweig-Serfaty-23,Blaustein-Jabin-Soler-24} for singular interactions.

In the present work, while mainly sticking to the simplest case of smooth interactions under the smallness condition $\kappa\ll1$, we aim to go beyond uniform-in-time chaos estimates by further estimating many-particle correlation functions, which provides finer information on the propagation of chaos in the system. The two-particle correlation function is defined as
\[G^{2,N}\,:=\,F^{2,N}-(F^{1,N})^{\otimes2},\]
which captures the defect to propagation of chaos~\eqref{eq:prop-chaos} at the level of two-particle statistics.
From the BBGKY hierarchy~\eqref{eq:BBGKY0}, we note that proving the mean-field limit $F^{1,N}_t\to\mu_t$ amounts to proving $G^{2,N}_t\to0$, which is precisely ensured by standard chaos estimates, cf.~\eqref{eq:size-chaos}.
Yet, two-particle correlations do not allow to reconstruct the full particle density $F^N$: in particular, understanding corrections to the mean-field limit requires to further estimate higher-order correlation functions $\{G^{k,N}\}_{2\le k\le N}$. Those are defined as suitable polynomial combinations of marginals of~$F^N$ in such a way that the full particle distribution $F^N$ be recovered in form of a cluster expansion,
\begin{equation}\label{eq:cluster-exp0}
F^N(z_1,\ldots,z_N)\,=\,\sum_{\pi\vdash \llbracket N\rrbracket}\prod_{A\in\pi}G^{\sharp A,N}(z_A),
\end{equation}
where $\pi$ runs through the list of all partitions of the index set $\llbracket N\rrbracket:=\{1,\ldots,N\}$, where $A$ runs through the list of blocks of the partition $\pi$, where $\sharp A$ is the cardinality of $A$, and where for \mbox{$A=\{i_1,\ldots, i_l\}\subset\llbracket N\rrbracket$} we write $z_A=(z_{i_1},\ldots,z_{i_l})$. As is easily checked, correlation functions are fully determined by prescribing~\eqref{eq:cluster-exp0} together with the ``maximality'' requirement $\int_\Dm G^{m,N}(z_1,\ldots,z_m)\,\ddr z_l=0$ for $1\le l\le m$. More explicitly, we can write
\begin{eqnarray*}
G^{3,N}&=&\Sym\big(F^{3,N}-3F^{2,N}\otimes F^{1,N}+2(F^{1,N})^{\otimes 3}\big),\\
G^{4,N}&=&\Sym\big(F^{4,N}-4F^{3,N}\otimes F^{1,N}-3F^{2,N}\otimes F^{2,N}+12F^{2,N}\otimes (F^{1,N})^{\otimes 2}-6(F^{1,N})^{\otimes 4} \big),
\end{eqnarray*}
and so on, where the symbol `$\Sym$' stands for the symmetrization of coordinates. More generally, for all $2\le m\le N$,
\begin{equation}\label{eq:def-cumGm}
G^{m,N}(z_1,\ldots,z_m)\,=\,\sum_{\pi\vdash \llbracket m \rrbracket}(\sharp\pi-1)!(-1)^{\sharp\pi-1}\prod_{A\in\pi}F^{\sharp A,N}(z_A),
\end{equation}
where we use a similar notation as in~\eqref{eq:cluster-exp0} and where $\sharp\pi$ stands for the number of blocks in a partition~$\pi$.
While standard propagation of chaos leads to $G^{2,N}_t=O(N^{-1})$, cf.~\eqref{eq:size-chaos}, and in fact \mbox{$G^{m,N}_t=O(N^{-1})$} for all $2\le m\le N$, a formal analysis of the BBGKY hierarchy~\eqref{eq:BBGKY0} further leads to expect
\begin{equation}\label{eq:refined-prop-chaos}
G_t^{m,N}\,=\,O(N^{1-m}),\qquad2\le m\le N.
\end{equation}
We refer to this as {\it higher-order propagation of chaos}:
such estimates provide a much deeper understanding of the structure of propagation of chaos and are key tools to describe deviations from mean-field theory, cf.~\cite{DSR-21,MD-21,Hess_Childs_2023,DJ-24}.

Estimates of the form~\eqref{eq:refined-prop-chaos} have been obtained in several settings up to an exponential time growth: non-Brownian particle systems were covered in~\cite{MD-21}, while in~\cite{Hess_Childs_2023} the Brownian dynamics~\eqref{eq:Brownian} was covered in the more general case of bounded non-smooth interactions.
In the present work, we obtain for the first time corresponding \emph{uniform-in-time} estimates, both for the Langevin and Brownian dynamics.
In fact, after this work was completed, in the simpler case of the Brownian dynamics~\eqref{eq:Brownian}, similar uniform-in-time estimates have been further obtained in~\cite{Xie-24} for bounded non-smooth interactions, thereby improving on our results, but to our knowledge those arguments do not extend easily to the Langevin setting~\eqref{eq:Langevin}.
Along the way, we establish a quantitative uniform-in-time central limit theorem (CLT) for the empirical measure with an optimal convergence rate, which is new even for the Brownian dynamics~\eqref{eq:Brownian}, and we also prove uniform-in-time concentration estimates.

From the technical perspective, {as described further in Section~\ref{sec:strategy} below,} we mainly take inspiration from a recent work by Delarue and Tse~\cite{Delarue_Tse_21},
where the uniform-in-time chaos estimate~\eqref{eq:size-chaos-bis} was established for the Brownian dynamics. The key idea is to consider the mean-field semigroup induced on functionals $\mu\mapsto\Phi(\mu)$ on the space of probability measures, and then to appeal to the so-called Lions calculus on this space to expand the expectation $\E[\Phi(\mu_t^N)]$ of functionals along the particle dynamics; see Lemma~\ref{lem:CST0} below. As noted in~\cite{Delarue_Tse_21}, the resulting Lions expansions can be combined with ergodic properties of the linearized mean-field equation to deduce uniform-in-time estimates.
In fact, this whole strategy is the heir to the functional framework originally developed by Mischler and Mouhot~\cite{Mischler_2012,Mischler_2013} (also in the context of Kac's program), where uniform-in-time propagation of chaos was obtained by comparing the particle dynamics with the mean-field semigroup induced directly on the space of measures --- rather than on {\it (nonlinear) functionals on} the space of measures, which is much more convenient for our analysis.
In order to control correlation functions $\{G^{m,N}\}_{2\le m\le N}$, we reduce the problem to estimating cumulants of functionals of the empirical measure $\{\kappa_m(\Phi(\mu_t^N))\}_{m\ge1}$, we apply (iterated) Lions expansions to cumulants, and we develop suitable diagrammatic tools to efficiently capture cancellations in such expansions and derive the desired estimates~\eqref{eq:refined-prop-chaos}.
To account for the effect of initial correlations, we further combine Lions expansions with the so-called Glauber calculus that we developed in~\cite{MD-21}.
In order to deduce uniform-in-time estimates, while only the case of the Brownian dynamics was considered in~\cite{Delarue_Tse_21}, we need to further appeal to hypocoercivity techniques to establish the relevant ergodic estimates for the linearized mean-field equation in case of the kinetic Langevin dynamics: for that purpose, we mainly draw inspiration from another work of Mischler and Mouhot~\cite{Mischler_2016}, which we are led to revisit in several ways; see Theorem~\ref{thm:ergodic} below.

We anticipate that our approach, based on capturing cancellations in Lions expansions for cumulants, and combining such expansions with Stein's method and Herbst's argument to further deduce a quantitative CLT and concentration estimates, might also be useful in other contexts, in particular for mean-field games.

\subsection{Main results}
We start with the statement of uniform-in-time higher-order propagation of chaos estimates with the optimal $N$-dependence~\eqref{eq:refined-prop-chaos}. Note that the smallness condition $\kappa\ll1$ for the interaction strength ensures the uniqueness of the steady state for the mean-field equation~\eqref{eq:VFP}, which is key to guarantee strong ergodic properties.

\begin{theor}[Uniform-in-time higher-order propagation of chaos]\label{thm:main}
Consider the Langevin dynamics~\eqref{eq:Langevin} and the associated correlation functions $\{G^{m,N}\}_{2\le m\le N}$ as defined in~\eqref{eq:def-cumGm}, and assume that the initial law has stretched exponential moments~\eqref{eq:mom-mu0} for some $\theta>0$.
For all $m\ge2$, there exist $\kappa_m>0$ (only depending on $d,\beta,W,A,\theta,m$), $C_m>0$ (further depending on $\mu_\circ$), and
$\ell_m>0$ (only depending on~$m$),
such that
given~$\kappa \in [0,\kappa_m]$
we have for all $N,t\ge0$,
\begin{equation}\label{eq:main_thm_eq}
\|G^{m,N}_t\|_{W^{-\ell_m,1}(\Dd^d)^{\otimes m}} \,\le\, C_mN^{1-m}.
\end{equation}
\end{theor}

\begin{rem}
While we focus here on the Langevin dynamics~\eqref{eq:Langevin}, we recall that all our results also hold in the simpler case of the Brownian dynamics~\eqref{eq:Brownian}. In that case, the proof is substantially simplified and a slightly stronger version of Theorem~\ref{thm:main} is actually obtained. On the one hand, the initial law~$\mu_\circ$ only needs to have {\it some} bounded moment, that is, $\int_\Xd|z|^{p}\mu_\circ(\ddr z)<\infty$ for some $p>0$. On the other hand, the smallness condition $\kappa\ll1$ can be taken independently of $m$: there exists $\kappa_0>0$ (only depending on $d,\beta,W,A$) such that the same result holds for any $\kappa\in[0,\kappa_0]$. As explained in Section~\ref{rem:gener}, this smallness condition can even be removed in some cases.
\end{rem}

Let us briefly emphasize some applications of the above correlation estimates. The main one is that these can be used to truncate the BBGKY hierarchy~\eqref{eq:BBGKY0} to an arbitrary precision in $N$.
Of course, the uniform-in-time estimate on two-particle correlations, $G^{2,N}=O(N^{-1})$, allows to recover the uniform-in-time chaos estimate~\eqref{eq:size-chaos}, thus recovering the result of~\cite{Lacker_LeFlem_2022} for smooth interactions (leaving aside the dependence on $m$). More interestingly, the estimate on three-particle correlations allows to further capture the so-called Bogolyubov correction to mean field, for which we obtain the first uniform-in-time result. We state it here for shortness at the level of the $1$-particle density, but this is easily extended to all marginals. Higher-order correlation estimates can be used similarly to capture corrections to an arbitrary order in $N$.

\begin{cor}[Bogolyubov correction to mean field]\label{cor:bogo}
\allowdisplaybreaks
Consider the Langevin dynamics~\eqref{eq:Langevin} and assume that the initial law has stretched exponential moments~\eqref{eq:mom-mu0} for some $\theta>0$.
There exist $\kappa_0>0$ (only depending on $d,\beta,W,A,\theta$), $C_0>0$ (further depending on $\mu_\circ$), and a universal constant~$\ell_0>0$, such that given $\kappa \in [0, \kappa_0]$ we have
for all~$N,t \ge 0$,
\begin{eqnarray}
\|F^{1,N}_t-\mu_t\|_{W^{-\ell_0,1}(\Dd^d)}&\le&C_0N^{-1},\label{eq:unif-MFL}\\
\|F^{1,N}_t-\tilde\mu_t^{N}\|_{W^{-\ell_0,1}(\Dd^d)}&\le&C_0N^{-2},\nonumber\\
\|NG^{2,N}_t-g^{(2)}_t\|_{W^{-\ell_0,1}(\Dd^d)^{\otimes2}}&\le&C_0N^{-1},\nonumber
\end{eqnarray}
where we recall that $\mu$ is the solution to the Vlasov--Fokker--Planck equation~\eqref{eq:VFP}, where $\tilde\mu^{N}$ satisfies the following corrected version of~\eqref{eq:VFP},
\begin{equation*}
\left\{\begin{array}{l}
\partial_t\tilde\mu^{N}+v\cdot\nabla_x\tilde\mu^{N}\,=\,\tfrac12\Div_v((\nabla_v+\beta v)\tilde\mu^{N})+(\nabla A+\kappa\nabla W\ast\tilde\mu^{N})\cdot\nabla_v\tilde\mu^{N}\\[1mm]
\hspace{8cm}+\tfrac\kappa N\int_{\Dd^d}\nabla W(x-x_*)\cdot\nabla_v(g^{(2)}-\mu^{\otimes2})(z,z_*)\,\ddr z_*,\\
\tilde\mu^N|_{t=0}=\mu_\circ,
\end{array}\right.
\end{equation*}
and where $g^{(2)}$ is the solution of the Bogolyubov equation
\begin{equation*}
\left\{\begin{array}{l}
\partial_tg^{(2)}\,=\,L_\mu^{(2)}g^{(2)}+\kappa\big(\nabla W(x_1-x_2)-\nabla W\ast\mu(x_1)\big)\cdot\nabla_{v_1}\mu^{\otimes2}\\[1mm]
\hspace{4cm}+\kappa\big(\nabla W(x_2-x_1)-\nabla W\ast\mu(x_2)\big)\cdot\nabla_{v_2}\mu^{\otimes2},\\
g^{(2)}|_{t=0}=0,
\end{array}\right.
\end{equation*}
where $L_\mu^{(2)}:=L_\mu\otimes\Id+\Id\otimes L_\mu$ stands for the $2$-particle linearized Vlasov operator at the mean-field solution $\mu$ with
\begin{gather*}
L_\mu h\,:=\,\tfrac12\Div_v((\nabla_v+\beta v)h) -v\cdot\nabla_xh+(\nabla A+\kappa\nabla W\ast\mu)\cdot\nabla_vh+\kappa(\nabla W\ast h)\cdot\nabla_v\mu.
\end{gather*}
\end{cor}

These higher-order corrections to mean-field are precisely correcting chaotic densities by taking into account weak particle correlations that build up over time due to interactions.
On long times, they actually allow to connect the mean-field approximation to Gibbs relaxation. More precisely, it is possible to show that there is $\lambda_0>0$ (only depending on $d,\beta,W,A$) such that given $\kappa\in[0,\kappa_0]$ we have for all~$N,t\ge0$,
\begin{equation}\label{eq:unif-MFL-refined}
\|F_t^{1,N}-\mu_t-R^{1,N}\|_{W^{-\ell_0,1}(\Dd^d)}\,\le\,C_0 e^{-\lambda_0 t}N^{-1},
\end{equation}
where $R^{1,N}:=M^{1,N}-M$ is the difference between the first marginal of the $N$-particle Gibbs ensemble and the mean-field equilibrium. Note that this equilibrium correction $R^{1,N}$ vanishes in the spatially-homogeneous setting on the torus.
This result~\eqref{eq:unif-MFL-refined} is new in the field and can be viewed as a ``cross estimate'' combining the mean-field approximation~\eqref{eq:unif-MFL} quantitatively with the convergence of the particle system to Gibbs equilibrium.
Its proof as an application of Theorem~\ref{thm:main} requires detailed computations of corrections to mean field and is postponed to a forthcoming work.

\medskip
As the proof of Theorem~\ref{thm:main} relies on a fine description of cumulants of the empirical measure along the particle dynamics, it can unsurprisingly be pushed further to derive other types of stochastic information on the empirical measure, such as a uniform-in-time quantitative CLT and concentration estimates.
We start with the former.
As expected from formal computations, leading fluctuations of the empirical measure are described by the Gaussian linearized Dean--Kawasaki SPDE, cf.~\eqref{eq:DK-Lang} below.
A qualitative CLT has actually been known to hold since the early days of mean-field theory~\cite{Tanaka-Hitsuda-81,Sznitman-84,Fernandez-Meleard-97}, and it has recently been extended to some singular interaction potentials as well~\cite{Wang-Zhao-Zhu-23,Jungel-24}.
In case of smooth interactions as considered in the present work, an optimal quantitative estimate for fluctuations already follows from~\cite{MD-21}, but we provide here the first uniform-in-time result. To our knowledge, this is new both in the Langevin and Brownian cases.

\begin{theor}[Uniform-in-time CLT]\label{th:CLT}
Consider the Langevin dynamics~\eqref{eq:Langevin} and the associated empirical measure~$\mu_t^N$, cf.~\eqref{eq:empirical-meas}, and assume that the initial law has stretched exponential moments~\eqref{eq:mom-mu0} for some $\theta>0$.
There exist $\lambda_0, \kappa_0 >0$ (only depending on $d,\beta,W,A,\theta$) such that the following holds for any $\kappa\in[0,\kappa_0]$:
for all $\phi\in C^\infty_c(\Dd^d)$, there exists $C_0>0$ (only depending on $d,\beta,W,A,\mu_\circ$) such that we have for all $N,t\ge0$,
\begin{equation*}
\ddr_2\bigg( N^\frac12\int_\Xd\phi\,(\mu_t^N-\mu_t)\,\,;\,\int_\Xd\phi\,\nu_t\bigg)
\,\le\,C_0 \Big(N^{-\frac12}+e^{-\lambda_0 t}N^{-\frac13}\Big)\|\phi\|_{W^{4,\infty}(\Xd)}^2,
\end{equation*}
where:
\begin{enumerate}[---]
\item $\ddr_2$ stands for the second-order Zolotarev distance between random variables,
\begin{equation}\label{eq:Zolo}
\quad\ddr_2(X;Y)\,:=\,\sup\Big\{\E[g(X)]-\E[g(Y)]~:~g\in C^2_b(\R),~g'(0)=0,~\|g''\|_{\Ld^\infty(\R)}=1\Big\};
\end{equation}
\item the limit fluctuation $\nu_t$ is the centered Gaussian process that is the unique almost sure distributional solution of the Gaussian linearized Dean--Kawasaki SPDE (see Section~\ref{sec:DK-def} for details),
\begin{align}\label{eq:DK-Lang}
\left\{\begin{array}{l}
\partial_t \nu_t+ v \cdot \nabla_x\nu_t\,=\, \Div_v(\sqrt{\mu_t} \xi_t) + \Div_v((\nabla_v+\beta v)\nu_t)\\[1mm]
\hspace{3.5cm}+\nabla A \cdot \nabla_v\nu_t + \kappa(\nabla W *\nu_t)\cdot\nabla_v\mu_t + \kappa(\nabla W *\mu_t)\cdot\nabla_v\nu_t,\\
\nu_t|_{t = 0}=\nu_\circ,
\end{array}\right.
\end{align}
where $\xi$ is a vector-valued space-time white noise on $\R^+\times\Dd^d$,
and where $\nu_\circ$ is the Gaussian field describing the fluctuations of the initial empirical measure in the sense that~$N^{1/2}\int_{\Dd^d}\phi\,(\mu^N_\circ-\mu_\circ)$ converges in law to $\int_{\Dd^d}\phi\,\nu_\circ$ for all $\phi\in C^\infty_c(\Dd^d)$.\footnote{In other words, this means that~$\nu_\circ$ is the random tempered distribution on $\Dd^d$ characterized by having Gaussian law with $\E[\int_{\Dd^d}\phi\nu_\circ]=0$ and $\Var[\int_{\Dd^d}\phi\nu_\circ]=\int_{\Dd^d}(\phi-\int_{\Dd^d}\phi \mu_\circ)^2\mu_\circ$ for all $\phi\in C^\infty_c(\Dd^d)$.}
\end{enumerate}
\end{theor}

\begin{rem}[Convergence rate]
The convergence rate in the above CLT can be improved to $O(N^{-1/2})$ under some non-degeneracy condition. More precisely, the proof shows
\begin{align*}
\ddr_2\bigg( N^\frac12\int_\Xd\phi\,(\mu_t^N-\mu_t)\,\,;\,\int_\Xd\phi\,\nu_t\bigg)
\,\le\,C_0N^{-\frac12}\bigg(1+\frac{e^{-3\lambda_0 t}}{c_t(\phi)+ N^{-\frac13}e^{-\lambda_0 t}}\bigg)^{\frac12}\,\|\phi\|_{W^{4,\infty}(\Xd)}^2,
\end{align*}
in terms of $c_t(\phi):=\|\phi\|_{W^{2,\infty}(\Xd)}^{-2}\Var[\int_\Xd\phi\tilde\nu_t]$, where $\tilde\nu$ is the solution of the linearized mean-field equation with initial condition $\nu_\circ$, that is, the solution of the corresponding equation~\eqref{eq:DK-Lang}  with the noise $\xi$ replaced by $0$. Under the non-degeneracy condition $c_t(\phi)\gtrsim1$, the rate thus becomes $O(N^{-1/2})$.
\end{rem}

\begin{rem}[Higher-order fluctuations]\label{rem:TCL}
In recent years, much work has been devoted to the justification of the non-Gaussian nonlinear Dean--Kawasaki equation, which is a highly singular SPDE formally expected to capture higher-order fluctuations; see in particular~\cite{KLvR-20,Cornalba-Fischer-23,Perkowski-22,Cornalba-Fischer-Raithel-2023}.
In contrast, the above result only focuses on {\it Gaussian leading fluctuations}, but it provides the first uniform-in-time justification.
Extensions to non-Gaussian corrections and the uniform-in-time justification of the nonlinear Dean--Kawasaki equation is postponed to a forthcoming work.
\end{rem}

We turn to concentration estimates for the empirical measure, which are easily recovered by combining ideas of the proof of Theorem~\ref{thm:main} together with Herbst's argument.
For the Langevin dynamics~\eqref{eq:Langevin}, the following result complements the concentration estimates obtained in~\cite[Theorem 5]{Bolley_2010} (removing the restriction to $r\gg N^{-1/2}+e^{-Ct}$, but requiring more smoothness of the test function~$\phi$).
In the simpler case of the Brownian dynamics~\eqref{eq:Brownian}, on the other hand, corresponding uniform-in-time concentration estimates are already well known: a uniform-in-time concentration estimate was first deduced in~\cite{Malrieu_2001} from a logarithmic Sobolev inequality in the case when $W$ is convex, which was then largely extended more recently in~\cite{Liu_2021}. We can also refer to the very recent work of Jackson and Zitridis~\cite{Jackson_Zitridis_2024} where concentration bounds for the Wasserstein distance $\mathcal{W}_p(\mu^N_t, \mu_t)$ are obtained from relative entropy estimates.

\begin{theor}[Uniform-in-time concentration]\label{th:concentr}
Consider the Langevin dynamics~\eqref{eq:Langevin} and the associated empirical measure~$\mu_t^N$, cf.~\eqref{eq:empirical-meas}, and assume for simplicity that the initial law~$\mu_\circ\in\calP(\Dd^d)$ has compact support.
There exists $\kappa_0>0$ (only depending on $d,\beta,W,A$) such that the following holds for any~$\kappa\in[0,\kappa_0]$:
there exists $C_0>0$ (only depending on $d,\beta,W,A,\mu_\circ$) such that we have for all $\phi\in C^\infty_c(\Dd^d)$ and $N,t,r\ge0$,
\begin{equation}
\label{eq:concentration_final}
\pr{\int_{\Dd^d} \phi\mu^N_t - \E \Big[ \int_{\Dd^d} \phi\mu^N_t\Big] \ge r} \,\le\, \exp\bigg(-\frac{N r^2}{2C_0\|\phi\|_{W^{3,\infty}(\Dd^d)}^2}\bigg).
\end{equation}
\end{theor}

Finally, we emphasize that, along our way in this work, we prove a general expansion result for the expectation of smooth functionals of the empirical measure along the particle dynamics, which is of independent interest and extends in particular the work of Chassagneux, Szpruch, and Tse~\cite{Chassagneux_2019} to the case of the kinetic Langevin dynamics.
More precisely, in the setting of Theorem~\ref{thm:main}, for all $m\ge1$ and $\kappa\in[0,\kappa_m]$, for all smooth functionals $\Phi$,
we obtain a truncated expansion of the following form, for all $N,t \ge 0$,
\[ \expec{\Phi(\mu^N_t)} - \Phi(\mu_t) \,=\, \sum_{j=1}^{m} \frac{C_{j,\Phi}(t,\mu_\circ)}{N^j} + O(N^{-m-1}), \]
with exact expressions for the coefficients $\{C_{j,\Phi}(t,\mu_\circ)\}_{j}$ independent of $N$.
As explained in~\cite[Section~1.1]{Chassagneux_2019}, by means of Romberg extrapolation, such an expansion can be used to accelerate the convergence of numerical schemes to estimate $\Phi(\mu_t)$ through the particle method.

\subsection{Extensions}
\label{rem:gener}
We focus for shortness on the well-traveled setting of interacting Brownian particles with small pairwise interactions derived from a smooth potential in the ambient space $\R^d$ with quadratic confinement. Yet, our analysis is pretty robust and can be adapted as long as suitable ergodic estimates are available for the linearized mean-field equation. We briefly describe several examples.

\subsubsection{Periodic setting}
\label{subsubsec:periodic}
Our analysis is easily adapted to particle systems on the torus $\T^d$, for instance in the spatially-homogeneous setting $A \equiv 0$, in which case our results still hold in the same form. In the proof, a minor difference appears in the Langevin setting, see Remark~\ref{rem:toroidal-erg}.
For the Brownian dynamics~\eqref{eq:Brownian}, thanks to the rich literature on ergodic properties of the linearized mean-field equation on the torus, see e.g.~\cite{Carrillo_2019,Delarue_Tse_21}, the smallness condition $\kappa\ll1$ can be lifted at least in two cases:
\begin{enumerate}[---]
\item \emph{Conservative interaction:}
If the interaction kernel $\nabla W$ is replaced by a smooth force kernel $K\in C^\infty_b(\T^d)^d$ with $K(-x)=-K(x)$, and if the force is conservative in the sense of $\Div K=0$, then we can rely on the ergodic estimates in~\cite[Section 3.6.2]{Delarue_Tse_21} and our results hold for the Brownian dynamics~\eqref{eq:Brownian} without the restriction $\kappa\ll1$.
\smallskip\item \emph{`H-stable' potential:} If the interaction potential $W\in C^\infty_b(\T^d)$ has nonnegative Fourier coefficients $\hat W(n)\ge0$ for all $n\in\Z^d$, then we can rely on the ergodic estimates in~\cite[Section 3.6.3]{Delarue_Tse_21} and our results hold for the Brownian dynamics~\eqref{eq:Brownian} without the restriction $\kappa\ll1$. In fact, the nonnegativity condition can even be relaxed to $1 +2 \kappa \inf_{n \in \Z^d} \hat W(n) > 0$.
\end{enumerate} 

\subsubsection{Non-quadratic confinement} 
Our results also hold in the same form if instead of the quadratic confinement~\eqref{eq:confinement-A} we choose $A(x)=a|x|^2+A'(x)$ for some $a>0$ and some smooth potential $A'\in C^\infty_b(\R^d)$ provided that $\|\nabla^2A'\|_{\Ld^\infty(\R^d)}$ is small enough (depending on $\beta,W,a$). In that case, we can still appeal to~\cite{Bolley_2010} to ensure the validity of Theorem~\ref{thm:ergodic}(i) below, while the rest of our approach can be adapted directly without major difficulties.

\subsubsection{More general interactions}
Although we focus for shortness on pairwise interactions between the particles, the term $\tfrac1N\sum_{j=1}^N\nabla W(X^{i,N}-X^{j,N})$ in~\eqref{eq:Langevin} (or respectively in~\eqref{eq:Brownian}) could be replaced by a more general interaction force of the form $b(X^{i,N},\mu^N)$ for some functional $b:\R^d\times\Pc(\Dd^d)\to\R^d$, where we recall that $\mu^N$ stands for the empirical measure~\eqref{eq:empirical-meas}. While the choice in~\eqref{eq:Langevin} amounts to $b(x,\mu)=\nabla W\ast\mu(x)$, our analysis is easily adapted to more general nonlinear functionals $b$ under suitable regularity assumptions both on $b$ and on its derivatives with respect to $\mu$, together with a corresponding smallness condition on $\kappa$. We skip the details and refer to~\cite{Delarue_Tse_21} for the needed adaptations (see in particular the regularity assumptions~(Reg) and~(Lip) in~\cite[Section~2.3.2]{Delarue_Tse_21}).

\subsection{Strategy and plan of the paper}\label{sec:strategy}
Let us briefly describe the strategy of the proof of Theorem~\ref{thm:main}.
It is well known that the estimation of correlation functions~$\{G^{m,N}\}_{2\le m\le N}$ can be reduced to the estimation of cumulants $\{\kappa^n(\int_{\Dd^d}\phi \mu_t^N)\}_{n\ge1}$ of linear functionals of the empirical measure $\mu_t^N$; see Lemma~\ref{lem:cumtocorrel}.
As the probability space is a product space accounting both for initial data and for Brownian forces, cumulants can be split through the law of total cumulance: we are thus led to consider separately ``initial'' and ``Brownian'' cumulants.
To estimate initial cumulants, we appeal to the machinery that we developed in~\cite{MD-21} based on so-called Glauber calculus; see Section~\ref{subsec:glauber}.
In order to estimate Brownian cumulants, on the other hand, we might try to appeal similarly to Malliavin calculus in the form of~\cite{Nourdin_2010}, but this would not easily combine with ergodic properties of the linearized mean-field equation to deduce uniform-in-time estimates.
Instead, we draw inspiration from the recent literature on mean-field games using the master equation formalism and the so-called Lions calculus on the space of probability measures, cf.~\cite{Carmona_2018,Cardaliaguet_2019,Chassagneux_2019,Chassagneux_2022}. In a nutshell, the idea is to consider the mean-field semigroup induced on functionals $\mu\mapsto\Phi(\mu)$ on the space of probability measures, and then use Lions calculus on that space to expand the Brownian expectation $\E_B[\Phi(\mu_t^N)]$ of functionals along the particle dynamics; see Lemma~\ref{lem:CST0}.
{Such expansions can be derived explicitly and be pursued to arbitrary order; see also~\cite{Tse_2021}. In addition, as noted by Delarue and Tse~\cite{Delarue_Tse_21}, they}
can be combined with ergodic properties of the linearized mean-field equation to obtain uniform-in-time estimates.
Yet, this does not immediately lead to the desired cumulant estimates $G_t^{m,N}=O(N^{1-m})$: for that purpose, we further need to capture underlying cancellations, which we achieve by developing new diagrammatic techniques in form of so-called Lions graphs; see Section~\ref{sec:brownian}.

While we appeal to recent developments on mean-field games for the expansion of $\E[\Phi(\mu^N_t)]$ along the particle dynamics,
let us emphasize that these techniques are in the same spirit as the functional framework originally developed by Mischler and Mouhot in the context of Kac's program, interacting diffusions, and jump processes with thermal bath~\cite{Mischler_2012,Mischler_2013}, where propagation of chaos was established by comparing the particle dynamics to the mean-field semigroup induced on the space of measures, and where (for the Kac particle system) uniform-in-time estimates were similarly deduced from ergodic properties of the linearized mean-field semigroup.
In the context of mean-field games~\cite{Carmona_2018,Cardaliaguet_2019}, it has appeared crucial to consider instead the corresponding mean-field semigroup induced on {\it functionals on} the space of measures, and this more nonlinear setting was conveniently exploited by Delarue and Tse~\cite{Delarue_Tse_21} to prove uniform-in-time chaos estimates~\eqref{eq:size-chaos-bis} for the Brownian dynamics~\eqref{eq:Brownian}. In our setting, even though we only consider cumulants of {\it linear} functionals of the empirical measure, higher-order expansions along the particle dynamics lead us naturally to turn to the same convenient nonlinear framework; see in particular Proposition~\ref{prop:decomp-exp-B}.

When we appeal to the ergodic properties of the linearized mean-field equation to deduce uniform-in-time estimates, some special care is actually needed.
Indeed, while ergodic estimates follow from the standard parabolic theory in the case of the Brownian dynamics, cf.~\cite{Delarue_Tse_21}, we have to further appeal to hypocoercivity techniques in the kinetic Langevin setting. For ergodic estimates on the weighted space $\Ld^2(M^{-1/2})$, where the weight is given by the mean-field steady state $M$, we can simply appeal to hypocoercivity in form of the theory of Dolbeault, Mouhot, and Schmeiser~\cite{Dolbeault_Hypocoercivity_2015}. For our estimations of cumulants, we rather need ergodic estimates on negative Sobolev spaces, and we easily check that the estimates on $\Ld^2(M^{-1/2})$ can be upgraded on $H^{-k}(M^{-1/2})$ for all $k\ge0$.
Yet, we would ideally rather need ergodic estimates on the larger space $W^{-k,1}(\Dd^d)$. Unfortunately, even the enlargement theory of Gualdani, Mischler, and Mouhot~\cite{Gualdani_2017, Mischler_2016} does not allow to reach such spaces. In Section~\ref{sec:ergodic}, we revisit enlargement techniques and show that we can actually reach $W^{-k,q}(\langle z\rangle^p)$ with arbitrarily small $q>1$ and $p>0$ provided $pq'\gg1$, which is of independent interest and is just enough for our purposes.

Finally, in order to deduce the quantitative CLT and the concentration estimates stated in Theorems~\ref{th:CLT} and~\ref{th:concentr}, we combine the same Lions expansions with Stein's method and with Herbst's argument, respectively.
{We believe this combination of techniques to be of independent interest for applications to other settings, and possibly for mean-field games.}
Note that the proof of Theorems~\ref{th:CLT} and~\ref{th:concentr} is actually much simpler than the proof of cumulant estimates in Theorem~\ref{thm:main} as it does not require to capture arbitrarily fine cancellations. For the quantitative CLT, for instance, the proof essentially boils down to the convergence of the variance and to the smallness of the third cumulant of the empirical measure, thus requiring no fine information on higher cumulants.

\subsubsection*{Plan of the paper}
We start in Section~\ref{sec:prelim} with the presentation of the main technical tools that are used to prove our main results, namely Lions and Glauber calculus, as well as ergodic estimates for the linearized mean-field equation.
In Section~\ref{sec:brownian}, we develop suitable diagrammatic representations for iterated Lions expansions of Brownian cumulants of the empirical measure, which allows us to systematically capture the needed cancellations.
The correlation estimates of Theorem~\ref{thm:main} are concluded in Section~\ref{sec:refined_chaos}, as well as Corollary~\ref{cor:bogo}. The quantitative CLT of Theorem~\ref{th:CLT} is then proved in Section~\ref{sec:clt} and the concentration estimates of Theorem~\ref{th:concentr} are established in Section~\ref{sec:concentration}. 
Finally, Section~\ref{sec:ergodic} is devoted to the proof of our new ergodic estimates stated in Section~\ref{sec:prelim}. In Appendix~\ref{sec:app-pr}, we include proof details for other preliminary results from Section~\ref{sec:prelim} when necessary.

\subsection{Notation}\label{sec:notation}
For notational convenience, we consider a general framework that covers both the Langevin and the Brownian dynamics~\eqref{eq:Langevin} and~\eqref{eq:Brownian} as special cases.
More precisely, we denote by $\{Z_t^{i,N}\}_{1\le i\le N}$ the set of particle trajectories in the space $\Xd:=\Dd^d$ or $\R^d$, as given by the following system of coupled SDEs: for $1\le i\le N$,
\begin{equation}\label{eq:system-abs}
\left\{\begin{array}{ll}
\ddr Z^{i,N}_t =b(Z^{i,N}_t, \mu^N_t)\ddr t + \sigma_0 \ddr B^{i}_t, \qquad t \ge 0, \\[1mm]
Z^{i,N}_t|_{t=0}=Z_\circ^{i,N},
\end{array}\right.
\end{equation}
where $\mu_t^N$ stands for the empirical measure
\[\textstyle\mu_t^N \,:=\, \frac1{N} \sum_{i=1}^N \delta_{Z_t^{i,N}}~~\in~\Pc(\Xd),\]
where $b:\Xd\times\Pc(\Xd)\to\R^d$ is a smooth functional (in a sense that will be made clear later on), where $\{B^i\}_i$ are i.i.d.\@ Brownian motions in $\Xd$, and where $\sigma_0$ is a constant matrix.
We assume that initial data $\{Z_\circ^{i,N}\}_{1\le i\le N}$ are i.i.d.\@ with law $\mu_\circ\in\Pc(\Xd)$, and that the latter has streched exponential moments~\eqref{eq:mom-mu0}.
The associated mean-field equation takes form of the following McKean--Vlasov equation,
\begin{equation}\label{eq:MFL-McKean}
\left\{\begin{array}{ll}
\partial_t\mu+\Div(b(\cdot,\mu)\,\mu)=\tfrac12\Div(a_0\nabla \mu),&\text{in $\R^+\times\Xd$},\\
\mu|_{t=0}=\mu_\circ,&\text{in $\Xd$},
\end{array}\right.
\end{equation}
with $a_0:=\sigma_0\sigma_0^T$, and we denote the well-posed solution operator on $\Pc(\Xd)$ by
\begin{equation}\label{eq:MF-solop}
\mu_t\,:=\,m(t;\mu_\circ).
\end{equation}
This general framework allows us to consider both systems of interest~\eqref{eq:Langevin} and~\eqref{eq:Brownian} at once:
the Langevin dynamics~\eqref{eq:Langevin} is given by
\begin{equation}\label{eq:Langevin-par}
\Xd=\Dd^d,\quad b((x,v),\mu)=\big(v,-\tfrac\beta2v- (\nabla A+\kappa\nabla W\ast\mu)(x)\big),\quad\sigma_0=\begin{pmatrix}0_{\R^d}&0_{\R^d}\\0_{\R^d}&\Id_{\R^d}\end{pmatrix},
\end{equation}
and the Brownian dynamics~\eqref{eq:Brownian} by
\begin{equation}\label{eq:Brownian-par}
\Xd=\R^d,\quad b(x,\mu)=- (\nabla A+\kappa\nabla W\ast\mu)(x) ,\quad\sigma_0=\Id_{\R^d}.
\end{equation}
Note that the diffusion matrix $a_0=\sigma_0\sigma_0^T$ is degenerate in the Langevin case, which is why specific hypocoercivity techniques are then needed.
Most of our work can actually be performed in the general framework~\eqref{eq:system-abs} without any structural assumption on $\Xd,b,\sigma_0$, except when establishing ergodic estimates in Section~\ref{sec:ergodic}. More precisely, our different main results hold for any system of the form~\eqref{eq:system-abs}, under suitable smoothness assumptions for $b$, provided that the ergodic estimates of Theorem~\ref{thm:ergodic} are available. For the latter, we restrict to the setting of the Langevin or Brownian dynamics in the weak coupling regime $\kappa\ll1$. Under mere smoothness assumptions on $b$, if ergodic estimates are not available, we note that our analysis can at least be repeated to obtain non-uniform estimates with exponential time growth.

\medskip\noindent
Finally, let us briefly list the main notation used throughout this work:
\begin{enumerate}[---]
\item We denote by $C\ge1$ any constant that only depends on the space dimension $d$. We use the notation~$\lesssim$ (resp.~$\gtrsim$) for~$\le C\times$ (resp.~$\ge\frac1C\times$) up to such a multiplicative constant~$C$. We write~$\simeq$ when both~$\lesssim$ and~$\gtrsim$ hold. We add subscripts to $C,\lesssim,\gtrsim,\simeq$ to indicate dependency on other parameters.
\smallskip\item The underlying probability space $(\Omega, \mathbb{P})$ splits as a product $(\Omega,\Pm)=(\Omega_\circ,\Pm_\circ)\times(\Omega_B,\Pm_B)$, where the first factor accounts for random initial data and where the second factor accounts for Brownian forces. The space $\Omega^\circ$ is endowed with the $\sigma$-algebra $\mathcal{F}_\circ = \sigma(Z^{1,N}_\circ, \dots, Z^{N,N}_\circ)$ generated by initial data, while $\Omega_B$ is endowed with the $\sigma$-algebra $\sigma((B^1_t, \dots, B^N_t)_{t\ge0})$. We also denote by $\mathcal{F}^B_t := \sigma((B^1_s, \dots, B^N_s)_{0 \le s \le t})$ the Brownian filtration.
We use $\E[\cdot]$ and $\kappa^m[\cdot]$ to denote the expectation and the cumulant of order $m$ with respect to $\Pm$, and we similarly denote by $\E_\circ[\cdot],\kappa^m_\circ[\cdot]$ and by $\E_B[\cdot],\kappa^m_B[\cdot]$ the expectation and cumulants with respect to $\Pm_\circ$ and $\Pm_B$, respectively.
\smallskip\item For any two integers $b\ge a\ge0$, we use the short-hand notation $\llbracket a,b \rrbracket:=\{a, a+1,\dots, b\}$, and in addition for any integer $a\ge1$ we set $\llbracket a \rrbracket := \llbracket 1, a \rrbracket$.
\smallskip\item For all $z \in \R^d$, we use the notation $\langle z \rangle := (1 + |z|^2)^{\frac12}$. 
\smallskip\item For a measure $\mu \in \calP(\Xd)$, we use the following short-hand notation for its second moment,
\begin{equation}\label{eq:QE-not}
Q(\mu)\,:=\,\Big(\int\langle z\rangle^2\mu(\ddr z)\Big)^{\frac12}.
\end{equation}
\end{enumerate}

%%%%%%%%%%%%%%%%%%%%%%%%%%%%%%%%%%%%%%%%%%%%
%%%%%%%%%%%%%%%%%%%%%%%%%%%%%%%%%%%%%%%%%%%%
%%%%%%%%%%%%%%%%%%%%%%%%%%%%%%%%%%%%%%%%%%%%

\section{Preliminary}
\label{sec:prelim}
This section is devoted to the presentation and development of the main technical tools used in this work. First, we include several probabilistic tools: we start with an account of the master equation formalism and of Lions calculus for functionals on the space of probability measures, then we turn to the relation between correlation functions and cumulants of the empirical measure, and we also recall useful tools from Glauber calculus.
Next, we present two ingredients from kinetic theory: moment estimates for the particle dynamics and, most importantly, new ergodic estimates for the linearized mean-field equation in appropriate weighted Sobolev spaces. The latter is 
the key ingredient for our uniform-in-time estimates; its proof
requires quite some developments and is the topic of Section~\ref{sec:ergodic}.

\subsection{Lions calculus}
\label{subsec:linear_derivatives}
We recall several notions of derivatives for continuous functionals on the space $\Pc(\Xd)$ of probability measures (endowed with weak convergence), and how they can be used to expand functionals along the particle dynamics.

\subsubsection{Linear derivative}
We start with the notion of \emph{linear derivative}, as used for instance by Lions in his course at Coll\`ege de France~\cite{Cardaliaguet_2013}; see also~\cite[Chapter 5]{Carmona_2018} for a slightly different exposition.
A functional $\Vc : \Pc(\Xd) \to \R$ is said to be continuously differentiable if there exists a continuous map $\Lind{\Vc}:\Pc(\Xd) \times \Xd\to\R$ such that, for all $\mu_0, \mu_1 \in \Pc(\Xd)$, 
\begin{align}
	\label{eq:def_lind}
	\Vc(\mu_0) - \Vc(\mu_1) \,=\, \int_0^1 \int_{\Xd} \Lind{\Vc}\big(s\mu_0 + (1-s)\mu_1, y\big) \,(\mu_0 - \mu_1)(\ddr y) \,\ddr s,
\end{align}
and we then call $\Lind{\Vc}$ the linear functional derivative of~$\Vc$. This definition holds up to a constant, which we fix by setting
\[ \int_{\Xd} \Lind{\Vc}(\mu,y) \,\mu(\ddr y) \,=\, 0,\qquad\text{for all $\mu\in\Pc(\Xd)$.}\]
The denomination ``linear derivative'' is understood as it is precisely defined to satisfy for all $\mu\in\Pc(\Xd)$ and~$y\in\Xd$,
\begin{equation}\label{eq:infinitesimal_linear_der}
\lim \limits_{h \to 0} \frac{\Vc((1-h)\mu + h \delta_y) - \Vc(\mu) }{h} \,=\, \Lind{\Vc}(\mu,y).
\end{equation}
Higher-order linear derivatives are defined by induction: for all integers $p \ge 1$, if the functional~$\Vc$ is $p$-times continuously differentiable, we say that it is $(p+1)$-times continuously differentiable if there exists a continuous map $\Lindk{\Vc}{p+1}:\Pc(\Xd)\times\Xd^{p+1}\to\R$ such that for all
$\mu, \mu'$ in $\Pc(\Xd)$ and~$y \in\Xd^{p}$,
\begin{align*}
\Lindk{\Vc}{p}(\mu,y) - \Lindk{\Vc}{p}(\mu',y) \,=\, \int_0^1 \int_{\Xd} \Lindk{\Vc}{p+1}\big((s\mu + (1-s)\mu', y, y'\big) \,(\mu - \mu')(\ddr y')\, \ddr s.
\end{align*}
Once again, to ensure the uniqueness of the $(p+1)$th linear functional derivative $\Lindk{\Vc}{p+1}$, we choose the convention
\begin{align*}
\int_{\Xd} \Lindk{\Vc}{p+1}(\mu,y_1,\dots,y_{p+1}) \,\mu(\ddr y_{p+1}) \,=\, 0, \qquad \text{for all $\mu\in\Pc(\Xd)$ and $y_1,\dots,y_{p} \in \Xd$.} 
\end{align*}

\subsubsection{L-derivative}
We further recall the notion of so-called \emph{L-derivatives} (or Lions derivatives, or intrinsic derivatives), as developed in~\cite{Lions_2014}. We refer e.g.\@ to~\cite[Section~2.2]{Cardaliaguet_2019} for the link to the Otto calculus on Wasserstein space~\cite{Otto-01,Ambrosio-Gigli-Savare-05}. For a continuously differentiable functional $\Vc:\calP(\Xd)\to\R$, if the map $y \mapsto \Lind{\Vc}(\mu,y)$ is of class $C^1$ on $\Xd$, the L-derivative of $\Vc$ is defined as 
\begin{align}\label{eq:def_L_derivative}
\partial_\mu\Vc(\mu,y) \,:=\, \nabla_y \Lind{\Vc}(\mu,y). 
\end{align}
We also define corresponding higher-order derivatives: for all $\mu \in \calP(\Xd)$ and $y_1, \ldots, y_p \in \Xd$, we define, provided that it makes sense,
\begin{align*}
\partial^p_\mu \Vc(\mu,y_1,\dots,y_p)
\,:=\, \nabla_{y_1} \ldots \nabla_{y_p} \Lindk{\Vc}{p}\big(\mu, y_1,\dots, y_p\big). 
\end{align*}

\subsubsection{Master equation formalism}
In terms of the above calculus on the space $\Pc(\Xd)$ of probability measures, we now introduce the so-called {master equation formalism} to describe the evolution of functionals on $\Pc(\Xd)$ along the mean-field flow.
For a smooth functional $\Phi:\Pc(\Xd) \to \R$, we define
\begin{align}\label{eq:def-Utmu-0intr}
U(t,\mu) \,:=\, P_t \Phi(\mu)\,:=\, \Phi(m(t,\mu)), \qquad t \ge 0, \quad \mu \in \Pc(\Xd), 
\end{align}
where we recall the notation $m(t,\mu)$ for the mean-field solution operator~\eqref{eq:MF-solop}.
This defines a semigroup $(P_t)_{t \ge 0}$ acting on bounded measurable functionals on $\Pc(\Xd)$.
From~\cite[Theorem 7.2]{Buckdahn_2017}, using the regularity of $b$, and assuming corresponding regularity of $\Phi$, we find that $U(t,\mu)$ satisfies the following master equation, which is viewed as an evolution equation for functionals on $\Pc(\Xd)$,
\begin{equation}\label{eq:master_eq}
\left\{
\begin{array}{l}
\partial_t U(t,\mu) = \int_{\Xd} \Big[b(x,\mu) \cdot \nabla_x \Lind{U}(t,\mu,x) + \frac12 a_0 : \nabla^2_x \Lind{U}(t,\mu,x)\Big] \mu(\ddr x),\\[2mm]
U(0,\mu) = \Phi(\mu),
\end{array}\right.
\end{equation}
where we recall $a_0=\sigma_0\sigma_0^T$.

\subsubsection{Expansions along the particle dynamics}
We recall the following useful result that allows to expand functionals along the particle dynamics in terms of the corresponding mean-field flow, cf.~\cite[(5.131)]{Carmona_2018} or~\cite[Lemma~2.8]{Chassagneux_2019}. (Note that the proof in~\cite{Chassagneux_2019} only relies on the master equation~\eqref{eq:master_eq} and on~\cite[Proposition 3.1]{Chassagneux_2022}, so that in particular there is no uniform ellipticity requirement for the diffusivity $a_0=\sigma_0\sigma_0^T$.)

\begin{lem}[see~\cite{Carmona_2018,Chassagneux_2019}]
\label{lem:CST0}
Let $\Phi: \calP(\Xd) \to \R$ be a smooth functional and let $U(t,\mu)$ be defined in~\eqref{eq:def-Utmu-0intr}.
Then for all $0\le s\le t$ we have
\begin{equation}\label{eq:Lions-expand-00}
U(t-s,\mu^N_s) \,=\, U(t,\mu^N_0) \,+\, \frac{1}{2N} \int_0^s \int_{\Xd} \Tr \Big[ a_0\, \partial^2_\mu U(t-u,\mu^N_u) (z,z) \Big] \, \mu^N_u(\ddr z) \, \ddr u \,+\, M^N_{t,s},
\end{equation}
where $(M^N_{t,s})_{s\ge0}$ is a {square-integrable} $(\mathcal{F}^B_s)_{s \ge 0}$-martingale with $M^N_{t,0} = 0$, which is explicitly given by
\begin{equation*}
M^N_{t,s}\,:=\,\frac1N\sum_{i=1}^N\int_0^{s \wedge t} \partial_\mu U(t-u,\mu_u^N)(Z^{i,N}_u)\cdot\sigma_0\,\ddr B^i_u.
\end{equation*}
\end{lem}

This expansion will be used throughout this work to compare the empirical measure to the corresponding mean-field semigroup. More precisely, we shall abundantly use the following immediate consequences.

\begin{cor}\label{cor:CST}
Let $\Phi: \calP(\Xd) \to \R$ be a smooth functional and let $U(t,\mu)$ be defined in~\eqref{eq:def-Utmu-0intr}.
\begin{enumerate}[(i)]
\item For all $t\ge0$, we have
\begin{align*}
\qquad\big|\E[\Phi(\mu_t^N)]-\E_\circ[\Phi(m(t,\mu^N_0))]\big|\,\lesssim\,N^{-1}\,\E\bigg[\int_0^t \int_{\Xd}\big|\partial^2_\mu U(t-u,\mu_u^N) (z,z)\big|\, \mu^N_u(\ddr z) \, \ddr u\bigg].
\end{align*}
\item For all $t\ge0$, we have
\begin{multline*}
\qquad\|\Phi(\mu_t^N)-\Phi(m(t,\mu_0^N))\|_{\Ld^2(\Omega_B)}\,\lesssim\,
N^{-\frac12}\,\E_B \bigg[  \int_0^t\int_\Xd\big|\partial_\mu U(t-u,\mu^N_u)(z) \big|^2 \, \mu_u^N(\ddr z)\,\ddr u \bigg]^\frac12\\
+N^{-1}\,\E_B \bigg[ \Big( \int_0^t\int_\Xd \big|\partial_\mu^2U(t-u,\mu^N_u)(z,z)\big| \, \mu_u^N(\ddr z) \, \ddr u \Big)^2 \bigg].
\end{multline*}
\end{enumerate}
\end{cor}

\begin{proof}
Taking the expectation $\E=\E_\circ\E_B$ in~\eqref{eq:Lions-expand-00}, using $\E_B[M_{t,s}^N]=0$, and setting $s=t$, we are led in particular to the following expansion for the expectation of a functional of the empirical measure,
\begin{align*}
\E[\Phi(\mu_t^N)]\,=\,\E_\circ\big[\Phi(m(t,\mu^N_0))\big] \,+\,\frac{1}{2N} \int_0^t \E\bigg[\int_{\Xd}\Tr \Big[ a_0\, \partial^2_\mu U(t-u,\mu^N_u) (z,z) \Big] \, \mu^N_u(\ddr z) \bigg] \ddr u,
\end{align*}
and item~(i) immediately follows.
Next, taking the $\Ld^2(\Omega_B)$ norm in~\eqref{eq:Lions-expand-00}, noting that Jensen's inequality yields
\[\qquad\E_B[(M_{t,s}^N)^2]\,\le\,N^{-1} \E_B \bigg[\int_0^{s\wedge t}\int_\Xd\big|\sigma_0^T(\partial_\mu U)(t-u,\mu_u^N)(z)\big|^2\,\mu_u^N(\ddr z)\,\ddr u \bigg],\]
and setting $s=t$, we similarly obtain item~(ii).
\end{proof}

Due to the above result, as emphasized in~\cite{Carmona_2018,Chassagneux_2019,Delarue_Tse_21}, Lions calculus provides a natural starting point to study propagation of chaos, which was indeed successfully used in particular in~\cite{Delarue_Tse_21} to establish uniform-in-time weak chaos estimates for the Brownian dynamics.
More precisely, in order to obtain weak chaos estimates of the form~\eqref{eq:size-chaos-bis}, that is,
\[\E[\Phi(\mu_t^N)]-\Phi(m(t,\mu_\circ))\,=\,O(N^{-1}),\]
we can appeal to item~(i) above and it remains to compare $\E_\circ[\Phi(m(t,\mu^N_0))]$ to $\Phi(m(t,\mu_\circ))$. The missing estimate is provided by the following general result due to~\cite[Theorem~2.11]{Chassagneux_2019}.
We emphasize that it provides some $O(N^{-1})$ convergence rate in the law of large numbers for the empirical measure associated with i.i.d.\@ data in some weak topology:
this contrasts with the much weaker rates that are sharply obtained for the expectation of Wasserstein distances e.g.\@ in~\cite{Fournier_15, Fournier_2023}. 

\begin{lem}[see~\cite{Chassagneux_2019}]
\label{lem:chassagneux}
For any smooth functional $\Phi:\Pc(\Xd)\to\R$, we have
\[\E[\Phi(\mu_0^N)]-\Phi(\mu_\circ)\,=\,\frac1N\int_0^1\int_0^1\int_\Xd\E\bigg[\frac{\delta^2\Phi}{\delta\mu^2}(\tilde\mu^N_{s,u,z},z,z)-\frac{\delta^2\Phi}{\delta\mu^2}(\tilde\mu^N_{s,u,z},z,Z_\circ^{1,N})\bigg]\,\mu_\circ(\ddr z)\,\ddr u\, s\,\ddr s,\]
in terms of
$\tilde\mu_{s,u,z}^N:=\frac{su}N(\delta_{z}-\delta_{Z_\circ^{1,N}})+\mu_\circ+s(\mu_0^N-\mu_\circ)$.
In particular, this implies
\[|\E[\Phi(\mu_0^N)]-\Phi(\mu_\circ)|\,\le\, 2N^{-1}\sup_{\mu\in\Pc(\Xd)}\Big\|\frac{\delta^2\Phi}{\delta\mu^2}\Big\|_{\Ld^\infty(\Xd^2)}.\]
\end{lem}

\subsection{Correlation functions and cumulants}
In order to estimate the many-particle correlation functions $\{G^{k,N}\}_{1\le k\le N}$ defined in~\eqref{eq:def-cumGm}, we shall proceed by estimating the cumulants of the empirical measure, which have a more exploitable probabilistic content.
We recall that the $m$th cumulant of a bounded random variable $X$ is defined by
\[ \kappa^m[X] := \Big( (\tfrac{\ddr}{\ddr t})^m \log \E \big[ e^{tX} \big] \Big)\Big|_{t = 0}, \]
that is,
\begin{align}
	\kappa^1[X] &= \E[X], \nonumber\\
	\kappa^2[X] &= \E[X^2] - \E[X]^2 = \Var[X], \nonumber\\
	\kappa^3[X] &= \E[X^3] - 3\E[X^2]\E[X] + 2 \E[X]^3, \nonumber\\
	\kappa^4[X] &= \E[X^4] - 4\E[X^3]\E[X] - 3\E[X^2]^2 + 12 \E[X^2]\E[X]^2 - 6 \E[X]^4,\label{eq:cumfrommom-inv0}
\end{align}
and so on {(see~\eqref{eq:cumfrommom-inv} for a general formula).}
We also define the joint cumulant of a family of bounded random variables $X_1,\ldots,X_m$ as
\begin{align*}
\kappa^m[X_1, \ldots, X_m]\, :=\,\Big( \tfrac{\ddr^m}{\ddr t_1 \dots \ddr t_m} \log\expecm{e^{\sum_{j=1}^m t_j X_j}}\Big)\Big|_{t_1 = \ldots = t_m = 0}.
\end{align*}
Since we consider in this work a product probability space $(\Omega,\Pm)=(\Omega_\circ,\Pm_\circ)\times(\Omega_B,\Pm_B)$, where the first factor accounts for random initial data and the second for Brownian forces,
we appeal to the following law of total cumulance in order to split cumulants accordingly.

\begin{lem}[see~\cite{Brillinger_1969}]\label{lem:total_cumulance}
For all $m \ge 2$ and all bounded random variables $X$, we have
\begin{equation*}
\kappa^m[X]\, =\, \sum_{\pi\vdash\llbracket m \rrbracket} \kappa_\circ^{\sharp\pi} \Big[ \big(\kappa^{\sharp A}_B[X]\big)_{A\in\pi}\Big],
\end{equation*}
where we recall that $\kappa_\circ$ and $\kappa_B$ stand for cumulants with respect to $\Pm_\circ$ and $\Pm_B$, respectively, and where we use similar notation as in~\eqref{eq:cluster-exp0}.
\end{lem}

\subsubsection{Moments and cumulants}
While cumulants are defined as polynomial expressions involving moments, cf.~\eqref{eq:cumfrommom-inv0}, those relations are easily inverted: similarly as in~\eqref{eq:cluster-exp0}, moments can be recovered from cumulants in form of a cluster expansion,
\begin{equation}\label{eq:cluster-exp}
\E[X^m]\,=\,\sum_{\pi\vdash\llbracket m \rrbracket}\prod_{A\in\pi}\kappa^{\sharp A}[X].
\end{equation}
For later purposes, we state the following recurrence relation between moments and cumulants: it immediately implies the above cluster expansion by induction, and it will be useful in this form in the sequel. A short proof is included in Appendix~\ref{sec:app-pr} for convenience.

\begin{lem}\label{lem:mom-cum}
For all $m \ge 2$ and all bounded random variables $X_1, \dots, X_m$, we have
\begin{equation}
\label{eq:esp_product}
\E[X_1 \dots X_m] \,=\, \sum_{J \subset \llbracket 2,m \rrbracket} \kappa[X_1,X_J] \,\E \bigg[ ~\prod_{j \in \llbracket 2,m \rrbracket \setminus J}~ X_j \bigg],
\end{equation}
where we use the standard convention $\prod_{j\in\varnothing}X_j=1$ for the empty product.
In particular, for all $m\ge1$ and all bounded random variables $X$, we have
\[\E[X^m]\,=\,\sum_{j=1}^{m}\binom{m-1}{j-1}\kappa^{j}[X]\,\E[X^{m-j}].\]
\end{lem}

\subsubsection{From cumulants to correlations}
We work out the standard link between cumulants of the empirical measure and correlation functions.
We state it in form of an inequality that can be directly iterated to bound successive correlation functions in terms of cumulants of the empirical mesure.
This was used for instance in~\cite[Section~4]{MD-21}, but we provide a self-contained statement, as well as a short proof in Appendix~\ref{sec:app-pr} for convenience.

\begin{lem}\label{lem:cumtocorrel}
For all $1\le m\le N$, there exists $C_m > 0$ such that for all $\phi \in C^\infty_c(\Xd)$ and $t \ge 0$, we have
\begin{multline*}
\bigg|\int_{\Xd^{m}}\phi^{\otimes m}G^{m,N}_t \bigg|\,\le\, \bigg|\kappa^m\bigg[\int_\Xd\phi\,\ddr\mu^N_t \bigg]\bigg|\\
+C_m\sum_{\substack {\pi\vdash\llbracket m\rrbracket \\ \sharp\pi<m}} \sum_{\rho\vdash\pi} N^{\sharp\pi-\sharp\rho-m+1}\bigg|\int_{\Xd^{\sharp\pi}}\Big(\bigotimes_{B\in\pi}\phi^{\sharp B}\Big)\Big(\bigotimes_{D\in\rho}G^{\sharp D,N}_t(z_D)\Big)\,\ddr z_\pi\bigg|.
\end{multline*}
\end{lem}

\subsection{Glauber calculus}\label{subsec:glauber}
We recall some useful tools from the so-called Glauber calculus on $(\Omega_\circ,\Pm_\circ)$ as developed by the second-named author in~\cite{MD-21} (see also~\cite{Decreusefond-Halconruy-19,DGO}).
Given a random variable $X\in\Ld^2(\Omega_\circ)$,
we then define its Glauber derivative at $j\in\llbracket N \rrbracket$ as
\begin{equation*}
D^\circ_j X \,:=\,X-\E\big[X\big|(Z_\circ^{i,N})_{i:i\ne j}\big].
\end{equation*}
The full gradient $D^\circ X = (D^\circ_j X)_{j\in\llbracket N \rrbracket}$ is viewed as an element of $\ell^2(\llbracket N \rrbracket;\Ld^2(\Omega_\circ))$.
As the initial data $(Z_\circ^{j,N})_{1\le j\le N}$ are i.i.d.,
a straightforward computation shows that $D^\circ_j$ is self-adjoint on $\Ld^2(\Omega_\circ)$ and satisfies
\[ D^\circ_jD^\circ_j = D^\circ_j, \qquad D^\circ_jD^\circ_k = D^\circ_k D^\circ_j, \qquad \text{for all $j,k\in\llbracket N \rrbracket$}.\]
We then define the Glauber Laplacian
\begin{align*}
\mathcal{L}_\circ \,:=\,(D^\circ)^*D^\circ\,=\, \sum_{j=1}^N (D^\circ_j)^*  D^\circ_j  \,=\, \sum_{j=1}^N D^\circ_j,
\end{align*}
which is a nonnegative self-adjoint operator on $\Ld^2(\Omega_\circ)$.
We recall some fundamental properties of this operator; see~\cite[Lemmas~2.5 and~2.6]{MD-21}.

\begin{lem}[see~\cite{MD-21}]\label{lem:L0-prop}\
\begin{enumerate}[(i)]
\item The kernel of $\calL_\circ$ is reduced to constants, $\ker \calL_\circ = \R$.
Moreover, $\calL_\circ$ has a unit spectral gap above~$0$, and its spectrum is the set $\N$. 
\smallskip\item The restriction of $\calL_\circ$ to $(\ker \calL_\circ)^\bot = \{X \in \Ld^2(\Omega_\circ): \E_\circ[X] = 0\}$ admits a well-defined inverse~$\Lc_\circ^{-1}$, which is a nonnegative self-adjoint contraction on $(\ker \calL_\circ)^\bot$. 
Moreover, this inverse operator satisfies for all $1<p<\infty$ and $X \in \Ld^p(\Omega_\circ)$ with $\E_\circ[X] = 0$,
\begin{align}
\label{eq:ineq_T}
\| \Lc_\circ^{-1} X\|_{\Ld^p(\Omega_\circ)} \,\lesssim\, \tfrac{p^2}{p-1} \|X\|_{\Ld^p(\Omega_\circ)}.
\end{align} 
\item The following Helffer-Sj\"ostrand representation holds for covariances: for all $X, Y\in\Ld^2(\Omega_\circ)$,
\begin{equation}\label{eq:HS-rep}
\Cov_\circ[X, Y] \,=\, \sum_{j=1}^N \mathbb{E}\big[ (D^\circ_j X) \Lc_\circ^{-1} (D^\circ_j Y) \big]. 
\end{equation}
\end{enumerate}
\end{lem}

Combining the spectral gap for $\Lc_\circ$ and the Helffer--Sj\"ostrand inequality~\eqref{eq:HS-rep}, we recover in particular the following well-known variance inequality due to Efron and Stein~\cite{Efron-Stein-81}: for all $X\in\Ld^2(\Omega_\circ)$,
\begin{equation}\label{eq:Poinc-Glauber}
\Var_\circ[X]\,\le\,\sum_{j=1}^N\E_\circ[|D^\circ_j X|^2].
\end{equation}

\subsubsection{Cumulant estimates via Glauber calculus}
As shown in~\cite[Theorem~2.2]{MD-21},
cumulants can be controlled in terms of higher-order Glauber derivatives, which can be viewed as a higher-order version of Poincar\'e's inequality~\eqref{eq:Poinc-Glauber} on~$\Ld^2(\Omega_\circ)$ with respect to Glauber calculus.
We claim that this can also be extended to the multivariate case, that is, to joint cumulants of families of random variables;
a proof is included in Appendix~\ref{sec:app-pr}.

\begin{prop}\label{prop:control_Glauber}
For all $n\ge0$ and bounded $\sigma((Z_\circ^{j,N})_{1\le j\le N})$-measurable random variables $X_1,\ldots,X_{n+1}$, we have
\begin{align*}
\kappa_\circ^{n+1}[X_1,\ldots,X_{n+1}] \,\lesssim_n\, \sum_{k=0}^{n-1} N^{k+1} \sum_{\substack{a_1, \dots, a_{n+1} \ge 1 \\ \sum_j a_j = n+k+1}} \prod_{j=1}^{n+1} \big\| (D^\circ)^{a_j} X_j \big\|_{\ell^{\infty}_{\ne}\big( \Ld^{\frac1{a_j} (n+k+1)}(\Omega_\circ) \big)},
\end{align*}
where we have set
\[ \big\| (D^\circ)^m Z\|_{\ell^\infty_{\ne}(\Ld^p(\Omega_\circ))} \,:=\, \sup_{\substack{j_1 \dots j_m \\ \text{distinct}}} \big\|D^\circ_{j_1} \dots D^\circ_{j_m} Z \big\|_{\Ld^p(\Omega_\circ)}.\]
\end{prop}

\subsubsection{Asymptotic normality via Glauber calculus}
As the approximate normality of a random variable essentially follows from the smallness of its cumulants of order $\ge3$, there is no surprise that it can be quantified as well by means of Glauber calculus. The following result is typically known in the literature as a ``second-order Poincar\'e inequality'' for approximate normality. It was first established by Chatterjee~\cite[Theorem~2.2]{Chat08} based on Stein's method for the $1$-Wasserstein distance, while the corresponding bound on the Kolmogorov distance is due to~\cite[Theorem~4.2]{LRP-15}. For our purposes in this work, we claim that the same result also holds for the Zolotarev distance; a short proof is included in Appendix~\ref{sec:app-pr} for convenience.

\begin{prop}[Second-order Poincar\'e inequality~\cite{Chat08,LRP-15}]\label{prop:2ndP}
For all bounded $\sigma((Z_\circ^{j,N})_j)$-measurable random variable $Y$, setting $\sigma_Y^2:=\var{Y}$, there holds
\begin{multline*}
\ddr_2\bigg(\frac{Y-\E_\circ[Y]}{\sigma_Y};\Nc\bigg)+\dW{\frac{Y-\E_\circ[Y]}{\sigma_Y}}{\Nc}+\dK{\frac{Y-\E_\circ[Y]}{\sigma_Y}}{\Nc}\\
\,\lesssim\,\frac1{\sigma_Y^{3}}\sum_{j=1}^N \E_\circ[|D_j^\circ Y|^6]^\frac12+\frac1{\sigma_Y^2}\bigg(\sum_{j=1}^N \Big(\sum_{l=1}^N\E_\circ[|D_l^\circ Y|^4]^\frac14\E_\circ[|D_j^\circ D_l^\circ Y|^4]^\frac14\Big)^2\bigg)^\frac12,
\end{multline*}
where $\dW\cdot\Nc$ and $\dK\cdot\Nc$ stand for the $1$-Wasserstein and the Kolmogorov distances\footnote{Recall that for two real-valued random variables $X,Y$ the $1$-Wasserstein and Kolmogorov distances are defined as
\[\dW XY=\sup\big\{\E[g(X)]-\E[g(Y)]:g\in C(\R),\,\operatorname{Lip}(g)\le1\big\},\qquad\dK XY=\sup_{t\in\R}|\Pm[X\le t]-\Pm[Y\le t]|.\]} to a standard Gaussian random variable $\Nc$, respectively, and where we recall that $\ddr_2(\cdot;\Nc)$ stands for the corresponding second-order Zolotarev distance defined in~\eqref{eq:Zolo}.
\end{prop}

\subsubsection{Concentration via Glauber calculus}
We show the following concentration estimate for random variables in $\Ld^2(\Omega_\circ)$. It follows from some degraded version of a log-Sobolev inequality, combined with Herbst's argument.
Note however that we do not have an exact log-Sobolev inequality with respect to Glauber calculus, cf.~\cite{Ledoux-99}, which is why we need to require an \emph{almost sure} a priori bound on the Glauber derivative. The proof is postponed to Appendix~\ref{sec:app-pr}.

\begin{prop}
\label{lem:concentration_Glauber}
Let $X\in\Ld^2(\Omega_\circ)$ be $\sigma((Z^{j,N}_\circ)_{1 \le j \le N})$-measurable with $\E_\circ[X]=0$ and $|D^\circ_j X|\le \frac12L$ almost surely for all $1\le j\le N$, for some constant $L>0$.
Then for all $\lambda>0$ we have
\[\E_\circ[e^{\lambda X}]\,\le\,\exp\Big(\tfrac N2\lambda L(e^{\lambda L}-1)\Big).\]
In particular, this entails
\[\Pm_\circ[X>r]\,\le\,\exp\Big(-\tfrac{r}{4L}\log\big(1+\tfrac{r}{NL}\big)\Big),\]
where the right-hand side is $\le\exp(-\frac{r^2}{8NL^2})$ as long as $r\le NL$.
\end{prop}

\subsubsection{Link to linear derivatives}
As the following lemma shows, Glauber derivatives can be estimated in terms of linear derivatives. This is particularly convenient in the sequel to unify notations when both Glauber and linear derivatives are involved.

\begin{lem}\label{lem:glauber_general}
Given a smooth functional $\Phi:\calP(\Xd) \to \R$,
we have almost surely for all $k\in \llbracket N \rrbracket$ and all distinct indices~$j_1, \ldots, j_k \in \llbracket N \rrbracket$,
\begin{equation*}
\big| D_{j_1}^\circ \dots D_{j_k}^\circ \Phi(\mu^N_0) \big| \,\le\, N^{-k}\,2^k\sup_{\mu \in \calP(\Xd)} \Big\| \Lindk{\Phi}{k}(\mu,\cdot) \Big\|_{\Ld^\infty(\Xd^k)}.
\end{equation*}
\end{lem}

\begin{proof}
For all~$j \in \llbracket N \rrbracket$, by definition of the Glauber derivative and of the linear derivative, we can compute
\begin{eqnarray}
\label{eq:decompo_Glauber_derivative}
D_{j}^\circ \Phi(\mu^N_0) &=& \Phi(\mu^N_0)-\int_{\Xd}\Phi \Big(\mu_0^N+\tfrac1N(\delta_{z}-\delta_{Z_0^{j,N}})\Big)\,\mu_\circ(\ddr z)\nonumber\\
&=&N^{-1}\int_0^1 \int_{\Xd} \int_{\Xd} \Lind{\Phi}\Big(\mu_0^N+\tfrac{1-s}N(\delta_{z}-\delta_{Z_0^{j,N}})\,,\,y\Big)\,(\delta_{Z_0^{j,N}}-\delta_{z})(\ddr y)\,\mu_\circ(\ddr z)\,\ddr s.\label{eq:glauber_general-rel}
\end{eqnarray}
By induction, we are led to the following representation formula for iterated Glauber derivatives: for all $k\ge1$ and all distinct indices $j_1,\ldots,j_k\in \llbracket N\rrbracket$,
\begin{multline}\label{eq:multi_d_Glauber}
D_{j_1}^\circ\ldots D_{j_k}^\circ \Phi(\mu^N_0)
 \,=\,N^{-k}  \int_{([0,1] \times \Xd\times\Xd)^k} \Lindk{\Phi}{k} \bigg(\mu_0^N+\sum_{i=1}^k\tfrac{1-s_i}N(\delta_{z_i}-\delta_{Z_0^{j_i,N}})\,,\, y_1, \ldots, y_k \bigg)\\[-2mm]
\times  \prod_{i=1}^k (\delta_{Z^{j_i,N}_0} - \delta_{z_i}) (\ddr y_i) \,\mu_\circ(\ddr z_i) \, \ddr s_i,
\end{multline}
and the conclusion follows.
\end{proof}

\subsection{Moment estimates on particle dynamics}
The following uniform-in-time moment estimates are stated both for the particle dynamics and for its mean-field counterpart.
The proof easily follows along the lines of the analysis in~\cite{Bolley_2010} and is postponed to Appendix~\ref{sec:app-pr}. Recall the notation $Q$ from~\eqref{eq:QE-not}. 

\begin{lem}\label{lem:unif-mom-est}
There exist $\lambda,C>0$ (only depending on $d,\beta,a$, $\|\nabla W\|_{\Ld^{\infty}(\R^d)}$) such that the following moment estimates hold.
\begin{enumerate}[(i)]
\item \emph{Algebraic moments:}
For all $t\ge0$, $N,k\ge1$, and $\mu\in\Pc(\Xd)$, we have
\begin{eqnarray}
\E_B\big[Q(\mu_t^N)^k\big]&\le&(Ck)^\frac k2(1+\tfrac kN)^\frac{k}{2}\Big(\int_\Xd\langle e^{-\lambda t}z\rangle^{2}\,\mu^N_0(\ddr z)\Big)^\frac k2,
\label{eq:estim-mom-Brownian}\\
\int_\Xd|z|^{k}\,m(t,\mu)&\le&(Ck)^k\int_\Xd\langle e^{-\lambda t}z\rangle^{k}\,\mu(\ddr z).\nonumber
\end{eqnarray}
\item \emph{Subgaussian moments:}
For all $t,\alpha\ge0$, $N,k\ge1$, $\mu\in\Pc(\Xd)$, and $0\le\delta<2$, we have
\begin{eqnarray}\label{eq:estim-mom-part-exp1}
\E_B\big[e^{\alpha Q(\mu_t^N)^\delta}\big]&\le&\exp\Big(C_\delta \alpha(\tfrac{\alpha} N+1)^{\frac\delta{2-\delta}}\Big)\exp\bigg(C\alpha\Big(\int_\Xd |e^{-\lambda t}z|^2\mu^N_0(\ddr z)\Big)^{\frac\delta2}\bigg),\\
\int_\Xd e^{\alpha|z|^{\delta}}m(t,\mu)&\le&\exp\Big(C_\delta \alpha^{\frac2{2-\delta}}\Big)\int_\Xd e^{C\alpha |e^{-\lambda t}z|^{\delta}}\mu(\ddr z),\nonumber
\end{eqnarray}
for some constant $C_\delta$ only depending on $d,\beta,a,\delta$, $\|\nabla W\|_{\Ld^{\infty}(\R^d)}$.
\smallskip\item \emph{Gaussian moments:}
For all $t\ge0$, $N,k\ge1$, and $\mu\in\Pc(\Xd)$, we have for all $0\le\alpha\le\frac1{C}N$,
\begin{equation}\label{eq:estim-mom-part-exp}
\E_B\big[e^{\alpha Q(\mu_t^N)^2}\big]
~\le~C\exp\Big(C\alpha\int_\Xd|e^{-\lambda t}z|^2\mu_0^N(\ddr z)\Big),
\end{equation}
and for all $0\le\alpha\le\frac1{C}$,
\[\int_\Xd e^{\alpha|z|^2}\, m(t,\mu)~\le~C\int_\Xd e^{C\alpha |e^{-\lambda t} z|^2} \mu(\ddr z).\]
\end{enumerate}
\end{lem}

\subsection{Ergodic estimates for mean field}\label{sec:prel-ergodic}
The last main ingredient takes form of ergodic estimates for the linearized mean-field equation, which is indeed the key tool for our uniform-in-time estimates in the spirit of~\cite{Delarue_Tse_21}.
Given $\mu \in \calP(\Xd)$, the linearized mean-field McKean--Vlasov operator at $\mu$
is defined as follows: for all $h\in C^\infty_c(\Xd)$ with $\int_\Xd h=0$,
\begin{align}
\label{eq:def_linearized}
L_\mu h \,:=\, 
\tfrac12 \Div(a_0 \nabla h)  - \Div\big( b(\cdot,\mu) h \big) - \Div\Big(\mu \int_{\Xd} \Lind{b}(\cdot,\mu,z) \,h(z)\,\ddr z \Big).
\end{align}
In the Langevin setting~\eqref{eq:Langevin-par}, this means for all $h$ on $\Xd=\R^d\times\R^d$,
\begin{equation}\label{eq:def_linearized-Lang}
L_\mu h\,=\,\tfrac12\Div_v((\nabla_v+\beta v)h) -v\cdot\nabla_xh+(\nabla A+\kappa\nabla W\ast\mu)\cdot\nabla_vh+\kappa(\nabla W\ast h)\cdot\nabla_v\mu,
\end{equation}
and in the Brownian setting~\eqref{eq:Brownian-par}, this means for all $h$ on $\Xd=\R^d$,
\[L_\mu h\,=\,\tfrac12 \triangle h + \Div(h\nabla A) +\kappa\Div(h\nabla W\ast\mu)+\kappa\Div(\mu \nabla W\ast h).\]
For our purposes in this work, we shall establish ergodic estimates in a weighted Sobolev framework with arbitrary integrability, negative regularity, and polynomial weight: more precisely,
for all $1\le q\le 2$ and $p\ge0$, we consider the space $\Ld^q(\langle z\rangle^p)$ as the weighted Lebesgue space with the norm
\[\|h\|_{\Ld^q(\langle z\rangle^p)}\,:=\,\|\langle z\rangle^ph\|_{\Ld^q(\Xd)}\,=\,\Big(\int_\Xd|h(z)|^q\,\langle z\rangle^{pq}\,\ddr z\Big)^\frac1q,\]
and, for all $k\ge0$, we consider the space $W^{-k,q}(\langle z\rangle^p)$ as the weighted negative Sobolev space associated with the dual norm
\begin{equation}\label{eq:def-W-kqzp}
\|h\|_{W^{-k,q}(\langle z\rangle^p)}\,:=\,\sup\bigg\{\int_\Xd hh'\langle z\rangle^p\,:\,\|h'\|_{W^{k,q'}(\Xd)}=1\bigg\},
\end{equation}
where $q':=\frac{q}{q-1}$ is the dual integrability exponent and where $W^{k,q'}(\Xd)$ is the standard Sobolev space with norm
\[\|h\|_{W^{k,q'}(\Xd)}\,:=\,\sup_{0\le j\le k}\|\nabla^jh\|_{\Ld^{q'}(\Xd)}.\]
In these terms, our ergodic estimates take on the following guise.
While item~(i) is well known (see e.g.~\cite{Bolley_2010} and the discussion below), our main contribution is to prove the Sobolev estimates of item~(ii).
Note that the restriction $pq'\gg1$ in the Langevin setting is fairly natural: indeed, we note for example that the restriction $pq'>d$ precisely ensures that the spatial density $\rho_h(x):=\int_{\R^d} h(x,v)\,\ddr v$ is defined in~$\Ld^1_\loc(\R^d)$ for all $h\in W^{-k,q}(\langle z\rangle^p)$.

\begin{theor}\label{thm:ergodic}
There exist $\kappa_0,\lambda_0>0$ (only depending on $d,\beta,a,\|W\|_{W^{1,\infty}(\R^d)}$)
such that the following results hold for any $\kappa\in[0,\kappa_0]$.
\begin{enumerate}[(i)]
\item There is a unique steady state $M$ for the mean-field evolution~\eqref{eq:MFL-McKean}, and the solution operator~\eqref{eq:MF-solop} satisfies for all $t\ge0$,
\begin{equation}\label{eq:relax-mt-ass}
\calW_2\big(m(t,\mu_\circ), M\big) \,\lesssim_{W,\beta,a}\, e^{-\lambda_0 t}\,\calW_2(\mu_\circ, M),
\end{equation}
where $\calW_2$ stands for the $2$-Wasserstein distance and where the multiplicative constant only depends on $d,\beta,a$, and $\|W\|_{W^{1,\infty}(\R^d)}$.
\smallskip\item
Given $1<q\le2$ and $0<p\le1$ with $pq'\gg_{\beta,a}1$ large enough (only depending on $d,\beta,a$), for all $k\ge2$ and all $f_\circ\in W^{-k,q}(\langle z\rangle^p)$ with $\int_\Xd f_\circ=0$, there is a unique solution $f\in\Ld^\infty_\loc(\R^+;W^{-k,q}(\langle z\rangle^p))$ to the Cauchy problem
\begin{align}\label{eq:erg_q}
\left\{
\begin{array}{l}
\partial_t f_t = L_{m(t,\mu_\circ)}f_t,\\
f_t|_{t=0} = f_\circ, 
\end{array}
\right.
\end{align}
and it satisfies for all $t\ge0$,
\begin{equation}\label{eq:estim-concl-qt}
\|f_t\|_{W^{-k,q}(\langle z\rangle^p)}
\,\lesssim_{W,\beta,k,p,q,a,\mu_\circ}\,
e^{-p\lambda_0t}\|f_\circ\|_{W^{-k,q}(\langle z\rangle^p)},
\end{equation}
where the multiplicative constant only depends on $d,\beta,k,p,q,a,\|W\|_{W^{k+d+1,\infty}(\R^d)}$, and $Q(\mu_\circ)$ (recall the notation~\eqref{eq:QE-not}).
In addition, the dependence on $Q(\mu_\circ)$ can be made at most exponential provided that $\kappa$ is small enough. More precisely, given $\theta\in(0,1]$, for all $k\ge2$ and $0<p\le\frac14\theta$, there is some $\kappa_{k,p}\in(0,\kappa_0]$ (only depending on $d,\beta,\theta,k,p,a$, $\|W\|_{W^{d+3,\infty}\cap W^{k+2,\infty}(\R^d)}$) such that the following holds provided $\kappa\in[0,\kappa_{k,p}]$: given $1<q\le2$ with $pq'\gg_{\beta,a}1$ large enough (only depending on $d,\beta,a$), we have for all $t\ge0$,
\begin{equation}\label{eq:estim-concl-qt-re}
\|f_t\|_{W^{-k,q}(\langle z\rangle^p)}
\,\lesssim_{W,\beta,k,p,q,a}\,
e^{pQ(\mu_\circ)^\theta}e^{-p\lambda_0t}\|f_\circ\|_{W^{-k,q}(\langle z\rangle^p)},
\end{equation}
where the multiplicative constant now only depends on $d,\beta,k,p,q,a$, $\|W\|_{W^{k+d+1,\infty}(\R^d)}$.
\end{enumerate}
\end{theor}

The convergence to equilibrium stated in item~(i) is well known: it was proven for instance by Bolley, Guillin, and Malrieu~\cite{Bolley_2010} by an elementary coupling argument
(see also~\cite{Malrieu_2001, Herau_2007, Bouchut_1995} for earlier results).
For corresponding results relying on convexity rather than on smallness of the interaction, we refer to~\cite{Monmarche_2017, Guillin_2021c} in the Langevin setting, and to~\cite{Malrieu_2001,Malrieu-03,Carrillo_2003, Carrillo_2005,Cattiaux-Guillin-Malrieu-08} in the Brownian setting.
Perturbations of the strictly convex case have also been investigated e.g.\@ in~\cite{Bolley_2012,Butkovsky_2014,Eberle_2018}.

Regarding the ergodic estimates stated in item~(ii), in the Brownian setting,
they easily follow from classical parabolic theory~\cite{Friedman_2008, Garroni_1993}. In the periodic case with $A\equiv0$, such estimates can be found in~\cite[Lemma~7.4]{Cardaliaguet_2013b} on the space $\Ld^\infty(\T^d)$, and in~\cite{Delarue_Tse_21} on the space $W^{-k,1}(\T^d)$ with $0\le k<2$. Those results are easily generalized to the case of a nontrivial confinement in the whole space $\R^d$, and they can be checked to hold on the space $W^{-k,q}(\langle z\rangle^p)$ for all $k\ge1$, $1\le q\le2$, and $0\le p\le1$. We emphasize in particular that they also hold on the unweighted space $W^{-k,1}(\R^d)$ for all $k\ge1$, in which case the dependence on $Q(\mu_\circ)$ in~\eqref{eq:estim-concl-qt} can be completely lifted.
We skip the details as it is similar to~\cite{Delarue_Tse_21}.
Note that the control of higher-order correlation functions indeed requires ergodic estimates in Sobolev spaces with \emph{arbitrary} negative regularity $k\ge1$.

The main challenge is to obtain the corresponding ergodic estimates in the kinetic Langevin setting, where parabolic tools are no longer available due to hypocoercivity.
This has been a very active area of research over the last two decades.
In the PDE community, the convergence to equilibrium for linear kinetic equations was first studied in~\cite{Helffer_2005,Herau_2003}.
General hypocoercivity techniques were developed in~\cite{Villani_2006,Dolbeault_2009,Dolbeault_Hypocoercivity_2015}, where the linear kinetic Fokker--Planck equation served as a prototypical example and where the exponential convergence to equilibrium was obtained both on the spaces~$\Ld^2(M^{-1/2})$ and $H^1(M^{-1/2})$.
Combining hypocoercivity techniques with so-called enlargement theory,
Gualdani, Mischler and Mouhot~\cite{Gualdani_2017, Mischler_2016} later obtained corresponding estimates on larger spaces.
While ergodic estimates in the Brownian setting hold on $W^{-k,1}(\Xd)$ for all~$k\ge1$, hypocoercivity techniques in the kinetic Langevin setting actually require working on weighted spaces $W^{-k,q}(\langle z\rangle^p)$ with integrability exponent $q>1$ and with $p>0$.
More precisely, enlargement theory as developed in~\cite{Mischler_2016} leads to estimates on $W^{-k,q}(\langle z\rangle^p)$ for all $q>1$, $k\in\{-1,0,1\}$, and for large enough weight exponents $p\gg1$.
Yet, for our purposes in this work, it is critical to be able to cover arbitrarily small $p>0$ when the integrability exponent $q$ is close enough to $1$.
This has led us to revisit and partially improve the work of Mischler and Mouhot~\cite{Mischler_2016}: our ergodic estimates are proven to hold for all $q>1$ and $p>0$ under the sole restriction that $pq'$ be large enough, which is of independent interest.
In addition, the control of higher-order correlation functions requires to cover arbitrary negative regularity $k\ge1$.
The proof is postponed to Section~\ref{sec:ergodic}.

\begin{rem}[Periodic setting]\label{rem:toroidal-erg}
As mentioned in the introduction, cf.\@ Section~\ref{subsubsec:periodic}, the above result can essentially be adapted to the corresponding periodic setting on the torus $\T^d$, for instance in the spatially-homogeneous setting $A\equiv0$, but some special care is then needed in the Langevin setting.
Indeed, the nonlinear hypocoercivity result that is available in that case is slightly weaker, cf.~\cite[Theorem 56]{Villani_2009}: it only yields a convergence rate~$t^{-\infty}$ in~\eqref{eq:relax-mt-ass}, thus leading to a similar decay rate~$t^{-\infty}$ instead of exponential in~\eqref{eq:estim-concl-qt}. Fortunately, the resulting subexponential estimates are still enough to repeat the proofs of Theorems~\ref{thm:main}, \ref{th:CLT}, and~\ref{th:concentr}, which can be checked to hold in the same form, replacing all exponential rates in time with superlinear ones. 
\end{rem}

%%%%%%%%%%%%%%%%%%%%%%%%%%%%%%%%%%%%%%%%%%%%
%%%%%%%%%%%%%%%%%%%%%%%%%%%%%%%%%%%%%%%%%%%%
%%%%%%%%%%%%%%%%%%%%%%%%%%%%%%%%%%%%%%%%%%%%

\section{Representation of Brownian cumulants}\label{sec:brownian}
This section is devoted to the representation of Brownian cumulants by means of Lions calculus.
More precisely, our starting point is the Lions expansion of Lemma~\ref{lem:CST0}: following~\cite{Chassagneux_2019}, it leads to an expansion of quantities of the form~$\E_B[\Phi(\mu^N_t)]$ as power series in $\frac1N$.
We introduce so-called \emph{L-graphs} (or Lions graphs) as a new diagrammatic representation that allows to efficiently capture cancellations in moment computations, leading us to a useful representation of Brownian cumulants.
Note that this is strictly unrelated to the so-called Lions forests developed by Delarue and Salkeld in~\cite{Delarue-Salkeld-21}.
In the sequel, we shall denote the $n$th time-integration simplex as
\[\triangle_t^n\,:=\,\big\{(t_1,\ldots,t_n)\in[0,t]^n:0< t_n<\ldots< t_1< t\big\},\]
and we also define
\[\triangle^n\,:=\,\big\{(t,\tau):t>0,\,\tau\in\triangle_t^n\big\}.\]

\subsection{Lions expansion along the flow}\label{subsec:iterated_expansion}
We appeal to Lemma~\ref{lem:CST0} similarly as in~\cite{Chassagneux_2019} to expand the Brownian expectation $\E_B[\Phi(\mu^N_t)]$ as a power series in $\frac1N$. To this aim, we start with the following iterative definition, which describes the natural quantities that appear in the expansion.

\begin{defin0}\label{def:Phi_U}
Given $n\ge1$ and a smooth functional $\Phi: \triangle^n \times \calP(\Xd) \to \R$, we define the sequence $(\Uc_\Phi^{(m)}, \Phi^{(m)})_{m\ge 1}$ as follows:
\begin{enumerate}[$\bullet$]
\item For $m = 1$, we set for all
$t>0$, $\tau = (\tau_1,\dots, \tau_n)\in\triangle_t^n$, $0 < s < \tau_n$, and $\mu \in \calP(\Xd)$,
\[ \Uc_\Phi^{(1)}\big((t,\tau,s), \mu\big) \,:=\, \Phi\big( (t,\tau), m(\tau_n -s ,\mu)\big),\] 
and 
\[ \Phi^{(1)}\big((t,\tau,s),\mu\big) \,:=\, \int_{\Xd} \Tr \Big[ a_0\, \partial^2_\mu \Uc_\Phi^{(1)}\big((t,\tau,s),\mu\big)(z,z) \Big] \, \mu(\ddr z). \]
\item For $m\ge2$, we iteratively define for all $t>0$, $\tau = (\tau_1, \ldots, \tau_{n+m-1})\in \triangle^{n+m-1}_t$, $0<s<\tau_{n+m-1}$, and~$\mu \in \calP(\Xd)$, 
\begin{align*}
\Uc_\Phi^{(m)}\big((t,\tau,s), \mu\big)\,:=\, \Phi^{(m-1)}\big((t,\tau), m(\tau_{n+m-1} - s, \mu)\big), 
\end{align*} 
and 
\begin{align*}
\Phi^{(m)}((t,\tau,s),\mu) \,:=\, \int_{\Xd} \Tr \Big[ a_0\, \partial^2_\mu \Uc_\Phi^{(m)}\big((t,\tau,s), \mu\big)(z,z) \Big] \, \mu(\ddr z). 
\end{align*}
\end{enumerate}
By convention, for $n=0$, given a smooth functional $\Phi: \calP(\Xd) \to \R$, we identify it with the functional $\tilde \Phi: \triangle^0 \times \calP(\Xd) \to \R$ given by~$\tilde \Phi(t,\mu) := \Phi(\mu)$, and we then set
\[\Uc_\Phi^{(1)}\big((t,s), \mu \big) \,:=\, \Phi\big(m(t-s,\mu)\big),\qquad 0<s<t,\]
from which we can define $\Uc_\Phi^{(m)},\Phi^{(m)}$ iteratively as above.
\end{defin0}

This definition is a minor extension of~\cite{Chassagneux_2019}, where only the case $n=0$ was considered.
By a straightforward adaptation of~\cite[Theorems~2.15--2.16]{Chassagneux_2019}, we emphasize that this definition always makes sense
with our smoothness assumptions.
Moreover, for all $n\ge0$ and all smooth functionals $\Phi: \triangle^n \times \calP(\Xd) \to \R$, it is clear from the definition that
for all $m, p \ge 1$, $t > 0$, $\tau \in \triangle^{n+m+p}_t$, and $\mu\in\calP(\Xd)$ we have
\begin{equation}\label{eq:relation_U_Phi}
\Uc_\Phi^{(m+p)} \big((t,\tau),\mu\big) \,=\, \Uc_{\Phi^{(m)}}^{(p)} \big((t,\tau), \mu \big),
\end{equation}
and similarly,
\begin{equation*}
\Phi^{(m+p)}\big((t,\tau), \mu\big) \,=\, (\Phi^{(m)})^{(p)}\big((t,\tau), \mu \big). 
\end{equation*}
In these terms, we can now state the following expansion result for functionals of the empirical measure along the particle dynamics.
This is similar to the so-called weak error expansion in~\cite[Theorem~2.9]{Chassagneux_2019}; we include a short proof for completeness.

\begin{prop}\label{prop:decomp-exp-B}
Given a smooth functional $\Phi : \calP(\Xd) \to \R$, we have for all~$n\ge0$ and $t>0$,
\begin{multline}\label{eq:expansion_order_n}
\E_B[\Phi(\mu^N_t)] ~=~ \sum_{m=0}^n \frac{1}{(2N)^m} \int_{\triangle^{m}_t} \Uc_\Phi^{(m+1)} \big( (t, \tau, 0), \mu^N_0 \big) \, \ddr \tau \\
+ \frac{1}{(2N)^{n+1}} \int_{\triangle^{n+1}_t} \E_B \Big[ \Uc_{\Phi}^{(n+2)} \big((t,\tau,\tau_{n+1}), \mu^N_{\tau_{n+1}} \big) \Big] \, \ddr \tau.
\end{multline}
\end{prop}

\begin{proof}
We proceed by induction and split the proof into two steps.

\medskip
\step1 Case~$n=0$.\\
By Lemma~\ref{lem:CST0}, we find for all $0 \le s \le t$ and $\mu \in \calP(\Xd)$,
\[ \Phi\big(m(t-s,\mu_s^N)\big)\,=\,\Phi\big(m(t,\mu_0^N)\big)+\frac1{2N}\int_0^s\int_\Xd\Tr\Big[ a_0\,\partial_\mu^2\Uc_\Phi^{(1)}\big((t,u),\mu_u^N\big)(z,z)\Big]\,\mu_u^N(\ddr z)\,\ddr u+M_{t,s}^N,\]
for some square-integrable martingale $(M_{t,s}^N)_s$ with $M^N_{t,0} = 0$, where we use the notation $\Uc_\Phi^{(1)}$ from Definition~\ref{def:Phi_U}.
Hence, taking the expectation with respect to $\mathbb{P}_B$ and choosing $s=t$, we get
\[ \E_B [\Phi(\mu_t^N)] \,=\, \Phi\big(m(t,\mu_0^N)\big)+ \frac{1}{2N} \int_0^t \E_B \bigg[\int_{\Xd} \Tr \big[a_0\, \partial^2_\mu \Uc_\Phi^{(1)}\big((t,u),\mu_u^N\big)(z,z) \big] \, \mu^N_u(\ddr z) \bigg] \, \ddr u.\]
With the notation of Definition~\ref{def:Phi_U}, this means
\[ \E_B [\Phi(\mu_t^N)] \,=\, \Uc_\Phi^{(1)}\big((t,0),\mu_0^N\big)+ \frac{1}{2N} \int_0^t \E_B \Big[\Phi^{(1)}\big((t,s),\mu^N_s\big)\Big]\, \ddr s.\]
As $\Phi^{(1)}\big((t,s),\mu_s^N\big)=\Uc_\Phi^{(2)}\big((t,s,s),\mu_s^N\big)$, this proves~\eqref{eq:expansion_order_n} with $n=0$.

\medskip
\step2 General case.\\
We argue by induction.
Suppose that~\eqref{eq:expansion_order_n} has been established for some $n\ge0$.
Let $t > 0$ and $\tau = (\tau_1,\dots, \tau_{n+1}) \in \triangle^{n+1}_t$ be fixed.
Applying Lemma~\ref{lem:CST0} with $\mu\mapsto\Phi(\mu)$ replaced by $\mu\mapsto\Phi^{(n+1)}((t,\tau),\mu)$, we find similarly as in Step~1,
\begin{multline*}
\E_B\big[\Phi^{(n+1)}\big((t,\tau),\mu_{\tau_{n+1}}^N\big)\big]
\,=\,\Phi^{(n+1)}\big((t,\tau),m(\tau_{n+1},\mu_0^N)\big)\\
+\frac1{2N}\int_0^{\tau_{n+1}} \E_B \bigg[ \int_\Xd\Tr\Big[a_0\,\partial_\mu^2\Uc_{\Phi^{(n+1)}}^{(1)}\big((t,\tau,u),\mu_u^N\big)(z,z)\Big] \,\mu_u^N(\ddr z)\bigg]\,\ddr u.
\end{multline*}
Using this to further decompose the remainder term in~\eqref{eq:expansion_order_n}, using the notation of Definition~\ref{def:Phi_U}, and recalling~\eqref{eq:relation_U_Phi}, we precisely deduce that~\eqref{eq:expansion_order_n} also holds with $n$ replaced by $n+1$.
\end{proof}

\subsection{Graphical notation and definition of L-graphs}\label{subsec:diagram}
We introduce a graphical notation associated with Definition~\ref{def:Phi_U}, defining the notion of L-graphs (or Lions graphs), which will considerably simplify combinatorial manipulations in the sequel.
Let $\Phi:\calP(\Xd)\to\R$ be a reference smooth functional.

\medskip
\noindent
$\bullet$ {\it Base point.} Given $n \ge0$ and a smooth functional $\Psi: \triangle^n \times \calP(\Xd) \to \R$, we set for all $t > 0$, $\tau= (\tau_1,\dots,\tau_n) \in \triangle^n_t$, and $0<s<\tau_n$,
\begin{align*}
\RS{p}_{[\Psi]}\big((t,\tau,s), \mu\big)
\,:=\, \Uc_\Psi^{(1)}\big((t,\tau, s), \mu\big)
\,=\, \Psi\big((t,\tau),m(\tau_n-s, \mu)\big).
\end{align*}
In case of the reference smooth functional $\Psi=\Phi$ , we drop the subscript and simply set
\[\RS{p}\big((t,s),\mu\big)
\,:=\,\RS{p}_{[\Phi]}\big((t,s),\mu\big)\,=\,\Phi\big(m(t-s,\mu)\big).\]

\noindent
$\bullet$ {\it Round edge.}
In view of Proposition~\ref{prop:decomp-exp-B}, the key operation that we want to account for in our graphical representation is 
\begin{align*}
\Uc_\Psi^{(k)}
\,\mapsto\, \Uc_\Psi^{(k+1)},
\end{align*}
cf.\@ Definition~\ref{def:Phi_U}. This will be represented with the symbol $\RS{b}$, which we henceforth call ``round edge''. More precisely, given $n \ge0$ and a smooth functional $\Psi: \triangle^n \times \calP(\Xd) \to \R$, we define for all $(t,\tau)\in\triangle^n$, $0<s<\tau_n$, and $\mu\in\Pc(\Xd)$,
\begin{equation}\label{eq:def-round-edge}
\begin{tikzpicture}[blue,baseline = -.7ex]
\begin{scope}[local bounding box=foo]
\draw node (0,0){\text{\color{black}$\Psi$}};
\end{scope}
\draw[rounded corners] (foo.south west) rectangle (foo.north east);
\end{tikzpicture}
\big((t,\tau,s),\mu\big)
\,:=\,\bigg(\int_\Xd\Tr\Big[a_0\,\partial_\mu^2\Psi\big((t,\tau),\nu\big)(z,z)\Big]\,\nu(\ddr z)\bigg)\bigg|_{\nu= m(\tau_n-s,\mu)},
\end{equation}
hence for instance
\begin{equation*}
\RS{pb} \,=\, \Uc_\Phi^{(2)},\qquad
\RS{pbB}\,=\,\Uc_\Phi^{(3)},
\end{equation*}
and so on.
In particular, we emphasize that
\[ \begin{tikzpicture}[blue,baseline=-.7ex]
\begin{scope}[local bounding box=foo]
\draw node (0,0){\text{\color{black}$\Psi$}};
\end{scope}
\draw[rounded corners] (foo.south west) rectangle (foo.north east);
\end{tikzpicture}
\,\ne\,
\begin{tikzpicture}[blue,baseline = -.7ex]
\begin{scope}[local bounding box=foo]
\draw node (0,0){\text{$\RS{p}_{\color{black}[\Psi]}$}};
\end{scope}
\draw[rounded corners] (foo.south west) rectangle (foo.north east);
\end{tikzpicture}.\]
When iterating this operation, we add a subscript $(m)$ to indicate the number $m$ of iterations, that is,
\begin{align*}
\RS{pb}_{(0)} \,:=\, \RS{p}, \qquad \RS{pb}_{(1)} \,:=\, \RS{pb}, \qquad \quad \RS{pb}_{(m+1)} \,:=\,
\begin{tikzpicture}[blue,baseline = -.7ex]
\begin{scope}[local bounding box=foo]
\draw node (0,0){\text{\color{black}$\RS{pb}_{(m)}$}};
\end{scope}
\draw[rounded corners] (foo.south west) rectangle (foo.north east);
\end{tikzpicture}.
\end{align*}
With this notation, the identity~\eqref{eq:relation_U_Phi} takes on the following guise, for all $m, p \ge0$, 
\begin{align*}
\RS{pb}_{(m+p)}
\,=\,
\begin{tikzpicture}[blue,baseline = -.7ex]
\begin{scope}[local bounding box=foo]
\draw node (0,0){\text{\color{black}$\RS{pb}_{(m)}$}};
\end{scope}
\draw[rounded corners] (foo.south west) rectangle (foo.north east);
\end{tikzpicture}_{(p)}.
\end{align*}

\noindent
$\bullet$ {\it Time integration.}
We introduce a short-hand notation for ordered time integrals: given $n,m\ge0$ and given a smooth functional $\Psi:\triangle^{n+m}\times\Pc(\Xd)\to\R$, we define for all $(t,\tau)
\in\triangle^n$ and $\mu\in\Pc(\Xd)$,
\[\Big(\int_{\triangle^m}\Psi\Big)\big((t,\tau),\mu\big)
\,:=\,\Big(\int_{\triangle_t^m}\Psi\Big)(\tau,\mu)
\,:=\,\int_{\triangle^m_{\tau_n}}\Psi\big((t,\tau,\sigma),\mu\big)\,\ddr\sigma,\]
and also, for $m\ge1$,
\[\Big(\int_{\triangle^{m-1}\times0}\Psi\Big)\big((t,\tau),\mu\big)
\,:=\,\Big(\int_{\triangle^{m-1}_t\times0}\Psi\Big)(\tau,\mu)
\,:=\,\int_{\triangle^{m-1}_{\tau_n}}\Psi\big((t,\tau,\sigma,0),\mu\big)\,\ddr\sigma.\]
In these terms, the result of Proposition~\ref{prop:decomp-exp-B} takes on the following guise: for all $n\ge0$ and $t>0$,
\begin{multline}\label{eq:decomp-exp-B-reform-graph}
\E_B[\Phi(\mu_t^N)]~=~\sum_{m=0}^n\frac1{(2N)^m}\Big(\int_{\triangle_t^m\times0}
{\begin{tikzpicture}[blue,baseline = -.7ex]
\begin{scope}[local bounding box=box1]
\node[circle,draw,fill,inner sep=0pt,minimum size=3pt] (1) at (0,0) {};
\draw[white] (1) circle(0.2);
\end{scope}
\draw[rounded corners] (box1.south west) rectangle (box1.north east);
\end{tikzpicture}}
_{(m)}\Big)(\mu_0^N)\\
~+~\frac1{(2N)^{n+1}}\int_{\triangle_t^{n+1}}\E_B\Big[
{\begin{tikzpicture}[blue,baseline = -.7ex]
\begin{scope}[local bounding box=box1]
\node[circle,draw,fill,inner sep=0pt,minimum size=3pt] (1) at (0,0) {};
\draw[white] (1) circle(0.2);
\end{scope}
\draw[rounded corners] (box1.south west) rectangle (box1.north east);
\end{tikzpicture}}
_{(n+1)}\big((t,\tau,\tau_{n+1}),\mu_{\tau_{n+1}}^N\big)\Big]\,\ddr\tau.
\end{multline}
In order to compute cumulants of functionals of the empirical measure along the flow, we shall need to apply~\eqref{eq:decomp-exp-B-reform-graph} with $\Phi$ replaced by powers $\Phi^k$ with $k\ge1$, and try to recognize the structure of the relation between moments and cumulants, cf.~Lemma~\ref{lem:mom-cum}. This will be performed in Section~\ref{sec:rep-Gausscum} and motivates the further notation that we introduce below.

\medskip\noindent
$\bullet$ {\it Products.}
Products of different functionals are simply denoted by juxtaposing the corresponding graphs.
Some care is however needed on how to identify respective time variables. For that purpose, we include subscripts with angular brackets indicating labels of the time variables for the different subgraphs: given $n,m\ge0$, given functionals $\Psi:\triangle^n\times\Pc(\Xd)\to\R$ and $\Theta:\triangle^m\times\Pc(\Xd)\to\R$, and given $i_1<\ldots<i_n$ and $j_1<\ldots<j_m$ with $\{i_1,\ldots,i_n\}\cup\{j_1,\ldots,j_m\}=\llbracket p\rrbracket$ (possibly not disjoint), we set for all $(t,\tau)\in\triangle^{p}$ and $\mu\in\Pc(\Xd)$,
\begin{equation}\label{eq:def-prod-graph}
\Psi_{\langle i_1,\ldots,i_n\rangle}
\Theta_{\langle j_1,\ldots,j_m\rangle}
\big((t,\tau),\mu\big)
\,:=\,
\Psi\big((t,\tau_{i_1},\ldots,\tau_{i_n}),\mu\big)
\Theta\big((t,\tau_{j_1},\ldots,\tau_{j_m}),\mu\big).
\end{equation}
For instance,
\[{\begin{tikzpicture}[blue,baseline = -.7ex]
\begin{scope}[local bounding box=box1]
\node[circle,draw,fill,inner sep=0pt,minimum size=3pt] (1) at (0,0) {};
\draw[white] (1) circle(0.2);
\end{scope}
\draw[rounded corners] (box1.south west) rectangle (box1.north east);
\end{tikzpicture}}
_{\langle 1,4\rangle}
{\begin{tikzpicture}[blue,baseline = -.7ex]
\begin{scope}[local bounding box=box1]
\node[circle,draw,fill,inner sep=0pt,minimum size=3pt] (1) at (0,0) {};
\draw[white] (1) circle(0.2);
\end{scope}
\draw[rounded corners] (box1.south west) rectangle (box1.north east);
\end{tikzpicture}}
_{(2)\langle2,3,4\rangle}
\,\big((t,\tau_1,\tau_2,\tau_3,\tau_4),\mu\big)
\,=\,
\Uc_\Phi^{(2)}\big((t,\tau_1,\tau_4),\mu\big)
\,\Uc_\Phi^{(3)}\big((t,\tau_2,\tau_3,\tau_4),\mu\big).\]
We also occasionally use indices to label time variables in subgraphs, for instance
\[{\begin{tikzpicture}[blue,baseline = -.7ex]
\begin{scope}[local bounding box=box1]
\node (1) at (0,0) {$
{\begin{tikzpicture}[blue,baseline = -.7ex]
\node[circle,draw,fill,inner sep=0pt,minimum size=3pt] (1) at (0,0) {};
\draw[white] (1) circle(0.2);
\end{tikzpicture}}
{\color{black}_{\!\!\langle1\rangle}}
$};
\end{scope}
\draw[rounded corners] (box1.south west) rectangle (box1.north east);
\end{tikzpicture}}
_{\langle 4\rangle}
{\begin{tikzpicture}[blue,baseline = -.7ex]
\begin{scope}[local bounding box=box1]
\node (1) at (0,0) {\text{$
{\begin{tikzpicture}[blue,baseline = -.7ex]
\begin{scope}[local bounding box=box1]
\node[circle,draw,fill,inner sep=0pt,minimum size=3pt] (1) at (0,0) {};
\draw[white] (1) circle(0.2);
\end{scope}
\draw[rounded corners] (box1.south west) rectangle (box1.north east);
\end{tikzpicture}}
{\color{black}_{\langle2,3\rangle}}$}};
\end{scope}
\draw[rounded corners] (box1.south west) rectangle (box1.north east);
\end{tikzpicture}}
_{\langle4\rangle}
\,\equiv\,
{\begin{tikzpicture}[blue,baseline = -.7ex]
\begin{scope}[local bounding box=box1]
\node[circle,draw,fill,inner sep=0pt,minimum size=3pt] (1) at (0,0) {};
\draw[white] (1) circle(0.2);
\end{scope}
\draw[rounded corners] (box1.south west) rectangle (box1.north east);
\end{tikzpicture}}
_{\langle 1,4\rangle}
{\begin{tikzpicture}[blue,baseline = -.7ex]
\begin{scope}[local bounding box=box1]
\node[circle,draw,fill,inner sep=0pt,minimum size=3pt] (1) at (0,0) {};
\draw[white] (1) circle(0.2);
\end{scope}
\draw[rounded corners] (box1.south west) rectangle (box1.north east);
\end{tikzpicture}}
_{(2)\langle2,3,4\rangle}.\]

\noindent
$\bullet$ {\it Straight edges.}
When applying the round edge~\eqref{eq:def-round-edge} to a power of a given functional, we are led to defining another type of operation, which we represent by a straight edge between subgraphs. More precisely, given $n,m\ge0$, given smooth functionals $\Psi:\triangle^{n+1}\times\Pc(\Xd)\to\R$ and $\Theta:\triangle^{m+1}\times\Pc(\Xd)\to\R$, and given a partition $\{i_1,\ldots,i_{n}\}\uplus\{j_1,\ldots,j_{m}\}=\llbracket n+m\rrbracket$ with $i_1<\ldots<i_n$ and $j_1<\ldots<j_m$, we set for all $(t,\tau,s,s')\in\triangle^{n+m+2}$,
\begin{multline}\label{eq:def-straight-edge}
\begin{tikzpicture}[blue,baseline = -.7ex]
\begin{scope}[local bounding box=box1]
\draw node(1) at (0,0){\text{\color{black}$\Psi_{\langle i_1,\ldots, i_n,m+n+1\rangle}$}};
\end{scope}
\begin{scope}[local bounding box=box2]
\draw node(2) at (4.4,0){\text{\color{black}$\Theta_{\langle j_1,\ldots, j_m,m+n+1\rangle}$}};
\end{scope}
\draw node(1) at (2.2,-0.2){\text{\tiny\color{black}$\langle m\!+\!n\!+\!2\rangle$}};
\path[-] (box1) edge (box2);
\draw[rounded corners,dotted] (box1.south west) rectangle (box1.north east);
\draw[rounded corners,dotted] (box2.south west) rectangle (box2.north east);
\end{tikzpicture}
\,\big((t,\tau,s,s',\mu\big)\\
\,:=\,\bigg(\int_\Xd\Big(\partial_\mu\Psi\big((t,\tau_{i_1},\ldots,\tau_{i_n},s),\nu\big)(z)
\cdot a_0\partial_\mu\Theta\big((t,\tau_{j_1},\ldots,\tau_{j_m},s),\nu\big)(z)\Big)\,\nu(\ddr z)\bigg)\bigg|_{\nu= m(s-s',\mu)}.
\end{multline}
Note that the dotted boxes around $\Psi$ and $\Theta$ above have no particular meaning in the graphical notation: they are just meant to emphasize that the straight edge is between the subgraphs corresponding to $\Psi$ and $\Theta$;
we remove them when no confusion is possible, simply writing for instance
\begin{equation*}
{\begin{tikzpicture}[blue,baseline = -.7ex]
\node[circle,draw,fill,inner sep=0pt,minimum size=3pt] (1) at (0,0) {};
\node[circle,draw,fill,inner sep=0pt,minimum size=3pt] (2) at (0.85,0) {};
\node (1') at (0,-0.2) {{\tiny\color{black}$\langle1\rangle$}};
\node (3') at (0.425,-0.2) {{\tiny\color{black}$\langle2\rangle$}};
\node (2') at (0.85,-0.2) {{\tiny\color{black}$\langle1\rangle$}};
\path[-] (1) edge (2);
\end{tikzpicture}}
\,\equiv\,
\begin{tikzpicture}[blue,baseline = -.7ex]
\begin{scope}[local bounding box=box1]
\node[circle,draw,fill,inner sep=0pt,minimum size=3pt] (1) at (0,0) {};
\node (1') at (0,-0.2) {{\tiny\color{black}$\langle1\rangle$}};
\node (1'') at (0,0.1) {{\tiny\color{white}$.$}};
\end{scope}
\begin{scope}[local bounding box=box2]
\node[circle,draw,fill,inner sep=0pt,minimum size=3pt] (2) at (1.1,0) {};
\node (2') at (1.1,-0.2) {{\tiny\color{black}$\langle1\rangle$}};
\node (2'') at (1.1,0.1) {{\tiny\color{white}$.$}};
\end{scope}
\draw (box1) -- (box2);
\node (3') at (0.55,-0.3) {{\tiny\color{black}$\langle2\rangle$}};
\draw[rounded corners,dotted] (box1.south west) rectangle (box1.north east);
\draw[rounded corners,dotted] (box2.south west) rectangle (box2.north east);
\end{tikzpicture}\,,
\qquad
\begin{tikzpicture}[blue,baseline = -.7ex]
\begin{scope}[local bounding box=box1]
\node[circle,draw,fill,inner sep=0pt,minimum size=3pt] (1) at (0,0) {};
\draw[white] (1) circle(0.2);
\end{scope}
\draw[rounded corners] (box1.south west) rectangle (box1.north east);
\node[circle,draw,fill,inner sep=0pt,minimum size=3pt] (2) at (0.85,0) {};
\node (1') at (0,-0.4) {{\tiny\color{black}$\langle1,2\rangle$}};
\node (3') at (0.48,-0.2) {{\tiny\color{black}$\langle3\rangle$}};
\node (2') at (0.85,-0.2) {{\tiny\color{black}$\langle2\rangle$}};
\path[-] (box1) edge (2);
\end{tikzpicture}
\,\equiv\,
\begin{tikzpicture}[blue,baseline = -.7ex]
\begin{scope}[local bounding box=box11]
\begin{scope}[local bounding box=box1]
\node[circle,draw,fill,inner sep=0pt,minimum size=3pt] (1) at (0,0) {};
\draw[white] (1) circle(0.2);
\end{scope}
\draw[rounded corners] (box1.south west) rectangle (box1.north east);
\node (1') at (0,-0.4) {{\tiny\color{black}$\langle1,2\rangle$}};
\node (1'') at (0,0.23) {{\tiny\color{white}$.$}};
\end{scope}
\begin{scope}[local bounding box=box2]
\node[circle,draw,fill,inner sep=0pt,minimum size=3pt] (2) at (1.2,0.052) {};
\node (2') at (1.2,-0.22) {{\tiny\color{black}$\langle2\rangle$}};
\node (2'') at (1.2,0.14) {{\tiny\color{white}$.$}};
\end{scope}
\draw[rounded corners,dotted] (box11.south west) rectangle (box11.north east);
\draw[rounded corners,dotted] (box2.south west) rectangle (box2.north east);
\path[-] (box11) edge (box2);
\node (3') at (0.67,-0.28) {{\tiny\color{black}$\langle3\rangle$}};
\end{tikzpicture}\,.
\end{equation*}

\noindent
$\bullet$ {\it Graphical notation and terminology.}
Starting from a number of base points, surrounding some subgraphs by round edges, and connecting some subgraphs via straight edges, we are led to a class of diagrams that we shall call {\it L-graphs} (or Lions graphs).
Viewing round edges as loops,
\[\begin{tikzpicture}[blue,baseline = -.7ex]
\tikzset{every loop/.style={}}
\begin{scope}[local bounding box=box1]
\draw node(1) at (0,0){\text{\color{black}$\Psi$}};
\end{scope}
\path[-] (box1) edge [loop above] (box1);
\draw[rounded corners,dotted] (box1.south west) rectangle (box1.north east);
\end{tikzpicture}
\,\equiv\,
\begin{tikzpicture}[blue,baseline = -.7ex]
\begin{scope}[local bounding box=box1]
\draw node(1) at (0,0){\text{\color{black}$\Psi$}};
\end{scope}
\draw[rounded corners] (box1.south west) rectangle (box1.north east);
\end{tikzpicture}\,,
\qquad
\begin{tikzpicture}[blue,baseline = -.7ex]
\tikzset{every loop/.style={}}
\begin{scope}[local bounding box=box1]
\draw node(1) at (0,0){\text{\color{black}$\Psi$}};
\end{scope}
\path[-] (box1) edge [loop above] (box1);
\path[-] (box1) edge [loop right] (box1);
\draw[rounded corners,dotted] (box1.south west) rectangle (box1.north east);
\end{tikzpicture}
\,\equiv\,
\begin{tikzpicture}[blue,baseline = -.7ex]
\begin{scope}[local bounding box=box1]
\draw node(1) at (0,0){
\begin{tikzpicture}[blue,baseline = -.7ex]
\begin{scope}[local bounding box=box1]
\draw node(1) at (0,0){\text{\color{black}$\Psi$}};
\end{scope}
\draw[rounded corners] (box1.south west) rectangle (box1.north east);
\end{tikzpicture}
};
\end{scope}
\draw[rounded corners] (box1.south west) rectangle (box1.north east);
\end{tikzpicture}\,,\]
we can view L-graphs as (undirected) multi-hypergraphs that satisfy a number of properties.
First recall that a multi-hypergraph is a pair $(V,E)$ where:
\begin{enumerate}[---]
\item $V$ is a set of elements called base points or vertices;
\smallskip\item $E$ is a multiset of elements called edges, which are pairs of non-empty subsets of $V$. The two subsets that are connected by an edge are called the ends of the edge.
\end{enumerate}
We then formally define an {\it (unlabeled) L-graph} as a multi-hypergraph $\Psi=(V,E)$ such that the following three properties hold:
\begin{enumerate}[---]
\item an edge in $E$ is either a loop (so-called round edge) or it connects {\it disjoint} vertex subsets (so-called straight edge): in other words, for all~$\{A,B\}\in E$, we have either $A=B$ or~$A\cap B=\varnothing$;
\smallskip\item a vertex subset $S\subset V$ can only be the end of at most one straight edge, but it can at the same time be the end of several round edges; in particular, each straight edge is simple, but round edges can be multiple;
\smallskip\item if a vertex subset $S\subset V$ is the end of some edge (round or straight), then strict subsets of $S$ can only be connected to other strict subsets of $S$: in other words, for all~$\{A,B\}\in E$ and $A'\subsetneq A$, the condition $\{A',A''\}\in E$ implies $A''\subsetneq A$.
\end{enumerate}
A {\it labeled L-graph} is an unlabeled L-graph endowed with a time labeling $V\sqcup E\to\N$, which associates a time label to each vertex and each edge.
Note that each labeled L-graph is uniquely associated to a functional that can be obtained by iterating the round edge operation~\eqref{eq:def-round-edge}, the straight edge operation~\eqref{eq:def-straight-edge}, and by taking products~\eqref{eq:def-prod-graph}.
We introduce some further useful terminology:
\begin{enumerate}[---]
\item {\it Induced subgraphs.} Given an L-graph $(V,E)$ and a subset $S\subset V$, we define the {\it L-graph induced by $S$} as the pair $(S,E_S)$ where $E_S$ is the multiset of all edges $\{A,B\}\in E$ with $A,B\subset S$. We define the {\it L-graph strictly induced by $S$} as the pair $(S,E_S')$ where $E_S'$ is now the multiset of all edges $\{A,B\}\in E$ with $A,B\subsetneq S$. By definition, we note that induced L-graphs are indeed L-graphs themselves, and moreover $E_S'$ coincides with $E_S$ after removing all occurrences of the round edge~$\{S\}$.
We call {\it L-subgraph} of $(V,E)$ any L-graph $(S,F)$ with $S\subset V$ and $E_S'\subset F\subset E_S$.
\smallskip\item {\it Stability.} 
Given an L-graph $(V,E)$, an L-subgraph $(S,F)$ is said to be {\it stable} if for all $\{A,B\}\in E$ with $A\subsetneq S$ we also have $B\subsetneq S$. In particular, by definition of an L-graph, a vertex subset that is the end of an edge is automatically inducing a stable subgraph.
\smallskip\item {\it Connectedness.}
An L-graph $(V,E)$ is said to be {\it connected} if there is no partition $V=A\cup B$ with $A,B\ne\varnothing$, $A\cap B=\varnothing$, and $E=E_{A}\cup E_{B}$. An L-graph can be uniquely decomposed into its connected components.
\smallskip\item {\it Irreducibility.}
An L-graph $(V,E)$ is said to be {\it irreducible} if for all straight edges $\{A,B\}\in E$ the induced subgraphs $(A,E_A)$ and $(B,E_B)$ are both connected and if for all round edge $\{A\}\in E$ the strictly induced subgraph $(A,E_A')$ is connected.
\end{enumerate}

\noindent
$\bullet$ {\it Rules for time ordering and symmetrization.}
Given an L-graph $(V,E)$, we consider the following four rules that restrict the possibilities for the time labeling $V\sqcup E\to\N$:
\begin{enumerate}[(R1)]
\item[(R1)] {\it Round edges.} In accordance with definition~\eqref{eq:def-round-edge}, the time label of a round edge $\{A\}\in E$ must always be larger than time labels of the strictly induced subgraph $(A,E_A')$; in other words, it must be larger than time labels of the subgraph that the round edge `surrounds'.
\smallskip\item[(R2)] {\it Straight edges.} In accordance with definition~\eqref{eq:def-straight-edge}, the time label of a straight edge $\{A,B\}\in E$ must always be larger than time labels of the two induced subgraphs $(A,E_A)$ and $(B,E_B)$. In addition, the last time label of the two subgraphs must coincide.
\smallskip\item[(R3)] {\it Products.} In any stable subgraph $(S,F)$, decomposing it into its connected components, the last time label of each component coincides.
\smallskip\item[(R4)] {\it No other repetition and no gap.} Apart from equalities of time labels imposed by the above three rules~(R1)--(R3), all time labels must be different. In addition, the set of time labels, that is, the image of the time labeling map $V\sqcup E\to\N$, must be of the form $\llbracket n \rrbracket=\{1,\ldots,n\}$ for some~$n\ge1$ (that is, without gap).
\end{enumerate}
In our notation, an unlabeled L-graph will be understood as the arithmetic average of all the labeled L-graphs that can be obtained by endowing the graph with a time labeling that satisfies the above four rules (R1)--(R4). For instance,
\begin{eqnarray*}
{\begin{tikzpicture}[blue,baseline = -.7ex]
\begin{scope}[local bounding box=box1]
\node[circle,draw,fill,inner sep=0pt,minimum size=3pt] (1) at (0,0) {};
\draw[white] (1) circle(0.2);
\end{scope}
\draw[rounded corners] (box1.south west) rectangle (box1.north east);
\end{tikzpicture}}
\,\,
{\begin{tikzpicture}[blue,baseline = -.7ex]
\begin{scope}[local bounding box=box1]
\node[circle,draw,fill,inner sep=0pt,minimum size=3pt] (1) at (0,0) {};
\draw[white] (1) circle(0.2);
\end{scope}
\draw[rounded corners] (box1.south west) rectangle (box1.north east);
\end{tikzpicture}}
&\equiv&
\tfrac12\Big(
{\begin{tikzpicture}[blue,baseline = -.7ex]
\begin{scope}[local bounding box=box1]
\node[circle,draw,fill,inner sep=0pt,minimum size=3pt] (1) at (0,0) {};
\draw[white] (1) circle(0.2);
\end{scope}
\draw[rounded corners] (box1.south west) rectangle (box1.north east);
\end{tikzpicture}}
_{\langle 1,3\rangle}
{\begin{tikzpicture}[blue,baseline = -.7ex]
\begin{scope}[local bounding box=box1]
\node[circle,draw,fill,inner sep=0pt,minimum size=3pt] (1) at (0,0) {};
\draw[white] (1) circle(0.2);
\end{scope}
\draw[rounded corners] (box1.south west) rectangle (box1.north east);
\end{tikzpicture}}
_{\langle2,3\rangle}
+
{\begin{tikzpicture}[blue,baseline = -.7ex]
\begin{scope}[local bounding box=box1]
\node[circle,draw,fill,inner sep=0pt,minimum size=3pt] (1) at (0,0) {};
\draw[white] (1) circle(0.2);
\end{scope}
\draw[rounded corners] (box1.south west) rectangle (box1.north east);
\end{tikzpicture}}
_{\langle 2,3\rangle}
{\begin{tikzpicture}[blue,baseline = -.7ex]
\begin{scope}[local bounding box=box1]
\node[circle,draw,fill,inner sep=0pt,minimum size=3pt] (1) at (0,0) {};
\draw[white] (1) circle(0.2);
\end{scope}
\draw[rounded corners] (box1.south west) rectangle (box1.north east);
\end{tikzpicture}}
_{\langle1,3\rangle}
\Big)
~=~
{\begin{tikzpicture}[blue,baseline = -.7ex]
\begin{scope}[local bounding box=box1]
\node[circle,draw,fill,inner sep=0pt,minimum size=3pt] (1) at (0,0) {};
\draw[white] (1) circle(0.2);
\end{scope}
\draw[rounded corners] (box1.south west) rectangle (box1.north east);
\end{tikzpicture}}
_{\langle 1,3\rangle}
{\begin{tikzpicture}[blue,baseline = -.7ex]
\begin{scope}[local bounding box=box1]
\node[circle,draw,fill,inner sep=0pt,minimum size=3pt] (1) at (0,0) {};
\draw[white] (1) circle(0.2);
\end{scope}
\draw[rounded corners] (box1.south west) rectangle (box1.north east);
\end{tikzpicture}}
_{\langle2,3\rangle}\,,\\
{\begin{tikzpicture}[blue,baseline = -.7ex]
\begin{scope}[local bounding box=box1]
\node[circle,draw,fill,inner sep=0pt,minimum size=3pt] (1) at (0,0) {};
\draw[white] (1) circle(0.2);
\end{scope}
\draw[rounded corners] (box1.south west) rectangle (box1.north east);
\end{tikzpicture}}
\,\,{\begin{tikzpicture}[blue,baseline = -.7ex]
\begin{scope}[local bounding box=box1]
\node[circle,draw,fill,inner sep=0pt,minimum size=3pt] (1) at (0,0) {};
\draw[white] (1) circle(0.2);
\end{scope}
\draw[rounded corners] (box1.south west) rectangle (box1.north east);
\end{tikzpicture}}
_{(2)}
&\equiv&
\tfrac13\Big(
{\begin{tikzpicture}[blue,baseline = -.7ex]
\begin{scope}[local bounding box=box1]
\node[circle,draw,fill,inner sep=0pt,minimum size=3pt] (1) at (0,0) {};
\draw[white] (1) circle(0.2);
\end{scope}
\draw[rounded corners] (box1.south west) rectangle (box1.north east);
\end{tikzpicture}}
_{\langle 1,4\rangle}
{\begin{tikzpicture}[blue,baseline = -.7ex]
\begin{scope}[local bounding box=box1]
\node[circle,draw,fill,inner sep=0pt,minimum size=3pt] (1) at (0,0) {};
\draw[white] (1) circle(0.2);
\end{scope}
\draw[rounded corners] (box1.south west) rectangle (box1.north east);
\end{tikzpicture}}
_{(2)\langle2,3,4\rangle}
+
{\begin{tikzpicture}[blue,baseline = -.7ex]
\begin{scope}[local bounding box=box1]
\node[circle,draw,fill,inner sep=0pt,minimum size=3pt] (1) at (0,0) {};
\draw[white] (1) circle(0.2);
\end{scope}
\draw[rounded corners] (box1.south west) rectangle (box1.north east);
\end{tikzpicture}}
_{\langle 2,4\rangle}
{\begin{tikzpicture}[blue,baseline = -.7ex]
\begin{scope}[local bounding box=box1]
\node[circle,draw,fill,inner sep=0pt,minimum size=3pt] (1) at (0,0) {};
\draw[white] (1) circle(0.2);
\end{scope}
\draw[rounded corners] (box1.south west) rectangle (box1.north east);
\end{tikzpicture}}
_{(2)\langle1,3,4\rangle}
+
{\begin{tikzpicture}[blue,baseline = -.7ex]
\begin{scope}[local bounding box=box1]
\node[circle,draw,fill,inner sep=0pt,minimum size=3pt] (1) at (0,0) {};
\draw[white] (1) circle(0.2);
\end{scope}
\draw[rounded corners] (box1.south west) rectangle (box1.north east);
\end{tikzpicture}}
_{\langle 3,4\rangle}
{\begin{tikzpicture}[blue,baseline = -.7ex]
\begin{scope}[local bounding box=box1]
\node[circle,draw,fill,inner sep=0pt,minimum size=3pt] (1) at (0,0) {};
\draw[white] (1) circle(0.2);
\end{scope}
\draw[rounded corners] (box1.south west) rectangle (box1.north east);
\end{tikzpicture}}
_{(2)\langle1,2,4\rangle}
\Big),\\
\begin{tikzpicture}[blue,baseline = -.7ex]
\begin{scope}[local bounding box=box1]
\node[circle,draw,fill,inner sep=0pt,minimum size=3pt] (1) at (0,0) {};
\draw[white] (1) circle(0.2);
\end{scope}
\draw[rounded corners] (box1.south west) rectangle (box1.north east);
\begin{scope}[local bounding box=box2]
\node[circle,draw,fill,inner sep=0pt,minimum size=3pt] (2) at (0.5,-0.3) {};
\draw[white] (2) circle(0.2);
\end{scope}
\draw[rounded corners] (box2.south west) rectangle (box2.north east);
\node[circle,draw,fill,inner sep=0pt,minimum size=3pt] (3) at (0.5,0.3) {};
\path[-] (box2) edge (3);
\end{tikzpicture}
&\equiv&
\tfrac13\bigg(
\begin{tikzpicture}[blue,baseline = -.7ex]
\begin{scope}[local bounding box=box1]
\node[circle,draw,fill,inner sep=0pt,minimum size=3pt] (1) at (0,0) {};
\draw[white] (1) circle(0.2);
\end{scope}
\draw[rounded corners] (box1.south west) rectangle (box1.north east);
\begin{scope}[local bounding box=box2]
\node[circle,draw,fill,inner sep=0pt,minimum size=3pt] (2) at (0.65,-0.4) {};
\draw[white] (2) circle(0.2);
\end{scope}
\draw[rounded corners] (box2.south west) rectangle (box2.north east);
\node[circle,draw,fill,inner sep=0pt,minimum size=3pt] (3) at (0.65,0.4) {};
\node (1') at (0,-0.4) {{\tiny\color{black}$\langle1,4\rangle$}};
\node (2') at (1.2,-0.4) {{\tiny\color{black}$\langle2,3\rangle$}};
\node (3') at (0.95,0.4) {{\tiny\color{black}$\langle3\rangle$}};
\node (4') at (0.9,0) {{\tiny\color{black}$\langle4\rangle$}};
\path[-] (box2) edge (3);
\end{tikzpicture}
+
\begin{tikzpicture}[blue,baseline = -.7ex]
\begin{scope}[local bounding box=box1]
\node[circle,draw,fill,inner sep=0pt,minimum size=3pt] (1) at (0,0) {};
\draw[white] (1) circle(0.2);
\end{scope}
\draw[rounded corners] (box1.south west) rectangle (box1.north east);
\begin{scope}[local bounding box=box2]
\node[circle,draw,fill,inner sep=0pt,minimum size=3pt] (2) at (0.65,-0.4) {};
\draw[white] (2) circle(0.2);
\end{scope}
\draw[rounded corners] (box2.south west) rectangle (box2.north east);
\node[circle,draw,fill,inner sep=0pt,minimum size=3pt] (3) at (0.65,0.4) {};
\node (1') at (0,-0.4) {{\tiny\color{black}$\langle2,4\rangle$}};
\node (2') at (1.2,-0.4) {{\tiny\color{black}$\langle1,3\rangle$}};
\node (3') at (0.95,0.4) {{\tiny\color{black}$\langle3\rangle$}};
\node (4') at (0.9,0) {{\tiny\color{black}$\langle4\rangle$}};
\path[-] (box2) edge (3);
\end{tikzpicture}
+
\begin{tikzpicture}[blue,baseline = -.7ex]
\begin{scope}[local bounding box=box1]
\node[circle,draw,fill,inner sep=0pt,minimum size=3pt] (1) at (0,0) {};
\draw[white] (1) circle(0.2);
\end{scope}
\draw[rounded corners] (box1.south west) rectangle (box1.north east);
\begin{scope}[local bounding box=box2]
\node[circle,draw,fill,inner sep=0pt,minimum size=3pt] (2) at (0.65,-0.4) {};
\draw[white] (2) circle(0.2);
\end{scope}
\draw[rounded corners] (box2.south west) rectangle (box2.north east);
\node[circle,draw,fill,inner sep=0pt,minimum size=3pt] (3) at (0.65,0.4) {};
\node (1') at (0,-0.4) {{\tiny\color{black}$\langle3,4\rangle$}};
\node (2') at (1.2,-0.4) {{\tiny\color{black}$\langle1,2\rangle$}};
\node (3') at (0.95,0.4) {{\tiny\color{black}$\langle2\rangle$}};
\node (4') at (0.9,0) {{\tiny\color{black}$\langle4\rangle$}};
\path[-] (box2) edge (3);
\end{tikzpicture}
\bigg).
\end{eqnarray*}
As we shall see in the next section, the whole point of this graphical notation is that it allows for quick and easy computations to derive representation formulas for Brownian cumulants.
We summarize the main graphical computation rules in the following lemma.
Note in particular that item~(ii) below implies that any L-graph is equal to a linear combination of {\it irreducible} L-graphs with the same number of vertices and edges.

\begin{lem}[Graphical computation rules]\label{lem:comput-graph}\
\begin{enumerate}[(i)]
\item For all $k\ge1$, a base point associated with the power $\Phi^k$ is equivalent to the product of $k$ copies of the basepoint associated with $\Phi$,
\begin{align*}
\RS{p}_{[\Phi]}\,=\,\RS{p},
\qquad
\RS{p}_{[\Phi^2]}\,=\,\RS{p}\RS{p},
\qquad
\RS{p}_{[\Phi^k]}\,=\,(\RS{p})^k.
\end{align*}
\item For any L-graph $\Psi: \triangle^{n} \times \calP(\Xd) \to \R$, for all $t > 0$, $\tau = (\tau_1, \ldots, \tau_n) \in \triangle^n_t$ and $0 < s < \tau_n$, for all $\mu \in \calP(\Xd)$, 
\begin{align}
    \label{eq:semigroup_prop_L_graph}
    \Psi \big((t, \tau), m(\tau_n-s, \mu)\big) = \Psi \big((t,\tau_1, \ldots, \tau_{n-1}, s), \mu \big). 
\end{align} 
\item Given two L-graphs $\Psi:\triangle^{n+1}\times\Pc(\Xd)\to\R$ and $\Theta:\triangle^{m+1}\times\Pc(\Xd)\to\R$,
and given $i_1<\ldots<i_n$ and $j_1<\ldots<j_m$ with $\{i_1,\ldots,i_n\}\cup\{j_1,\ldots,j_m\}=\llbracket p\rrbracket$ for some $0\le p\le n+m$, we find
\begin{eqnarray}
\begin{tikzpicture}[blue,baseline = -.7ex]
\begin{scope}[local bounding box=foo]
\node (1) at (0,0){\text{\color{black}$\Psi_{\langle i_1,\ldots,i_n,p+1\rangle}$}};
\node (2) at (2.5,0){\text{\color{black}$\Theta_{\langle j_1,\ldots,j_m,p+1\rangle}$}};
\end{scope}
\draw[rounded corners] (foo.south west) rectangle (foo.north east);
\end{tikzpicture}
_{\langle p+2\rangle}
&=&
\begin{tikzpicture}[blue,baseline = -.7ex]
\begin{scope}[local bounding box=box1]
\node (1) at (0,0){\text{\color{black}$\Psi$}};
\end{scope}
\draw[rounded corners] (box1.south west) rectangle (box1.north east);
\end{tikzpicture}
_{\langle i_1,\ldots,i_n,p+1,p+2\rangle}
\Theta_{\langle j_1,\ldots,j_m,p+2\rangle}
\nonumber\\
&&
+\,
\Psi_{\langle i_1,\ldots,i_n,p+2\rangle}
~\begin{tikzpicture}[blue,baseline = -.8ex]
\begin{scope}[local bounding box=box1]
\node (1) at (0,0){\text{\color{black}$\Theta$}};
\end{scope}
\draw[rounded corners] (box1.south west) rectangle (box1.north east);
\end{tikzpicture}
_{\langle j_1,\ldots,j_m,p+1,p+2\rangle}
\nonumber\\
&&
+\,2\,
\begin{tikzpicture}[blue,baseline = -.7ex]
\begin{scope}[local bounding box=box1]
\node (1) at (0,0){{\color{black}$\Psi_{\langle i_1,\ldots,i_n,p+1\rangle}$}};
\end{scope}
\begin{scope}[local bounding box=box2]
\node (2) at (3.5,0){{\color{black}$\Theta_{\langle j_1,\ldots,j_m,p+1\rangle}$}};
\end{scope}
\path[-] (box1) edge (box2);
\draw[rounded corners,dotted] (box1.south west) rectangle (box1.north east);
\draw[rounded corners,dotted] (box2.south west) rectangle (box2.north east);
\node (3) at (1.75,-0.2){{\color{black}\tiny$\langle p\!+\!2\rangle$}};
\end{tikzpicture}.
\label{eq:decomp-round-PsiTheta}
\end{eqnarray}
In particular, in the first right-hand side term, we note that the penultimate time label~$p+1$ of~$\Psi$ is larger than all the time labels $j_1,\ldots,j_m$ in~$\Theta$ (and conversely in the second term), thus adding nontrivial time ordering not implied by the basic rules \emph{(R1)}--\emph{(R4)}.
Using symmetrized notations, we find for instance
\begin{eqnarray}
{\begin{tikzpicture}[blue,baseline = -.7ex]
\begin{scope}[local bounding box=box1]
\node[circle,draw,fill,inner sep=0pt,minimum size=3pt] (1) at (0,0) {};
\node[circle,draw,fill,inner sep=0pt,minimum size=3pt] (2) at (0.3,0) {};
\draw[white] (1) circle(0.2);
\draw[white] (2) circle(0.2);
\end{scope}
\draw[rounded corners] (box1.south west) rectangle (box1.north east);
\end{tikzpicture}}
&=&
2\,{\begin{tikzpicture}[blue,baseline = -.7ex]
\begin{scope}[local bounding box=box1]
\node[circle,draw,fill,inner sep=0pt,minimum size=3pt] (1) at (0,0) {};
\draw[white] (1) circle(0.2);
\end{scope}
\node[circle,draw,fill,inner sep=0pt,minimum size=3pt] (2) at (-0.4,0) {};
\draw[rounded corners] (box1.south west) rectangle (box1.north east);
\end{tikzpicture}}
\,+\,
2\,{\begin{tikzpicture}[blue,baseline = -.7ex]
\node[circle,draw,fill,inner sep=0pt,minimum size=3pt] (1) at (0,0) {};
\node[circle,draw,fill,inner sep=0pt,minimum size=3pt] (2) at (0.3,0) {};
\path[-] (1) edge (2);
\end{tikzpicture}}\,,
\label{eq:comp-gr-1(..)}
\\
{\begin{tikzpicture}[blue,baseline = -.7ex]
\begin{scope}[local bounding box=box1]
\node (1) at (0,0) {
\begin{tikzpicture}[blue,baseline = -.7ex]
\begin{scope}[local bounding box=box1]
\node[circle,draw,fill,inner sep=0pt,minimum size=3pt] (1) at (0,0) {};
\node[circle,draw,fill,inner sep=0pt,minimum size=3pt] (2) at (0.3,0) {};
\draw[white] (1) circle(0.2);
\draw[white] (2) circle(0.2);
\end{scope}
\draw[rounded corners] (box1.south west) rectangle (box1.north east);
\end{tikzpicture}
};
\end{scope}
\draw[rounded corners] (box1.south west) rectangle (box1.north east);
\end{tikzpicture}}
&=&
2\,{\begin{tikzpicture}[blue,baseline = -.7ex]
\begin{scope}[local bounding box=box1]
\node (1) at (0,0) {
\begin{tikzpicture}[blue,baseline = -.7ex]
\begin{scope}[local bounding box=box1]
\node[circle,draw,fill,inner sep=0pt,minimum size=3pt] (1) at (0,0) {};
\draw[white] (1) circle(0.2);
\end{scope}
\draw[rounded corners] (box1.south west) rectangle (box1.north east);
\end{tikzpicture}
};
\end{scope}
\node[circle,draw,fill,inner sep=0pt,minimum size=3pt] (2) at (-0.5,0) {};
\draw[rounded corners] (box1.south west) rectangle (box1.north east);
\end{tikzpicture}}
\,+\,
2\,{\begin{tikzpicture}[blue,baseline = -.7ex]
\begin{scope}[local bounding box=box1]
\node[circle,draw,fill,inner sep=0pt,minimum size=3pt] (1) at (0,0) {};
\draw[white] (1) circle(0.2);
\end{scope}
\begin{scope}[local bounding box=box2]
\node[circle,draw,fill,inner sep=0pt,minimum size=3pt] (2) at (0.5,0) {};
\draw[white] (2) circle(0.2);
\end{scope}
\draw[rounded corners] (box1.south west) rectangle (box1.north east);
\draw[rounded corners] (box2.south west) rectangle (box2.north east);
\end{tikzpicture}}
\,+\,
4\,{\begin{tikzpicture}[blue,baseline = -.7ex]
\begin{scope}[local bounding box=box1]
\node[circle,draw,fill,inner sep=0pt,minimum size=3pt] (1) at (0.5,0) {};
\draw[white] (1) circle(0.2);
\end{scope}
\node[circle,draw,fill,inner sep=0pt,minimum size=3pt] (2) at (0,0) {};
\draw[rounded corners] (box1.south west) rectangle (box1.north east);
\path[-] (2) edge (box1);
\end{tikzpicture}}
\,+\,
2\,
{\begin{tikzpicture}[blue,baseline = -.6ex]
\begin{scope}[local bounding box=box1]
\node[circle,draw,fill,inner sep=0pt,minimum size=3pt] (1) at (0,0) {};
\node[circle,draw,fill,inner sep=0pt,minimum size=3pt] (2) at (0.3,0) {};
\draw[white] (1) circle(0.2);
\draw[white] (2) circle(0.2);
\path[-] (1) edge (2);
\end{scope}
\draw[rounded corners] (box1.south west) rectangle (box1.north east);
\end{tikzpicture}}\,.
\label{eq:comp-gr-2(..)}
\end{eqnarray}
This naturally generalizes to products of more than two functionals; we skip the details for conciseness.
\smallskip\item Given functionals $\Psi:\triangle^{n+1}\times\Pc(\Xd)\to\R$ and $\Theta:\triangle^{m+1}\times\Pc(\Xd)\to\R$, the time integral of their symmetrized product can be factorized as
\[\binom{n+m}{n}\int_{\triangle^{n+m}\times0}(\,\Psi~\Theta\,)~=~\Big(\int_{\triangle^{n}\times0}\Psi\Big)\Big(\int_{\triangle^{m}\times0}\Theta\Big).\qedhere\]
\end{enumerate}
\end{lem}

\begin{proof}
All four items are direct consequences of the definitions.
First, the definition of the basepoint in the graphical notation means
\[\RS{p}_{[\Phi^k]}\big((t,s),\mu\big)
\,=\,\Uc_{\Phi^k}^{(1)}\big((t,s),\mu\big)
\,=\,(\Phi^k)\big(t,m(t-s,\mu)\big)
\,=\,\big(\Phi\big(t,m(t-s,\mu)\big)\big)^k,\]
which proves item~(i). Second, recalling the semigroup property
\[ m(\tau_{n-1} - \tau_n, m(\tau_n - s, \mu)) = m(\tau_{n-1}-s, \mu), \]
we find from item~(i) that $\RS{p}_{[\Phi^k]}$ satisfies \eqref{eq:semigroup_prop_L_graph}. From the definitions \eqref{eq:def-round-edge} and \eqref{eq:def-straight-edge}, this property is clearly conserved when surrounding an L-graph satisfying \eqref{eq:semigroup_prop_L_graph} by a round edge, and when connecting two such L-graphs by a straight edge, which yields item~(ii). The proof of items~(iii) and~(iv) is divided into the following two steps.

\medskip
\step1 Proof of~(iii).\\
Given smooth functionals $\Psi:\triangle^{n+1}\times\Pc(\Xd)\to\R$ and $\Theta:\triangle^{m+1}\times\Pc(\Xd)\to\R$,
and given $i_1<\ldots<i_n$ and $j_1<\ldots<j_m$ with $\{i_1,\ldots,i_n\}\cup\{j_1,\ldots,j_m\}=\llbracket p\rrbracket$, the definition of the round edge notation yields for all $(t,\tau)\in\triangle^{p+2}$ and $\mu\in\Pc(\Xd)$,
\begin{multline*}
\begin{tikzpicture}[blue,baseline = -.7ex]
\begin{scope}[local bounding box=foo]
\node (1) at (0,0){\text{\color{black}$\Psi_{\langle i_1,\ldots,i_n,p+1\rangle}$}};
\node (2) at (2.3,0){\text{\color{black}$\Theta_{\langle j_1,\ldots,j_m,p+1\rangle}$}};
\end{scope}
\draw[rounded corners] (foo.south west) rectangle (foo.north east);
\end{tikzpicture}
_{\langle p+2\rangle}\big((t,\tau),\mu\big)
\\
\,=\,
\bigg(\int_{\Xd}\Tr\Big[a_0\,\partial_\mu^2\Big(\Psi\big((t,\tau_{i_1},\ldots,\tau_{i_n},\tau_{p+1}),\nu\big)\\[-3mm]
\times\Theta\big((t,\tau_{ j_1},\ldots,\tau_{j_m},\tau_{p+1}),\nu\big)\Big)(z,z)\Big]\,\nu(\ddr z)\bigg)\bigg|_{\nu=m(\tau_{p+1}-\tau_{p+2},\mu)},
\end{multline*}
and thus, using the chain rule for the L-derivative,
\begin{eqnarray*}
\lefteqn{\begin{tikzpicture}[blue,baseline = -.7ex]
\begin{scope}[local bounding box=foo]
\node (1) at (0,0){\text{\color{black}$\Psi_{\langle i_1,\ldots,i_n,p+1\rangle}$}};
\node (2) at (2.3,0){\text{\color{black}$\Theta_{\langle j_1,\ldots,j_m,p+1\rangle}$}};
\end{scope}
\draw[rounded corners] (foo.south west) rectangle (foo.north east);
\end{tikzpicture}
_{\langle p+2\rangle}\big((t,\tau),\mu\big)
}\\
&=&
\bigg(\int_{\Xd}\Tr\Big[a_0\,(\partial_\mu^2\Psi)\big((t,\tau_{i_1},\ldots,\tau_{i_n},\tau_{p+1}),\nu\big)(z,z)\Big]\\[-3mm]
&&\hspace{3cm}\times\Theta\big((t,\tau_{ j_1},\ldots,\tau_{j_m},\tau_{p+1}),\nu\big)\,\nu(\ddr z)\bigg)\bigg|_{\nu=m(\tau_{p+1}-\tau_{p+2},\mu)}
\\
&+&\bigg(\int_{\Xd}\Tr\Big[a_0\,(\partial_\mu^2\Theta)\big((t,\tau_{j_1},\ldots,\tau_{j_m},\tau_{p+1}),\nu\big)(z,z)\Big]\\[-3mm]
&&\hspace{3cm}\times\Psi\big((t,\tau_{i_1},\ldots,\tau_{i_n},\tau_{p+1}),\nu\big)\,\nu(\ddr z)\bigg)\bigg|_{\nu=m(\tau_{p+1}-\tau_{p+2},\mu)}
\\
&+&2\,\bigg(\int_{\Xd}\Big((\partial_\mu\Psi)\big((t,\tau_{i_1},\ldots,\tau_{i_n},\tau_{p+1}),\nu\big)(z)\\[-3mm]
&&\hspace{3cm}\cdot a_0(\partial_\mu\Theta)\big((t,\tau_{j_1},\ldots,\tau_{j_m},\tau_{p+1}),\nu\big)(z)\Big]\,\nu(\ddr z)\bigg)\bigg|_{\nu=m(\tau_{p+1}-\tau_{p+2},\mu)},
\end{eqnarray*}
or equivalently,
\begin{eqnarray*}
\lefteqn{\begin{tikzpicture}[blue,baseline = -.7ex]
\begin{scope}[local bounding box=foo]
\node (1) at (0,0){\text{\color{black}$\Psi_{\langle i_1,\ldots,i_n,p+1\rangle}$}};
\node (2) at (2.5,0){\text{\color{black}$\Theta_{\langle j_1,\ldots,j_m,p+1\rangle}$}};
\end{scope}
\draw[rounded corners] (foo.south west) rectangle (foo.north east);
\end{tikzpicture}
_{\langle p+2\rangle}\big((t,\tau,\tau_{p+2}),\mu\big)
}\\
&=&\Theta\big((t,\tau_{j_1},\ldots,\tau_{j_m},\tau_{p+1}),m(\tau_{p+1}-\tau_{p+2},\mu)\big)
\\
&&\hspace{2cm}\times\bigg(\int_{\Xd}\Tr\Big[a_0\,(\partial_\mu^2\Psi)\big((t,\tau_{i_1},\ldots,\tau_{i_n},\tau_{p+1}),\nu\big)(z,z)\Big]\,\nu(\ddr z)\bigg)\bigg|_{\nu=m(\tau_{p+1}-\tau_{p+2},\mu)}
\\
&+&\Psi\big((t,\tau_{i_1},\ldots,\tau_{i_n},\tau_{p+1}),m(\tau_{p+1}-\tau_{p+2},\mu)\big)\\
&&\hspace{2cm}\times\bigg(\int_{\Xd}\Tr\Big[a_0\,(\partial_\mu^2\Theta)\big((t,\tau_{j_1},\ldots,\tau_{j_m},\tau_{p+1}),\nu\big)(z,z)\Big]\,\nu(\ddr z)\bigg)\bigg|_{\nu=m(\tau_{p+1}-\tau_{p+2},\mu)}
\\
&+&2\,\bigg(\int_{\Xd}\Big((\partial_\mu\Psi)\big((t,\tau_{i_1},\ldots,\tau_{i_n},\tau_{p+1}),\nu\big)(z)\\[-3mm]
&&\hspace{3cm}\cdot a_0(\partial_\mu\Theta)\big((t,\tau_{j_1},\ldots,\tau_{j_m},\tau_{p+1}),\nu\big)(z)\Big]\,\nu(\ddr z)\bigg)\bigg|_{\nu=m(\tau_{p+1}-\tau_{p+2},\mu)}.
\end{eqnarray*}
Further using item~(ii), the identity~\eqref{eq:decomp-round-PsiTheta} follows.

\medskip
\step2 Proof of~(iv).\\
By definition of the symmetrized product, we have for all $(t,\tau)\in\triangle^{n+m}$ and $0<s<\tau_{n+m}$,
\begin{multline*}
(\Psi\,\Theta)\big((t,\tau,s),\mu\big)\,=\,\binom{n+m}{n}^{-1}
\sum_{\sigma\in\Sym(n+m)}
\mathds1_{\sigma(1)<\ldots<\sigma(n)}\,\mathds1_{\sigma(n+1)<\ldots<\sigma(n+m)}\\
\times\Psi\big((t,\tau_{\sigma(1)},\ldots,\tau_{\sigma(n)},s),\mu\big)\,\Theta\big((t,\tau_{\sigma(n+1)},\ldots,\tau_{\sigma(n+m)},s),\mu\big),
\end{multline*}
and thus, taking the time integral,
\begin{multline*}
\int_{\triangle^{n+m}_t}(\Psi\,\Theta)\,\big((t,\tau,0),\mu\big)\,\ddr\tau\,=\,\binom{n+m}{n}^{-1}
\sum_{\substack{\sigma\in\Sym(n+m)}}
\mathds1_{\sigma(1)<\ldots<\sigma(n)}\,\mathds1_{\sigma(n+1)<\ldots<\sigma(n+m)}\\
\times\int_{\triangle_t^{n+m}}\Psi\big((t,\tau_{\sigma(1)},\ldots,\tau_{\sigma(n)},0),\mu\big)\,\Theta\big((t,\tau_{\sigma(n+1)},\ldots,\tau_{\sigma(n+m)},0),\mu\big)\,\ddr\tau.
\end{multline*}
Noting that the sum over permutations allows to reconstruct the full product of integrals, we obtain
\begin{equation*}
\int_{\triangle^{n+m}_t}(\Psi\,\Theta)\,\big((t,\tau,0),\mu\big)\,\ddr\tau\,=\,\binom{n+m}{n}^{-1}
\Big(\int_{\triangle_t^{n}}\Psi\big((t,\tau,0),\mu\big)\,\ddr\tau\Big)
\Big(\int_{\triangle_t^m}
\Theta\big((t,\tau,0),\mu\big)\,\ddr\tau\Big),
\end{equation*}
which is precisely the statement of item~(iv).
\end{proof}

\subsection{Graphical representation of Brownian cumulants}\label{sec:rep-Gausscum}
Starting from Proposition~\ref{prop:decomp-exp-B} in form of~\eqref{eq:decomp-exp-B-reform-graph}, we can use the above graphical notation to easily compute cumulants of functionals of the empirical measure along the particle dynamics. We first illustrate this by a direct computation of the leading contribution to the variance and to the third cumulant.

\begin{lem}
\label{lem:first_brownian_moments}
Given a smooth functional $\Phi:\Pc(\Xd)\to\R$, we can represent the variance and the third cumulant along the particle dynamics as
\begin{eqnarray*}
\Var_B[\Phi(\mu_t^N)]&=&\frac1{N}\int_{\triangle_t\times0}
{\begin{tikzpicture}[blue,baseline = -.7ex]
\node[circle,draw,fill,inner sep=0pt,minimum size=3pt] (1) at (0,0) {};
\node[circle,draw,fill,inner sep=0pt,minimum size=3pt] (2) at (0.3,0) {};
\path[-] (1) edge (2);
\end{tikzpicture}}
~(\mu_0^N)
~+~
\frac{E_{N,2}}{(2N)^2},\\
\kappa_B^3\big[\Phi(\mu_t^N)\big]&=&
\frac{3}{N^2}\,\int_{\triangle_t^2\times0}{\begin{tikzpicture}[blue,baseline = -.6ex]
\node[circle,draw,fill,inner sep=0pt,minimum size=3pt] (1) at (0,0) {};
\begin{scope}[local bounding box=box1]
\node[circle,draw,fill,inner sep=0pt,minimum size=3pt] (2) at (0.5,-0.2) {};
\node[circle,draw,fill,inner sep=0pt,minimum size=3pt] (3) at (0.5,0.2) {};
\draw[white] (2) circle(0.2);
\draw[white] (3) circle(0.2);
\path[-] (2) edge (3);
\end{scope}
\draw[rounded corners,dotted] (box1.south west) rectangle (box1.north east);
\path[-] (1) edge (box1);
\end{tikzpicture}}\,(\mu_0^N)
~+~\frac{E_{N,3}}{(2N)^3},
\end{eqnarray*}
where the error terms $E_{N,2},E_{N,3}$ are given explicitly by
\begin{eqnarray*}
E_{N,2}&:=&
A_{2,2}
-
2A_{1,2}\,\E_B[\Phi(\mu_t^N)]
+
\frac{(A_{1,2})^2}{(2N)^2}
-
\Big(\int_{\triangle_t\times0}
{\begin{tikzpicture}[blue,baseline = -.7ex]
\begin{scope}[local bounding box=box1]
\node[circle,draw,fill,inner sep=0pt,minimum size=3pt] (1) at (0,0) {};
\draw[white] (1) circle(0.2);
\end{scope}
\draw[rounded corners] (box1.south west) rectangle (box1.north east);
\end{tikzpicture}}
\,(\mu_0^N)\Big)^2,\\
E_{N,3}&:=&
A_{3,3} - 3 A_{2,2} \, \E_B[\Phi(\mu^N_t)] - A_{1,3}\,\E_B[\Phi(\mu^N_t)^2] + 4 A_{1,3} \, \E_B[\Phi(\mu_t^N)]^2 - \frac{2}{(2N)^3} \, A_{1,3}^2 \E_B[\Phi(\mu_t^N)]  \\
&&
\,-\,\int_{\triangle_t^3\times0} \bigg( 4 \, 
{\begin{tikzpicture}[blue,baseline = -.7ex]
\node[circle,draw,fill,inner sep=0pt,minimum size=3pt] (1) at (-1,0) {};
\begin{scope}[local bounding box=box1]
\node[circle,draw,fill,inner sep=0pt,minimum size=3pt] (1) at (-0.6,0) {};
\draw[white] (1) circle(0.2);
\end{scope}
\begin{scope}[local bounding box=box2]
\node (2) at (0,0) {
{\begin{tikzpicture}[blue,baseline = -.7ex]
\begin{scope}[local bounding box=box3]
\node[circle,draw,fill,inner sep=0pt,minimum size=3pt] (3) at (0,0) {};
\draw[white] (3) circle(0.2);
\end{scope}
\draw[rounded corners] (box3.south west) rectangle (box3.north east);
\end{tikzpicture}}
};
\end{scope}
\draw[rounded corners] (box1.south west) rectangle (box1.north east);
\draw[rounded corners] (box2.south west) rectangle (box2.north east);
\end{tikzpicture}}
\, + \,  
{\begin{tikzpicture}[blue,baseline = .2ex]
\begin{scope}[local bounding box=box1]
\node[circle,draw,fill,inner sep=0pt,minimum size=3pt] (1) at (0,0) {};
\draw[white] (1) circle(0.2);
\end{scope}
\begin{scope}[local bounding box=box2]
\node[circle,draw,fill,inner sep=0pt,minimum size=3pt] (2) at (0.5,0) {};
\draw[white] (2) circle(0.2);
\end{scope}
\begin{scope}[local bounding box=box3]
\node[circle,draw,fill,inner sep=0pt,minimum size=3pt] (3) at (0.25,0.4) {};
\draw[white] (3) circle(0.2);
\end{scope}
\draw[rounded corners] (box1.south west) rectangle (box1.north east);
\draw[rounded corners] (box2.south west) rectangle (box2.north east);
\draw[rounded corners] (box3.south west) rectangle (box3.north east);
\end{tikzpicture}}
\,+\,
6\,
{\begin{tikzpicture}[blue,baseline = -.6ex]
\begin{scope}[local bounding box=box1]
\node[circle,draw,fill,inner sep=0pt,minimum size=3pt] (1) at (0,0) {};
\draw[white] (1) circle(0.2);
\end{scope}
\begin{scope}[local bounding box=box2]
\node[circle,draw,fill,inner sep=0pt,minimum size=3pt] (2) at (0.5,-0.15) {};
\node[circle,draw,fill,inner sep=0pt,minimum size=3pt] (3) at (0.5,0.15) {};
\draw[white] (2) circle(0.2);
\draw[white] (3) circle(0.2);
\path[-] (3) edge (2);
\end{scope}
\draw[rounded corners] (box1.south west) rectangle (box1.north east);
\draw[rounded corners] (box2.south west) rectangle (box2.north east);
\end{tikzpicture}}
\,+\;
12\,
{\begin{tikzpicture}[blue,baseline = -.6ex]
\begin{scope}[local bounding box=box1]
\node[circle,draw,fill,inner sep=0pt,minimum size=3pt] (1) at (0,0) {};
\draw[white] (1) circle(0.2);
\end{scope}
\begin{scope}[local bounding box=box2]
\node[circle,draw,fill,inner sep=0pt,minimum size=3pt] (2) at (0.5,-0.2) {};
\draw[white] (2) circle(0.2);
\end{scope}
\node[circle,draw,fill,inner sep=0pt,minimum size=3pt] (3) at (0.5,0.2) {};
\draw[rounded corners] (box1.south west) rectangle (box1.north east);
\draw[rounded corners] (box2.south west) rectangle (box2.north east);
\path[-] (3) edge (box2);
\end{tikzpicture}}
\,+\;
6\,
{\begin{tikzpicture}[blue,baseline = -.6ex]
\begin{scope}[local bounding box=box1]
\node (1) at (0,0) {
{\begin{tikzpicture}[blue,baseline = -.6ex]
\begin{scope}[local bounding box=box1]
\node[circle,draw,fill,inner sep=0pt,minimum size=3pt] (1) at (0,0) {};
\draw[white] (1) circle(0.2);
\end{scope}
\draw[rounded corners] (box1.south west) rectangle (box1.north east);
\end{tikzpicture}}
};
\end{scope}
\node[circle,draw,fill,inner sep=0pt,minimum size=3pt] (2) at (0.5,-0.15) {};
\node[circle,draw,fill,inner sep=0pt,minimum size=3pt] (3) at (0.5,0.15) {};
\draw[rounded corners] (box1.south west) rectangle (box1.north east);
\path[-] (3) edge (2);
\end{tikzpicture}}
\bigg)
\,(\mu_0^N)
\\
&&
\,-\,
\frac1{2N}
\int_{\triangle_t^4\times0}\bigg(
\,
{\begin{tikzpicture}[blue,baseline =-.6ex]
\begin{scope}[local bounding box=box1]
\node[circle,draw,fill,inner sep=0pt,minimum size=3pt] (1) at (0,-0.25) {};
\draw[white] (1) circle(0.2);
\end{scope}
\begin{scope}[local bounding box=box2]
\node (2) at (0.6,0) {
{\begin{tikzpicture}[blue,baseline = .2ex]
\begin{scope}[local bounding box=box0]
\node[circle,draw,fill,inner sep=0pt,minimum size=3pt] (0) at (0,0) {};
\draw[white] (0) circle(0.2);
\end{scope}
\draw[rounded corners] (box0.south west) rectangle (box0.north east);
\end{tikzpicture}}
};
\end{scope}
\begin{scope}[local bounding box=box3]
\node[circle,draw,fill,inner sep=0pt,minimum size=3pt] (3) at (0,0.25) {};
\draw[white] (3) circle(0.2);
\end{scope}
\draw[rounded corners] (box1.south west) rectangle (box1.north east);
\draw[rounded corners] (box2.south west) rectangle (box2.north east);
\draw[rounded corners] (box3.south west) rectangle (box3.north east);
\end{tikzpicture}}
\,+\,
2\,
{\begin{tikzpicture}[blue,baseline =-.6ex]
\node[circle,draw,fill,inner sep=0pt,minimum size=3pt] (1) at (0,0) {};
\begin{scope}[local bounding box=box2]
\node (2) at (0.5,0) {
{\begin{tikzpicture}[blue,baseline = .2ex]
\begin{scope}[local bounding box=box0]
\node[circle,draw,fill,inner sep=0pt,minimum size=3pt] (0) at (0,0) {};
\draw[white] (0) circle(0.2);
\end{scope}
\draw[rounded corners] (box0.south west) rectangle (box0.north east);
\end{tikzpicture}}
};
\end{scope}
\begin{scope}[local bounding box=box3]
\node (3) at (1.3,0) {
{\begin{tikzpicture}[blue,baseline = .2ex]
\begin{scope}[local bounding box=box4]
\node[circle,draw,fill,inner sep=0pt,minimum size=3pt] (4) at (0,0) {};
\draw[white] (4) circle(0.2);
\end{scope}
\draw[rounded corners] (box4.south west) rectangle (box4.north east);
\end{tikzpicture}}
};
\end{scope}
\draw[rounded corners] (box2.south west) rectangle (box2.north east);
\draw[rounded corners] (box3.south west) rectangle (box3.north east);
\end{tikzpicture}}
\,+\,
12\,
{\begin{tikzpicture}[blue,baseline = -.6ex]
\begin{scope}[local bounding box=box1]
\node (1) at (0,0) {
{\begin{tikzpicture}[blue,baseline = -.6ex]
\begin{scope}[local bounding box=box0]
\node[circle,draw,fill,inner sep=0pt,minimum size=3pt] (0) at (0,0) {};
\draw[white] (0) circle(0.2);
\end{scope}
\draw[rounded corners] (box0.south west) rectangle (box0.north east);
\end{tikzpicture}}
};
\end{scope}
\begin{scope}[local bounding box=box2]
\node[circle,draw,fill,inner sep=0pt,minimum size=3pt] (2) at (0.6,-0.2) {};
\draw[white] (2) circle(0.2);
\end{scope}
\node[circle,draw,fill,inner sep=0pt,minimum size=3pt] (3) at (0.6,0.2) {};
\draw[rounded corners] (box1.south west) rectangle (box1.north east);
\draw[rounded corners] (box2.south west) rectangle (box2.north east);
\path[-] (3) edge (box2);
\end{tikzpicture}}
\,+\;
6\,
{\begin{tikzpicture}[blue,baseline = -.6ex]
\begin{scope}[local bounding box=box1]
\node (1) at (0,0) {
{\begin{tikzpicture}[blue,baseline = -.6ex]
\begin{scope}[local bounding box=box1]
\node[circle,draw,fill,inner sep=0pt,minimum size=3pt] (1) at (0,0) {};
\draw[white] (1) circle(0.2);
\end{scope}
\draw[rounded corners] (box1.south west) rectangle (box1.north east);
\end{tikzpicture}}
};
\end{scope}
\begin{scope}[local bounding box=box2]
\node[circle,draw,fill,inner sep=0pt,minimum size=3pt] (2) at (0.6,-0.15) {};
\node[circle,draw,fill,inner sep=0pt,minimum size=3pt] (3) at (0.6,0.15) {};
\draw[white] (2) circle(0.2);
\draw[white] (3) circle(0.2);
\path[-] (3) edge (2);
\end{scope}
\draw[rounded corners] (box1.south west) rectangle (box1.north east);
\draw[rounded corners] (box2.south west) rectangle (box2.north east);
\end{tikzpicture}}
\bigg)(\mu_0^N)
\\ 
&&
\,+\, \E_B[\Phi(\mu_t^N)] \Big( \int_{\triangle^3_t \times 0} 4 \, {\begin{tikzpicture}[blue,baseline = -.7ex]
\begin{scope}[local bounding box=box1]
\node[circle,draw,fill,inner sep=0pt,minimum size=3pt] (1) at (-0.6,0) {};
\draw[white] (1) circle(0.2);
\end{scope}
\begin{scope}[local bounding box=box2]
\node (2) at (0,0) {
{\begin{tikzpicture}[blue,baseline = -.7ex]
\begin{scope}[local bounding box=box3]
\node[circle,draw,fill,inner sep=0pt,minimum size=3pt] (3) at (0,0) {};
\draw[white] (3) circle(0.2);
\end{scope}
\draw[rounded corners] (box3.south west) rectangle (box3.north east);
\end{tikzpicture}}
};
\end{scope}
\draw[rounded corners] (box1.south west) rectangle (box1.north east);
\draw[rounded corners] (box2.south west) rectangle (box2.north east);
\end{tikzpicture}} \,(\mu^N_0) 
\, + \, 
\frac1{N} \,
{\begin{tikzpicture}[blue,baseline =-.6ex]
\begin{scope}[local bounding box=box2]
\node (2) at (0.5,0) {
{\begin{tikzpicture}[blue,baseline = .2ex]
\begin{scope}[local bounding box=box0]
\node[circle,draw,fill,inner sep=0pt,minimum size=3pt] (0) at (0,0) {};
\draw[white] (0) circle(0.2);
\end{scope}
\draw[rounded corners] (box0.south west) rectangle (box0.north east);
\end{tikzpicture}}
};
\end{scope}
\begin{scope}[local bounding box=box3]
\node (3) at (1.3,0) {
{\begin{tikzpicture}[blue,baseline = .2ex]
\begin{scope}[local bounding box=box4]
\node[circle,draw,fill,inner sep=0pt,minimum size=3pt] (4) at (0,0) {};
\draw[white] (4) circle(0.2);
\end{scope}
\draw[rounded corners] (box4.south west) rectangle (box4.north east);
\end{tikzpicture}}
};
\end{scope}
\draw[rounded corners] (box2.south west) rectangle (box2.north east);
\draw[rounded corners] (box3.south west) rectangle (box3.north east);
\end{tikzpicture}} \,(\mu^N_0) \bigg) 
\end{eqnarray*}
where for shortness we have defined
\[A_{k,m}\,:=\,\int_{\triangle_t^{m}}\E_B\Big[\,
{\begin{tikzpicture}[blue,baseline = -1.25ex]
\begin{scope}[local bounding box=box1]
\node[circle,draw,fill,inner sep=0pt,minimum size=3pt] (1) at (0,0) {};
\draw[white] (1) circle(0.2);
\node (2) at (0.15,-0.2) {{\tiny\color{black}$[\Phi^k]$}};
\end{scope}
\draw[rounded corners] (box1.south west) rectangle (box1.north east);
\end{tikzpicture}}
_{(m)}\!\big((t,\tau,\tau_m),\mu_{\tau_m}^N\big)\Big]\,\ddr\tau.\qedhere\]
\end{lem}

\begin{proof}
We split the proof into three steps.

\medskip
\step1 Formula for variance.\\
Using Proposition~\ref{prop:decomp-exp-B} in form of~\eqref{eq:decomp-exp-B-reform-graph} to accuracy $O(N^{-2})$, we find
\begin{equation}\label{eq:expand-Phi2-prop}
\E_B[\Phi(\mu_t^N)^2]
~=~
{\begin{tikzpicture}[blue,baseline = -.7ex]
\node[circle,draw,fill,inner sep=0pt,minimum size=3pt] (1) at (0,0) {};
\node[circle,draw,fill,inner sep=0pt,minimum size=3pt] (2) at (0.3,0) {};
\draw[white] (1) circle(0.2);
\draw[white] (2) circle(0.2);
\end{tikzpicture}}\!
\big((t,0),\mu_0^N\big)
+
\frac1{2N}\int_{\triangle_t\times0}
{\begin{tikzpicture}[blue,baseline = -.7ex]
\begin{scope}[local bounding box=box1]
\node[circle,draw,fill,inner sep=0pt,minimum size=3pt] (1) at (0,0) {};
\node[circle,draw,fill,inner sep=0pt,minimum size=3pt] (2) at (0.3,0) {};
\draw[white] (1) circle(0.2);
\draw[white] (2) circle(0.2);
\end{scope}
\draw[rounded corners] (box1.south west) rectangle (box1.north east);
\end{tikzpicture}}
\,(\mu_0^N)
~+~\frac{A_{2,2}}{(2N)^{2}},
\end{equation}
with the notation for $A_{2,2}$ in the statement.
By Lemma~\ref{lem:comput-graph}(iii) in form of~\eqref{eq:comp-gr-1(..)},
we get
\begin{multline}\label{eq:expand-Phi2-appllem}
\E_B[\Phi(\mu_t^N)^2]
~=~
\Big(\!{\begin{tikzpicture}[blue,baseline = -.7ex]
\node[circle,draw,fill,inner sep=0pt,minimum size=3pt] (1) at (0,0) {};
\draw[white] (1) circle(0.2);
\end{tikzpicture}}
\big((t,0),\mu_0^N\big)\Big)^2
\\
+
\frac2{2N}\Big(\!{\begin{tikzpicture}[blue,baseline = -.7ex]
\node[circle,draw,fill,inner sep=0pt,minimum size=3pt] (1) at (0,0) {};
\draw[white] (1) circle(0.2);
\end{tikzpicture}}
\big((t,0),\mu_0^N\big)\Big)\Big(\int_{\triangle_t\times0}
{\begin{tikzpicture}[blue,baseline = -.7ex]
\begin{scope}[local bounding box=box1]
\node[circle,draw,fill,inner sep=0pt,minimum size=3pt] (1) at (0,0) {};
\draw[white] (1) circle(0.2);
\end{scope}
\draw[rounded corners] (box1.south west) rectangle (box1.north east);
\end{tikzpicture}}
\,(\mu_0^N)\Big)
+
\frac2{2N}\int_{\triangle_t\times0}
{\begin{tikzpicture}[blue,baseline = -.7ex]
\node[circle,draw,fill,inner sep=0pt,minimum size=3pt] (1) at (0,0) {};
\node[circle,draw,fill,inner sep=0pt,minimum size=3pt] (2) at (0.3,0) {};
\path[-] (1) edge (2);
\end{tikzpicture}}
~(\mu_0^N)
+\frac{A_{2,2}}{(2N)^{2}}.
\end{multline}
On the other hand, using again Proposition~\ref{prop:decomp-exp-B} in form of~\eqref{eq:decomp-exp-B-reform-graph} to accuracy $O(N^{-2})$, we also have
\begin{equation*}
\E_B[\Phi(\mu_t^N)]\,=\,
{\begin{tikzpicture}[blue,baseline = -.7ex]
\node[circle,draw,fill,inner sep=0pt,minimum size=3pt] (1) at (0,0) {};
\draw[white] (1) circle(0.2);
\end{tikzpicture}}
\!\big((t,0),\mu_0^N\big)
+
\frac1{2N}\int_{\triangle_t\times0}
{\begin{tikzpicture}[blue,baseline = -.7ex]
\begin{scope}[local bounding box=box1]
\node[circle,draw,fill,inner sep=0pt,minimum size=3pt] (1) at (0,0) {};
\draw[white] (1) circle(0.2);
\end{scope}
\draw[rounded corners] (box1.south west) rectangle (box1.north east);
\end{tikzpicture}}
\,(\mu_0^N)
~+~\frac{A_{1,2}}{(2N)^{2}}.
\end{equation*}
Taking the square of this identity, and comparing it to~\eqref{eq:expand-Phi2-appllem}, the formula for the variance follows after straightforward simplifications.

\medskip
\step2 Next-order formula for variance.\\
Before turning to the third cumulant, we expand the formula for the variance to the next order.
Instead of~\eqref{eq:expand-Phi2-prop}, we start from Proposition~\ref{prop:decomp-exp-B} with accuracy $O(N^{-3})$, in form of
\begin{equation*}
\E_B[\Phi(\mu_t^N)^2]
~=~
{\begin{tikzpicture}[blue,baseline = -.7ex]
\node[circle,draw,fill,inner sep=0pt,minimum size=3pt] (1) at (0,0) {};
\node[circle,draw,fill,inner sep=0pt,minimum size=3pt] (2) at (0.3,0) {};
\draw[white] (1) circle(0.2);
\draw[white] (2) circle(0.2);
\end{tikzpicture}}\!
\big((t,0),\mu_0^N\big)
+
\frac1{2N}\int_{\triangle_t\times0}
{\begin{tikzpicture}[blue,baseline = -.7ex]
\begin{scope}[local bounding box=box1]
\node[circle,draw,fill,inner sep=0pt,minimum size=3pt] (1) at (0,0) {};
\node[circle,draw,fill,inner sep=0pt,minimum size=3pt] (2) at (0.3,0) {};
\draw[white] (1) circle(0.2);
\draw[white] (2) circle(0.2);
\end{scope}
\draw[rounded corners] (box1.south west) rectangle (box1.north east);
\end{tikzpicture}}
\,(\mu_0^N)
+
\frac1{(2N)^2}\int_{\triangle_t^2\times0}
{\begin{tikzpicture}[blue,baseline = -.7ex]
\begin{scope}[local bounding box=box1]
\node (1) at (0,0) {
{\begin{tikzpicture}[blue,baseline = -.7ex]
\begin{scope}[local bounding box=box1]
\node[circle,draw,fill,inner sep=0pt,minimum size=3pt] (1) at (0,0) {};
\node[circle,draw,fill,inner sep=0pt,minimum size=3pt] (2) at (0.3,0) {};
\draw[white] (1) circle(0.2);
\draw[white] (2) circle(0.2);
\end{scope}
\draw[rounded corners] (box1.south west) rectangle (box1.north east);
\end{tikzpicture}}
};
\end{scope}
\draw[rounded corners] (box1.south west) rectangle (box1.north east);
\end{tikzpicture}}
\,(\mu_0^N)
\\
+
\frac{A_{2,3}}{(2N)^{3}}.
\end{equation*}
Then further appealing to Lemma~\ref{lem:comput-graph}(iii) in form of~\eqref{eq:comp-gr-1(..)}, as well as to Lemma~\ref{lem:comput-graph}(iv),
we obtain, instead of~\eqref{eq:expand-Phi2-appllem},
\begin{multline*}
\E_B[\Phi(\mu_t^N)^2]
~=~
\Big({\begin{tikzpicture}[blue,baseline = -.7ex]
\node[circle,draw,fill,inner sep=0pt,minimum size=3pt] (1) at (0,0) {};
\draw[white] (1) circle(0.2);
\end{tikzpicture}}\!
\big((t,0),\mu_0^N\big)\Big)^2
+
\frac2{2N}\Big(\!{\begin{tikzpicture}[blue,baseline = -.7ex]
\node[circle,draw,fill,inner sep=0pt,minimum size=3pt] (1) at (0,0) {};
\draw[white] (1) circle(0.2);
\end{tikzpicture}}
\big((t,0),\mu_0^N\big)\Big)
\Big(\int_{\triangle_t\times0}
{\begin{tikzpicture}[blue,baseline = -.7ex]
\begin{scope}[local bounding box=box1]
\node[circle,draw,fill,inner sep=0pt,minimum size=3pt] (1) at (0,0) {};
\draw[white] (1) circle(0.2);
\end{scope}
\draw[rounded corners] (box1.south west) rectangle (box1.north east);
\end{tikzpicture}}
\,(\mu_0^N)\Big)
+\frac2{2N}\int_{\triangle_t\times0}
{\begin{tikzpicture}[blue,baseline = -.7ex]
\node[circle,draw,fill,inner sep=0pt,minimum size=3pt] (1) at (0,0) {};
\node[circle,draw,fill,inner sep=0pt,minimum size=3pt] (2) at (0.3,0) {};
\path[-] (1) edge (2);
\end{tikzpicture}}
~(\mu_0^N)
\\
+
\frac2{(2N)^2}
\Big({\begin{tikzpicture}[blue,baseline = -.7ex]
\node[circle,draw,fill,inner sep=0pt,minimum size=3pt] (1) at (0,0) {};
\draw[white] (1) circle(0.2);
\end{tikzpicture}}\!
\big((t,0),\mu_0^N\big)\Big)
\Big(\int_{\triangle_t^2\times0}
{\begin{tikzpicture}[blue,baseline = -.7ex]
\begin{scope}[local bounding box=box1]
\node (1) at (0,0) {
{\begin{tikzpicture}[blue,baseline = -.7ex]
\begin{scope}[local bounding box=box1]
\node[circle,draw,fill,inner sep=0pt,minimum size=3pt] (1) at (0,0) {};
\draw[white] (1) circle(0.2);
\end{scope}
\draw[rounded corners] (box1.south west) rectangle (box1.north east);
\end{tikzpicture}}
};
\end{scope}
\draw[rounded corners] (box1.south west) rectangle (box1.north east);
\end{tikzpicture}}
\,(\mu_0^N)\Big)
+
\frac1{(2N)^2}\Big(\int_{\triangle_t\times0}
{\begin{tikzpicture}[blue,baseline = -.7ex]
\begin{scope}[local bounding box=box1]
\node[circle,draw,fill,inner sep=0pt,minimum size=3pt] (1) at (0,0) {};
\draw[white] (1) circle(0.2);
\end{scope}
\draw[rounded corners] (box1.south west) rectangle (box1.north east);
\end{tikzpicture}}
\,(\mu_0^N)\Big)^2
\\
+
\frac1{(2N)^2}\int_{\triangle_t^2\times0}
\Big(4\,{\begin{tikzpicture}[blue,baseline = -.7ex]
\begin{scope}[local bounding box=box1]
\node[circle,draw,fill,inner sep=0pt,minimum size=3pt] (1) at (0.5,0) {};
\draw[white] (1) circle(0.2);
\end{scope}
\node[circle,draw,fill,inner sep=0pt,minimum size=3pt] (2) at (0,0) {};
\draw[rounded corners] (box1.south west) rectangle (box1.north east);
\path[-] (2) edge (box1);
\end{tikzpicture}}
\,+\,
2\,
{\begin{tikzpicture}[blue,baseline = -.6ex]
\begin{scope}[local bounding box=box1]
\node[circle,draw,fill,inner sep=0pt,minimum size=3pt] (1) at (0,0) {};
\node[circle,draw,fill,inner sep=0pt,minimum size=3pt] (2) at (0.3,0) {};
\draw[white] (1) circle(0.2);
\draw[white] (2) circle(0.2);
\path[-] (1) edge (2);
\end{scope}
\draw[rounded corners] (box1.south west) rectangle (box1.north east);
\end{tikzpicture}}\Big)\,(\mu_0^N)
+
\frac{A_{2,3}}{(2N)^{3}}.
\end{multline*}
On the other hand, using again Proposition~\ref{prop:decomp-exp-B} in form of~\eqref{eq:decomp-exp-B-reform-graph} to accuracy $O(N^{-3})$, we also have
\begin{equation}\label{eq:formula-exp-nextorder}
\E_B[\Phi(\mu_t^N)]\,=\,
{\begin{tikzpicture}[blue,baseline = -.7ex]
\node[circle,draw,fill,inner sep=0pt,minimum size=3pt] (1) at (0,0) {};
\draw[white] (1) circle(0.2);
\end{tikzpicture}}
\!\big((t,0),\mu_0^N\big)
+
\frac1{2N}\int_{\triangle_t\times0}
{\begin{tikzpicture}[blue,baseline = -.7ex]
\begin{scope}[local bounding box=box1]
\node[circle,draw,fill,inner sep=0pt,minimum size=3pt] (1) at (0,0) {};
\draw[white] (1) circle(0.2);
\end{scope}
\draw[rounded corners] (box1.south west) rectangle (box1.north east);
\end{tikzpicture}}
\,(\mu_0^N)
+
\frac1{(2N)^2}\int_{\triangle_t\times0}
{\begin{tikzpicture}[blue,baseline = -.7ex]
\begin{scope}[local bounding box=box1]
\node (1) at (0,0) {
{\begin{tikzpicture}[blue,baseline = -.7ex]
\begin{scope}[local bounding box=box1]
\node[circle,draw,fill,inner sep=0pt,minimum size=3pt] (1) at (0,0) {};
\draw[white] (1) circle(0.2);
\end{scope}
\draw[rounded corners] (box1.south west) rectangle (box1.north east);
\end{tikzpicture}}
};
\end{scope}
\draw[rounded corners] (box1.south west) rectangle (box1.north east);
\end{tikzpicture}}
\,(\mu_0^N)
+\frac{A_{1,3}}{(2N)^{3}}.
\end{equation}
Taking the square of this identity, and comparing it to the previous one for $\E_B[\Phi(\mu_t^N)^2]$, we are led to
\begin{equation}\label{eq:formula-var-nextorder}
\Var_B[\Phi(\mu_t^N)]
~=~
\frac2{2N}\int_{\triangle_t\times0}
{\begin{tikzpicture}[blue,baseline = -.7ex]
\node[circle,draw,fill,inner sep=0pt,minimum size=3pt] (1) at (0,0) {};
\node[circle,draw,fill,inner sep=0pt,minimum size=3pt] (2) at (0.3,0) {};
\path[-] (1) edge (2);
\end{tikzpicture}}
\,(\mu_0^N)
+
\frac1{(2N)^2}\int_{\triangle_t^2\times0}
\Big(4\,{\begin{tikzpicture}[blue,baseline = -.7ex]
\begin{scope}[local bounding box=box1]
\node[circle,draw,fill,inner sep=0pt,minimum size=3pt] (1) at (0.5,0) {};
\draw[white] (1) circle(0.2);
\end{scope}
\node[circle,draw,fill,inner sep=0pt,minimum size=3pt] (2) at (0,0) {};
\draw[rounded corners] (box1.south west) rectangle (box1.north east);
\path[-] (2) edge (box1);
\end{tikzpicture}}
\,+\,
2\,
{\begin{tikzpicture}[blue,baseline = -.6ex]
\begin{scope}[local bounding box=box1]
\node[circle,draw,fill,inner sep=0pt,minimum size=3pt] (1) at (0,0) {};
\node[circle,draw,fill,inner sep=0pt,minimum size=3pt] (2) at (0.3,0) {};
\draw[white] (1) circle(0.2);
\draw[white] (2) circle(0.2);
\path[-] (1) edge (2);
\end{scope}
\draw[rounded corners] (box1.south west) rectangle (box1.north east);
\end{tikzpicture}}\Big)\,(\mu_0^N)
+\frac{R_N}{(2N)^3}.
\end{equation}
where the error term $R_N$ is given by
\begin{equation*}
R_N\,:=\,
A_{2,3}
-2A_{1,3}\,\E_B[\Phi(\mu_t^N)]
+\frac{(A_{1,3})^2}{(2N)^3}
-
2
\int_{\triangle_t^3\times0}
{\begin{tikzpicture}[blue,baseline = -.7ex]
\begin{scope}[local bounding box=box1]
\node[circle,draw,fill,inner sep=0pt,minimum size=3pt] (1) at (-0.6,0) {};
\draw[white] (1) circle(0.2);
\end{scope}
\begin{scope}[local bounding box=box2]
\node (2) at (0,0) {
{\begin{tikzpicture}[blue,baseline = -.7ex]
\begin{scope}[local bounding box=box3]
\node[circle,draw,fill,inner sep=0pt,minimum size=3pt] (3) at (0,0) {};
\draw[white] (3) circle(0.2);
\end{scope}
\draw[rounded corners] (box3.south west) rectangle (box3.north east);
\end{tikzpicture}}
};
\end{scope}
\draw[rounded corners] (box1.south west) rectangle (box1.north east);
\draw[rounded corners] (box2.south west) rectangle (box2.north east);
\end{tikzpicture}}
\,(\mu_0^N)
-
\frac1{2N}
\Big(\int_{\triangle_t^2\times0}
{\begin{tikzpicture}[blue,baseline = -.7ex]
\begin{scope}[local bounding box=box1]
\node (1) at (0,0) {
{\begin{tikzpicture}[blue,baseline = -.7ex]
\begin{scope}[local bounding box=box1]
\node[circle,draw,fill,inner sep=0pt,minimum size=3pt] (1) at (0,0) {};
\draw[white] (1) circle(0.2);
\end{scope}
\draw[rounded corners] (box1.south west) rectangle (box1.north east);
\end{tikzpicture}}
};
\end{scope}
\draw[rounded corners] (box1.south west) rectangle (box1.north east);
\end{tikzpicture}}
\,(\mu_0^N)\Big)^2.
\end{equation*}

\medskip
\step3 Formula for third cumulant.\\
Using Proposition~\ref{prop:decomp-exp-B} in form of~\eqref{eq:decomp-exp-B-reform-graph}, we find
\begin{equation*}
\E_B[\Phi(\mu_t^N)^3]
~=~
{\begin{tikzpicture}[blue,baseline = -.1ex]
\node[circle,draw,fill,inner sep=0pt,minimum size=3pt] (1) at (0,0) {};
\node[circle,draw,fill,inner sep=0pt,minimum size=3pt] (2) at (0.3,0) {};
\node[circle,draw,fill,inner sep=0pt,minimum size=3pt] (2) at (0.15,0.25) {};
\draw[white] (1) circle(0.2);
\draw[white] (2) circle(0.2);
\end{tikzpicture}}
\,\big((t,0),\mu_0^N\big)
+
\frac1{2N}\Big(\int_{\triangle_t\times0}
{\begin{tikzpicture}[blue,baseline = -.1ex]
\begin{scope}[local bounding box=box1]
\node[circle,draw,fill,inner sep=0pt,minimum size=3pt] (1) at (0,0) {};
\node[circle,draw,fill,inner sep=0pt,minimum size=3pt] (2) at (0.3,0) {};
\node[circle,draw,fill,inner sep=0pt,minimum size=3pt] (3) at (0.15,0.25) {};
\draw[white] (1) circle(0.2);
\draw[white] (2) circle(0.2);
\draw[white] (3) circle(0.2);
\end{scope}
\draw[rounded corners] (box1.south west) rectangle (box1.north east);
\end{tikzpicture}}
\,(\mu_0^N)\Big)
+
\frac1{(2N)^2}\Big(\int_{\triangle_t^2\times0}
{\begin{tikzpicture}[blue,baseline = -.4ex]
\begin{scope}[local bounding box=box1]
\node (1) at (0,0) {
{\begin{tikzpicture}[blue,baseline = -.1ex]
\begin{scope}[local bounding box=box1]
\node[circle,draw,fill,inner sep=0pt,minimum size=3pt] (1) at (0,0) {};
\node[circle,draw,fill,inner sep=0pt,minimum size=3pt] (2) at (0.3,0) {};
\node[circle,draw,fill,inner sep=0pt,minimum size=3pt] (3) at (0.15,0.25) {};
\draw[white] (1) circle(0.2);
\draw[white] (2) circle(0.2);
\draw[white] (3) circle(0.2);
\end{scope}
\draw[rounded corners] (box1.south west) rectangle (box1.north east);
\end{tikzpicture}}
};
\end{scope}
\draw[rounded corners] (box1.south west) rectangle (box1.north east);
\end{tikzpicture}}
\,(\mu_0^N)\Big)
+
\frac{A_{3,3}}{(2N)^{3}}.
\end{equation*}
To compute the different right-hand side terms, we appeal to Lemma~\ref{lem:comput-graph} in form of
\begin{eqnarray*}
{\begin{tikzpicture}[blue,baseline =.1ex]
\begin{scope}[local bounding box=box1]
\node[circle,draw,fill,inner sep=0pt,minimum size=3pt] (1) at (0,0) {};
\node[circle,draw,fill,inner sep=0pt,minimum size=3pt] (2) at (0.3,0) {};
\node[circle,draw,fill,inner sep=0pt,minimum size=3pt] (3) at (0.15,0.25) {};
\draw[white] (1) circle(0.2);
\draw[white] (2) circle(0.2);
\draw[white] (3) circle(0.2);
\end{scope}
\draw[rounded corners] (box1.south west) rectangle (box1.north east);
\end{tikzpicture}}
&=&
3\,{\begin{tikzpicture}[blue,baseline = -.7ex]
\begin{scope}[local bounding box=box1]
\node[circle,draw,fill,inner sep=0pt,minimum size=3pt] (1) at (0,0) {};
\draw[white] (1) circle(0.2);
\end{scope}
\node[circle,draw,fill,inner sep=0pt,minimum size=3pt] (2) at (0.38,0.15) {};
\node[circle,draw,fill,inner sep=0pt,minimum size=3pt] (3) at (0.38,-0.15) {};
\draw[rounded corners] (box1.south west) rectangle (box1.north east);
\end{tikzpicture}}
\,+\,
6\,{\begin{tikzpicture}[blue,baseline = -.7ex]
\node[circle,draw,fill,inner sep=0pt,minimum size=3pt] (1) at (0,0) {};
\node[circle,draw,fill,inner sep=0pt,minimum size=3pt] (2) at (0.22,0.15) {};
\node[circle,draw,fill,inner sep=0pt,minimum size=3pt] (3) at (0.22,-0.15) {};
\path[-] (2) edge (3);
\end{tikzpicture}}\,,
\\
{\begin{tikzpicture}[blue,baseline = -.4ex]
\begin{scope}[local bounding box=box0]
\node (1) at (0,0) {
{\begin{tikzpicture}[blue,baseline = -.1ex]
\begin{scope}[local bounding box=box1]
\node[circle,draw,fill,inner sep=0pt,minimum size=3pt] (1) at (0,0) {};
\node[circle,draw,fill,inner sep=0pt,minimum size=3pt] (2) at (0.3,0) {};
\node[circle,draw,fill,inner sep=0pt,minimum size=3pt] (3) at (0.15,0.25) {};
\draw[white] (1) circle(0.2);
\draw[white] (2) circle(0.2);
\draw[white] (3) circle(0.2);
\end{scope}
\draw[rounded corners] (box1.south west) rectangle (box1.north east);
\end{tikzpicture}}
};
\end{scope}
\draw[rounded corners] (box0.south west) rectangle (box0.north east);
\end{tikzpicture}}
&=&
3\,
{\begin{tikzpicture}[blue,baseline =-.5ex]
\begin{scope}[local bounding box=box1]
\node (1) at (0,0) {
\begin{tikzpicture}[blue,baseline = -.1ex]
\begin{scope}[local bounding box=box1]
\node[circle,draw,fill,inner sep=0pt,minimum size=3pt] (1) at (0,0) {};
\draw[white] (1) circle(0.2);
\end{scope}
\draw[rounded corners] (box1.south west) rectangle (box1.north east);
\end{tikzpicture}
};
\end{scope}
\draw[rounded corners] (box1.south west) rectangle (box1.north east);
\node[circle,draw,fill,inner sep=0pt,minimum size=3pt] (2) at (0.55,0.15) {};
\node[circle,draw,fill,inner sep=0pt,minimum size=3pt] (3) at (0.55,-0.15) {};
\end{tikzpicture}}
+ 
3\, 
\big({\begin{tikzpicture}[blue,baseline = -.7ex]
\begin{scope}[local bounding box=box1]
\node[circle,draw,fill,inner sep=0pt,minimum size=3pt] (1) at (0,0) {};
\draw[white] (1) circle(0.2);
\end{scope}
;
\draw[rounded corners] (box1.south west) rectangle (box1.north east);
\end{tikzpicture}}\big)^2
{\begin{tikzpicture}[blue]
\node[circle,draw,fill,inner sep=0pt,minimum size=3pt] (2) at (0,0) {};
\end{tikzpicture}} \, 
+ \,
12\,
{\begin{tikzpicture}[blue,baseline = -.5ex]
\begin{scope}[local bounding box=box1]
\node[circle,draw,fill,inner sep=0pt,minimum size=3pt] (1) at (0.35,-0.15) {};
\draw[white] (1) circle(0.2);
\end{scope}
\node[circle,draw,fill,inner sep=0pt,minimum size=3pt] (2) at (0,0) {};
\node[circle,draw,fill,inner sep=0pt,minimum size=3pt] (3) at (0.35,0.25) {};
\draw[rounded corners] (box1.south west) rectangle (box1.north east);
\path[-] (3) edge (box1);
\end{tikzpicture}}
\,+\,
6\,
{\begin{tikzpicture}[blue,baseline = -.6ex]
\begin{scope}[local bounding box=box1]
\node[circle,draw,fill,inner sep=0pt,minimum size=3pt] (1) at (0,0) {};
\draw[white] (1) circle(0.2);
\end{scope}
\node[circle,draw,fill,inner sep=0pt,minimum size=3pt] (2) at (0.4,-0.2) {};
\node[circle,draw,fill,inner sep=0pt,minimum size=3pt] (3) at (0.4,0.2) {};
\draw[rounded corners] (box1.south west) rectangle (box1.north east);
\path[-] (2) edge (3);
\end{tikzpicture}}
\,+\,
6\,{\begin{tikzpicture}[blue,baseline = -.6ex]
\node[circle,draw,fill,inner sep=0pt,minimum size=3pt] (1) at (0,0) {};
\begin{scope}[local bounding box=box1]
\node[circle,draw,fill,inner sep=0pt,minimum size=3pt] (2) at (0.4,-0.2) {};
\node[circle,draw,fill,inner sep=0pt,minimum size=3pt] (3) at (0.4,0.2) {};
\draw[white] (2) circle(0.2);
\draw[white] (3) circle(0.2);
\path[-] (2) edge (3);
\end{scope}
\draw[rounded corners] (box1.south west) rectangle (box1.north east);
\end{tikzpicture}}
\,+\,
12\,{\begin{tikzpicture}[blue,baseline = -.6ex]
\node[circle,draw,fill,inner sep=0pt,minimum size=3pt] (1) at (0,0) {};
\begin{scope}[local bounding box=box1]
\node[circle,draw,fill,inner sep=0pt,minimum size=3pt] (2) at (0.5,-0.2) {};
\node[circle,draw,fill,inner sep=0pt,minimum size=3pt] (3) at (0.5,0.2) {};
\draw[white] (2) circle(0.2);
\draw[white] (3) circle(0.2);
\path[-] (2) edge (3);
\end{scope}
\draw[rounded corners,dotted] (box1.south west) rectangle (box1.north east);
\path[-] (1) edge (box1);
\end{tikzpicture}} 
\end{eqnarray*}
Inserting these identities into the above,
comparing with~\eqref{eq:formula-exp-nextorder} and~\eqref{eq:formula-var-nextorder}, 
and recalling that the third cumulant is given by
\begin{equation*}
\kappa_3[\Phi(\mu_t^N)]
\,=\,
\E_B[\Phi(\mu_t^N)^3]
-3\Var_B[\Phi(\mu_t^N)]\E_B[\Phi(\mu_t^N)]
-\E_B[\Phi(\mu_t^N)]^3,
\end{equation*}
the formula in the statement follows after straightforward simplifications.
\end{proof}

We now show how the above explicit diagrammatic computation can be pursued systematically to higher orders.
First note that starting from Proposition~\ref{prop:decomp-exp-B} in form of~\eqref{eq:decomp-exp-B-reform-graph} and appealing to the computation rules of Lemma~\ref{lem:comput-graph} to expand each L-graph into a sum of irreducible graphs, the $k$th moment of a smooth functional along the flow can be expanded as a power series in $N^{-1}$, where the term of order $O(N^{-m})$ is given by a sum of all irreducible L-graphs with $k$ vertices and $m$ edges (see indeed~\eqref{eq:mom-expand-L} below).
In the above lemma, for the first cumulants, we manage to capture cancellations showing that the power series for the variance starts at order $O(N^{-1})$ and that the power series for the third cumulant starts at $O(N^{-2})$.
In the following result, we unravel the underlying combinatorial structure and show how cancellations can be systematically captured to higher orders: in a nutshell, the power series for cumulants takes the same form as for moments, except that the sum over irreducible L-graphs is restricted to {\it connected} graphs, cf.~\eqref{eq:cum-expand-L}.
In particular, given that for $m<k-1$ there is no connected L-graph with $k$ vertices and only $m$ edges, we deduce that the power series for the $k$th cumulant must start at order $O(N^{1-k})$.

\begin{prop}\label{prop:expand-Phi-NjLions}
Given a smooth functional $\Phi:\Pc(\Xd)\to\R$, for all $k\ge1$, we can expand as follows the $k$th Brownian moment along the flow: for all $n\ge0$,
\begin{equation}\label{eq:mom-expand-L}
\E_B[\Phi(\mu_t^N)^k]~=~\sum_{m=0}^{n}\frac1{(2N)^m}\sum_{\Psi\in\Gamma(k,m)}\gamma(\Psi)\int_{\triangle_t^{m}\times0}\Psi~+~\frac{R_{N,n}^k(t)}{(2N)^{n+1}},
\end{equation}
where $\Gamma(k,m)$ stands for the set of all (unlabeled) irreducible L-graphs with $k$ vertices and $m$ edges,
where $\gamma$ is some map from $\Gamma(k,m)$ to $\N$,
and where the remainder is given by
\begin{equation}\label{eq:def-RNnt}
R_{N,n}^k(t)\,:=\,\int_{\triangle_t^{n+1}}\E_B\Big[
{\begin{tikzpicture}[blue,baseline = -1.25ex]
\begin{scope}[local bounding box=box1]
\node[circle,draw,fill,inner sep=0pt,minimum size=3pt] (1) at (0,0) {};
\draw[white] (1) circle(0.2);
\node (2) at (0.15,-0.2) {{\tiny\color{black}$[\Phi^{k}]$}};
\end{scope}
\draw[rounded corners] (box1.south west) rectangle (box1.north east);
\end{tikzpicture}}
_{(n+1)}\big((t,\tau,\tau_{n+1}),\mu_{\tau_{n+1}}^N\big)\Big]\,\ddr\tau.
\end{equation}
Moreover, with this notation, the $k$th cumulant can be expanded as follows: for all $n\ge0$,
\begin{equation}\label{eq:cum-expand-L}
\kappa_B^k[\Phi(\mu_t^N)]~=~
\mathds1_{n\ge k-1}\sum_{m=k-1}^{n}\frac1{(2N)^m}\sum_{\Psi\in\Gamma_\circ(k,m)}\gamma(\Psi)\int_{\triangle_t^{m}\times0}\Psi~+~\frac{\tilde R_{N,n}^k(t)}{(2N)^{n+1}},
\end{equation}
where the sum is now restricted to $\Gamma_\circ(k,m)\subset\Gamma(k,m)$, which stands for the subset of all \emph{connected} (unlabeled) irreducible L-graphs with $k$ vertices and $m$ edges,
and where the remainder $\tilde R_{N,n}^k(t)$ can be expressed as a linear combination of elements of the set
\[\bigg\{S_{N,n}^{k-j}(t)\, \prod_{i=1}^j\E_B[\Phi(\mu_t^N)^{i}]^{\alpha_i}\,:\,0\le j\le k~\text{and}~\alpha_1,\ldots,\alpha_j\in\N~\text{with}~\sum_{i=1}^ji\alpha_i=j\bigg\},\]
with bounded coefficients independent of $N,\Phi,\mu_0^N,t$,
where the factors $\{S_{N,n}^{k}(t)\}_k$ are defined by
\begin{equation}\label{eq:def-SNnt}
S_{N,n}^k(t)
\,:=\,
R_{N,n}^k(t)
-\sum_{j=1}^k\binom{k-1}{j-1}\sum_{m=0}^nR_{N,n-m}^{k-j}(t)\sum_{\Psi\in\Gamma_\circ(j,m)}\gamma(\Psi)\int_{\triangle_t^m\times0}\Psi(\mu_0^N).\qedhere
\end{equation}
\end{prop}

\begin{proof}
We split the proof into three steps.

\medskip
\step1 Proof of~\eqref{eq:mom-expand-L}.\\
By Proposition~\ref{prop:decomp-exp-B} in form of~\eqref{eq:decomp-exp-B-reform-graph}, we recall that we have for all $k\ge1$ and $n\ge0$,
\begin{equation}\label{eq:moment-phimutN-expand}
\E_B[\Phi(\mu_t^N)^k]~=~\sum_{m=0}^{n}\frac1{(2N)^m}\Big(\int_{\triangle_t^m\times0}
{\begin{tikzpicture}[blue,baseline = -1.25ex]
\begin{scope}[local bounding box=box1]
\node[circle,draw,fill,inner sep=0pt,minimum size=3pt] (1) at (0,0) {};
\draw[white] (1) circle(0.2);
\node (2) at (0.15,-0.2) {{\tiny\color{black}$[\Phi^{k}]$}};
\end{scope}
\draw[rounded corners] (box1.south west) rectangle (box1.north east);
\end{tikzpicture}}
_{(m)}(\mu_0^N)\Big)
~+~\frac{R_{N,n}^k(t)}{(2N)^{n+1}},
\end{equation}
with remainder $R_{N,n}^k(t)$ defined in~\eqref{eq:def-RNnt}.
In order to prove~\eqref{eq:mom-expand-L}, it remains to use Lemma~\ref{lem:comput-graph}(iii) to expand the L-graphs in the above right-hand side as sums of \emph{irreducible} L-graphs.
By a direct induction argument,
we note that all time labelling satisfying the basic rules (R1)--(R4) appear symmetrically in the expansion, thus proving that for all $k\ge1$ and $m\ge0$ we can expand
\begin{equation}\label{eq:decomp-Phik-m}
{\begin{tikzpicture}[blue,baseline = -1.25ex]
\begin{scope}[local bounding box=box1]
\node[circle,draw,fill,inner sep=0pt,minimum size=3pt] (1) at (0,0) {};
\draw[white] (1) circle(0.2);
\node (2) at (0.15,-0.2) {{\tiny\color{black}$[\Phi^k]$}};
\end{scope}
\draw[rounded corners] (box1.south west) rectangle (box1.north east);
\end{tikzpicture}}
_{(m)}
\,=\,
\sum_{\Psi\in\Gamma(k,m)}\gamma(\Psi)\Psi,
\end{equation}
for some map $\gamma:\Gamma(k,m)\to\N$, where as in the statement $\Gamma(k,m)$ stands for the set of all (unlabeled) irreducible L-graphs with $k$ vertices and $m$ edges.
This already proves~\eqref{eq:mom-expand-L}.

\medskip
\step2 Proof that for all $k\ge1$ and $m\ge0$ the map $\gamma:\Gamma(k,m)\to\N$ in~\eqref{eq:decomp-Phik-m} satisfies for all L-graphs $\Psi\in\Gamma(k,m)$,
\begin{equation}\label{eq:gamma-split-form}
\gamma(\Psi)\,=\,\sum_{\substack{\Theta\subset\Psi\\\text{connected}\\\text{component}}}\binom{k-1}{V(\Theta)-1}\binom{m}{E(\Theta)}\,\gamma(\Theta)\,\gamma(\Psi\setminus\Theta),
\end{equation}
where the sum runs over all connected components $\Theta$ of the L-graph $\Psi$,
where $V(\Theta)$ and $E(\Theta)$ stand for the number of vertices and the number of edges in $\Theta$, respectively,
and where $\Psi\setminus \Theta$ stands for the L-subgraph obtained by removing the component $\Theta$ from $\Psi$.

\medskip\noindent
To prove this identity, we start by noting that, when appealing to Lemma~\ref{lem:comput-graph}(iii) to iteratively prove~\eqref{eq:decomp-Phik-m}, the map $\gamma$ can be given an explicit interpretation:
for all $\Psi\in\Gamma(k,m)$, the coefficient $\gamma(\Psi)$ is the positive integer given by
\[\gamma(\Psi)\,:=\,2^{SE(\Psi)}N(\Psi),\]
where $SE(\Psi)$ is the number of straight edges in $\Psi$ and where $N(\Psi)$ is the number of ways to obtain the graph $\Psi$ by starting from $k$ labeled vertices and by iteratively adding round or straight edges between stable subgraphs.
Conditioning on the connected component that the vertex with the first label belongs to, the identity~\eqref{eq:gamma-split-form} immediately follows from this interpretation.

\medskip
\step3 Proof of~\eqref{eq:cum-expand-L}.\\
Given $k\ge1$ and $m\ge0$, the result~\eqref{eq:gamma-split-form} of Step~2 implies
\[\sum_{\Psi\in\Gamma(k,m)}\gamma(\Psi)\Psi\,=\,\sum_{j=1}^k\binom{k-1}{j-1}\sum_{p=0}^m\binom{m}{p}\bigg(\sum_{\substack{\Psi\in\Gamma(j,p)\\\text{connected}}}\gamma(\Psi)\Psi\bigg)\bigg(\sum_{\Psi\in\Gamma(k-j,m-p)}\gamma(\Psi)\Psi\bigg),\]
and thus, by~\eqref{eq:decomp-Phik-m},
\[{\begin{tikzpicture}[blue,baseline = -1.25ex]
\begin{scope}[local bounding box=box1]
\node[circle,draw,fill,inner sep=0pt,minimum size=3pt] (1) at (0,0) {};
\draw[white] (1) circle(0.2);
\node (2) at (0.15,-0.2) {{\tiny\color{black}$[\Phi^k]$}};
\end{scope}
\draw[rounded corners] (box1.south west) rectangle (box1.north east);
\end{tikzpicture}}
_{(m)}
\,=\,\sum_{j=1}^k\binom{k-1}{j-1}\sum_{p=0}^m\binom{m}{p}\bigg(\sum_{\substack{\Psi\in\Gamma(j,p)\\\text{connected}}}\gamma(\Psi)\Psi\bigg){\begin{tikzpicture}[blue,baseline = -1.25ex]
\begin{scope}[local bounding box=box1]
\node[circle,draw,fill,inner sep=0pt,minimum size=3pt] (1) at (0,0) {};
\draw[white] (1) circle(0.2);
\node (2) at (0.15,-0.2) {{\tiny\color{black}$[\Phi^{k-j}]$}};
\end{scope}
\draw[rounded corners] (box1.south west) rectangle (box1.north east);
\end{tikzpicture}}
_{(m-p)}.\]
Taking the time integral and appealing to Lemma~\ref{lem:comput-graph}(iii), this leads us to
\begin{equation*}
\int_{\triangle_t^m\times0}
{\begin{tikzpicture}[blue,baseline = -1.25ex]
\begin{scope}[local bounding box=box1]
\node[circle,draw,fill,inner sep=0pt,minimum size=3pt] (1) at (0,0) {};
\draw[white] (1) circle(0.2);
\node (2) at (0.15,-0.2) {{\tiny\color{black}$[\Phi^k]$}};
\end{scope}
\draw[rounded corners] (box1.south west) rectangle (box1.north east);
\end{tikzpicture}}
_{(m)}(\mu)
\,=\,\sum_{j=1}^k\binom{k-1}{j-1}\sum_{p=0}^m\bigg(\sum_{\substack{\Psi\in\Gamma(j,p)\\\text{connected}}}\gamma(\Psi)\int_{\triangle_t^p\times0}\Psi(\mu)\bigg)
\bigg(
\int_{\triangle_t^{m-p}\times0}{\begin{tikzpicture}[blue,baseline = -1.25ex]
\begin{scope}[local bounding box=box1]
\node[circle,draw,fill,inner sep=0pt,minimum size=3pt] (1) at (0,0) {};
\draw[white] (1) circle(0.2);
\node (2) at (0.15,-0.2) {{\tiny\color{black}$[\Phi^{k-j}]$}};
\end{scope}
\draw[rounded corners] (box1.south west) rectangle (box1.north east);
\end{tikzpicture}}
_{(m-p)}(\mu)\bigg).
\end{equation*}
Now recalling~\eqref{eq:moment-phimutN-expand}, and defining
\[L_{N,n}^k(t,\mu_0^N)\,:=\,\sum_{m=0}^{n}\frac1{(2N)^{m}}\sum_{\substack{\Psi\in\Gamma(k,m)\\\text{connected}}}\gamma(\Psi)\int_{\triangle_t^m\times0}\Psi(\mu_0^N),\]
we deduce
\begin{equation}\label{eq:rel-mom-cum-phimunt-imit}
\E_B[\Phi(\mu_t^N)^k]\\
\,=\,
\sum_{j=1}^k\binom{k-1}{j-1}L_{N,n}^j(t,\mu_0^N)
\,\E_B[\Phi(\mu_t^N)^{k-j}]
+\frac{S_{N,n}^k(t)}{(2N)^{n+1}},
\end{equation}
with remainder $S_{N,n}^k(t)$ defined in~\eqref{eq:def-SNnt}.
Finally, we recall the recurrence relation of Lemma~\ref{lem:mom-cum} between moments and cumulants: for all $k\ge1$, we have
\[\E_B[\Phi(\mu_t^N)^k]\,=\,\sum_{j=1}^{k}\binom{k-1}{j-1}\kappa^{j}_B[\Phi(\mu_t^N)]\,\E_B[\Phi(\mu_t^N)^{k-j}].\]
Comparing this with the identity~\eqref{eq:rel-mom-cum-phimunt-imit} above, the conclusion follows by a direct induction.
\end{proof}

\subsection{Error estimates}
We turn to the uniform-in-time estimation of error terms in expansions such as~\eqref{eq:mom-expand-L} or~\eqref{eq:cum-expand-L}. 
We start with the following lemma describing
L-derivatives
of the solution of the mean-field McKean--Vlasov equation~\eqref{eq:MFL-McKean}. {Note that the representation formula below
corresponds to the one in~\cite{Tse_2021} for the Brownian dynamics.}

\begin{lem}\label{lem:diff-mk}
For all $t\ge0$ and $\phi\in C^\infty_c(\Xd)$, the functional $\mu\mapsto \int_\Xd\phi\, m(t,\mu)$ is smooth and its linear functional derivatives can be represented as follows: for all $k\ge1$, $\mu\in\Pc(\Xd)$, and~$y_1,\ldots,y_k\in\Xd$,
\begin{equation}\label{eq:rep-dmu-k-m}
\Lindk{}{k}\Big(\mu\mapsto\int_\Xd\phi\, m(t,\mu)\Big)(\mu,y_1,\ldots,y_k)\,=\,\int_\Xd \phi \,m^{(k)}(t,\mu,y_1,\ldots,y_k),
\end{equation}
where $m^{(k)}(\cdot,\mu,y_1,\ldots,y_k):\R^+\times\Xd\to\R$ is a distributional solution of the linear Cauchy problem
\begin{equation}\label{eq:cauchy_mk}
\left\{
\begin{array}{l}
\partial_t m^{(k)}(t,\mu,y_1,\ldots,y_k) - L_{m(t,\mu)}m^{(k)}(t,\mu,y_1,\ldots,y_k)=F_k(t,\mu,y_1,\ldots,y_k),\\
m^{(k)}(t,\mu,y_1,\ldots,y_k)|_{t=0} = (-1)^{k-1}(\delta_{y_k} - \mu),
\end{array}
\right.
\end{equation}
where for all $\mu\in\Pc(\Xd)$ we recall that $L_\mu$ stands for the linearized McKean--Vlasov operator at~$\mu$, cf.~\eqref{eq:def_linearized},
and where the source term $F_k$ is given by
\begin{multline}\label{eq:def_F_lambda}
F_k(t,\mu,y_1,\ldots,y_k)\,:=\,
\sum_{j=1}^{k-1}\sum_{\substack{I_0\cup I_1=\llbracket k\rrbracket\\\text{disjoint}}}\mathds1_{\sharp I_0=j,\sharp I_1=k-j}\\
\times\Div\bigg(m^{(j)}(t,\mu,y_{I_0})\int_{\Xd}\Lind{b}(\cdot,m(t,\mu),z)\, m^{(k-j)}((t,z),\mu,y_{I_1})\,\ddr z\bigg),
\end{multline}
where for $I=\{i_1,\ldots,i_r\}$ with $1\le i_1<\ldots<i_r\le k$ we set $y_I:=(y_{i_1},\ldots,y_{i_r})$.
Note that for $k=1$ we have $F_1\equiv0$.
In addition, given $\kappa_0, \lambda_0$ as in Theorem \ref{thm:ergodic}, we have the following uniform-in-time estimates: given $\kappa \in [0,\kappa_0]$, $1 < q \le 2$, and $0 < p \le 1$, further assuming in the Langevin setting that $p q' \gg_{\beta,a} 1$ is large enough (only depending on $d,\beta,a$), we have for all $k \ge 1$, $\alpha_1,\ldots,\alpha_k\ge0$, $y_1,\ldots, y_k \in \Xd$, for all $t\ge 0$, $\mu \in \calP(\Xd)$,
\begin{equation}\label{eq:estim-decay-mklin}
\|\nabla_{y_1}^{\alpha_1}\ldots \nabla_{y_k}^{\alpha_k}m^{(k)}(t,\mu,y_1,\ldots,y_k)\|_{W^{-((k+\max_j\alpha_j)\vee 2),q}(\langle z\rangle^p)}
\,\lesssim\,
\langle y_1\rangle^p\ldots \langle y_k\rangle^p e^{-p\lambda_0 t},
\end{equation}
where the multiplicative constant only depends on $d,W,\beta,k,p,q,a$, $\max_j\alpha_j$, and $Q(\mu)$.
In the Brownian setting, the dependence on $Q(\mu)$ can be lifted in case of estimates in the unweighted spaces $W^{-k,1}(\R^d)$.
In the Langevin setting, the dependence on $Q(\mu)$ can be made at most exponential provided that $\kappa$ is small enough:
given $\theta\in(0,1]$, for all $\ell\ge2$ and $0<p\le\frac14\theta$, given $\kappa\in[0,\kappa_{\ell,p}]$ with $\kappa_{\ell,p}$ as in Theorem~\ref{thm:ergodic}(ii), given $1 < q \le 2$ with $pq' \gg_{\beta,a} 1$ large enough (only depending on $d,\beta,a$), we have for all $k \ge 1$, $\alpha_1,\ldots, \alpha_k \ge 0$ with $k+\max_j\alpha_j\le\ell$, for all $y_1,\ldots, y_k \in \Xd$ and $t \ge 0$, 
\begin{equation}\label{eq:estim-decay-mklin-Q}
\|\nabla_{y_1}^{\alpha_1}\ldots \nabla_{y_k}^{\alpha_k}m^{(k)}(t,\mu,y_1,\ldots,y_k)\|_{W^{-((k+\max_j\alpha_j)\vee 2),q}(\langle z\rangle^p)}
\,\lesssim\,
\langle y_1\rangle^p\ldots \langle y_k\rangle^p e^{-p\lambda_0 t} \, e^{2kp Q(\mu)^\theta},
\end{equation}
where the multiplicative constant only depends on $d,W,\beta,\theta,\ell,k,p,q,a$, $\max_j\alpha_j$.
\end{lem}

\begin{proof}
By successively taking linear derivatives in the McKean--Vlasov equation~\eqref{eq:MFL-McKean}, the representation~\eqref{eq:rep-dmu-k-m}--\eqref{eq:cauchy_mk} in terms of linearized equations is straightforward with source term given by
\begin{multline*}
F_k(t,\mu,y_1,\ldots,y_k)\,:=\,
\sum_{l=1}^k\sum_{a_0=0}^{k-1}\sum_{\substack{a_1+\ldots+a_l=k-a_0\\ 1\le a_1\le \ldots\le a_l<k}}\sum_{\substack{I_0\cup\ldots\cup I_l=\llbracket k\rrbracket\\\text{disjoint}}}\mathds1_{\forall 0\le r\le l:\sharp I_r=a_r}\\
\times\Div\bigg(m^{(j)}(t,\mu,y_{I_0})\int_{\Xd^l}\Lindk{b}{l}(\cdot,m(t,\mu),z_1,\ldots,z_l)\\
\times m^{(a_1)}((t,z_1),\mu,y_{I_1})\ldots m^{(a_l)}((t,z_l),\mu,y_{I_l})\,\ddr z_1\ldots\ddr z_l\bigg),
\end{multline*}
where we recall the notation $y_I=(y_{i_1},\ldots,y_{i_r})$ for $I=\{i_1,\ldots,i_r\}$.
The pairwise structure of interactions actually yields various simplifications: noting that
\[\Lindk{(W\ast\mu)}{l}(\cdot,\mu,z_1,\ldots,z_l)=(-1)^lW\ast(\delta_{z_l}-\mu),\]
and noting that $\int_{\Xd} m^{(k)}(t,\mu,y_1,\ldots,y_k)=0$ for all $k\ge1$, the above expression for $F_k$ reduces precisely to~\eqref{eq:def_F_lambda}.
We emphasize that this simplification is not essential, but it slightly simplifies the computations.
It remains to deduce the decay estimate~\eqref{eq:estim-decay-mklin}:
we argue by induction and split the proof into two steps. Let $1 < q \le 2$ and $0 < p \le1$ be fixed and assume $pq' \gg_{\beta, a} 1$ large enough in the Langevin setting.

\medskip
\step1 Preliminary: we note some properties of the spaces $W^{-k,q}(\langle z\rangle^p)$ and their duals~$W^{k,q'}(\Xd)$,
\begin{enumerate}[---]
\item for all $\ell\ge0$ and $h\in C^\infty_c(\Xd)$,
\begin{equation}\label{eq:norm-Hs-weig-1}
\quad\|\nabla h\|_{W^{\ell,q'}(\Xd)}\,\le\,\|h\|_{W^{\ell+1,q'}(\Xd)};
\end{equation}
\begin{equation}\label{eq:norm-Hs-weig-2}
\|\nabla h\|_{W^{-\ell, q}(\langle z\rangle^p)}\,\lesssim\,\|h\|_{W^{1-\ell, q}(\langle z\rangle^p)};
\end{equation}
\item for all $\ell>\frac{1}{q'}\dim(\Xd)$, $h\in C^\infty_c(\Xd)$, and $y\in\Xd$,
\begin{equation}\label{eq:init}
\|h\|_{W^{-\ell,q}(\langle z\rangle^p)}\lesssim_{\ell,q}\int_{\Xd} |h|\langle z\rangle^p,\qquad
\|\delta_y\|_{W^{-\ell,q}(\langle z\rangle^p)}\,\lesssim_{\ell,  q}\,\langle y \rangle^p.
\end{equation}
\end{enumerate}
The claim~\eqref{eq:norm-Hs-weig-1} is a direct consequence of the definition of $W^{k,q'}(\Xd)$. The claim~\eqref{eq:norm-Hs-weig-2} is a direct consequence of \eqref{eq:norm-Hs-weig-1} by definition of dual norms.
The claim~\eqref{eq:init} follows from the Sobolev embedding and the definition of dual norms.

\medskip
\step2 Conclusion.\\
Applying $\nabla_{y_1}^{\alpha_1}\ldots\nabla_{y_k}^{\alpha_k}$ to both sides of equation~\eqref{eq:cauchy_mk}, setting $n = \max_j \alpha_j$, and appealing to Theorem~\ref{thm:ergodic}(ii) together with Duhamel's principle, we obtain for all $\ell\ge2$,
\begin{multline*}
\|\nabla_{y_1}^{\alpha_1}\ldots\nabla_{y_k}^{\alpha_k}m^{(k)}(t,\mu,y_1,\ldots,y_k)\|_{W^{-\ell,q}(\langle z\rangle^p)}\\
\,\lesssim_{W,\beta,\ell,p,q,a}\, C_{\ell,p}(\mu) \bigg( \,
e^{-p\lambda_0 t} \|\nabla_{y_1}^{\alpha_1}\ldots\nabla_{y_k}^{\alpha_k}(\delta_{y_k}-\mu)\|_{W^{-\ell,q}(\langle z\rangle^p)}\\
+\int_0^t \, e^{-p\lambda_0(t-s)} \|\nabla_{y_1}^{\alpha_1}\ldots\nabla_{y_k}^{\alpha_k}F_k(s,\mu,y_1,\ldots,y_k)\|_{W^{-\ell,q}(\langle z\rangle^p)}\ddr s \bigg),
\end{multline*}
up to a multiplicative constant only depending on $d,\beta,\ell,p,q,a$, $\|W\|_{W^{k+d+1,\infty}(\R^d)}$, and where the constant $C_{\ell,p}(\mu)$ further depends on $Q(\mu)$ and can be chosen as $C_{p,\mu}:=e^{pQ(\mu)^\theta}$ in case $p\in(0,\frac14\theta]$ and $\kappa\in[0,\kappa_{\ell,p}]$. Without loss of generality, we may assume that $\kappa_{\ell,p}$ is decreasing in $\ell$ and that $C_{\ell,p}(\mu)$ is increasing in~$\ell$.
The first right-hand side term is bounded by
\[\|\nabla_{y_1}^{\alpha_1}\ldots\nabla_{y_k}^{\alpha_k}(\delta_{y_k}-\mu)\|_{W^{-\ell,q}(\langle z\rangle^p)}\,\le\,\mathds1_{\alpha_1=\ldots=\alpha_{k-1}=0}\|\nabla^{\alpha_k}\delta_{y_k}\|_{W^{-\ell,q}(\langle z\rangle^p)}+\mathds1_{\alpha_1=\ldots=\alpha_{k}=0}\|\mu\|_{W^{-\ell,q}(\langle z\rangle^p)},\]
and thus,
by~\eqref{eq:norm-Hs-weig-2} and~\eqref{eq:init}, for $\ell>\alpha_k+\frac{1}{q'}\dim\Xd$,
\[\|\nabla_{y_1}^{\alpha_1}\ldots\nabla_{y_k}^{\alpha_k}(\delta_{y_k}-\mu)\|_{W^{-\ell,q}(\langle z\rangle^p)}\,\lesssim_{\ell,q}\,\langle y_k\rangle^p+Q(\mu)^p.\]
As by assumption we have $q'\ge pq'\gg1$, this estimate holds for any $\ell\ge\alpha_k+1$, say.
To shorten notation, let us introduce the following norms: given~$k\ge1$, we define for $\ell,n\ge0$ and $H:\Xd\times\Xd^k\to\R$,
\[\fatnorm{H}_{\ell,n}\,:=\,\sup_{0\le\alpha_1,\ldots,\alpha_k\le n}\sup_{y_1,\ldots,y_k\in\Xd}\Big(\langle y_1 \rangle^{-p}\ldots\langle y_k \rangle^{-p}\|\nabla_{y_1}^{\alpha_1}\ldots \nabla_{y_k}^{\alpha_k}H(\cdot,y_1,\ldots,y_k)\|_{W^{-\ell,q}(\langle z\rangle^p)}\Big).\]
In these terms, the above reads as follows, for all $k\ge1$, $n \ge 0$, and $\ell \ge (n+1) \vee 2$,
%/ $n+2 \le \ell \le n + \ell_0$,
\begin{equation}\label{eq:mkFk-estim-synth}
\fatnorm{m^{(k)}(t,\mu)}_{\ell,n}
\,\lesssim_{W,\beta,\ell,p,q,a}\, C_{\ell,p}(\mu)\Big( Q(\mu)^p e^{-p \lambda_0 t} 
+ \int_0^t e^{-p\lambda_0(t-s)} \fatnorm{F_k(s,\mu)}_{\ell,n}\ddr s \Big).
\end{equation}
We turn to the estimation of the source term $F_k$ as defined in~\eqref{eq:def_F_lambda}.
Recalling the choice of $b$, cf.~\eqref{eq:Langevin-par} and~\eqref{eq:Brownian-par}, and using again~\eqref{eq:norm-Hs-weig-2},
we easily find for all $\ell,\ell',n\ge0$,
\begin{equation*}
\fatnorm{F_k(t,\mu)}_{\ell+1,n}
\,\lesssim_{W,\beta,\ell,\ell',k,a}\,\max_{1\le j\le k-1}\fatnorm{m^{(j)}(t,\mu)}_{\ell,n}\fatnorm{m^{(k-j)}(t,\mu)}_{\ell',n}.
\end{equation*}
Inserting this into~\eqref{eq:mkFk-estim-synth}, and recalling that $F_1=0$, we deduce by induction for all $k\ge1$, $n\ge0$, $\ell\ge (n + k) \vee 2$, and $t \ge 0$, 
\begin{equation*}
\fatnorm{m^{(k)}(t,\mu)}_{\ell,n}
\,\lesssim_{W,\beta,\ell,k,p,q,a}\,
C_{\ell,p}(\mu)^{2k-1} Q(\mu)^{kp} e^{-p\lambda_0 t}.
\end{equation*}
By definition of $C_{\ell,p}(\mu)$, this concludes the proof of~\eqref{eq:estim-decay-mklin-Q}.
\end{proof}

With the above estimates at hand, we may now turn to the estimation of the L-derivative of smooth functionals along the particle dynamics.
For that purpose, we define the following hierarchy of norms:
for any smooth functional $\Phi: \Pc(\Xd) \to\R$, we define for $\mu \in \calP(\Xd)$, $k \ge 1$, $n\ge0$, and $p \ge 0$,
\[\3\Phi\3_{k,n,p,\mu}\,:=\,\max_{1\le j\le k}~\max_{0\le\alpha_1,\ldots,\alpha_j\le n} \bigg(~\sup_{y_1,\ldots,y_j\in\Xd} \langle y_1 \rangle^{-p} \ldots \langle y_j \rangle^{-p} \Big|\nabla_{y_1}^{\alpha_1}\ldots\nabla_{y_j}^{\alpha_j}\Lindk{\Phi}{j}\big(\mu,y_1,\ldots,y_j\big)\Big|\bigg).\]
The following result can be iterated to estimate arbitrary Lions graphs.

\begin{lem}\label{lem:control_graph}
Let $\kappa_0, \lambda_0$ be as in Theorem \ref{thm:ergodic} and let $\kappa \in [0,\kappa_0]$. Given~$m \ge 0$ and a smooth functional $\Psi: \triangle^m \times \calP(\Xd) \mapsto \R$, we have for all $k\ge1$, $n\ge0$, for all $0<p\le \frac12$, $(t,\tau,s) \in \triangle^{m+1}$,
\begin{align}\label{eq:estim-3norm-UPsi}
\3 \Uc_\Psi^{(1)}((t,\tau,s),\cdot )\3_{k,n,p,\mu} \,\le C_{W,\beta,k,p,n,a,\mu}\, e^{-p \lambda_0 (\tau_m - s)} \, \3 \Psi((t,\tau),\cdot) \3_{k,n+k+1,\frac p2, m(\tau_m-s,\mu)},
\end{align}
and in addition 
\begin{align*}
\vertiii{ {\begin{tikzpicture}[blue,baseline = -.75ex]
\begin{scope}[local bounding box=box1]
\draw node (0,0){\text{\color{black}$\Psi$}};
\end{scope}
\draw[rounded corners] (box1.south west) rectangle (box1.north east);
\end{tikzpicture}}
((t,\tau,s),\cdot) }_{k, n, p, \mu}
\,\le C_{W,\beta,k,p,n,a, \mu}\, e^{-p \lambda_0(\tau_m-s)}
\, \3 \Psi((t,\tau),\cdot) \3_{k+2,n+k+2,\frac p2, m(\tau_m-s,\mu)}.
\end{align*}
Moreover, given~$m, m' \ge 0$, smooth functionals $\Psi: \triangle^{m+1} \times \Xd \mapsto \R$ and $\Theta: \triangle^{m'+1} \times \Xd \mapsto \R$, and given a partition $\{i_1,\ldots, i_m\} \cup \{j_1, \ldots, j_{m'}\} = \llbracket m + m' \rrbracket$ with $i_1 < \ldots < i_m$ and $j_1 < \ldots < j_{m'}$, we have for all $k \ge 1$, $n \ge 0$, for all $0<p \le \frac12$, $(t,\tau,s,s') \in \triangle^{m+m'+2}$,
\begin{multline*}
{\vertiii{\begin{tikzpicture}[blue,baseline = -.7ex]
\begin{scope}[local bounding box=box1]
\draw node(1) at (0,0){\text{\color{black}$\Psi_{\langle i_1,\ldots, i_m,m+m'+1\rangle}$}};
\end{scope}
\begin{scope}[local bounding box=box2]
\draw node(2) at (4.4,0){\text{\color{black}$\Theta_{\langle j_1,\ldots, j_{m'},m+m'+1\rangle}$}};
\end{scope}
\draw node(1) at (2.2,-0.2){\text{\tiny\color{black}$\langle m\!+\!m'\!+\!2\rangle$}};
\path[-] (box1) edge (box2);
\draw[rounded corners,dotted] (box1.south west) rectangle (box1.north east);
\draw[rounded corners,dotted] (box2.south west) rectangle (box2.north east);
\end{tikzpicture}((t,\tau,s,s'), \cdot)}_{k,n,p,\mu}}\\
\,\le\,C_{W,\beta,k,p,n,a, \mu} e^{-p \lambda_0(s-s')}
\3 \Psi \big((t,\tau_{i_1}, \ldots, \tau_{i_m},s),\cdot\big)\3_{k+1,n+k+2,\frac p3,m(s-s',\mu)} \\
\times \3 \Theta \big((t,\tau_{j_1}, \ldots, \tau_{j_{m'}},s),\cdot\big)\3_{k+1,n+k+2,\frac p3,m(s-s', \mu)}.
\end{multline*}
In these estimates, the multiplicative constant $C_{W,\beta,k,p,n,a, \mu}$ only depends on $d, \beta, W, k, p, n, a$, and~$Q(\mu)$.
In the Langevin setting, the dependence on $Q(\mu)$ can be made at most exponential provided that $\kappa$ is small enough: given $\theta\in(0,1]$, for all $\ell\ge 2$ and $0<p\le\frac14\theta$, given $\kappa\in[0,\kappa_{\ell,p}]$ with $\kappa_{\ell,p}$ as in Theorem~\ref{thm:ergodic}(ii), for all $k\ge1$ and $n\ge0$ with $n+k\le\ell$, the above inequalities hold with multiplicative constants of the form
\[ C_{\beta, W, k, p, n, a, \mu} \,=\,e^{2kpQ(\mu)^\theta}\,C_{\beta,W,k,p,n,a},\]
for some constant $C_{\beta,W,k,p,n,a}$ only depending on $d,\beta, W,k,p,n,a$.
\end{lem}

\begin{proof}
We focus on the result for $\kappa \in [0,\kappa_0]$, while the precise dependence on $Q(\mu)$ for small $\kappa$ follows from a straightforward adaptation using Lemma \ref{lem:diff-mk} in form of~\eqref{eq:estim-decay-mklin-Q} instead of~\eqref{eq:estim-decay-mklin}. 
Given $m\ge0$ and a smooth functional $\Psi:\triangle^m\times\Pc(\Xd)\to\R$, recalling that $\Uc_\Psi^{(1)}$ is defined in Definition~\ref{def:Phi_U}, we can compute
\begin{align}\label{eq:formula_lind_U}
\Lind{\Uc_\Psi^{(1)}}\big((t,\tau,s),\mu,y\big)\,=\,\int_{\Xd}\Lind{\Psi}\big((t,\tau),m(\tau_m-s,\mu),\cdot\big)\,m^{(1)}(\tau_m-s,\mu,y),
\end{align}
with $m^{(1)}$ as defined in Lemma~\ref{lem:diff-mk}.
Recalling the definition of dual norms and applying Lemma~\ref{lem:diff-mk}, given $1<q\le2$ and $0<p\le1$ with $pq'\gg_{\beta,a}1$ large enough,
we deduce for all $(t,\tau,s)\in\triangle^{m+1}$, $\mu \in \calP(\Xd)$, and $y \in \Xd$,
\begin{multline*}
\Big| \Lind{\Uc_\Psi^{(1)}}\big((t,\tau,s),\mu,y\big) \Big|
\,\le\, \|m^{(1)}(\tau_m -s, \mu, y) \|_{W^{-2, q}(\langle z\rangle^{p})} \, \Big\|\langle \cdot \rangle^{-p}\Lind{\Psi}\big((t,\tau), m(\tau_m-s, \mu), \cdot\big)\Big\|_{W^{2, q'}(\Xd)} \\
\,\lesssim_{W,\beta,p,q,a,\mu}\,\langle y \rangle^{p} e^{-p \lambda_0 (\tau_m-s)}
\3 \Psi \3_{1, 2,\frac p2, m(\tau_m-s,\mu)},
\end{multline*}
where in the last estimate we further used $pq'>2\dim\Xd$.
By induction, on top of~\eqref{eq:formula_lind_U}, we find for all $k\ge1$ and $y_1,\ldots,y_k \in \Xd$,
\begin{multline}\label{eq:formula_lindk_u}
\Lindk{\Uc_\Psi^{(1)}}{k}\big((t,\tau,s),\mu,y_1,\ldots,y_k\big)\,=\,\sum_{l=1}^k\sum_{\substack{a_1+\ldots+a_l=k\\1\le a_1\le \ldots\le a_l\le k}}\sum_{\substack{I_1\cup\ldots\cup I_l=\llbracket k\rrbracket\\\text{disjoint}}}\mathds1_{\forall1\le r\le l:\sharp I_r=a_r}\\
\times\int_{\Xd^l}\Lindk{\Psi}{l}\big((t,\tau),m(\tau_m-s,\mu),z_1,\ldots,z_l\big)\,m^{(a_1)}\big((\tau_m-s,z_1),\mu,y_{I_1}\big)\\
\ldots m^{(a_l)}\big((\tau_m-s,z_l),\mu,y_{I_l}\big)\,\ddr z_1\ldots\ddr z_l,
\end{multline}
and the conclusion~\eqref{eq:estim-3norm-UPsi} then follows similarly using Lemma~\ref{lem:diff-mk}.
We turn to the estimation of the round edge.
By definition~\eqref{eq:def-round-edge}, we can write
\begin{equation*}
\begin{tikzpicture}[blue,baseline = -.7ex]
\begin{scope}[local bounding box=foo]
\draw node (0,0){\text{\color{black}$\Psi$}};
\end{scope}
\draw[rounded corners] (foo.south west) rectangle (foo.north east);
\end{tikzpicture}
\big((t,\tau,s),\mu\big)
\,=\,\Uc_{\Psi'}^{(1)}\big((t,\tau,s),\mu\big),
\end{equation*}
in terms of
\[\Psi'\big((t,\tau),\mu\big):=\int_\Xd \Tr\Big[a_0\partial_\mu^2\Psi\big((t,\tau),\mu\big)(z,z)\Big]\mu(\ddr z).\]
By a similar induction as the one performed to get~\eqref{eq:formula_lindk_u},
we find for all $k\ge1$ and $y_1,\ldots,y_k \in \Xd$,
\begin{multline*}
\Lindk{\Uc_{\Psi'}^{(1)}}{k}\big((t,\tau,s),\mu, y_1,\dots,y_k\big) = \sum_{j=0}^{k}\sum_{l=1}^j \sum_{\substack{a_1+\ldots+a_l=j\\1\le a_1\le \ldots\le a_l\le j}}\sum_{\substack{I_1\cup\ldots\cup I_l\subset\llbracket k\rrbracket\\\text{disjoint}}}\mathds1_{\forall1\le r\le l:\sharp I_r=a_r}\\
\times\int_{\Xd^{l+1}} \Lindk{}{l} \Tr \Big[a_0 \partial^2_\mu \Psi \big((t,\tau),m(\tau_m-s,\mu),z,z \big) \Big] (z_1,\ldots,z_l)\,m^{(a_1)}\big((\tau_m-s,z_1),\mu,y_{I_1}\big)\\
\ldots m^{(a_l)}\big((\tau_m-s,z_l),\mu,y_{I_l}\big)\, m^{(k-j)}\big((\tau_m-s,z), \mu, y_{\llbracket k \rrbracket \setminus (I_1\cup\ldots\cup I_l)}\big)\,\ddr z_1\ldots\ddr z_l \, \ddr z.
\end{multline*}
For $j<k$, the terms
can be estimated as before using Lemma~\ref{lem:diff-mk}. For $j=k$,
we use the following: for any bounded function $\varphi$, by Lemma~\ref{lem:unif-mom-est}, we have
\[\int_{\Xd} \varphi(z,z) \, m\big((\tau_m-s,z),\mu\big) \, \ddr z \,\lesssim_{W,\beta,a}\, \sup_{z\in \Xd} \Big( \langle z \rangle^{-p} \, \varphi(z,z) \Big) \int_{\Xd} \langle z\rangle^{p} \, \mu(\ddr z),\]
and the conclusion then follows by Jensen's inequality.
The argument for the straight edge is similar and we skip the details for shortness.
\end{proof}

The above result can be iterated to estimate arbitrary Lions graphs. Combining it with the diagrammatic representation of moments and cumulants in Proposition~\ref{prop:expand-Phi-NjLions}, we obtain the following corollary. We emphasize that it is crucial here that all constants in the previous results depend at most exponentially on the second moments of the empirical measure, since the expectation of these factors can then be estimated along the particle dynamics by Lemma~\ref{lem:unif-mom-est}.

\begin{cor}[Truncated Lions expansions]\label{prop:truncated_expansion}
Assume that the initial law $\mu_\circ$ has stretched exponential moments~\eqref{eq:mom-mu0} for some $\theta>0$.
For all $n \ge 0$, there exists $\kappa_n \in (0,\kappa_0]$ (only depending on $d, W,\beta,\theta, n,a$) such that, given $\kappa \in [0,\kappa_n]$, we can expand as follows Brownian moments and cumulants along the particle dynamics: for all smooth functionals $\Phi:\Pc(\Xd)\to\R$ and all $k\ge1$,
\begin{eqnarray}
\E_\circ \bigg[\Big|\E_B[\Phi(\mu_t^N)^k]~-~\sum_{m=0}^{n}\frac1{(2N)^m}\sum_{\Psi\in\Gamma(k,m)}\gamma(\Psi)\int_{\triangle_t^{m}\times0} \Psi\Big|\bigg]&\lesssim&N^{-n-1},\label{eq:estim-truncated-expansion-mom}\\
\E_\circ \bigg[\Big|\kappa_B^k[\Phi(\mu_t^N)]~-~
\mathds1_{n\ge k-1}\sum_{m=k-1}^{n}\frac1{(2N)^m}\sum_{\Psi\in\Gamma_\circ(k,m)}\gamma(\Psi)\int_{\triangle_t^{m}\times0} \Psi\Big|\bigg]&\lesssim&N^{-n-1},\label{eq:estim-truncated-expansion-cum}
\end{eqnarray}
where we recall that $\Gamma(k,m)$ stands for the set of all (unlabeled) irreducible L-graphs with $k$ vertices and $m$ edges
and that $\Gamma_\circ(k,m)$ stands for the subset of all \emph{connected} (unlabeled) irreducible L-graphs with $k$ vertices and $m$ edges,
where $\gamma$ is some map $\Gamma(k,m)\to\N$, and where multiplicative constants only depend on $d, W,\beta,k,n, a,\mu_\circ$, and on
\[ \sup_{\mu \in \calP(\Xd)} \3 \Phi \3_{2(n+1), (n+1)(n+4),3^{-n-4}\theta,\mu}.\]
\end{cor}

\begin{proof}
Let $0 < p \le \frac14\theta$ be fixed. Given $n \ge 1$, define
\[ \kappa_n \,:=\, \min_{1 \le j \le n} ~\min_{1\le\ell\le (n+1)(n+4)} \kappa_{2\ell, 2^{-j} p} \,>\,0,\]
where the $\kappa_{\ell,p}$'s are as in Theorem~\ref{thm:ergodic}(ii), and let then $\kappa \in [0, \kappa_n]$ be fixed.
By definition~\eqref{eq:def-round-edge} of the round edge, and by Lemma~\ref{lem:unif-mom-est},
given $m\ge0$ and a smooth functional $\Psi: \triangle^m\times \calP(\Xd) \to \R$, we have for all $(t,\tau,s) \in \triangle^{m+1}$, and 
$\mu \in \calP(\Xd)$,
\begin{multline*}
\big|{\begin{tikzpicture}[blue,baseline = -.75ex]
\begin{scope}[local bounding box=box1]
\draw node (0,0){\text{\color{black}$\Psi$}};
\end{scope}
\draw[rounded corners] (box1.south west) rectangle (box1.north east);
\end{tikzpicture}}((t,\tau,s), \mu)\big|\\
\,\lesssim_{W,\beta,p,a}\,\Big(\int_{\Xd} \langle z \rangle^{p} \mu(\ddr z)\Big)
\3 \Psi((t,\tau), \cdot) \3_{2, 1,\frac p2, m(\tau_m-s,\mu)}
\,\lesssim_{W,\beta,p,a}\,Q(\mu)^p
\3 \Psi((t,\tau), \cdot) \3_{2, 1,\frac p2, m(\tau_m-s,\mu)}.
\end{multline*}
In particular, for all $k\ge1$, $0 \le m \le n$, $(t,\tau, s) \in \triangle^{m+2}$,
\begin{equation*}
\Big|{\begin{tikzpicture}[blue,baseline = -1.25ex]
\begin{scope}[local bounding box=box1]
\node[circle,draw,fill,inner sep=0pt,minimum size=3pt] (1) at (0,0) {};
\draw[white] (1) circle(0.2);
\node (2) at (0.15,-0.2) {{\tiny\color{black}$[\Phi^k]$}};
\end{scope}
\draw[rounded corners] (box1.south west) rectangle (box1.north east);
\end{tikzpicture}}_{(m+1)}((t,\tau,s), \mu^N_s)\Big|
\,\lesssim_{W,\beta,p,a}\,
Q(\mu_s^N)^p\,
\Vertiii{ {\begin{tikzpicture}[blue,baseline = -1.25ex]
\begin{scope}[local bounding box=box1]
\node[circle,draw,fill,inner sep=0pt,minimum size=3pt] (1) at (0,0) {};
\draw[white] (1) circle(0.2);
\node (2) at (0.15,-0.2) {{\tiny\color{black}$[\Phi^k]$}};
\end{scope}
\draw[rounded corners] (box1.south west) rectangle (box1.north east);
\end{tikzpicture}}_{(m)}((t,\tau), \cdot)}_{2,1,\frac p2, m(\tau_{m+1}-s, \mu^N_s)}.
\end{equation*}
Now repeatedly applying Lemma~\ref{lem:control_graph} to control the right-hand side, and using Lemma~\ref{lem:unif-mom-est}(i) in form of $Q(m(u,\mu))\lesssim Q(\mu)$ for all $u\ge0$, we get for all $k\ge1$ and $(t,\tau,s)\in\triangle^{m+2}$,
\begin{multline*}
\Big|{\begin{tikzpicture}[blue,baseline = -1.25ex]
\begin{scope}[local bounding box=box1]
\node[circle,draw,fill,inner sep=0pt,minimum size=3pt] (1) at (0,0) {};
\draw[white] (1) circle(0.2);
\node (2) at (0.15,-0.2) {{\tiny\color{black}$[\Phi^k]$}};
\end{scope}
\draw[rounded corners] (box1.south west) rectangle (box1.north east);
\end{tikzpicture}}_{(m+1)}((t,\tau,s), \mu^N_s)\Big|
\,\lesssim_{W,\beta,m,p,a}\,
Q(\mu_s^N)^{p} \Big( \prod_{i=1}^{m} e^{4i2^{-i}pQ(\mu^N_s)^\theta} \Big) \Big( \prod_{i=1}^m e^{-2^{i-m-1}p\lambda_0(\tau_{i} - \tau_{i+1})} \Big) \\
\times\vertiii{\RS{p}_{[\Phi^k]}((t,\tau_1),\cdot)}_{2(m+1),(m+1)(m+2)-1,2^{-m-1}p,m(\tau_1-s, \mu^N_s)}.
\end{multline*}
Recalling $\RS{p}_{[\Phi^k]}=\Uc^{(1)}_{\Phi^k}=(\Uc^{(1)}_{\Phi})^k$, using again Lemma~\ref{lem:control_graph}, and noting that $\sum_{i=1}^m4 \, i\, 2^{-i}\le8$ and $Q^p\le e^{p\theta^{-1}Q^\theta}$,
this means, setting $\tau_0 := t$,
%and using $x^p e^{\frac32 x} \le (4p)^{-p} e^{\frac3{16p}} \, e^x$ for all $x$ and $(2 i) 2^{-1} \le 1$ for all $i \ge 1$,
\begin{multline}
\label{eq:temp_truncation_k_0}
\Big|{\begin{tikzpicture}[blue,baseline = -1.25ex]
\begin{scope}[local bounding box=box1]
\node[circle,draw,fill,inner sep=0pt,minimum size=3pt] (1) at (0,0) {};
\draw[white] (1) circle(0.2);
\node (2) at (0.15,-0.2) {{\tiny\color{black}$[\Phi^k]$}};
\end{scope}
\draw[rounded corners] (box1.south west) rectangle (box1.north east);
\end{tikzpicture}}_{(m+1)}((t,\tau,s), \mu^N_s)\Big|
\,\lesssim_{W,\beta,m,p,a}\,
 e^{3Q(\mu^N_s)^\theta} \Big( \prod_{i=0}^m e^{-2^{i-m-1}p\lambda_0(\tau_{i} - \tau_{i+1})} \Big) \\
\times\3\Phi\3_{2(m+1),(m+1)(m+4),2^{-m-2}p,m(t-s, \mu^N_s)}^k.
\end{multline}
Using Lemma~\ref{lem:unif-mom-est}(ii) with exponent $\delta =\theta$, we have
\begin{align*}
\E[e^{3Q(\mu^N_s)}]\, \lesssim_{W,\beta,\theta,a}\,\E_\circ[ e^{CQ(\mu^N_0)^\theta}] \,\lesssim_{W,\beta,\theta,a}\,\int_{\Xd}e^{C|z|^\theta}\mu_\circ(\ddr z),
\end{align*}
for some constant $C > 0$ only depending on $d,\beta,a$, $\|\nabla W\|_{\Ld^\infty(\R^d)}$, and some multiplicative constant further depending on $\theta$.
Taking the expectation in \eqref{eq:temp_truncation_k_0} and using the moment assumption on the initial law, we thus find
\begin{equation*}
\E\bigg[\Big|{\begin{tikzpicture}[blue,baseline = -1.25ex]
\begin{scope}[local bounding box=box1]
\node[circle,draw,fill,inner sep=0pt,minimum size=3pt] (1) at (0,0) {};
\draw[white] (1) circle(0.2);
\node (2) at (0.15,-0.2) {{\tiny\color{black}$[\Phi^k]$}};
\end{scope}
\draw[rounded corners] (box1.south west) rectangle (box1.north east);
\end{tikzpicture}}_{(m+1)}((t,\tau,s), \mu^N_s)\Big|\bigg]
\,\lesssim_{W,\beta,m,p,a,\mu_\circ}\, e^{-2^{-m-1}p\lambda_0(t - \tau_{m+1})}\sup_{\mu \in \calP(\Xd)} \3\Phi\3_{2(m+1),(m+1)(m+4),2^{-m-2}p,\mu}^k,
\end{equation*}
and the conclusion~\eqref{eq:estim-truncated-expansion-mom} then follows from Proposition~\ref{prop:expand-Phi-NjLions} in form of~\eqref{eq:mom-expand-L}, choosing e.g.\@ \mbox{$p=\frac14\theta$}.
Noting that similar a priori bounds on any irreducible L-graph can be obtained iteratively from Lemmas~\ref{lem:unif-mom-est} and~\ref{lem:control_graph}, the proof of~\eqref{eq:estim-truncated-expansion-cum} follows similarly from~\eqref{eq:cum-expand-L}.
\end{proof}

%%%%%%%%%%%%%%%%%%%%%%%%%%%%%%%%%%%%%%%%%%%%
%%%%%%%%%%%%%%%%%%%%%%%%%%%%%%%%%%%%%%%%%%%%
%%%%%%%%%%%%%%%%%%%%%%%%%%%%%%%%%%%%%%%%%%%%

\section{Higher-order propagation of chaos}\label{sec:refined_chaos}

This section is devoted to the proof of the correlation estimates of Theorem~\ref{thm:main}, as well as of their consequences on corrections to mean field as stated in Corollary~\ref{cor:bogo}.

\begin{proof}[Proof of Theorem~\ref{thm:main}]
Let $m_0 \ge 2$ and let $\kappa \in [0,\kappa_{m_0-2}]$ be fixed with $\kappa_{m_0-2}$ as in Corollary~\ref{prop:truncated_expansion}. For $t\ge0$ and $\phi\in C^\infty_c(\Xd)$, consider the random variables
\[X_t^N(\phi):=\int_\Xd\phi\,\mu_t^N.\]
In the spirit of Lemma~\ref{lem:cumtocorrel}, we start by estimating cumulants of $X_t^N(\phi)$.
By the law of total cumulance, cf.~Lemma~\ref{lem:total_cumulance}, they can be decomposed as follows, for all $m\ge2$,
\begin{equation}\label{eq:tot-cum-Xtphi}
\kappa^m[X_t^N(\phi)]\, =\, \sum_{\pi\vdash\llbracket m \rrbracket} \kappa_\circ^{\sharp \pi} \Big[ \big(\kappa^{\sharp A}_B[X_t^N(\phi)]\big)_{A\in\pi}\Big].
\end{equation}
We appeal to Corollary~\ref{prop:truncated_expansion} with $\Phi(\mu):=\int_\Xd\phi\mu$ to expand Brownian cumulants of $X_t^N(\phi)=\Phi(\mu_t^N)$: for all $1\le k\le m\le m_0$, we find
\[\E_\circ \bigg[\Big|\kappa_B^{k}[X_t^N(\phi)]-\mathds1_{k<m}\sum_{p=k-1}^{m-2}\frac1{(2N)^p}\sum_{\Psi\in\Gamma_\circ(k,p)}\gamma(\Psi)\int_{\triangle_t^{p}\times0}\Psi\Big|\bigg]\,\lesssim_{W,\beta,\phi,m,a,\mu_\circ}\,N^{1-m},\]
where we recall that $\Gamma_\circ(k,m)$ stands for the set of all connected irreducible L-graphs with $k$ vertices and $m$ edges built from the reference base point $\Phi$,
and where the multiplicative constant only depends on $d,W,\beta,m,a,\mu_\circ$, and on the $W^{r,\infty}(\Xd)$ norm of $\phi$ for some $r$ only depending on $m$.
Inserting this approximation into~\eqref{eq:tot-cum-Xtphi}, we deduce
\begin{multline}\label{eq:decomp-cum-B0-BB}
\bigg|\kappa^m[X_t^N(\phi)]\,-\, \sum_{s=2}^m\sum_{\{A_1,\ldots,A_s\}\vdash\llbracket m \rrbracket}~\sum_{p_1=\sharp A_1-1}^{m-2}\ldots\sum_{p_s=\sharp A_s-1}^{m-2}\frac1{(2N)^{p_1+\ldots+p_s}}\\
\times\sum_{\Psi_1\in\Gamma_\circ(\sharp A_1,p_1)}\ldots\sum_{\Psi_s\in\Gamma_\circ(\sharp A_s,p_s)}\gamma(\Psi_1)\ldots\gamma(\Psi_s)\\
\times\kappa_\circ^s\Big[\int_{\triangle_t^{p_1}\times0}\Psi_1,\ldots,\int_{\triangle_t^{p_s}\times0}\Psi_s\Big]\bigg|
\,\lesssim_{W,\beta,\phi,m,a,\mu_\circ}\,N^{1-m}.
\end{multline}
It remains to estimate the joint Glauber cumulants in this expression.
For that purpose, we appeal to the higher-order Poincar\'e inequality of Proposition~\ref{prop:control_Glauber}: recalling that Glauber derivatives can be bounded by linear derivatives, cf.~\eqref{eq:multi_d_Glauber}, we get
\begin{multline*}
\kappa_\circ^{s}\Big[\int_{\triangle_t^{p_1}\times0}\Psi_1,\ldots,\int_{\triangle_t^{p_s}\times0}\Psi_s\Big]\\
\,\lesssim_s\,N^{1-s}
\sum_{k=0}^{s-2} \sum_{\substack{a_1, \dots, a_{s} \ge 1 \\ \sum_j a_j = s+k}} \prod_{j=1}^{s}
\bigg\| \int_{([0,1]\times\Xd\times\Xd)^{a_j}}\Lindk{}{a_j}\Big(\int_{\triangle_t^{p_j}\times0}\Psi_j\Big)(m^{N,s_1,\ldots,s_{a_j}}_0,y_1,\ldots,y_{a_j})\\[-2mm]
\times\prod_{l=1}^{a_j}(\delta_{Z_0^{l,N}}-\delta_{z_l})(\ddr y_l)\,\mu_\circ(\ddr z_l)\,\ddr s_l\,\bigg\|_{\Ld^\frac{s+k}{a_j}(\Omega_\circ)},
\end{multline*}
where we have set for abbreviation
\[m^{N,s_1,\ldots,s_{a_j}}_0\,:=\,\mu_0^N+\sum_{l=1}^{a_j}\frac{1-s_l}{N}(\delta_{z_l}-\delta_{Z_0^{l,N}}).\]
Norms of linear derivatives of each $\Psi_j\in\Gamma_\circ(\sharp A_j,p_j)$ can be estimated using Lemmas~\ref{lem:unif-mom-est} and \ref{lem:control_graph}, together with the moment assumption on $\mu_\circ$. Inserting the result into~\eqref{eq:decomp-cum-B0-BB}, we conclude for all $1\le m\le m_0$,
\begin{equation*}
\kappa^m[X_t^N(\phi)]\,\lesssim_{W,\beta,\phi,m,a,\mu_\circ}\,N^{1-m}.
\end{equation*}
We now appeal to Lemma~\ref{lem:cumtocorrel} to turn this into an estimate on correlation functions: the above cumulant estimate implies for all $1\le m\le m_0$,
\begin{equation*}
\Big|\int_{\Xd^{m}}\phi^{\otimes m}G^{m,N}\Big|
\,\lesssim_{W,\beta,\phi,m,a,\mu_\circ}\, N^{1-m}
+\sum_{\substack {\pi\vdash\llbracket m\rrbracket \\ \sharp\pi<m}} \sum_{\rho\vdash\pi} N^{\sharp\pi-\sharp\rho-m+1}\bigg|\int_{\Xd^{\sharp\pi}}\Big(\bigotimes_{B\in\pi}\phi^{\sharp B}\Big)\Big(\bigotimes_{D\in\rho}G^{\sharp D,N}(z_D)\Big)\,\ddr z_\pi\bigg|,
\end{equation*}
and a direct induction argument then yields
\[\Big|\int_{\Xd^{m}}\phi^{\otimes m}G^{m,N}\Big|\,\lesssim_{W,\beta,\phi,m,a,\mu_\circ}\,N^{1-m}.\]
As the multiplicative constant only depends on $\phi$ via its $W^{r,\infty}(\Xd)$ norm for some $r$ only depending on~$m$, the conclusion of Theorem~\ref{thm:main} follows by duality.
\end{proof}

We now turn to the proof of Corollary~\ref{cor:bogo} as a straightforward consequence of the correlation estimates of Theorem~\ref{thm:main} using the BBGKY hierarchy.
For the proof, we shall need an additional ingredient from Section~\ref{sec:ergodic}: on top of the ergodic estimates of Theorem~\ref{thm:ergodic}(ii) for the linearized mean-field operator $L_\mu$, we will also use a reduced version in form of Proposition~\ref{prop:estim-Vts-Wk1*} for the modified operator~$R_\mu$ defined in~\eqref{eq:def-Rmu}.

\begin{proof}[Proof of Corollary~\ref{cor:bogo}]
Let $\lambda_0>0$ be as in Theorem~\ref{thm:ergodic}, and assume that it is chosen smaller than the exponent $\lambda$ in Lemma~\ref{lem:unif-mom-est} and than the exponent $\lambda_1$ in Proposition~\ref{prop:estim-Vts-Wk1*}.
We split the proof into two steps.

\medskip
\step1 Proof that for all $k\gg1$, $1<q\le2$, and $0<p\ll_\theta1$ with $pq'\gg_{\beta,a}1$, provided $\kappa\ll_{W,\beta,p,a}1$, we have
\begin{equation}\label{eq:estim-F1Nmu}
\|F_t^{1,N}-\mu_t\|_{W^{-k,q}(\langle z\rangle^p)}\,\lesssim\,N^{-1},\qquad
\|\tilde\mu_t^N-\mu_t\|_{W^{-k,q}(\langle z\rangle^p)}\,\lesssim\,N^{-1}.
\end{equation}
We focus on the estimate on $F^{1,N}-\mu$, while that on $\tilde\mu^N-\mu$ is obtained similarly.
From the BBGKY hierarchy~\eqref{eq:BBGKY0}, decomposing $F^{2,N}=(F^{1,N})^{\otimes2}+G^{2,N}$ and using the operator~$R_\mu$ defined in~\eqref{eq:def-Rmu}, we find that $F^{1,N}$ satisfies
\begin{equation}\label{eq:hier-F1N}
\partial_tF^{1,N}\,=\,R_{F^{1,N}}F^{1,N}
+\kappa\int_{\Xd}\nabla W(x-x_*)\cdot\nabla_{v}(G^{2,N}-\tfrac1NF^{2,N})(\cdot,z_*)\,\ddr z_*,
\end{equation}
while the mean-field equation~\eqref{eq:VFP} takes the form
\[\partial_t\mu\,=\,R_{\mu}\mu.\]
Taking the difference between those two equations, we deduce
\begin{equation*}
\partial_t(F^{1,N}-\mu)\,=\,R_{\mu}(F^{1,N}-\mu)
+\kappa h^N,
\end{equation*}
where we have set
\begin{equation*}
h^N\,:=\,\nabla W\ast(F^{1,N}-\mu)\cdot\nabla_vF^{1,N}\\
+\int_{\Xd}\nabla W(x-x_*)\cdot\nabla_{v}(G^{2,N}-\tfrac1NF^{2,N})(\cdot,z_*)\,\ddr z_*.
\end{equation*}
Note that $F^{1,N}|_{t=0} =\mu|_{t=0}= \mu_\circ$.
Given $1<q\le2$ and $0<p\le1$ with $pq'\gg_{\beta,a}1$, the ergodic estimate of Proposition~\ref{prop:estim-Vts-Wk1*} then yields for all $k\ge0$ and $t\ge0$,
\begin{equation}\label{eq:estim-F1N-mu-gronw}
\|F^{1,N}_t-\mu_t\|_{W^{-k,q}(\langle z\rangle^p)}\,\le\, \kappa \int_0^te^{-p\lambda_1(t-s)}\Big(C_{k,p,q}+C_{k,p,q,\mu_\circ}e^{-p\lambda_0s}\Big)\|h_s^N\|_{W^{-k,q}(\langle z\rangle^p)}\,\ddr s,
\end{equation}
for some constant $C_{k,p,q}$ only depending on $d,W,\beta,k,p,q,a$, and some constant $C_{k,p,q,\mu_\circ}$ further depending on~$Q(\mu_\circ)$.
Now note that by definition of $h^N$ we can bound for $k\ge1$,
\begin{eqnarray*}
\lefteqn{\|h^N\|_{W^{-k,q}(\langle z\rangle^p)}}\\
&\lesssim_{W,k}&\|\nabla W\ast(F^{1,N}-\mu)\|_{W^{k-1,\infty}(\R^d)}\|F^{1,N}\|_{W^{1-k,q}(\langle z\rangle^p)}
+\|G^{2,N}-\tfrac1NF^{2,N}\|_{W^{1-k,q}(\langle z\rangle^p)^{\otimes2}}\\
&\lesssim_{W,k}&\|F^{1,N}-\mu\|_{W^{-k,q}(\langle z\rangle^p)}\|F^{1,N}\|_{W^{1-k,q}(\langle z\rangle^p)}
+\|G^{2,N}-\tfrac1NF^{2,N}\|_{W^{1-k,q}(\langle z\rangle^p)^{\otimes2}},
\end{eqnarray*}
where we have used in particular $pq'\ge4d$. For $k\ge2$, note that the Sobolev embedding and Lemma~\ref{lem:unif-mom-est}(i) yield
\[\|F^{1,N}_s\|_{W^{1-k,q}(\langle z\rangle^p)}\,\lesssim\,\int_\Xd\langle z\rangle^pF^{1,N}_s\,=\,\E\Big[\int_\Xd\langle z\rangle^p\mu^{N}_s\Big]\,\lesssim_{W,\beta,a}\,1+e^{-p\lambda_0s}Q(\mu_\circ)^p.\]
In addition, although only stated in $W^{-k,1}(\Xd)$ for $k\gg1$, note that the proof of the correlation estimates in Theorem~\ref{thm:main}  (for $G^{2,N}$, say) actually holds in $W^{-k,q}(\langle z\rangle^p)$ for all $k\gg1$, $1<q\le2$, and $0<p\ll_\theta1$ with $pq'\gg_{\beta,a}1$ provided $\kappa\ll_{W,\beta,p,a}1$.
Recalling $F^{2,N} = G^{2,N} + F^{1,N} \otimes F^{1,N}$, the above then yields
\begin{equation*}
\|h^N_s\|_{W^{-k,q}(\langle z\rangle^p)}
\,\le\,\Big(C_{k,p,q}+C_{k,p,q,\mu_\circ}e^{-p\lambda_0s}\Big)\|F^{1,N}_s-\mu_s\|_{W^{-k,q}(\langle z\rangle^p)}
+C_{k,p,q,\mu_\circ}N^{-1}.
\end{equation*}
Inserting this into~\eqref{eq:estim-F1N-mu-gronw}, we obtain for all $t\ge0$,
\begin{multline*}
\|F^{1,N}_t-\mu_t\|_{W^{-k,q}(\langle z\rangle^p)}\,\le\,C_{k,p,q,\mu_\circ}N^{-1}\\
+\kappa \int_0^te^{-p\lambda_1(t-s)}\Big(C_{k,p,q}+C_{k,p,q,\mu_\circ}e^{-p\lambda_0s}\Big)\|F^{1,N}_s-\mu_s\|_{W^{-k,q}(\langle z\rangle^p)}\,\ddr s,
\end{multline*}
and the claim~\eqref{eq:estim-F1Nmu} then follows from Gr\"onwall's inequality provided $\kappa C_{k,p,q}\le p\lambda_0$ (which is independent of $Q(\mu_\circ)$).

\medskip
\step2 Conclusion.\\
From the BBGKY hierarchy~\eqref{eq:BBGKY0}, decomposing again $F^{2,N}=(F^{1,N})^{\otimes2}+G^{2,N}$, further using the cluster expansion~\eqref{eq:cluster-exp0} for $F^{3,N}$,
and using the operators~$L_\mu$ and~$R_\mu$ defined in~\eqref{eq:def_linearized-Lang} and~\eqref{eq:def-Rmu}, we find that the correlation function~$G^{2,N}$ satisfies
\begin{multline*}
\partial_tG^{2,N}\,=\,L_{F^{1,N}}^{(2)}G^{2,N}
+\tfrac\kappa N\sum_{i\ne j}\Big(\nabla W(x_i-x_j)-\nabla W\ast F^{1,N}(x_i)\Big)\cdot\nabla_{v_i}(F^{1,N})^{\otimes2}\\
+\tfrac\kappa N\sum_{i\ne j}\Big(\nabla W(x_i-x_j)-2\nabla W\ast F^{1,N}(x_i)\Big)\cdot\nabla_{v_i}G^{2,N}
\\
-\tfrac{\kappa}N\sum_{i\ne j}\int_{\Dd^d}\nabla W(x_i-x_*)\cdot\Big(2G^{2,N}(z_j,z_*)\nabla_{v_i}F^{1,N}(z_i)+F^{1,N}(z_j)\nabla_{v_i}G^{2,N}(z_i,z_*)\Big)\,\ddr z_*\\[-4mm]
+\kappa\tfrac{N-2}N\sum_{i}\int_{\Dd^d}\nabla W(x_i-x_*)\cdot\nabla_{v_i}G^{3,N}(z_1,z_2,z_*)\,\ddr z_*,
\end{multline*}
where we recall the short-hand notation $L_\mu^{(2)}=L_\mu\otimes\Id+\Id\otimes L_\mu$.
Comparing this, as well as~\eqref{eq:hier-F1N}, with the defining equations for $\tilde\mu^N$ and~$g^{(2)}$ in the statement, we are led to the following equations for the approximation errors,
\begin{eqnarray*}
\partial_t(F^{1,N}-\tilde\mu^{N})&=&L_{\mu}(F^{1,N}-\tilde\mu^N)
+\tfrac\kappa N\int_{\Dd^d}\nabla W(x-x_*)\cdot\nabla_v(NG^{2,N}-g^{(2)})(\cdot,z_*)\,\ddr z_*
+\kappa \tilde h^{1,N},\\
\partial_t(NG^{2,N}-g^{(2)})&=&L_\mu^{(2)}(NG^{2,N}-g^{(2)})+\kappa h^{2,N},
\end{eqnarray*}
where we have set
\begin{align*}
\tilde h^{1,N}
&:=
-\tfrac1N\int_{\Dd^d}\nabla W(x-x_*)\cdot\nabla_v(F^{2,N}-\mu^{\otimes2})(\cdot,z_*)\,\ddr z_*\\
&\quad +\nabla W\ast(F^{1,N}-\tilde\mu^N)\cdot\nabla_v(F^{1,N}-\mu)
+\nabla W\ast(\tilde\mu^N-\mu)\cdot\nabla_v(F^{1,N}-\tilde\mu^N),\\[2mm]
h^{2,N}&:=\sum_{i\ne j}\nabla W(x_i-x_j)\cdot\nabla_{v_i}(F^{2,N}-\mu^{\otimes2})\\[-1mm]
&\quad -\sum_{i}\Big(\nabla W\ast F^{1,N}(z_i)\cdot\nabla_{v_i}(F^{1,N})^{\otimes2}-\nabla W\ast\mu(x_i)\cdot\nabla_{v_i}\mu^{\otimes2}\Big)\\
&\quad +\sum_{i}\nabla W\ast(\tfrac{N-2}NF^{1,N}-\mu)(x_i)\cdot\nabla_{v_i}(NG^{2,N})\\[-1mm]
&\quad +\sum_{i\ne j}\nabla_{v_i}(\tfrac{N-2}NF^{1,N}-\mu)(z_i)\cdot\int_{\Dd^d}\nabla W(x_i-x_*)(NG^{2,N})(z_j,z_*)\,\ddr z_*\\[-1mm]
&\quad -\sum_{i\ne j}F^{1,N}(z_j)\int_{\Dd^d}\nabla W(x_i-x_*)\cdot\nabla_{v_i}G^{2,N}(z_i,z_*)\,\ddr z_*  \\
&\quad + (N-2) \sum_{i}\int_{\Dd^d}\nabla W(x_i-x_*)\cdot\nabla_{v_i}G^{3,N}(z_1,z_2,z_*)\,\ddr z_*.
\end{align*}
Note that $F^{1,N}|_{t=0} =\tilde \mu^N|_{t=0}= \mu_\circ$ and $NG^{2,N}|_{t=0}=g^{(2)}|_{t=0}=0$.
Given $1<q\le2$ and $0<p\le1$ with $pq'\gg_{\beta,a}1$, and $\kappa\ll_{W,\beta,a}1$, the ergodic estimates of Theorem~\ref{thm:ergodic} then yield for all $k\ge2$ and $t\ge0$,
\begin{eqnarray*}
\|F^{1,N}_t-\tilde\mu^N_t\|_{W^{-k,q}(\langle z\rangle^p)}
&\lesssim_{W,\beta,k,p,q,a,\mu_\circ}&\int_0^te^{-p\lambda_0(t-s)}\Big(N^{-1}\|NG^{2,N}_s-g^{(2)}_s\|_{W^{1-k,q}(\langle z\rangle^p)}\\[-2mm]
&&\hspace{5cm}+\|\tilde h^{1,N}_s\|_{W^{-k,q}(\langle z\rangle^p)}\Big)\,\ddr s,\\[-2mm]
\|NG^{2,N}_t-g^{(2)}_t\|_{W^{-k,q}(\langle z\rangle^p)}&\lesssim_{W,\beta,k,p,q,a,\mu_\circ}&\int_0^te^{-p\lambda_0(t-s)}\|h^{2,N}_s\|_{W^{-k,q}(\langle z\rangle^p)}\,\ddr s.
\end{eqnarray*}
By definition of $\tilde h^{1,N}$ and $h^{2,N}$, appealing to the result~\eqref{eq:estim-F1Nmu} of Step~1, together with the correlation estimates of Theorem~\ref{thm:main}, we find for all $k\gg1$, $1<q\le2$, and $0<p\ll_\theta1$ with $pq'\gg_{\beta,a}1$, provided~$\kappa\ll_{W,\beta,p,a}1$,
\begin{equation*}
\|\tilde h^{1,N}\|_{W^{-k,q}(\langle z\rangle^p)}\,\lesssim_{W,\beta,k,p,q,a,\mu_\circ}\,N^{-2},\qquad
\|h^{2,N}\|_{W^{-k,q}(\langle z\rangle^p)}\,\lesssim_{W,\beta,k,p,q,a,\mu_\circ}\,N^{-1}.
\end{equation*}
Inserting this into the above, the conclusion follows.
\end{proof}

%%%%%%%%%%%%%%%%%%%%%%%%%%%%%%%%%%%%%%%%%%%%
%%%%%%%%%%%%%%%%%%%%%%%%%%%%%%%%%%%%%%%%%%%%
%%%%%%%%%%%%%%%%%%%%%%%%%%%%%%%%%%%%%%%%%%%%

\section{Quantitative central limit theorem}\label{sec:clt}
This section is devoted to the proof of Theorem~\ref{th:CLT}. For $t\ge0$ and $\phi\in C^\infty_c(\R^d)$, consider the centered random variables
\[S^N_t(\phi)\,:=\,\sqrt N Y^N_t(\phi)\,:=\,\sqrt N\bigg(\int_{\Xd}\phi\mu^N_t-\E\Big[\int_{\Xd}\phi\mu^N_t\Big]\bigg).\]
We shall start by using a Lions expansion to split the contributions from initial data and from Brownian forces in the fluctuations.
From there, we separately analyze initial and Brownian fluctuations, using tools from Glauber and Lions calculus, respectively.

\subsection{Gaussian Dean--Kawasaki equation}\label{sec:DK-def}
We consider the Gaussian Dean--Kawasaki SPDE~\eqref{eq:DK-Lang}. With our general notation, covering the Langevin and Brownian settings at the same time, this reads as follows,
\begin{align}\label{eq:DK-gen}
\left\{\begin{array}{ll}
\partial_t \nu_t
\,=\, L_{\mu_t}\nu_t +\Div(\sqrt{\mu_t} \sigma_0\xi_t),&\text{for $t\ge0$},\\
\nu_t|_{t = 0}=\nu_\circ,&
\end{array}\right.
\end{align}
where:
\begin{enumerate}[---]
\item $L_{\mu}$ is the linearized mean-field operator defined in~\eqref{eq:def_linearized};
\smallskip\item $\mu_t:=m(t,\mu_\circ)$ is the solution of the mean-field McKean--Vlasov equation~\eqref{eq:MFL-McKean};
\smallskip\item $\nu_\circ$ is the Gaussian field describing the fluctuations of the initial empirical measure,
in the sense that~$\sqrt N\int_\Xd\phi (\mu^N_0-\mu_\circ)$ converges in law to $\int_\Xd\phi\,\nu_\circ$ for all $\phi\in C^\infty_c(\Xd)$; in other words, ~$\nu_\circ$ is the random tempered distribution on $\Xd$ characterized by having Gaussian law with
\begin{equation}\label{eq:init-fluct}
\quad\Var\Big[\int_\Xd\phi\,\nu_\circ\Big]\,=\,\int_\Xd \Big(\phi-\int_\Xd\phi\,\mu_\circ\Big)^2\mu_\circ,\qquad\E\Big[\int_\Xd\phi\,\nu_\circ\Big]\,=\,0,\qquad\text{for all $\phi\in C^\infty_c(\Xd)$;}
\end{equation}
\item $\xi_t$ is a Gaussian white noise on $\R\times\Xd$ and is taken independent of $\nu_\circ$.
\end{enumerate}
As $\nu_\circ$ and $\xi$ are random tempered distributions on $\Xd$ and $\R\times\Xd$, respectively, equation~\eqref{eq:DK-gen} is naturally understood in the distributional sense almost surely, and its solution~$\nu$ must itself be sought as a random tempered distribution on $\R^+\times\Xd$.

In order to solve equation~\eqref{eq:DK-gen}, we start by introducing some notation.
Given $\kappa\in[0,\kappa_0]$, $k \ge 2$, $1<q\le2$, and $0<p\le\frac14\theta$ with $pq'\gg_{\beta,a}1$ large enough, for all $h \in W^{-k,q}(\langle z \rangle^p)$, we can define $\{U_{t,s}[h]\}_{s \le t}$ as the unique weak solution in $C_\loc([s,\infty);W^{-k,q}(\langle z \rangle^p))$ of
\begin{equation}\label{eq:def_U}
\left\{\begin{array}{ll}
\partial_t U_{t,s}[h] = L_{\mu_t} U_{t,s}[h],&\text{for $t \ge s$},\\
U_{t,s}[h]|_{t=s}= h,&
\end{array}\right.
\end{equation}
see Theorem~\ref{thm:ergodic}(ii).
We also consider the dual evolution $\{U_{t,s}^*\}_{t \ge s}$ on $W^{-k,q}(\langle z\rangle^p)^*$, where for all $t\ge s$ we define $U_{t,s}^*$ as the adjoint of $U_{t,s}$,
\[ \int_{\Xd} U_{t,s}^*[g]h\,=\, \int_{\Xd} U_{t,s}[h]g.\]
Recall that the condition $pq'>d$ ensures that the dual space $W^{-k,q}(\langle z\rangle^p)^*$ contains $W^{k,\infty}(\Xd)$.
For any $g \in W^{-k,q}(\langle z\rangle^p)^*$ and $t\ge0$, the dual flow $s \mapsto U_{t,s}^\ast[g]$ naturally belongs to $C([0, t]; W^{-k,q}(\langle z\rangle^p)^*)$, where $W^{-k,q}(\langle z\rangle^p)^*$ is endowed with the weak topology. Denoting by $L_\mu^*$ the adjoint of the linearized mean-field operator $L_\mu$, we find that the dual flow satisfies the backward Cauchy problem
\begin{equation}\label{eq:dual-test-lin-semigr}
\left\{\begin{array}{ll}
\partial_sU_{t,s}^\ast[g]=-L_{\mu_s}^*U_{t,s}^*[g],&\text{for $0\le s\le t$},\\
U_{t,s}^\ast[g]|_{s=t}=g.
\end{array}\right.
\end{equation}

In these terms, using the theory of Da Prato and Zabczyk~\cite{Prato_1992}, and more specifically its non-autonomous extension by Seidler~\cite{Seidler_1993}, we can check that the Gaussian Dean--Kawasaki equation~\eqref{eq:DK-gen} admits a unique weak solution that is a random element in $C(\R^+;\Sc'(\Xd))$, and it can be expressed by Duhamel's principle
\begin{equation*}
\nu_t\,:=\,U_{t,0}[\nu_\circ]+\int_0^tU_{t,s}\big[\Div\big(\sqrt{\mu_s}\sigma_0\xi_s\big)\big]\,\ddr s.
\end{equation*}
In particular, the law of the solution is characterized by its covariance structure
\begin{eqnarray}
\Var\bigg[\int_\Xd \phi\nu_t\bigg]
&=&\Var\bigg[\int_\Xd U_{t,0}^*[\phi]\, \nu_\circ\bigg]+\Var\bigg[\int_0^t\Big(\int_\Xd \sqrt{\mu_s}\sigma_0^T\nabla U_{t,s}^*[\phi] \cdot \xi_s\Big)\,\ddr s\bigg]\nonumber\\
&=&\int_\Xd\Big(U_{t,0}^*[\phi]-\int_\Xd U_{t,0}^*[\phi]\,\mu_\circ\Big)^2\mu_\circ\,+\,\int_0^t\Big(\int_\Xd \big|\sigma_0^T\nabla U_{t,s}^*[\phi] \big|^2 \mu_s\Big)\,\ddr s.\label{eq:cov-DK}
\end{eqnarray}

\subsection{Splitting fluctuations}
By means of a Lions expansion, we start by showing that fluctuations can be split neatly into contributions from initial data and from Brownian forces.

\begin{lem}\label{lem:fluct-split}
There exists $\kappa_1 \in (0,\kappa_0]$ (only depending on $d,\beta,W,\theta,a$) such that given $\kappa\in[0,\kappa_1]$
%and assume that the initial law $\mu_\circ$ satisfies $\int_\Xd|z|^{p_0}\mu_\circ(\ddr z)<\infty$ for some $p_0>0$.
we have for all $\phi\in C^\infty_c(\Xd)$ and~$t\ge0$,
\[  \big\|S_t^N(\phi)-C_t^N(\phi)-D_t^N(\phi)\big\|_{\Ld^2(\Omega)}\,\lesssim_{W,\beta,a,\mu_\circ}\,N^{-\frac12}\|\phi\|_{W^{3,\infty}(\Xd)},\]
in terms of
\begin{eqnarray*}
C_t^N(\phi)&:=&\sqrt N\bigg(\int_{\Xd}\phi\,m(t,\mu^N_0)-\E_\circ\Big[\int_{\Xd}\phi\,m(t,\mu^N_0)\Big]\bigg),\\
D_t^N(\phi)&:=&\frac1{\sqrt N}\sum_{i=1}^N\int_0^t\partial_\mu U(t-s,\mu_s^N)(Z^{i,N}_s)\cdot\sigma_0\ddr B^i_s,
\end{eqnarray*}
where we have set for abbreviation $U(t,\mu):=\int_\Xd \phi\, m(t,\mu)$.
\end{lem}

\begin{proof}
Let~$\kappa_0,\lambda_0$ be as in Theorem~\ref{thm:ergodic} and let $1 < q \le 2$ and $0 < p \le \frac14\theta$ be fixed with $pq' \gg_{\beta,a} 1$ large enough. For this value of $p$, using the notations of Theorem~\ref{thm:ergodic}(ii), we set $\kappa_1:= \min(\kappa_{2,p}, \kappa_{3,p})$ which belongs to $(0,\kappa_0]$ and we let $\kappa \in [0, \kappa_1]$ be fixed. 
Starting point is the Lions expansion of Lemma~\ref{lem:CST0},
\begin{multline*}
\int_{\Xd}\phi\,\mu^N_t\,=\,\int_{\Xd}\phi\, m(t,\mu_0^N)
+\frac1N\sum_{i=1}^N\int_0^t\partial_\mu U(t-s,\mu_s^N)(Z^{i,N}_s)\cdot\sigma_0\ddr B^i_s\\
+\frac1{2N}\int_0^t\int_\Xd\Tr\Big[a_0\,\partial_\mu^2U(t-s,\mu_s^N)(z,z)\Big]\mu_s^N(\ddr z)\,\ddr s.
\end{multline*}
Multiplying by $\sqrt N$, subtracting the expectation to both sides of the identity, taking the $\Ld^2(\Omega)$ norm, and recognizing the definition of $C_t^N(\phi)$ and $D_t^N(\phi)$, we are led to
\begin{equation}\label{eq:temp_decompo_YN}
\big\|S_t^N(\phi)-C_t^N(\phi)-D_t^N(\phi)\big\|_{\Ld^2(\Omega)}\,\le\,
\frac1{2\sqrt N}\bigg\|\int_0^t\int_\Xd\Tr\Big[a_0\,\partial_\mu^2U(t-s,\mu_s^N)(z,z)\Big]\mu_s^N(\ddr z)\,\ddr s\bigg\|_{\Ld^2(\Omega)}.
\end{equation}
By Lemma~\ref{lem:diff-mk}, we get
\begin{equation*}
\int_0^t\int_\Xd\Tr\Big[a_0\,\partial_\mu^2U(t-s,\mu_s^N)(z,z)\Big]\mu_s^N(\ddr z)\,\ddr s
\, \lesssim_{W,\beta,p,a} \, \|\phi\|_{W^{3,\infty}(\Xd)}\int_0^t e^{-p\lambda_0(t-s)} Q(\mu^N_s)^{2p} e^{4pQ(\mu^N_s)^\theta} \,\ddr s.
\end{equation*}
Appealing to Lemma~\ref{lem:unif-mom-est} together with the exponential moment assumption for $\mu_\circ$, we deduce
\begin{equation*}
\bigg\|\int_0^t\int_\Xd\Tr\Big[a_0\,\partial_\mu^2U(t-s,\mu_s^N)(z,z)\Big]\mu_s^N(\ddr z)\,\ddr s\bigg\|_{\Ld^2(\Omega)}
\, \lesssim_{W,\beta,p,a, \mu_\circ} \, \|\phi\|_{W^{3,\infty}(\Xd)}.
\end{equation*}
Combined with~\eqref{eq:temp_decompo_YN}, this yields the conclusion.
\end{proof}

\subsection{Initial fluctuations}
We establish the following quantitative CLT for initial fluctuations $C_t^N(\phi)$. The proof is split into two separate parts: the asymptotic normality of $C_t^N(\phi)$ and the convergence of its variance structure. For both parts, we exploit tools from Glauber calculus: more precisely, the first part follows from Stein's method in form of the so-called second-order Poincar\'e inequality of Proposition~\ref{prop:2ndP}, while for the second part our starting point is the Helffer--Sj\"ostrand representation for the variance in Lemma~\ref{lem:L0-prop}(iii).

\begin{lem}\label{lem:init-fluct}
Let $\lambda_0$ be as in Theorem~\ref{thm:ergodic}, $\kappa_1$ as in Lemma~\ref{lem:fluct-split}, and let $\kappa\in[0,\kappa_1]$ be fixed.
The random variable~$C_t^N(\phi)$ defined in Lemma~\ref{lem:fluct-split} satisfies for all $\phi\in C^\infty_c(\Xd)$ and $t\ge0$,
\begin{multline*}
\ddr_2\Big({C_t^N(\phi)}\,,\,{\sigma^C_t(\phi,\mu_\circ)\Nc}\Big)\\[-2mm]
\,\lesssim_{W,\beta,a,\mu_\circ}\,N^{-\frac12}e^{-\frac12\theta\lambda_0 t}\|\phi\|^2_{W^{2,\infty}(\Xd)} \bigg(1+\Big(\tfrac{\sigma_t^C(\phi,\mu_\circ)}{\|\phi\|_{W^{2,\infty}(\Xd)}} + (N^{-\frac13}e^{-\frac12\theta\lambda_0 t})^{\frac12}\Big)^{-1}\bigg),
\end{multline*}
where the limit variance is defined by
\begin{equation}\label{eq:def-sigC}
\sigma_t^C(\phi,\mu_\circ)^2\,:=\,\Var_\circ\!\big[(U_{t,0}^*[\phi])(Z_\circ^{1,N})\big]\,=\,\int_\Xd\Big(U_{t,0}^*[\phi]-\int_\Xd U_{t,0}^*[\phi]\,\mu_\circ\Big)^2\mu_\circ,
\end{equation}
where we recall that $U_{t,0}^*$ is defined in~\eqref{eq:dual-test-lin-semigr},
that $\ddr_2$ is the second-order Zolotarev metric~\eqref{eq:Zolo}, and that $\Nc$ stands for a standard normal random variable.
\end{lem}

\begin{proof}
Let $1 < q \le 2$, $0 < p \le \frac14\theta$, and $\kappa_1$ be fixed as in the proof of Lemma~\ref{lem:fluct-split}, and let $\kappa \in [0,\kappa_1]$. 
We split the proof into three steps.

\medskip
\step1 Asymptotic normality: proof that for all $t\ge0$,
\begin{multline}\label{eq:as-norm-Glaub}
\ddr_2\bigg(\frac{C_t^N(\phi)}{\Var_\circ[C_t^N(\phi)]^\frac12}\,,\,\Nc\bigg)
\,+\,\ddr_{\operatorname{W}}\bigg(\frac{C_t^N(\phi)}{\Var_\circ[C_t^N(\phi)]^\frac12}\,,\,\Nc\bigg)
\,+\,\ddr_{\operatorname{K}}\bigg(\frac{C_t^N(\phi)}{\Var_\circ[C_t^N(\phi)]^\frac12}\,,\,\Nc\bigg)\\
\,\lesssim_{W,\beta,p,a,\mu_\circ}\, N^{-\frac12}e^{-2p\lambda_0 t}\|\phi\|_{W^{2,\infty}(\Xd)}^2\Var_\circ[C_t^N(\phi)]^{-1}\Big(1+\|\phi\|_{W^{2,\infty}(\Xd)}\Var_\circ[C_t^N(\phi)]^{-\frac12}\Big).
\end{multline}
Set for abbreviation
\[\hat C_t^N(\phi)\,:=\,\frac{C_t^N(\phi)}{\Var_\circ[C_t^N(\phi)]^\frac12}.\]
By Proposition~\ref{prop:2ndP}, we can estimate
\begin{multline*}
\ddr_2\big(\hat C_t^N(\phi),\Nc\big)
\,+\,\ddr_{\operatorname{W}}\big(\hat C_t^N(\phi),\Nc\big)
\,+\,\ddr_{\operatorname{K}}\big(\hat C_t^N(\phi),\Nc\big)
\,\lesssim\,\Var_\circ[C_t^N(\phi)]^{-\frac32}\sum_{j=1}^N\E_\circ\big[{|D^\circ_j C_t^N(\phi)|^6}\big]^\frac12\\
+\Var_\circ[C_t^N(\phi)]^{-1}\bigg(\sum_{j=1}^N\Big(\sum_{l=1}^N\E_\circ\big[{|D^\circ_l C_t^N(\phi)|^4}\big]^\frac14\E_\circ\big[{|D^\circ_j D^\circ_l C_t^N(\phi)|^4}\big]^\frac14\Big)^2\bigg)^\frac12.
\end{multline*}
Recalling the definition of $C_t^N(\phi)$ in Lemma~\ref{lem:fluct-split}, using~\eqref{eq:decompo_Glauber_derivative} and~\eqref{eq:multi_d_Glauber} to bound Glauber derivatives by means of linear derivatives, appealing to Lemma~\ref{lem:diff-mk} to estimate the latter, and distinguishing between the cases $l=j$ and $l\ne j$ in the second right-hand side term, we deduce
\begin{multline*}
\ddr_2\big(\hat C_t^N(\phi),\Nc\big)
\,+\,\ddr_{\operatorname{W}}\big(\hat C_t^N(\phi),\Nc\big)
\,+\,\ddr_{\operatorname{K}}\big(\hat C_t^N(\phi),\Nc\big)\\
\,\lesssim_{W,\beta,p,a}\,\Var_\circ[C_t^N(\phi)]^{-\frac32}N^{-\frac12}e^{-3p\lambda_0 t} \|\phi\|_{W^{2,\infty}(\Xd)}^3\E_\circ\big[ e^{CpQ(\mu^N_0)^\theta} \big] \\
+\Var_\circ[C_t^N(\phi)]^{-1}N^{-\frac12}e^{-2p\lambda_0 t} \|\phi\|_{W^{2,\infty}(\Xd)}^2\E_\circ\big[e^{CpQ(\mu^N_0)^\theta} \big].
\end{multline*}
and the claim follows from the exponential moment assumption on $\mu_\circ$. 

\medskip
\step2 Convergence of the variance: proof that for all $t\ge0$,
\begin{equation}\label{eq:var-CtN-exp}
\Big|\Var_\circ[C_t^N(\phi)]-\sigma_t^C(\phi,\mu_\circ)^2\Big|
\,\lesssim_{W,\beta,p,a,\mu_\circ}\,N^{-1}e^{-2p\lambda_0 t}\|\phi\|_{W^{2,\infty}(\Xd)}^2,
\end{equation}
where the limit variance is defined in~\eqref{eq:def-sigC}.

\medskip\noindent
By the definition of $C_t^N(\phi)$ in~Lemma~\ref{lem:fluct-split}, we have
\[\Var_\circ[C_t^N(\phi)]\,=\,N\Var_\circ\Big[\int_\Xd\phi\,m(t,\mu_0^N)\Big].\]
Appealing to the Helffer--Sj\"ostrand representation for the variance in terms of Glauber calculus, cf.~Lemma~\ref{lem:L0-prop}(iii), we get
\[\Var_\circ[C_t^N(\phi)]\,=\,N\sum_{j=1}^N\E_\circ\bigg[\Big(D^{\circ}_j\int_\Xd\phi\,m(t,\mu_0^N)\Big)\Lc_\circ^{-1}\Big(D^{\circ}_j\int_\Xd\phi\,m(t,\mu_0^N)\Big)\bigg].\]
By exchangeability, this is equivalently written as
\begin{equation*}
\Var_\circ[C_t^N(\phi)]\,=\,N^2\,\E_\circ\bigg[\Big(D^{\circ}_1\int_\Xd\phi\,m(t,\mu_0^N)\Big)\Lc_\circ^{-1}\Big(D^{\circ}_1\int_\Xd\phi\,m(t,\mu_0^N)\Big)\bigg].
\end{equation*}
Using the short-hand notation $\E_\circ^1:=\E_\circ[\,\cdot\,|\{Z_\circ^{i,N}\}_{2\le i\le N}]$, $\E_\circ^{\ne1}:=\E_\circ[\,\cdot\,|Z_{\circ}^{1,N}]$, and $\Var_\circ^{\ne1}:=\Var_\circ[\,\cdot\,|Z_{\circ}^{1,N}]$, and noting that $\E_\circ^{\ne1}\Lc_\circ^{-1}D_1^\circ=\Lc_\circ^{-1}\E_\circ^{\ne1}D_1^\circ$, we deduce from the triangle inequality
\begin{multline*}
\bigg|\Var_\circ[C_t^N(\phi)]-N^2\,\E_\circ\bigg[\bigg(\E_\circ^{\ne1}\Big[D^{\circ}_1\int_\Xd\phi\,m(t,\mu_0^N)\Big]\bigg)\Lc_\circ^{-1}\bigg(\E_\circ^{\ne1}\Big[D^{\circ}_1\int_\Xd\phi\,m(t,\mu_0^N)\Big]\bigg)\bigg]\bigg|\\
\,\le\,N^2\,\E_{\circ}^1\bigg[\Var_\circ^{\ne1}\Big[D_1^\circ \int_\Xd\phi\,m(t,\mu_0^N)\Big]\bigg].
\end{multline*}
Since we have $\Lc_\circ X=X$ for any $\sigma(Z_\circ^{1,N})$-measurable random variable $X$ with $\E_\circ[X]=0$, the operator~$\Lc_\circ^{-1}$ can be replaced by $\Id$ in the left-hand side.
Further appealing to the variance inequality~\eqref{eq:Poinc-Glauber} for Glauber calculus, we are led to
\begin{equation*}
\bigg|\Var_\circ[C_t^N(\phi)]-N^2\,\E_\circ\bigg[\E_\circ^{\ne1}\Big[D^{\circ}_1\int_\Xd\phi\,m(t,\mu_0^N)\Big]^2\bigg]\bigg|
\,\le\,N^3\,\E_{\circ}\bigg[\Big|D^{\circ}_2 D^{\circ}_1 \int_\Xd\phi\,m(t,\mu_0^N)\Big|^2\bigg].
\end{equation*}
Using~\eqref{eq:multi_d_Glauber} to bound Glauber derivatives by means of linear derivatives, appealing to Lemma~\ref{lem:diff-mk} to estimate the latter, and recalling the exponential moment assumption for~$\mu_\circ$, we obtain
\begin{equation}\label{eq:conv-var-preco}
\bigg|\Var_\circ[C_t^N(\phi)]-N^2\,\E_\circ\bigg[\E_\circ^{\ne1}\Big[D^{\circ}_1\int_\Xd\phi\,m(t,\mu_0^N)\Big]^2\bigg]\bigg|
\,\lesssim_{W,\beta,p,a,\mu_\circ}\,N^{-1}e^{-2p\lambda_0 t}\|\phi\|_{W^{2,\infty}(\Xd)}^2.
\end{equation}
It remains to evaluate the Glauber derivative in the left-hand side.
Recalling again the link between Glauber and linear derivatives, cf.~\eqref{eq:glauber_general-rel}, and appealing to Lemma~\ref{lem:diff-mk} for the computation of the linear derivative, we get
\begin{multline}\label{eq:decomp-derDminit}
{D^{\circ}_1 \int_{\Xd} \phi \, m(t,\mu^N_0)}\\
\,=\,N^{-1}\int_0^1\int_{\Xd}\int_{\Xd}\bigg(\int_\Xd \phi \, m^{(1)}\big(t,\mu_0^N+\tfrac{1-s}N(\delta_z-\delta_{Z_\circ^{1,N}}),y\big)\bigg)\,(\delta_{Z_\circ^{1,N}}-\delta_z)(\ddr y)\,\mu_\circ(\ddr z)\,\ddr s.
\end{multline}
Let us further appeal to the definition of linear derivative to replace the measure $\mu_0^N+\frac{1-s}N(\delta_z-\delta_{Z_\circ^{1,N}})$ in the argument of $m^{(1)}$ by
\[\mu_{0,z'}^N\,:=\,\mu_0^N+\tfrac{1}N(\delta_{z'}-\delta_{Z_\circ^{1,N}})\,=\,\tfrac1N\delta_{z'}+\tfrac1N\textstyle\sum_{j=2}^N\delta_{Z_\circ^{j,N}},\]
where $z'$ is a new variable integrated over with respect to $\mu_\circ$.
Using Lemma~\ref{lem:diff-mk} to estimate the additional linear derivative that constitutes the resulting error, we get
\begin{multline*}
\E_\circ\bigg[\Big|{D^{\circ}_1 \int_{\Xd} \phi \, m(t,\mu^N_0)}
-N^{-1}\int_{\Xd}\int_{\Xd}\int_{\Xd}\Big(\int_\Xd \phi \, m^{(1)}(t,\mu_{0,z'}^N,y)\Big)\,(\delta_{Z_\circ^{1,N}}-\mu_\circ)(\ddr y)\,\mu_\circ(\ddr z')\Big|^2\bigg]^\frac12\\
\,\lesssim_{W,\beta,p,a,\mu_\circ}\, N^{-2}e^{-p\lambda_0 t}\|\phi\|_{W^{2,\infty}(\Xd)}.
\end{multline*}
Inserting this into~\eqref{eq:conv-var-preco} and reorganizing expectations and integrals, we obtain
\begin{multline*}
\bigg|\Var_\circ[C_t^N(\phi)]-\int_\Xd\E_\circ\bigg[\int_{\Xd}\Big(\int_\Xd \phi \, m^{(1)}(t,\mu_{0}^N,y)\Big)\,(\delta_{z}-\mu_\circ)(\ddr y)\bigg]^2\,\mu_\circ(\ddr z)\bigg|\\
\,\lesssim_{W,\beta,p,a,\mu_\circ}\,N^{-1}e^{-2p\lambda_0 t}\|\phi\|_{W^{2,\infty}(\Xd)}^2.
\end{multline*}
We are now in position to appeal to Lemma~\ref{lem:chassagneux} to replace $\mu_0^N$ by $\mu_\circ$ in the expectation. Using as before Lemma~\ref{lem:diff-mk} and the exponential moment assumption for $\mu_\circ$ to estimate linear derivatives, this leads us to
\begin{multline}\label{eq:def-VarC-prim}
\bigg|\Var_\circ[C_t^N(\phi)]-\int_\Xd\bigg(\int_{\Xd}\Big(\int_\Xd \phi \, m^{(1)}(t,\mu_\circ,y)\Big)\,(\delta_{z}-\mu_\circ)(\ddr y)\bigg)^2\,\mu_\circ(\ddr z)\bigg|\\
\,\lesssim_{W,\beta,p,a,\mu_\circ}\,N^{-1}e^{-2p\lambda_0 t}\|\phi\|_{W^{2,\infty}(\Xd)}^2.
\end{multline}
Using the notation~\eqref{eq:def_U}, the definition of $m^{(1)}$ in Lemma~\ref{lem:diff-mk} amounts to
\begin{equation*}
m^{(1)}(t,\mu_\circ,y)=U_{t,0}[\delta_y-\mu_\circ],
\end{equation*}
and the claim~\eqref{eq:var-CtN-exp} follows with the limit variance $\sigma_t^C(\phi,\mu_\circ)$ defined in~\eqref{eq:def-sigC}.

\medskip
\step3 Conclusion.\\
By homogeneity of~$\ddr_2$ and by the triangle inequality, we can estimate
\begin{eqnarray*}
\lefteqn{\ddr_2\Big({C_t^N(\phi)}\,,\,{\sigma^C_t(\phi,\mu_\circ)\Nc}\Big)}\\
&=&\Var[C_t^N(\phi)]\,\ddr_2\bigg(\frac{C_t^N(\phi)}{\Var[C_t^N(\phi)]^{1/2}}\,,\,\frac{\sigma^C_t(\phi,\mu_\circ)}{\Var[C_t^N(\phi)]^{1/2}}\Nc\bigg)\\
&\le&\Var[C_t^N(\phi)]\,\ddr_2\bigg(\frac{C_t^N(\phi)}{\Var[C_t^N(\phi)]^{1/2}}\,,\,\Nc\bigg)+\tfrac12\Big|\Var[C_t^N(\phi)]-\sigma^C_t(\phi,\mu_\circ)^2\Big|.
\end{eqnarray*}
By the asymptotic normality~\eqref{eq:as-norm-Glaub} and by the convergence result~\eqref{eq:var-CtN-exp} for the variance, we then get
\begin{equation}\label{eq:estim-d2-preconcl-init}
\ddr_2\Big({C_t^N(\phi)}\,,\,{\sigma^C_t(\phi,\mu_\circ)\Nc}\Big)
\,\lesssim_{W,\beta,p,a,\mu_\circ}\, N^{-\frac12}e^{-2p\lambda_0 t}\|\phi\|_{W^{2,\infty}(\Xd)}^2\Big(1+\|\phi\|_{W^{2,\infty}(\Xd)}\Var[C_t^N(\phi)]^{-\frac12}\Big).
\end{equation}
It remains to deal with the last factor involving the inverse of the variance.
For that purpose, we distinguish between two cases:
\begin{enumerate}[---]
\item Case~1: $\sigma_t^C(\phi,\mu_\circ)^2\ge L\|\phi\|_{W^{2,\infty}(\Xd)}^2$.\\
In this case, the convergence result~\eqref{eq:var-CtN-exp} for the variance yields
\begin{eqnarray*}
\|\phi\|_{W^{2,\infty}(\Xd)}\Var[C_t^N(\phi)]^{-\frac12}&\le&\|\phi\|_{W^{2,\infty}(\Xd)}\Big(\sigma_t^C(\phi,\mu_\circ)^2-C N^{-1}e^{-2p\lambda_0 t}\|\phi\|_{W^{2,\infty}(\Xd)}^2 \Big)_+^{-\frac12}\\
&\le&(L-CN^{-1}e^{-2p\lambda_0 t})_+^{-\frac12},
\end{eqnarray*}
so that~\eqref{eq:estim-d2-preconcl-init} becomes
\begin{equation*}
\ddr_2\Big({C_t^N(\phi)}\,,\,{\sigma^C_t(\phi,\mu_\circ)\Nc}\Big) \\
\,\lesssim\,N^{-\frac12}e^{-2p\lambda_0 t}\|\phi\|_{W^{2,\infty}(\Xd)}^2\Big(1+ (L-CN^{-1}e^{-2p\lambda_0 t})_+^{-\frac12}\Big).
\end{equation*}
\item Case~2: $\sigma_t^C(\phi,\mu_\circ)^2\le L\|\phi\|_{W^{2,\infty}(\Xd)}^2$.\\
In this case, the convergence result~\eqref{eq:var-CtN-exp} for the variance yields
\begin{eqnarray*}
\ddr_2\Big({C_t^N(\phi)}\,,\,{\sigma^C_t(\phi,\mu_\circ)\Nc}\Big)
&\le&\sigma^C_t(\phi,\mu_\circ)^2+\Var_\circ[C_t^N(\phi)]\\
&\le&(2L+C   N^{-1}e^{-2p\lambda_0t})\|\phi\|_{W^{2,\infty}(\Xd)}^2.
\end{eqnarray*}
\end{enumerate}
Optimizing between those two cases and choosing $p=\frac14\theta$, the conclusion follows.
\end{proof}

\subsection{Brownian fluctuations}
We establish the following quantitative CLT for Brownian fluctuations $D_t^N(\phi)$.
The proof is based on the Lions expansion combined with a simple heat-kernel PDE argument.
The additional estimate~\eqref{eq:CLT-0-indep} in the statement below ensures that the CLT for $D_t^N(\phi)$ also holds given $C_t^N(\phi)$, which is key to deduce a joint result.

\begin{lem}[Brownian fluctuations]\label{prop:CLT-0}
Let $\kappa_1$ be as in Lemma \ref{lem:fluct-split}
and let~$\kappa\in[0,\kappa_1]$ be fixed. The random variable~$D_t^N(\phi)$ defined in Lemma~\ref{lem:fluct-split} satisfies for all $t\ge0$,
\begin{equation}\label{eq:CLT-0-main}
\ddr_2\big(\,D_t^N(\phi)\,,\,\sigma_t^D(\phi,\mu_\circ)\,\Nc\,\big)\,\lesssim_{W,\beta,a,\mu_\circ}\,N^{-\frac12} \|\phi\|_{W^{4,\infty}(\Xd)}^2,
\end{equation}
where the limit variance is given by
\begin{equation}\label{eq:def-limvar-sigmaD}
\sigma_t^D(\phi,\mu_\circ)^2\,:=\,\int_0^t\Big(\int_\Xd\big|\sigma_0^T\nabla U_{t,s}^*[\phi])\big|^2\,m(s,\mu_\circ)\Big)\,\ddr s,
\end{equation}
where we recall that $U_{t,s}^*$ is defined in~\eqref{eq:dual-test-lin-semigr},
that $\ddr_2$ is the second-order Zolotarev metric~\eqref{eq:Zolo}, and that $\Nc$ stands for a standard normal random variable.
In addition, for all $h\in C^2_b(\R^2)$ and $t\ge0$,
\begin{multline}\label{eq:CLT-0-indep}
\Big|\E\big[h\big(C_t^N(\phi),D_t^N(\phi)\big)\big]-\E_\circ\E_\Nc\big[h\big(C_t^N(\phi),\sigma_t^D(\phi,\mu_\circ)\Nc\big)\big]\Big|\\[-1mm]
\,\lesssim_{W,\beta,a,\mu_\circ}\,N^{-\frac12}\|\phi\|_{W^{4,\infty}(\Xd)}^2\|\partial_2^2h\|_{\Ld^\infty(\R^2)},
\end{multline}
where the standard normal variable $\Nc$ is taken independent both of initial data and of Brownian forces, and where we denote by $\E_\Nc$ the expectation with respect to $\Nc$.
\end{lem}

\begin{proof}
Let $\lambda_0$ be as in Theorem~\ref{thm:ergodic},
let $1 < q \le 2$, $0 < p \le \frac14\theta$, and $\kappa_1$ be fixed as in the proof of Lemma~\ref{lem:fluct-split}, and let $\kappa\in[0,\kappa_1]$.
We focus on the proof of~\eqref{eq:CLT-0-main}, while the additional statement~\eqref{eq:CLT-0-indep} can be obtained along the exact same lines --- simply replacing the test function $g$ below by $h(C_t^N(\phi),\cdot)$ and recalling that $C_t^N(\phi)$ is independent of Brownian forces.
We split the proof into two steps.

\medskip
\step1 Proof that for all $g\in C^2_b(\R)$ and $t\ge0$,
\begin{equation}\label{eq:ODE-char-re}
\Big|\E\big[g(D_{t}^N(\phi))\big]-\E_\circ\E_{\Nc}\Big[g\big(\sigma_t^D(\phi,\mu_0^N)\,\Nc\big)\Big]\Big|\,\lesssim_{W,\beta,a,\mu_\circ}\,N^{-\frac12}\|\phi\|_{W^{4,\infty}(\Xd)}^2 \|g''\|_{\Ld^\infty(\R)},
\end{equation}
where the limit variance is defined in~\eqref{eq:def-limvar-sigmaD},
and where as in the statement $\Nc$ stands for a standard normal random variable taken independent both of initial data and of Brownian forces.

\medskip\noindent
To prove this result, let us consider the $(\mathcal{F}^B_s)_{0\le s \le t}$-martingale $(D_{s,t}^N(\phi))_{0\le s\le t}$ given by
\[D_{s,t}^N(\phi)\,:=\,\frac1{\sqrt N}\sum_{i=1}^N\int_0^s(\partial_\mu U)(t-u,\mu_u^N)(Z^{i,N}_u)\cdot\sigma_0\ddr B_u^i,\]
which satisfies
\begin{equation*}
D_{0,t}^N(\phi)=0,\qquad D_{t,t}^N(\phi)=D_{t}^N(\phi).
\end{equation*}
Let $g\in C^2_b(\R)$ be fixed. By definition of $D_{s,t}^N(\phi)$, It\^o's lemma yields for all $\rho\in\R$ and $0\le s\le t$,
\begin{equation}\label{eq:ODE-char}
\frac{\ddr}{\ddr s}\E_B\big[g(\rho +D_{s,t}^N(\phi))\big]\,=\,\tfrac12\E_B\bigg[g''(\rho +D_{s,t}^N(\phi))\int_\Xd\big|\sigma_0^T(\partial_\mu U)(t-s,\mu_s^N)\big|^2\,\mu_s^N\bigg].
\end{equation}
Appealing to the Lions expansion in form of Corollary~\ref{cor:CST}(ii), we find for all $0\le s\le t$,
\begin{multline*}
\bigg\|\int_\Xd\big|\sigma_0^T(\partial_\mu U)(t-s,\mu_s^N)\big|^2\,\mu_s^N
-\int_\Xd\big|\sigma_0^T(\partial_\mu U)(t-s,m(s,\mu_0^N))\big|^2\,m(s,\mu_0^N)\bigg\|_{\Ld^2(\Omega_B)}\\
\,\lesssim\,N^{-\frac12}\E_B\bigg[\int_0^s\int_\Xd\big|\partial_\mu H_{t-s}(s-u,\mu_u^N)(z)\big|^2\mu_u^N(\ddr z)\,\ddr u\bigg]^\frac12\\
+N^{-1}\E_B\bigg[\Big(\int_0^s\int_\Xd\big|\partial_\mu^2H_{t-s}(s-u,\mu_u^N)(z,z)\big|\,\mu_u^N(\ddr z)\,\ddr u\Big)^2\bigg],
\end{multline*}
in terms of
\[H_{t-s}(u,\mu)\,:=\,\int_{\Xd} \Big| \sigma_0^T \partial_\mu U(t-s,m(u,\mu))\Big|^2 m(u,\mu).\]
Appealing to Lemma~\ref{lem:diff-mk} to estimate the multiple linear derivatives,
%using $x^3 \le (3/ep)^3 e^{px}$ and 
together with the moment bounds of Lemma~\ref{lem:unif-mom-est}, we get after straightforward computations for all $0\le s\le t$,
\begin{multline*}
\bigg\|\int_\Xd\big|\sigma_0^T(\partial_\mu U)(t-s,\mu_s^N)\big|^2\,\mu_s^N
-\int_\Xd\big|\sigma_0^T(\partial_\mu U)(t-s,m(s,\mu_0^N))\big|^2\,m(s,\mu_0^N)\bigg\|_{\Ld^2(\Omega_B)}\\
\,\lesssim_{W,\beta,p,a}\, N^{-\frac12}e^{-p\lambda_0(t-s)} e^{CpQ(\mu^N_0)^\theta}\|\phi\|_{W^{4,\infty}(\Xd)}^2.
\end{multline*}
Inserting this into~\eqref{eq:ODE-char}, and using the short-hand notation
\[\kappa_{s,t}(\phi,\mu)\,:=\,\int_\Xd\big|\sigma_0^T(\partial_\mu U)(t-s,m(s,\mu))\big|^2\,m(s,\mu),\]
we deduce for all $\rho\in\R$, $0\le s\le t$, and $\lambda \in [0,\frac18\lambda_0)$,
\begin{multline*}
\bigg|\frac{\ddr}{\ddr s}\E_B\big[g(\rho +D_{s,t}^N(\phi))\big]-\tfrac12\kappa_{s,t}(\phi, \mu_0^N)\,\frac{\ddr^2}{\ddr \rho^2}\E_B\big[g(\rho +D_{s,t}^N(\phi))\big]\bigg|\\
\,\lesssim_{W,\beta,p,a}\, N^{-\frac12}e^{-p\lambda_0 (t-s)}e^{CpQ(\mu^N_0)^\theta} \|\phi\|_{W^{4,\infty}(\Xd)}^2\|g''\|_{\Ld^\infty(\R)} .
\end{multline*}
In order to solve this approximate heat equation for the map $(s,\rho)\mapsto \E_B[g(\rho+D_{s,t}^N(\phi))]$ on $[0,t]\times\R$,
we can appeal for instance to the Feynman--Kac formula with $D_{0,t}^N(\phi)=0$. Equivalently, this amounts to integrating the above estimate with the associated heat kernel. It leads us to deduce for all~$\rho\in\R$ and $0\le s\le t$,
\begin{equation*}
\bigg|\E_B\big[g(\rho +D_{s,t}^N(\phi))\big]-\E_{\Nc}\Big[g\Big(\rho +\Nc\sqrt{\textstyle\int_0^s \kappa_{u,t}(\phi,\mu_0^N)\,\ddr u}\Big)\Big]\bigg|\,\lesssim_{W,\beta,p,a}\,N^{-\frac12}e^{CpQ(\mu_0^N)} \|\phi\|_{W^{4,\infty}(\Xd)}^2 \|g''\|_{\Ld^\infty(\R)}.
\end{equation*}
In particular, setting $\rho=0$ and $s=t$, recalling $D_{t,t}^N(\phi)=D_t^N(\phi)$, taking the expectation with respect to initial data, and recalling the exponential moment assumption for $\mu_\circ$, we get
\begin{equation*}
\bigg|\E\big[g(D_{t}^N(\phi))\big]-\E_\circ\E_{\Nc}\Big[g\Big(\Nc\sqrt{\textstyle\int_0^t \kappa_{u,t}(\phi,\mu_0^N)\,\ddr u}\Big)\Big]\bigg|\,\lesssim_{W,\beta,p,a,\mu_\circ}\,N^{-\frac12} \, \|\phi\|_{W^{4,\infty}(\Xd)}^2 \, \|g''\|_{\Ld^\infty(\R)}.
\end{equation*}
Using the notation~\eqref{eq:def_U} and recalling that the definition of $m^{(1)}$ in Lemma~\ref{lem:diff-mk} amounts to
\begin{equation*}
m^{(1)}(t,\mu_\circ,y)=U_{t,0}[\delta_y-\mu_\circ],
\end{equation*}
we recognize the definition~\eqref{eq:def-limvar-sigmaD} of $\sigma_t^D$,
\begin{eqnarray}
\int_0^t \kappa_{u,t}(\phi,\mu)\,\ddr u
&=&\int_0^t\Big(\int_\Xd\big|\sigma_0^T(\partial_\mu U)(t-s,m(s,\mu))\big|^2\,m(s,\mu) \Big)\,\ddr s\nonumber\\
&=&\int_0^t\Big(\int_\Xd\big|\sigma_0^T\nabla U_{t,s}^*[\phi])\big|^2\,m(s,\mu)\Big)\,\ddr s\nonumber\\
&=&\sigma_t^D(\phi,\mu)^2,\label{eq:redef-sigmaD}
\end{eqnarray}
and the claim~\eqref{eq:ODE-char-re} follows.

\medskip
\step2 Proof that for all $g\in C^2_b(\R)$ and $t\ge0$,
\begin{equation}\label{eq:sigma-t-0-appr}
\Big|\E_\circ\E_{\Nc}\Big[g\big(\sigma_t^D(\phi,\mu_0^N)\,\Nc\big)\Big]-\E_{\Nc}\Big[g\big(\sigma_t^D(\phi,\mu_\circ)\,\Nc\big)\Big]\Big|\,\lesssim_{W,\beta,a,\mu_\circ}\,N^{-\frac12} \|\phi\|_{W^{4,\infty}(\Xd)}^2\|g''\|_{\Ld^\infty(\R)}.
\end{equation}
Combining this with the result~\eqref{eq:ODE-char-re} of Step~1, and recalling the definition~\eqref{eq:Zolo} of the second-order Zolotarev metric, this will conclude the proof of~\eqref{eq:CLT-0-main}.
Set for shortness $\sigma_t^N:=\sigma_t^D(\phi,\mu_0^N)$ and~$\sigma_t:=\sigma_t^D(\phi,\mu_\circ)$.
We can decompose
\begin{equation*}
\E_{\Nc}\big[g(\sigma_t^N\Nc)\big]-\E_{\Nc}\big[g(\sigma_t\, \Nc)\big]
\,=\,(\sigma_t^N-\sigma_t)\int_0^1\E_{\Nc}\Big[\Nc g'\Big(\big(\sigma_t+\rho(\sigma_t^N-\sigma_t)\big)\Nc\Big)\Big]\,\ddr\rho,
\end{equation*}
and a Gaussian integration by parts then yields
\begin{equation*}
\E_{\Nc}\big[g(\sigma_t^N\Nc)\big]-\E_{\Nc}\big[g(\sigma_t \, \Nc)\big]
\,=\,(\sigma_t^N-\sigma_t)\int_0^1\big(\sigma_t+\rho(\sigma_t^N-\sigma_t)\big)\,\E_{\Nc}\Big[g''\Big(\big(\sigma_t+\rho(\sigma_t^N-\sigma_t)\big)\Nc\Big)\Big]\,\ddr\rho.
\end{equation*}
Hence,
\begin{equation*}
\Big|\E_{\Nc}\big[g(\sigma_t^N\Nc)\big]-\E_{\Nc}\big[g(\sigma_t \, \Nc)\big]\Big|
\,\le\,|\sigma_t^N-\sigma_t|\big(|\sigma_t|+|\sigma_t^N-\sigma_t|\big)\|g''\|_{\Ld^\infty(\R)}.
\end{equation*}
Taking the expectation with respect to initial data, the claim~\eqref{eq:sigma-t-0-appr} would follow provided that we could show for all $t\ge0$,
\begin{eqnarray}
|\sigma_t|&\lesssim_{W,\beta,a,\mu_\circ}& \|\phi\|_{W^{2,\infty}(\Xd)},\label{eq:estim-sigmat}\\
\E_\circ[|\sigma_t^N-\sigma_t|^2]^\frac12&\lesssim_{W,\beta,a,\mu_\circ}&N^{-\frac12}\|\phi\|_{W^{4,\infty}(\Xd)}.\label{eq:estim-sigmaN0t}
\end{eqnarray}
For that purpose, we first recall that by~\eqref{eq:redef-sigmaD} we can write 
\begin{equation}\label{eq:redef-sigmaD-bis}
\sigma_t^D(\phi,\mu)^2\,=\,\int_0^t\Big(\int_\Xd\big|\sigma_0^T(\partial_\mu U)(t-s,m(s,\mu))\big|^2\,m(s,\mu) \Big)\,\ddr s,
\end{equation}
with the short-hand notation $U(t,\mu):=\int_\Xd\phi\, m(t,\mu)$.
Applying Lemma~\ref{lem:diff-mk} to estimate the linear derivative, and combining it with the moment bounds of Lemma~\ref{lem:unif-mom-est}, the claim~\eqref{eq:estim-sigmat} follows after straightforward computations.
We turn to the proof of~\eqref{eq:estim-sigmaN0t}.
By the triangle inequality, we can decompose
\[\E_\circ\big[|\sigma_t^D(\phi,\mu_0^N)-\sigma_t^D(\phi,\mu_\circ)|^2\big]^\frac12\,\le\,\big|\E_\circ[\sigma_t^D(\phi,\mu_0^N)]-\sigma_t^D(\phi,\mu_\circ)\big|+\Var_\circ[\sigma_t^D(\phi,\mu_0^N)]^\frac12,\]
and we estimate both terms separately.
On the one hand, starting again from~\eqref{eq:redef-sigmaD-bis}, appealing to Lemma~\ref{lem:chassagneux}, using Lemma~\ref{lem:diff-mk} to estimate the multiple linear derivatives, and combining it with the bounds of Lemma~\ref{lem:unif-mom-est} and with the exponential moment assumption for $\mu_\circ$, we easily get
\[\big|\E_\circ[\sigma_t^D(\phi,\mu_0^N)]-\sigma_t^D(\phi,\mu_\circ)\big|\,\lesssim_{W,\beta,a,\mu_\circ}\, N^{-1} \|\phi\|_{W^{4,\infty}(\Xd)}^2.\]
On the other hand, using the variance inequality~\eqref{eq:Poinc-Glauber} for Glauber calculus, and appealing to~\eqref{eq:decompo_Glauber_derivative} to bound Glauber derivatives in terms of linear derivatives, we find
\begin{multline*}
\Var_\circ\big[\sigma_t^D(\phi,\mu_0^N)\big]
\,\le\,\sum_{j=1}^N\E_\circ\big[|D^{\circ}_j \sigma_t^D(\phi,\mu_0^N)|^2\big]\\ 
\,\lesssim\,N^{-1} \, \E_\circ \bigg[ \Big| \int_0^1 \int_\Xd\int_\Xd \Lind{\sigma_t^D}\Big(\phi,\mu^N_0 + \tfrac{1-s}{N}(\delta_z - \delta_{Z_0^{1,N}}), y\Big)\, (\delta_{Z^{1,N}_0} - \delta_z)(\ddr y) \,\mu_\circ(\ddr z)\,\ddr s \Big|^2 \bigg].
\end{multline*}
Further using Lemma~\ref{lem:diff-mk} to estimate the multiple linear derivatives, and combining it with the subgaussian bounds of Lemma~\ref{lem:unif-mom-est} and with the exponential moment assumption for $\mu_\circ$, we deduce
\[\Var_\circ[\sigma_t^D(\phi,\mu_0^N)]\,\lesssim_{W,\beta,a,\mu_\circ}\,N^{-1}\|\phi\|_{W^{3,\infty}(\Xd)}^2,\]
and the claim~\eqref{eq:estim-sigmaN0t} follows.
\end{proof}

\subsection{Proof of Theorem~\ref{th:CLT}}
Let~$\lambda_0$ be as in Theorem~\ref{thm:ergodic}, $\kappa_1$ as in Lemma~\ref{lem:fluct-split}, and let $\kappa\in[0,\kappa_1]$ be fixed.
Let also~$g\in C^2_b(\R)$ be momentarily fixed with $g'(0)=0$ and $\|g''\|_{\Ld^\infty(\R)}=1$.
By Lemma~\ref{lem:fluct-split}, we find
\[\Big|\E\big[g(S_t^N(\phi))\big]-\E\big[g\big(C_t^N(\phi)+D_t^N(\phi)\big)\big]\Big|\,\lesssim_{W,\beta,a,\mu_\circ}\, N^{-\frac12}\|\phi\|_{W^{3,\infty}(\Xd)}^2.\]
Next, appealing to~\eqref{eq:CLT-0-indep} in Lemma~\ref{prop:CLT-0} for the asymptotic normality of $D_t^N(\phi)$ given $C_t^N(\phi)$, we deduce
\begin{equation}\label{eq:CLT-Y-part-Brow}
\Big|\E\big[g(S_t^N(\phi))\big]-\E_\circ\E_\Nc\big[g\big(C_t^N(\phi)+\sigma_t^D(\phi,\mu_\circ)\Nc\big)\big]\Big|\,\lesssim_{W,\beta,a,\mu_\circ}\,N^{-\frac12} \|\phi\|_{W^{4,\infty}(\Xd)}^2,
\end{equation}
where $\Nc$ stands for a standard normal random variable taken independent both of initial data and of Brownian forces.
It remains to combine this with our analysis of fluctuations of $C_t^N(\phi)$.
Appealing to the asymptotic normality of $C_t^N(\phi)$ as stated in Lemma~\ref{lem:init-fluct}, and testing that result in Zolotarev metric with the function $\E_{\Nc}[g(\cdot+\sigma_t^D(\phi,\mu_\circ)\Nc)]\in C^2_b(\R)$, we deduce
\begin{multline*}
\Big|\E\big[g(S_t^N(\phi))\big]-\E_{\Nc}\E_{\Nc'}\big[g\big(\sigma_t^C(\phi,\mu_\circ)\Nc'+\sigma_t^D(\phi,\mu_\circ)\Nc\big)\big]\Big|\\
\,\lesssim_{W,\beta,a,\mu_\circ}\, N^{-\frac12}\|\phi\|_{W^{4,\infty}(\Xd)}^2\bigg(1+e^{-\frac12\theta\lambda_0 t}\Big(\tfrac{\sigma_t^C(\phi,\mu_\circ)}{\|\phi\|_{W^{2,\infty}(\Xd)}}+ (N^{-\frac13}e^{-\frac12\theta\lambda_0 t})^{\frac12}\Big)^{-1}\bigg),
\end{multline*}
where $\Nc'$ stands for another standard normal random variable taken independent of initial data, of Brownian forces, and of~$\Nc$.
Taking the supremum over $g$, and noting that $\sigma_t^C(\phi,\mu_\circ)\Nc'+\sigma_t^D(\phi,\mu_\circ)\Nc$ has the same distribution as $\sigma_t(\phi,\mu_\circ)\Nc$ with total variance
\[\sigma_t(\phi,\mu_\circ)^2\,:=\,\sigma_t^C(\phi,\mu_\circ)^2+\sigma_t^D(\phi,\mu_\circ)^2,\]
we conclude
\begin{align*}
\ddr_2\Big(S_t^N(\phi)\,,\,\sigma_t(\phi,\mu_\circ)\Nc\Big)
&~~\lesssim_{W,\beta,a,\mu_\circ}
N^{-\frac12}\|\phi\|_{W^{4,\infty}(\Xd)}^2\bigg(1+e^{-\frac12\theta\lambda_0 t}\Big(\tfrac{\sigma_t^C(\phi,\mu_\circ)}{\|\phi\|_{W^{2,\infty}(\Xd)}}+ (N^{-\frac13}e^{-\frac12\theta\lambda_0 t})^{\frac12}\Big)^{-1}\bigg)\\
&~~\le~~\Big(N^{-\frac12}+N^{-\frac13} e^{-\frac14\theta\lambda_0 t}\Big)\|\phi\|_{W^{4,\infty}(\Xd)}^2.
\end{align*}
Noting that the total variance $\sigma_t(\phi,\mu_\circ)$ coincides with the variance predicted by the Gaussian Dean--Kawasaki equation, cf.~\eqref{eq:cov-DK}, the conclusion follows.\qed

%%%%%%%%%%%%%%%%%%%%%%%%%%%%%%%%%%%%%%%%%%%%
%%%%%%%%%%%%%%%%%%%%%%%%%%%%%%%%%%%%%%%%%%%%
%%%%%%%%%%%%%%%%%%%%%%%%%%%%%%%%%%%%%%%%%%%%

\section{Concentration estimates}\label{sec:concentration}

This section is devoted to the proof of Theorem~\ref{th:concentr}. As the initial law $\mu_\circ$ is assumed here to have compact support, we can choose e.g.\@ $\theta = 1$. Let $1 < q \le 2$ and $0 < p \le \frac14$ be fixed with $pq' \gg_{\beta,a} 1$. Using the notation $\kappa_{\ell,p}$ from Theorem~\ref{thm:ergodic}(ii), we set $\kappa_1 := \min (\kappa_{2,p}, \kappa_{3,p}) \in (0,\kappa_0]$, and we let $\kappa \in [0,\kappa_1]$ be fixed.
Let also $\lambda_0$ be as in Theorem~\ref{thm:ergodic}.
For $t\ge0$ and $\phi\in C^\infty_c(\Xd)$, consider the centered random variables
\[Y_t^N(\phi)\,:=\,X_t^N(\phi)-\E[X_t^N(\phi)]\,=\,\int_\Xd\phi\,\mu_t^N-\E\Big[\int_\Xd\phi\,\mu_t^N\Big].\]
We establish concentration by means of moment estimates.
By the Lions expansion of~Lemma~\ref{lem:CST0}, we can decompose
\begin{equation*}
Y_t^N(\phi)\,=\,\tilde Y_t^N(\phi)+M_t^N(\phi)+\frac1{2N}\big(E_t^N(\phi)-\E[E_t^N(\phi)]\big),
\end{equation*}
in terms of
\begin{eqnarray*}
\tilde Y_t^N(\phi)&:=&\int_\Xd\phi\,m(t,\mu_0^N)-\E\Big[\int_\Xd\phi\,m(t,\mu_0^N)\Big],\\
M_t^N(\phi)&:=&\frac1N\sum_{i=1}^N\int_0^t\partial_\mu U(t-s,\mu_s^N)(Z_s^{i,N})\cdot\sigma_0\ddr B_s^i,\\
E_t^N(\phi)&:=&\int_0^t\int_\Xd\Tr\Big[a_0\partial_\mu^2U(t-s,\mu_s^N)(z,z)\Big]\mu_s^N(\ddr z)\,\ddr s,
\end{eqnarray*}
where we use the short-hand notation
\[U(t-s,\mu)\,:=\,\int_\Xd\phi\, m(t-s,\mu).\]
For all $k\ge1$, we may then decompose the $k$th moment of $Y_t^N(\phi)$ as
\begin{equation}\label{eq:mom-decompY}
\E\big[|Y_t^N(\phi)|^k\big]\,\le\,3^k\E_\circ\big[|\tilde Y_t^N(\phi)|^k\big]+3^k\E\big[|M_t^N(\phi)|^k\big]+3^kN^{-k}\E\big[|E_t^N(\phi)|^k\big].
\end{equation}
We separately analyze the three right-hand side terms and split the proof into four steps.
In the sequel, all constants $C$ are implicitly allowed to depend on $W,\beta,p,q,a$, as well as on the support of $\mu_\circ$, which is assumed here to be compact.

\medskip
\step{1} Proof that for all $1\le k\le N$, 
\begin{equation}\label{eq:bound-tildeYtN}
\E_\circ\big[|\tilde Y_t^N(\phi)|^k\big]\,\le\,\Big(\frac{Ck}{N}\Big)^\frac k2e^{-kp\lambda_0t}\|\langle z\rangle^{-p}\phi\|_{W^{2,q'}(\Xd)}^k.
\end{equation}
Using~\eqref{eq:multi_d_Glauber} to estimate the Glauber derivative by means of a linear derivative, using Lemma~\ref{lem:diff-mk} to control the latter, and noting that the support assumption for $\mu_\circ$ ensures that $Q(\mu_0^N)$ and $|Z_0^{j,N}|$ are uniformly bounded almost surely, we get
\begin{equation*}
|D^{\circ}_j \tilde Y_t^N(\phi)|\,\le\, CN^{-1}e^{-p\lambda_0 t}\|\langle z\rangle^{-p}\phi\|_{W^{2,q'}(\Xd)}.
\end{equation*}
We may then appeal to Proposition~\ref{lem:concentration_Glauber}, to the effect of
\begin{eqnarray*}
\E_\circ\big[|\tilde Y_t^N(\phi)|^k\big]
&\le&\inf_{\lambda> 0}\Big\{k!\lambda^{-k}\E_\circ\big[e^{\lambda\tilde Y_t^N(\phi)}\big]\Big\}\\
&\le&\inf_{\lambda> 0}\Big\{k!\lambda^{-k}\exp\Big(C\lambda e^{-p\lambda_0t} \|\langle z\rangle^{-p}\phi\|_{W^{2,q'}(\Xd)}\big(e^{2\lambda CN^{-1}e^{-p\lambda_0t}\|\langle z\rangle^{-p}\phi\|_{W^{2,q'}(\Xd)}}-1\big)\Big)\Big\}.
\end{eqnarray*}
Choosing $\lambda=(kN)^{1/2}(2Ce^{-p\lambda_0t}\|\langle z\rangle^{-p}\phi\|_{W^{2,q'}(\Xd)})^{-1}$, the claim~\eqref{eq:bound-tildeYtN} follows.

\medskip
\step{2} Proof that for all $1\le k\le N$,
\begin{equation}\label{eq:bound-martMtN}
\E [ |M^N_t(\phi)|^k] \,\le\,\Big(\frac{Ck}{N}\Big)^\frac k2\|\langle z\rangle^{-p}\phi\|_{W^{2,q'}(\Xd)}^k.
\end{equation}
We recall that $(M^N_t(\phi))_t$ is a martingale with quadratic variation given by
\begin{equation*}
\langle M^N_t(\phi) \rangle \,=\, \frac{1}{N^2} \sum_{i=1}^N \int_0^t \big| \partial_\mu U(t-s, \mu^N_s )(Z^{i,N}_s)\big|^2\, \ddr s.
\end{equation*}
Appealing to Lemma~\ref{lem:diff-mk} to estimate the L-derivative, we get with our choice of $p,q,\kappa$,
\begin{equation*}
\langle M^N_t(\phi) \rangle \,\le\, CN^{-1}\|\langle z\rangle^{-p}\phi\|_{W^{2,q'}(\Xd)}^2\int_0^te^{-2p\lambda_0(t-s)} Q(\mu_s^N)^{2p}e^{4pQ(\mu^N_s)}\ddr s.
\end{equation*}
By the Burkholder--Davis--Gundy inequality, see e.g.~\cite[Theorem~1]{Wang_1991}, we have for all $k\ge1$,
\[\E [ |M^N_t(\phi)|^k] \,\le\, (Ck)^{\frac{k}{2}} \E[ \langle M^N_t(\phi) \rangle^{\frac{k}{2}}],\]
and thus
\begin{equation*}
\E [ |M^N_t(\phi)|^k] \,\le\,\Big(\frac{Ck}{N}\Big)^\frac k2\|\langle z\rangle^{-p}\phi\|_{W^{2,q'}(\Xd)}^k\max_{0\le s\le t}\E\big[Q(\mu^N_s)^{kp}e^{2kpQ(\mu^N_s)}\big].
\end{equation*}
Appealing to Lemma \ref{lem:unif-mom-est} together with the compact support assumption on $\mu_\circ$, we find for $s\ge0$,
\begin{equation*}
\E \big[ Q(\mu^N_s)^{kp}e^{2kpQ(\mu^N_s)} \big]\,\le\,\E \big[e^{3kpQ(\mu^N_s)} \big] \,\le\, e^{Ck^2/N} \E_\circ \big[ e^{CkQ(\mu^N_0)}\big] \,\lesssim\, e^{Ck+Ck^2/N},
\end{equation*}
and the claim~\eqref{eq:bound-martMtN} follows.

\medskip
\step{3} Proof that for all $1\le k\le N$, 
\begin{equation}\label{eq:EN_moment}
\E[|E_t^N(\phi)|^k]\,\le\,C^k\|\langle z\rangle^{-p}\phi\|_{W^{3,q'}(\Xd)}^k.
\end{equation}
Recalling the definition of $E_t^N(\phi)$, and appealing to Lemma~\ref{lem:diff-mk} to estimate multiple L-derivatives, we get
\[|E_t^N(\phi)|\,\lesssim\,\|\langle z\rangle^{-p}\phi\|_{W^{3,q'}(\Xd)}\int_0^te^{-p\lambda_0(t-s)}Q(\mu^N_s)^{2p}e^{4pQ(\mu^N_s)}\,\ddr s.\]
Now appealing to the moment bounds of Lemma~\ref{lem:unif-mom-est}, together with Jensen's inequality and with the compact support assumption for $\mu_\circ$, the claim~\eqref{eq:EN_moment} follows.

\medskip
\step{4} Conclusion.\\
Inserting the results of the first three steps into~\eqref{eq:mom-decompY}, we obtain for all $1\le k\le N$, 
\[\E[|Y_t^N(\phi)|^k]\,\le\,\Big(\frac{Ck}{N}\Big)^\frac k2\|\phi\|_{W^{3,\infty}(\Xd)}^k. \]
Combining with the trivial bound $|Y_t^N(\phi)|\le2\|\phi\|_{L^\infty(\Xd)}$, we find that the above estimate actually holds for all $k\ge1$ (without the restriction to $k\le N$).
For $r\ge0$, we may then deduce by Markov's inequality, for all $k\ge1$,
\begin{equation*}
\pr{Y_t^N(\phi)\ge r}\,\le\,r^{-k}\,\E[|Y_t^N(\phi)|^k]\,\le\,\bigg(\frac{Ck}{Nr^2}\bigg)^\frac k2\|\phi\|_{W^{3,\infty}(\Xd)}^k.
\end{equation*}
Choosing
$k=N r^2\big(e C\|\phi\|_{W^{3,\infty}(\Xd)}^2\big)^{-1}$,
the conclusion follows.
\qed

%%%%%%%%%%%%%%%%%%%%%%%%%%%%%%%%%%%%%%%%%%%%
%%%%%%%%%%%%%%%%%%%%%%%%%%%%%%%%%%%%%%%%%%%%
%%%%%%%%%%%%%%%%%%%%%%%%%%%%%%%%%%%%%%%%%%%%

\section{Ergodic Sobolev estimates for mean field}\label{sec:ergodic}
This section is devoted to the proof of Theorem~\ref{thm:ergodic}.
As recalled in Section~\ref{sec:prel-ergodic}, the convergence to equilibrium in item~(i) of the statement is well known.
We focus here on the proof of the ergodic estimates in item~(ii) and we restrict to the kinetic Langevin setting~\eqref{eq:Langevin-par}, while we emphasize that the same arguments can be repeated and substantially simplified in the Brownian setting (in which case the estimates further hold on the unweighted spaces $W^{-k,1}(\R^d)$).

As stated in item~(i), recall that the mean-field evolution~\eqref{eq:VFP} has a unique steady state~$M$ for~$\kappa$ in $[0,\kappa_0]$, which can be characterized in the Langevin setting as the unique solution of the fixed-point Gibbs equation
\begin{equation*}
M(x,v) \,=\, c_M \exp\Big[-\beta \Big(\tfrac12|v|^2 + A(x) + \kappa W\ast M(x) \Big)\Big], \qquad z=(x,v) \in \Xd,
\end{equation*}
where $c_M$ is the normalizing constant such that $\int_{\Xd} M= 1$. Note that this fixed-point equation has indeed a unique solution provided that $\kappa \beta \|W\|_{\Ld^\infty(\R^d)} < 1$.

\subsection{Exponential relaxation for modified linearized operators}
We start by considering the following modified version of the linearized operator $L_\mu$ defined in~\eqref{eq:def_linearized-Lang}, where we remove the (compact) convolution term:
given a measure $\mu\in\Pc(\Xd)$, we define for all~$h\in C^\infty_c(\Xd)$,
\begin{equation}\label{eq:def-Rmu}
R_\mu h\,:=\,\tfrac12\Div_v((\nabla_v+\beta v)h)-v\cdot\nabla_xh+(\nabla A+\kappa\nabla W\ast\mu)\cdot\nabla_v h.
\end{equation}
Given $\mu_\circ\in\Pc(\Xd)$, $s\ge0$, and $h_s\in C^\infty_c(\Xd)$, recalling that $\mu_t:=m(t,\mu_\circ)$ stands for the solution of the mean-field equation~\eqref{eq:VFP},
we consider the following (non-autonomous) equation,
\begin{equation*}
\left\{\begin{array}{l}
\partial_th_t=R_{\mu_t}h_t,\quad\text{for $t\ge s$},\\
h_t|_{t=s}=h_s.
\end{array}\right.
\end{equation*}
It is easily checked that this linear parabolic equation is well-posed with $h\in C_\loc([s,\infty);\Ld^2(M^{-1/2}))$ whenever the initial condition~$h_s$ belongs to $\Ld^2(M^{-1/2})$.
We then consider the associated fundamental solution operators $\{V_{t,s}\}_{t\ge s\ge0}$ on~$\Ld^2(M^{-1/2})$ defined by
\begin{equation}\label{eq:def-Vts}
h_t\,=\,V_{t,s}h_s.
\end{equation}
Note that $\int_\Xd V_{t,s}h_s=\int_\Xd h_s$ for all $t\ge s$.
We establish the following exponential convergence result to the steady state.

\begin{prop}\label{prop:estim-Vts-Wk1*}
Let $\kappa_0$ be as in Theorem~\ref{thm:ergodic}(i) and let $\kappa\in[0,\kappa_0]$.
There exists $\lambda_1>0$ (only depending on $d,\beta,a$, $\|W\|_{W^{1,\infty}(\R^d)}$) such that the following holds: given $1 < q \le 2$ and $0 < p \le 1$ with $pq' \gg_{\beta,a} 1$ large enough (only depending on $d,\beta,a$), we have for all $k\ge0$, $h_s\in C^\infty_c(\Xd)$, $t\ge s\ge0$, and $\delta\in[0,1]$,
\begin{equation*}
\Big\|V_{t,s}h_s-M\int_\Xd h_s\Big\|_{W^{-k,q}(\langle z\rangle^p)}
\,\lesssim_{W,\beta,k,p,q,a}\,
e^{-(\delta\wedge p)\lambda_1(t-s)}\Big(1+e^{-\delta\lambda_1s}Q(\mu_\circ)^{\delta+\kappa^2C_k}\Big)\|h_s\|_{W^{-k,q}(\langle z\rangle^p)},
\end{equation*}
for some constant $C_k$ only depending on $d,\beta,k,a$, $\|W\|_{W^{k+2,\infty}(\R^d)}$,
and where the multiplicative constant further depends on $p,q$, $\|W\|_{W^{k+d+1,\infty}(\R^d)}$.
\end{prop}

To prove this result, we shall first establish the exponential decay on negative Sobolev spaces with the Gibbs weight~$M$, that is, the exponential decay on the smaller spaces $H^{-k}(M^{-1/2})$, and next we shall appeal to the enlargement theory of Gualdani, Mischler and Mouhot~\cite{Gualdani_2017, Mischler_2016} to conclude with the desired result on~$W^{-k,q}(\langle z\rangle^p)$. Here, for all $k\ge0$, the space~$H^{-k}(M^{-1/2})$ is defined as the weighted negative Sobolev space associated with the dual norm
\begin{equation}\label{eq:def-H-kM}
\|h\|_{H^{-k}(M^{-1/2})}\,:=\,\sup\bigg\{\int_\Xd hh'M^{-1}\,:\,\|h'\|_{H^k(M^{-1/2})}=1\bigg\},
\end{equation}
where $H^k(M^{-1/2})$ is the standard weighted Sobolev space with norm
\[\|h\|_{H^k(M^{-1/2})}\,:=\,\sup_{i,j\ge0,\,i+j\le k}\|\nabla_x^i\nabla_v^jh\|_{\Ld^2(M^{-1/2})},\qquad\|h\|_{\Ld^2(M^{-1/2})}\,:=\,\Big(\int_\Xd|h|^2M^{-1}\Big)^\frac12.\]
Note that the treatment of the weight in the definition of those weighted spaces differs slightly from the one in the definition of $W^{-k,q}(\langle z\rangle^p)$, cf.~\eqref{eq:def-W-kqzp}, but for convenience we stick to this slight inconsistency in the choice of definitions.

In order to appeal to enlargement theory, we start by introducing a suitable decomposition of the operator~$R_\mu$.
Let a cut-off function $\chi\in C^\infty_c(\Xd)$ be fixed with $\chi(z)=1$ for $|z|\le1$, and set
\begin{equation}\label{eq:def_chiR}
\chi_R(z):=\chi(\tfrac1Rz).
\end{equation}
In those terms, we split the operator $R_\mu$ as follows,
\begin{eqnarray}
R_\mu&:=& \mathcal{A}+B_\mu,\label{eq:def_calB}\\
\calA h&:=&\Lambda\chi_Rh,\nonumber\\
B_\mu h&:=&\tfrac12\Div_v((\nabla_v+\beta v)h)-v\cdot\nabla_xh+(\nabla A+\kappa\nabla W\ast\mu)\cdot\nabla_vh-\Lambda\chi_Rh,\nonumber
\end{eqnarray}
for some constants $\Lambda,R>0$ to be properly chosen later on (see Lemmas~\ref{lem:estim-Wts} and~\ref{lem:estim-reg} below).
Let us denote by $\{W_{t,s}\}_{t\ge s\ge0}$ the fundamental solution operators for the (non-autonomous) evolution equation associated with~$B_\mu$: for all~$s\ge0$ and $h_s\in C^\infty_c(\Xd)$, we define $h_t=W_{t,s}h_s$ as the solution of
\begin{equation*}
\left\{\begin{array}{l}
\partial_th_t=B_{\mu_t}h_t,\quad\text{for $t\ge s$},\\
h_t|_{t=s}=h_s.
\end{array}\right.
\end{equation*}
Again, it is easily checked that this equation is well-posed with $h\in C_\loc([s,\infty);\Ld^2(M^{-1/2}))$ whenever $h_s\in\Ld^2(M^{-1/2})$.
Our proof of Proposition~\ref{prop:estim-Vts-Wk1*} is based on the following three preliminary lemmas, the proofs of which are postponed to Sections~\ref{sec:pr-lem:estim-Vts-Hk}, \ref{sec:pr-lem:estim-Wts}, \ref{sec:pr-lem:estim-Wts-2},
and~\ref{sec:pr-lem:estim-reg}
below.

\begin{lem}[Exponential decay on restricted space]\label{lem:estim-Vts-Hk}
Let $\kappa_0$ be as in Theorem~\ref{thm:ergodic}(i) and let $\kappa\in[0,\kappa_0]$.
There exists $\lambda_1>0$ (only depending on $d,\beta,a$, $\|W\|_{W^{1,\infty}(\R^d)}$) such that the following holds: for all $k\ge0$, $h_s\in C^\infty_c(\Xd)$, $t\ge s\ge0$, and $\delta\in[0,1]$,
{\allowdisplaybreaks\begin{multline*}
\Big\|V_{t,s}h_s-M\int_\Xd h_s\Big\|_{H^{-k}(M^{-1/2})}\\
\,\le\,C_{k}e^{-\delta\lambda_1(t-s)}\|g_s\|_{H^{-k}(M^{-1/2})}\Big(1+\kappa\big( e^{-\lambda_1 s}Q(\mu_\circ)\big)^\delta\Big)\Big(1+e^{-\lambda_1s}\big(e^{\lambda_1t}\wedge Q(\mu_\circ) \big)\Big)^{\kappa^2C_{k}}\\
\,\le\,C_{k}e^{-\delta\lambda_1(t-s)}\|g_s\|_{H^{-k}(M^{-1/2})}\Big(1+e^{-\lambda_1 s}Q(\mu_\circ)\Big)^{\delta+\kappa^2C_{k}},
\end{multline*}}
for some constant $C_{k}$ only depending on $d,\beta,k,a$, $\|W\|_{W^{k+2,\infty}(\R^d)}$.
$\|W\|_{W^{k+2,\infty}(\R^d)}$.
\end{lem}

\begin{lem}[Exponential decay for modified operator]\label{lem:estim-Wts}
Let $\kappa_0$ be as in Theorem~\ref{thm:ergodic}(i) and let~$\kappa\in[0,\kappa_0]$.
There exists $\lambda_2>0$ (only depending on $d,\beta,a$, $\|W\|_{W^{1,\infty}(\R^d)}$), such that the following holds: given $1<q\le2$ and $0<p\le1$ with $pq'\gg_{\beta,a}1$ large enough (only depending on $d,\beta,a$), choosing $\Lambda,R$ large enough (only depending on $d,\beta,p,a$, $\|W\|_{W^{1,\infty}(\R^d)}$), we have for all $k\ge0$, $h_s\in C^\infty_c(\Xd)$, and~$t\ge s\ge0$,
\begin{eqnarray}
\|W_{t,s}h_s\|_{W^{-k,q}(\langle z\rangle^p)}&\lesssim_{W,\beta,k,p,q,a}&e^{-p \lambda_2(t-s)}\|h_s\|_{W^{-k,q}(\langle z\rangle^p)},\label{eq:estim-Wts-zp}\\
\|W_{t,s}h_s\|_{H^{-k}(M^{-1/2})}&\lesssim_{W,\beta,k,p,q,a}&e^{-\lambda_2(t-s)}\|h_s\|_{H^{-k}(M^{-1/2})},\label{eq:estim-Wts-M}
\end{eqnarray}
where the multiplicative constants only depend on $d,\beta,k,p,q,a$, $\|W\|_{W^{k+1,\infty}(\R^d)}$.
\end{lem}

\begin{lem}[Regularization estimate]\label{lem:estim-reg}
Let $\kappa_0,\lambda_2$ be as in Theorem~\ref{thm:ergodic}(i) and Lemma~\ref{lem:estim-Wts}, respectively, and let $\kappa\in[0,\kappa_0]$.
There is some $n\ge1$ large enough (only depending on $d$) such that the following holds: given $1<q\le2$ and $0<p\le1$ with $pq'\gg_{\beta,a}1$ large enough (only depending on $d,\beta,a$), choosing $\Lambda,R$ large enough (only depending on $d,\beta,p,a$, $\|W\|_{W^{1,\infty}(\R^d)}$),
we have for all $k\ge0$, $h_s\in C^\infty_c(\Xd)$, and $t\ge s\ge0$,
\begin{multline*}
\int_{s\le s_1\le \ldots\le s_n\le t}\|\calA W_{t,s_n}\calA W_{s_n,s_{n-1}}\ldots \calA W_{s_1,s}h_s\|_{H^{-k}(M^{-1/2})}\,\ddr s_1\ldots\ddr s_n\\
\,\lesssim_{W,\beta,k,p,q,a}\,e^{-p\lambda_2(t-s)}\|h_s\|_{W^{-k,q}(\langle z\rangle^p)},
\end{multline*}
where the multiplicative constant only depends on $d,\beta,k,p,q,a$, $\|W\|_{W^{k+d+1,\infty}(\R^d)}$.
\end{lem}

With those lemmas at hand, we are now in position to conclude the proof of Proposition~\ref{prop:estim-Vts-Wk1*} based on the enlargement theory of Gualdani, Mischler, and Mouhot~\cite{Gualdani_2017, Mischler_2016}.

\begin{proof}[Proof of Proposition~\ref{prop:estim-Vts-Wk1*}]
Let $\lambda_1,\lambda_2$ be defined in Lemmas~\ref{lem:estim-Vts-Hk} and~\ref{lem:estim-Wts}, respectively,
and let $1<q\le2$, $k\ge0$, and $0<p\le1$ be fixed with $pq'\gg1$ large enough in the sense of Lemmas~\ref{lem:estim-Wts} and~\ref{lem:estim-reg}.
We note that the space $H^{-k}(M^{-1/2})$ is continuously embedded in $W^{-k,q}(\langle z\rangle^p)$: by definition of dual norms, we find for all $h\in C^\infty_c(\Xd)$,
\begin{equation}\label{eq:embedding-W-k1-Hk}
\|h\|_{W^{-k,q}(\langle z\rangle^p)}\,\lesssim_{W,\beta,k,a}\,\|h\|_{H^{-k}(M^{-1/2})},
\end{equation}
where the constant only depends on $d,\beta,k,a$, $\|W\|_{W^{k,\infty}(\R^d)}$.
In this setting, we can appeal to enlargement theory to extend the estimates of Lemma~\ref{lem:estim-Vts-Hk} to $W^{-k,q}(\langle z\rangle^p)$:
by Lemmas~\ref{lem:estim-Vts-Hk}, \ref{lem:estim-Wts}, and~\ref{lem:estim-reg}, we can apply~\cite[Theorem~1.1]{Mischler_2016} and the conclusion precisely follows.
For completeness, we include a short proof as the present situation does not exactly fit in the semigroup setting of~\cite{Mischler_2016}.
Starting point is the following form of the Duhamel formula: based on the decomposition~\eqref{eq:def_calB}, the fundamental solution operators $\{V_{t,s}\}_{0\le s\le t}$ can be expanded around $\{W_{t,s}\}_{0\le s\le t}$ via
\[V_{t,s}\,=\,W_{t,s}+\int_s^tV_{t,u}\calA W_{u,s}\,\ddr u.\]
By iteration, we get for any $n\ge1$,
\[V_{t,s}\,=\,W_{t,s}+\sum_{j=1}^{n-1}\int_s^tW_{t,u}(\calA W)^j_{u,s}\,\ddr u+\int_s^tV_{t,u}(\calA W)^n_{u,s}\,\ddr u,\]
where we have set for abbreviation, for all $j\ge1$ and $0\le s_0\le s_j$,
\[(\calA W)^{ j}_{s_j,s_0}\,:=\,\int_{s_0\le s_1\le\ldots\le s_{j-1}\le s_j}\calA W_{s_j,s_{j-1}}\calA W_{s_{j-1},s_{j-2}}\ldots \calA W_{s_1,s_0}\,\ddr s_1\ldots\ddr s_{j-1}.\]
Given $h_s\in C^\infty_c(\Xd)$,
taking norms,
applying the exponential decay of Lemma~\ref{lem:estim-Wts} for $\{W_{t,s}\}_{0\le s\le t}$ on the space~$W^{-k,q}(\langle z\rangle^p)$, noting that $A$ is bounded on $W^{-k,q}(\langle z\rangle^p)$ and that $|\int_\Xd h|\lesssim_{k,p,q}\|h\|_{W^{-k,q}(\langle z\rangle^p)}$ provided $pq'>2d$, we get for any $n\ge1$,
\begin{multline*}
\Big\|V_{t,s}h_s-M\int_\Xd h_s\Big\|_{W^{-k,q}(\langle z\rangle^p)}
\,=\,\Big\|V_{t,s}h_s-M\int_\Xd V_{t,s}h_s\Big\|_{W^{-k,q}(\langle z\rangle^p)}\\
\,\lesssim_{W,\beta,k,p,q,n,a}\,\big(1+(t-s)^n\big)e^{-p\lambda_2(t-s)}\|h_s\|_{W^{-k,q}(\langle z\rangle^p)}\\
+\int_{s}^t\Big\|V_{t,u}(\calA W)^n_{u,s}h_s-M\int_\Xd (\calA W)^n_{u,s}h_s\Big\|_{W^{-k,q}(\langle z\rangle^p)}\,\ddr u.
\end{multline*}
In order to estimate the last right-hand side term, we recall the embedding~\eqref{eq:embedding-W-k1-Hk} and we use the exponential relaxation of Lemma~\ref{lem:estim-Vts-Hk} for $\{V_{t,s}\}_{0\le s\le t}$ on the space $H^{-k}(M^{-1/2})$: for any $\delta\in[0,1]$, we obtain
\begin{multline*}
\Big\|V_{t,s}h_s-M\int_\Xd h_s\Big\|_{W^{-k,q}(\langle z\rangle^p)}
\,\lesssim_{W,\beta,k,p,q,n,a}\,\big(1+(t-s)^{n}\big)e^{-p\lambda_2(t-s)}\|h_s\|_{W^{-k,q}(\langle z\rangle^p)}\\
+\Big(1+e^{-\lambda_1s}Q(\mu_\circ)\Big)^{\delta+\kappa^2C_k}\int_s^te^{-\delta\lambda_1(t-u)}\|(\calA W)^{ n}_{u,s}h_s\|_{H^{-k}(M^{-1/2})}\,\ddr u,
\end{multline*}
for some constant $C_k$ only depending on $d,\beta,k,a$, $\|W\|_{W^{k+2,\infty}(\R^d)}$.
Now using the regularization estimate of Lemma~\ref{lem:estim-reg} for $n\ge1$ large enough (only depending on $d$), we are led to
\begin{multline*}
\Big\|V_{t,s}h_s-M\int_\Xd h_s\Big\|_{W^{-k,q}(\langle z\rangle^p)}
\,\lesssim_{W,\beta,k,p,q,a}\,\big(1+(t-s)^{n}\big)e^{-p\lambda_2(t-s)}\|h_s\|_{W^{-k,q}(\langle z\rangle^p)}\\
+\|h_s\|_{W^{-k,q}(\langle z\rangle^p)}\Big(1+e^{-\lambda_1s}Q(\mu_\circ)\Big)^{\delta+\kappa^2C_k}\int_s^te^{-\delta\lambda_1(t-u)}e^{-p\lambda_2(u-s)}\,\ddr u,
\end{multline*}
and the conclusion follows (up to renaming $\frac12(\lambda_1\wedge\lambda_2)$ as $\lambda_1$).
\end{proof}

\subsection{Proof of Theorem~\ref{thm:ergodic}(ii)}
In this section, we establish Theorem~\ref{thm:ergodic}(ii) perturbatively as a consequence of Proposition~\ref{prop:estim-Vts-Wk1*}.
As a preliminary, we start with the following a priori estimates for the mean-field evolution~\eqref{eq:MFL-McKean} on the weighted spaces $W^{-k,q}(\langle z\rangle^p)$.

\begin{lem}\label{lem:dual_estimate_m}
Let $\kappa\in[0,1]$. There exists $\lambda>0$ (only depending on $\beta,a$) such that the following holds: given $1<q\le2$ and $0 < p \le 1$ with $pq'\ge4d$, we have for all $k\ge1$ and $t\ge0$,
\[\|m(t,\mu_\circ)\|_{W^{-k,q}(\langle z\rangle^p)}\,\lesssim\, \int_\Xd\langle z\rangle^p\,m(t,\mu_\circ)\,\lesssim\,1+e^{-p\lambda t}Q(\mu_\circ)^p,\]
where multiplicative constants only depend on $d,\beta,a$, $\|W\|_{W^{1,\infty}(\R^d)}$.
\end{lem}

\begin{proof}
Set $\mu_t:=m(t,\mu_\circ)$ for abbreviation.
Noting that for $q'\ge pq'\ge4d$ and $k\ge1$ the Sobolev embedding yields
\begin{equation*}
\|\mu_t\|_{W^{-k,q}(\langle z\rangle^{p})}\,\le\,\|\mu_t\|_{W^{-1,q}(\langle z\rangle^{p})}\,\lesssim\,\int_\Xd\langle z\rangle^{p}\,\mu_t(\ddr z)\,\le\,\Big(\int_\Xd\langle z\rangle^2\mu_t(\ddr z)\Big)^{\frac p2}\,=\,Q(\mu_t)^p,
\end{equation*}
the conclusion follows from Lemma~\ref{lem:unif-mom-est}(i).
\end{proof}

With this estimate at hand, we can now conclude the proof of Theorem~\ref{thm:ergodic}(ii) in the Langevin setting. As we treat the (compact) convolution term $(L_\mu-R_\mu)h=\kappa(\nabla W\ast h)\cdot\nabla_v\mu$ perturbatively, we emphasize that some care is needed to obtain a result that is valid for $\kappa \in [0,\kappa_0]$ independently of~$k,p,q$. This is achieved through an iterative procedure.

\begin{proof}[Proof of Theorem~\ref{thm:ergodic}(ii)]
By a standard approximation argument, it suffices to consider $f_\circ\in C^\infty_c(\Xd)$
with $\int_{\Xd} f_\circ = 0$.
In that case, the well-posedness of the Cauchy problem~\eqref{eq:erg_q} is standard and it remains to establish the stability estimate~\eqref{eq:estim-concl-qt}.
Let $\kappa_0,\lambda_1$ be as in Proposition~\ref{prop:estim-Vts-Wk1*}, and assume that~$\lambda_1$ is chosen to be smaller than the exponent~$\lambda$ in Lemma~\ref{lem:dual_estimate_m}.
In terms of the modified linearized operator~$R_\mu$ defined in~\eqref{eq:def-Rmu}, further setting $\mu_t:=m(t,\mu_\circ)$ for abbreviation,
and defining the operator
\[A_th\,:=\,(\nabla W\ast h)\cdot\nabla_v\mu_t,\]
the equation~\eqref{eq:erg_q}
can be reformulated as
\[\partial_tf_t\,=\,R_{\mu_t}f_t+\kappa A_tf_t,\]
hence, by Duhamel's formula,
\begin{equation*}
f_t=V_{t,0}f_\circ
+\kappa\int_0^tV_{t,s}A_sf_s\,\ddr s.
\end{equation*}
Iterating this formula,
we get for any $n\ge1$,
\begin{multline}\label{eq:ft-duhiterated}
f_t\,=\,
V_{t,0}f_\circ
+\sum_{m=2}^{n}\kappa^{m-1}\int_{0\le s_1\le\ldots\le s_{m-1}\le t}V_{t,s_{m-1}}A_{s_{m-1}}V_{s_{m-1},s_{m-2}}\ldots A_{s_1}V_{s_1,0}f_\circ\,\ddr s_1\ldots \ddr s_{m-1}\\
+\kappa^{n}\int_{0\le s_1\le \ldots\le s_n\le t}V_{t,s_n}A_{s_n}V_{s_n,s_{n-1}}\ldots A_{s_2}V_{s_2,s_1}A_{s_1}f_{s_1}\,\ddr s_1\ldots \ddr s_{n}.
\end{multline}
From here, we split the proof into two steps.

\medskip
\step1 Estimate on $A_t$:
given $1<q_0, q\le2$ and $0<p_0,p\le1$ with $p_0q_0',pq'\ge4d$, we have for all $k_0\ge2$, $k\ge0$, $h\in C_c^\infty(\Xd)$, and $t\ge0$,
\begin{equation}\label{eq:At-estim}
\|A_t h\|_{W^{-k_0,q_0}(\langle z \rangle^{p_0})}
\,\lesssim_{W,\beta,k,a}\,\Big(1+e^{-p_0\lambda_1t}Q(\mu_\circ)^{p_0}\Big)\|h\|_{W^{- k, q}(\langle z \rangle^{ p})},
\end{equation}
where the multiplicative constant only depends on $d,\beta,k,a$, $\|W\|_{W^{k+2,\infty}(\R^d)}$.

\medskip\noindent
By definition of $A_t$, given $1<q_0,q\le2$ and $0<p_0,p\le1$ with $pq'\ge4d$, we can bound for all $k_0\ge2$, $k\ge0$, $h\in C^\infty_c(\Xd)$, and $t\ge0$,
\begin{eqnarray*}
\|A_t h\|_{W^{-k_0,q_0}(\langle z \rangle^{p_0})}
&\le&\|(\nabla W\ast h)\mu_t\|_{W^{-1,q_0}(\langle z \rangle^{p_0})}\\
&\lesssim&\|\mu_t\|_{W^{-1,q_0}(\langle z\rangle^{p_0})}\|W\ast h\|_{W^{2,\infty}(\R^d)}\\
&\lesssim_k&\|\mu_t\|_{W^{-1,q_0}(\langle z\rangle^{p_0})}\|W\|_{W^{k+2,\infty}(\R^d)}\|h\|_{W^{- k, q}(\langle z \rangle^{ p})}.
\end{eqnarray*}
Now applying Lemma~\ref{lem:dual_estimate_m}, provided that $p_0q_0'\ge4d$ and recalling $\lambda \ge \lambda_1$, the claim follows.

\medskip
\step2 Conclusion.\\
Starting from~\eqref{eq:ft-duhiterated}, leaving the last term aside, appealing to~\eqref{eq:At-estim} with $k_0=k$, $q_0=q$, $p_0=p$, as well as to the ergodic estimates of Proposition~\ref{prop:estim-Vts-Wk1*}, we obtain for all $1<q\le2$ and $0<p\le1$ with $pq'\gg_{\beta,a}1$ large enough (only depending on $d,\beta,a$), $k\ge2$, $n\ge1$, $t\ge0$, and~$\delta\in[0,1]$,
\begin{multline}\label{eq:estim-ft-iterate}
\|f_t\|_{W^{-k,q}(\langle z\rangle^p)}
\,\le\,
C_{k,p,q}e^{-(\delta\wedge p)\lambda_1t}\|f_\circ\|_{W^{-k,q}(\langle z\rangle^p)}\sum_{m=0}^{n-1} \tfrac{(\kappa C_{k,p,q}t)^m}{m!}Q(\mu_\circ)^{m(p+\delta+\kappa^2C_k)}
\\
+\kappa^{n}\int_{0\le s_1\le \ldots\le s_n\le t}\|V_{t,s_n}A_{s_n}V_{s_n,s_{n-1}}\ldots A_{s_2}V_{s_2,s_1}A_{s_1}f_{s_1}\|_{W^{-k,q}(\langle z\rangle^p)}\,\ddr s_1\ldots \ddr s_{n},
\end{multline}
for some constant $C_k$ only depending on $d,\beta,k,a$, $\|W\|_{W^{k+2,\infty}(\R^d)}$,
and some constant $C_{k,p,q}$ further depending on $p,q$, $\|W\|_{W^{k+d+1,\infty}(\R^d)}$. 
It remains to estimate the last right-hand side term, for which more care is needed in order to avoid a bad dependency in the multiplicative constant.
For that purpose, given a reference exponent $0<p_0\le1$, first note that we can choose $1<q_0\le2$ with $p_0q_0'\gg_{\beta,a}1$ large enough (only depending on $d,\beta,a$), such that:
\begin{enumerate}[---]
\item both Lemma~\ref{lem:dual_estimate_m} and Proposition~\ref{prop:estim-Vts-Wk1*} hold in the space $W^{-2,q_0}(\langle z\rangle^{p_0})$;
\smallskip\item for all $k\ge2$, $1<q\le q_0$, and $0<p\le\frac12p_0$, we have the following embedding, for all $h\in C^\infty_c(\Xd)$,
\[\|h\|_{W^{-k,q}(\langle z\rangle^p)}\,\lesssim\, \|h\|_{W^{-2,q_0}(\langle z\rangle^{p_0})},\]
where the multiplicative constant only depends on $d$.
\end{enumerate}
With this choice, for all $k\ge2$, $1<q\le q_0$, $0<p\le\frac12 p_0$, and $\delta\in[0,1]$, applying the above embedding and then using iteratively Proposition~\ref{prop:estim-Vts-Wk1*} in $W^{-2,q_0}(\langle z\rangle^{p_0})$ and the result~\eqref{eq:At-estim} of Step~1 with $k=k_0=2$, $q=q_0$, and $p=p_0$, we can estimate
\begin{eqnarray*}
\lefteqn{\|V_{t,s_n}A_{s_n}V_{s_n,s_{n-1}}\ldots A_{s_2}V_{s_2,s_1}A_{s_1}f_{s_1}\|_{W^{-k,q}(\langle z\rangle^p)}}\\
&\lesssim&\|V_{t,s_n}A_{s_n}V_{s_n,s_{n-1}}\ldots A_{s_2}V_{s_2,s_1}A_{s_1}f_{s_1}\|_{W^{-2,q_0}(\langle z\rangle^{p_0})}\\
&\le&C_1^n\|A_{s_1}f_{s_1}\|_{W^{-2,q_0}(\langle z\rangle^{p_0})}e^{-(\delta\wedge p_0)\lambda_1(t-s_1)}\Big(1+e^{-\delta\lambda_1s_{1}}Q(\mu_\circ)^{\delta+\kappa^2C_0}\Big)\\
&&\qquad\times\prod_{j=2}^n\Big(1+e^{-(\delta\wedge p_0)\lambda_1s_{j}}Q(\mu_\circ)^{p_0+\delta+\kappa^2C_0}\Big),
\end{eqnarray*}
for some constant $C_0$ only depending on $d,\beta,a$, $\|W\|_{W^{4,\infty}(\R^d)}$, and some constant $C_1$ further depending on $p_0,q_0,\|W\|_{W^{d+3,\infty}(\R^d)}$.
Now applying~\eqref{eq:At-estim} with $k_0=2$, given $1<q\le q_0$ and $0<p\le\frac12p_0$ with~$pq'\ge4d$, we deduce for all $k\ge2$ and $\delta\in[0,1]$,
\begin{multline*}
\|V_{t,s_n}A_{s_n}V_{s_n,s_{n-1}}\ldots A_{s_2}V_{s_2,s_1}A_{s_1}f_{s_1}\|_{W^{-k,q}(\langle z\rangle^p)}\\
\,\le\,C_2C_1^n\|f_{s_1}\|_{W^{-k,q}(\langle z\rangle^{p})}e^{-(\delta\wedge p_0)\lambda_1(t-s_1)}\prod_{j=1}^n\Big(1+e^{-(\delta\wedge p_0)\lambda_1s_{j}}Q(\mu_\circ)^{p_0+\delta+\kappa^2C_0}\Big),
\end{multline*}
for some constant $C_2$ only depending on $d,\beta,k,a$, $\|W\|_{W^{k+2,\infty}(\R^d)}$.
Expanding the product, we get
\begin{multline*}
\|V_{t,s_n}A_{s_n}V_{s_n,s_{n-1}}\ldots A_{s_2}V_{s_2,s_1}A_{s_1}f_{s_1}\|_{W^{-k,q}(\langle z\rangle^p)}
\,\le\,C_2C_1^n\|f_{s_1}\|_{W^{-k,q}(\langle z\rangle^{p})}e^{-(\delta\wedge p_0)\lambda_1(t-s_1)}\\
+C_2(2C_1)^n\|f_{s_1}\|_{W^{-k,q}(\langle z\rangle^{p})}Q(\mu_\circ)^{n(p_0+\delta+\kappa^2C_0)}e^{-(\delta\wedge p_0)\lambda_1(t-s_1)}\sum_{j=1}^ne^{-(\delta\wedge p_0)\lambda_1s_j}.
\end{multline*}
Taking the time integral,
we are then led to
\begin{multline*}
\int_{0\le s_1\le \ldots\le s_n\le t}\|V_{t,s_n}A_{s_n}V_{s_n,s_{n-1}}\ldots A_{s_2}V_{s_2,s_1}A_{s_1}f_{s_1}\|_{W^{-k,q}(\langle z\rangle^p)}\ddr s_1\ldots\ddr s_n\\
\,\le\,C_2(2C_1)^n((\delta\wedge p_0)\lambda_1)^{1-n}\int_0^te^{-\frac12(\delta\wedge p_0)\lambda_1(t-s)}\|f_{s}\|_{W^{-k,q}(\langle z\rangle^{p})}\,\ddr s\\
+C_2(4C_1)^n((\delta\wedge p_0)\lambda_1)^{1-n}Q(\mu_\circ)^{n(p_0+\delta+\kappa^2C_0)}
\int_0^te^{-\frac12(\delta\wedge p_0)\lambda_1t}\|f_{s}\|_{W^{-k,q}(\langle z\rangle^{p})}\ddr s.
\end{multline*}
Inserting this into~\eqref{eq:estim-ft-iterate}, applying Gr\"onwall's inequality, renaming the constants, and recalling that~$q_0$ can be chosen only depending on $d,\beta,a,p_0$, we conclude the following: given $1<q\le2$ and~$0<p\le\frac12p_0\le\frac12$ with $pq'\gg_{\beta,a}1$ large enough (only depending on $d,\beta,a$), we have for all $n\ge1$, $k\ge2$, $\delta\in[0,1]$, and $t\ge0$,
\begin{multline*}
\|f_t\|_{W^{-k,q}(\langle z\rangle^p)}
\,\le\,
C_{k,p,q}e^{-\frac12(\delta\wedge p)\lambda_1t}\|f_\circ\|_{W^{-k,q}(\langle z\rangle^p)}\sum_{m=0}^{n-1} \tfrac{(\kappa C_{k,p,q}t)^m}{m!}Q(\mu_\circ)^{m(p+\delta+\kappa^2C_k)}\\
\times\exp\bigg(C_k(\kappa \delta^{-1}C_1)^{n}t+C_k(\kappa \delta^{-1}C_1)^{n}Q(\mu_\circ)^{n(p_0+\delta+\kappa^2C_0)}\bigg),
\end{multline*}
for some constant $C_0$ only depending on $d,\beta,a$, $\|W\|_{W^{4,\infty}(\R^d)}$, some constant $C_1$ further depending on $p_0$, $\|W\|_{W^{d+3,\infty}(\R^d)}$,
some constant $C_k$ only depending on $d,\beta,k,a$, $\|W\|_{W^{k+2,\infty}(\R^d)}$, and some constant $C_{k,p,q}$ further depending on $p,q$, $\|W\|_{W^{k+d+1,\infty}(\R^d)}$.
We can apply this last estimate for instance in the following two ways:
\begin{enumerate}[---]
\item given $\theta\in(0,1]$, choosing $p_0=\frac12\theta$, $\delta=\frac14\theta$, $n=1$, and requiring $\kappa$ to be small enough such that $\kappa \delta^{-1}C_1C_k\le\frac14(\delta\wedge p)\lambda_1$ and $\kappa^2 C_0 \le \frac14\theta$;
\item choosing $p_0=\delta=1$, requiring $\kappa$ to be small enough such that $\kappa\delta^{-1}C_1\le\frac12$, and then choosing $n\ge1$ such that $2^{-n}C_k\le\frac14p\lambda_1$;
\end{enumerate}
and the conclusion follows (with $\lambda_0:=\frac14\lambda_1$).
\end{proof}
\dd

\subsection{Proof of Lemma~\ref{lem:estim-Vts-Hk}: ergodic estimates with Gibbs weight}\label{sec:pr-lem:estim-Vts-Hk}
This section is devoted to the proof of Lemma~\ref{lem:estim-Vts-Hk}. We start by considering the standard kinetic Fokker--Planck operator
\[R_M h\,:=\,\tfrac12\Div_v((\nabla_v+\beta v)h)-v\cdot\nabla_xh+(\nabla A+\kappa\nabla W\ast M)\cdot\nabla_vh.\]
The exponential relaxation of the associated semigroup $\{e^{tR_M}\}_{t\ge0}$ on $\Ld^2(M^{-{1/2}})$ was established in the seminal work of Dolbeault, Mouhot, and Schmeiser~\cite{Dolbeault_Hypocoercivity_2015} based on hypocoercivity techniques.
We post-process this well-known result to further derive estimates on Sobolev spaces with arbitrary negative regularity. For that purpose, we appeal to a duality argument and argue by induction using parabolic estimates.

\begin{lem}\label{lem:relax-LM0-FP}
Let $\kappa_0$ be as in Theorem~\ref{thm:ergodic}(i) and let $\kappa\in[0,\kappa_0]$.
There exists $\lambda_1>0$ (only depending on $d,\beta,a$, $\|W\|_{W^{1,\infty}(\R^d)}$) such that for all $k\ge0$, $h\in C^\infty_c(\Xd)$, and $t\ge0$,
\begin{equation}\label{eq:res00}
\Big\|e^{tR_M}h-M\int_\Xd h\Big\|_{H^{-k}(M^{-{1/2}})}\,\lesssim_{W,\beta,k,a}\,e^{-\lambda_1 t}\|h\|_{H^{-k}(M^{-{1/2}})},
\end{equation}
where the multiplicative factor only depends on $d,\beta,k,a$, $\|W\|_{W^{k+1,\infty}(\R^d)}$.
\end{lem}

\begin{proof}
We set for abbreviation $\pi_M^\bot h:=h-M\int_\Xd h$, and we note that $\int_\Xd e^{tR_M}h=\int_\Xd h$ for all $t\ge0$.
By definition of dual norms, cf.~\eqref{eq:def-H-kM}, also recalling the definition of the steady state $M$, it suffices to show that there is some $\lambda_1>0$ such that for all $k\ge0$, $h\in C^\infty_c(\Xd)$, and $t\ge0$,
\begin{equation*}
\|\pi_M^\bot e^{tR_M^*}h\|_{H^k(M^{-{1/2}})}
\,\lesssim_{W,\beta,k,a}\,e^{-\lambda_1 t}\|h\|_{H^k(M^{-{1/2}})},
\end{equation*}
where $R_M^*$ stands for the dual Fokker--Planck operator
\begin{equation}\label{eq:def-LM0-star}
R_M^* h\,=\,\tfrac12\Div_v((\nabla_v+\beta v)h)+v\cdot\nabla_xh-(\nabla A+\kappa\nabla W\ast M)\cdot\nabla_v h.
\end{equation}
We shall actually prove the following more detailed estimate, further capturing the dissipation: there is some $\lambda_1>0$ such that for all $k\ge0$, $h\in C^\infty_c(\Xd)$, and $t\ge0$,
\begin{multline}
    \label{eq:res00k}
e^{\lambda_1 t}\|\pi_M^\bot e^{tR_M^*}h\|_{H^k(M^{-{1/2}})}
+\Big(\int_0^t e^{2\lambda_1 s}\|(\nabla_v+\beta v)e^{s R_M^*}h\|_{H^k(M^{-{1/2}})}^2\,\ddr s\Big)^\frac12 \\
\,\lesssim_{W,\beta,k,a}\,
\|h\|_{H^k(M^{-{1/2}})}.
\end{multline}
We split the proof into two steps.

\medskip
\step1 Case $k=0$: there is some $\lambda_1>0$ (only depending on $d,\beta,a$, $\|W\|_{W^{1,\infty}(\Xd)}$) such that for all $h\in C^\infty_c(\Xd)$ and $t\ge0$,
\begin{equation}\label{eq:res001}
\|\pi_M^\bot e^{tR_M^*}h\|_{\Ld^2(M^{-{1/2}})}
\,\lesssim_{W,\beta,a}\,e^{-\lambda_1 t}\|\pi_M^\bot h\|_{\Ld^2(M^{-{1/2}})}.
\end{equation}
This was precisely established by Dolbeault, Mouhot, and Schmeiser in~\cite[Theorem~10]{Dolbeault_Hypocoercivity_2015}.

\medskip
\step2 Conclusion: proof of~\eqref{eq:res00k}.\\
Given $h\in C^\infty_c(\Xd)$, we set for shortness $J_{t}^{\alpha,\gamma}:=\nabla_x^\alpha\nabla_v^\gamma \pi_M^\bot e^{tR_M^*}h$ for multi-indices $\alpha,\gamma\in\N^d$.
By definition, it satisfies
\begin{equation}\label{eq:Jtalphagamma}
\left\{\begin{array}{l}
\partial_tJ_{t}^{\alpha,\gamma}=R_M^*J_t^{\alpha,\gamma}+r_{t}^{\alpha,\gamma},\quad\text{for $t\ge0$},\\
J_{t}^{\alpha,\gamma}|_{t=0}=\nabla_x^\alpha\nabla_v^\gamma(\pi_M^\bot h),
\end{array}\right.
\end{equation}
where the source term $r_{t}^{\alpha,\gamma}$ is given by
\begin{equation}\label{eq:def-rtalphagamma-remainder}
r_{t}^{\alpha,\gamma}\,:=\,[\nabla_x^\alpha\nabla_v^\gamma ,R_M^*]\pi_M^\bot e^{tR_M^*}h.
\end{equation}
On the one hand, by Duhamel's formula in form of
\begin{equation*}
J_t^{\alpha,\gamma}\,=\,e^{tR_M^*}\nabla_x^\alpha\nabla_v^\gamma(\pi_M^\bot h)+\int_0^te^{(t-s)R_M^*}r_s^{\alpha,\gamma}\,\ddr s,
\end{equation*}
the exponential decay~\eqref{eq:res001} yields
\begin{equation}\label{eq:estim-dec-Jt-0}
\|J_t^{\alpha,\gamma}\|_{\Ld^2(M^{-{1/2}})}\,\lesssim_{W,\beta,a}\,e^{-\lambda_1t}\|\nabla_x^\alpha\nabla_v^\gamma (\pi_M^\bot h)\|_{\Ld^2(M^{-{1/2}})}+\int_0^te^{-\lambda_1(t-s)}\|r_s^{\alpha,\gamma}\|_{\Ld^2(M^{-{1/2}})}\,\ddr s.
\end{equation}
On the other hand, integrating by parts, the energy identity for equation~\eqref{eq:Jtalphagamma} takes the form
\begin{eqnarray}
\partial_t\|J_t^{\alpha,\gamma}\|_{\Ld^2(M^{-{1/2}})}^2
&=&
2\int_\Xd J_t^{\alpha,\gamma}(R_M^*J_t^{\alpha,\gamma})M^{-1}
+2\int_\Xd J_t^{\alpha,\gamma}r_t^{\alpha,\gamma}M^{-1}
\nonumber\\
&\le&
-\|(\nabla_v+\beta v)J_t^{\alpha,\gamma}\|_{\Ld^2(M^{-{1/2}})}^2
+2\|J_t^{\alpha,\gamma}\|_{\Ld^2(M^{-{1/2}})}\|r_t^{\alpha,\gamma}\|_{\Ld^2(M^{-{1/2}})}.
\label{eq:estim-energy-Jalphagamma}
\end{eqnarray}
Regarding the dissipation term in this last estimate, we make the following observation: integrating by parts and using $\nabla_vM^{-1}=\beta vM^{-1}$, we find for all $f\in C^\infty_c(\Xd)$,
\begin{eqnarray}
\int_{\Xd} |\nabla_vf|^2 M^{-1}
&=&-\int_{\Xd}  M^{-1}f\big((\nabla_v+\beta v)\cdot\nabla_vf\big)\nonumber\\
&=&-\int_{\Xd}  M^{-1}f\,\Div_v\big((\nabla_v+\beta v)f\big)+\beta d\int_{\Xd}  |f|^2M^{-1}\nonumber\\
&=&\int_{\Xd}  |(\nabla_v+\beta v)f|^2M^{-1}+\beta d\int_{\Xd}  |f|^2M^{-1}.\label{eq:dissip-decomp}
\end{eqnarray}
Using this to replace half of the dissipation term in~\eqref{eq:estim-energy-Jalphagamma}, we get
\begin{multline*}
\partial_t\|J_t^{\alpha,\gamma}\|_{\Ld^2(M^{-{1/2}})}^2
+\tfrac12\Big(\|\nabla_vJ_t^{\alpha,\gamma}\|_{\Ld^2(M^{-{1/2}})}^2
+\|(\nabla_v+\beta v)J_t^{\alpha,\gamma}\|_{\Ld^2(M^{-{1/2}})}^2\Big)\\
\,\le\,
\tfrac{\beta d}2\|J_t^{\alpha,\gamma}\|_{\Ld^2(M^{-{1/2}})}^2+2\|J_t^{\alpha,\gamma}\|_{\Ld^2(M^{-{1/2}})}\|r_t^{\alpha,\gamma}\|_{\Ld^2(M^{-{1/2}})},
\end{multline*}
and thus, by Gr\"onwall's inequality, for all $\lambda\ge0$,
\begin{multline*}
\sup_{0\le s\le t}\Big(e^{2\lambda s}\|J_s^{\alpha,\gamma}\|_{\Ld^2(M^{-{1/2}})}^2\Big)
+\int_0^te^{2\lambda s}\Big(\|\nabla_vJ_s^{\alpha,\gamma}\|_{\Ld^2(M^{-{1/2}})}^2+\|(\nabla_v+\beta v)J_s^{\alpha,\gamma}\|_{\Ld^2(M^{-{1/2}})}^2\Big)\,\ddr s\\
\,\lesssim_{\beta,d,\lambda}\,
\|\nabla_x^\alpha\nabla_v^\gamma(\pi_M^\bot h)\|_{\Ld^2(M^{-{1/2}})}^2
+\Big(\int_0^te^{\lambda s}\|r_s^{\alpha,\gamma}\|_{\Ld^2(M^{-{1/2}})}\,\ddr s\Big)^2
+\int_0^te^{2\lambda s}\|J_s^{\alpha,\gamma}\|_{\Ld^2(M^{-{1/2}})}^2\,\ddr s,
\end{multline*}
where the multiplicative constant is of the form $C_{d, \beta}(1+\lambda)$.
Now using~\eqref{eq:estim-dec-Jt-0} to bound the last term, we obtain for all $0\le\lambda<\lambda'<\lambda_1$,
\begin{multline*}
\sup_{0\le s\le t}\Big(e^{2\lambda s}\|J_s^{\alpha,\gamma}\|_{\Ld^2(M^{-{1/2}})}^2\Big)
+\int_0^te^{2\lambda s}\Big(\|\nabla_vJ_s^{\alpha,\gamma}\|_{\Ld^2(M^{-{1/2}})}^2+\|(\nabla_v+\beta v)J_s^{\alpha,\gamma}\|_{\Ld^2(M^{-{1/2}})}^2\Big)\,\ddr s\\
\,\lesssim_{W,\beta,\lambda,\lambda',a}\,
\|\nabla_x^\alpha\nabla_v^\gamma(\pi_M^\bot h)\|_{\Ld^2(M^{-{1/2}})}^2
+\Big(\int_0^te^{\lambda' s}\|r_s^{\alpha,\gamma}\|_{\Ld^2(M^{-{1/2}})}\,\ddr s\Big)^2.
\end{multline*}
By definition of $R_M^*$, cf.~\eqref{eq:def-LM0-star}, the source term $r^{\alpha,\gamma}$ defined in~\eqref{eq:def-rtalphagamma-remainder} takes the form
\begin{equation*}
r_{t}^{\alpha,\gamma}=
\sum_{i:e_i\le\gamma}\binom{\gamma}{e_i}J_t^{\alpha+e_i,\gamma-e_i}
+\!\!\sum_{(\alpha',\gamma')<(\alpha,\gamma)}\binom{\alpha}{\alpha'}\binom{\gamma}{\gamma'}\nabla_x^{\alpha-\alpha'}\nabla_v^{\gamma-\gamma'}\big(\tfrac\beta2v-\nabla A-\kappa\nabla W\ast M\big)\cdot\nabla_vJ_t^{\alpha',\gamma'},
\end{equation*}
so the above yields for all $0\le\lambda<\lambda'<\lambda_1$,
\begin{multline*}
\sup_{0\le s\le t}\Big(e^{2\lambda s}\|J_s^{\alpha,\gamma}\|_{\Ld^2(M^{-{1/2}})}^2\Big)
+\int_0^te^{2\lambda s}\Big(\|\nabla_vJ_s^{\alpha,\gamma}\|_{\Ld^2(M^{-{1/2}})}^2+\|(\nabla_v+\beta v)J_s^{\alpha,\gamma}\|_{\Ld^2(M^{-{1/2}})}^2\Big)\,\ddr s\\
\,\lesssim_{W,\beta, \lambda,\lambda', \alpha,\gamma,a}\,
\|\nabla_x^\alpha\nabla_v^\gamma(\pi_M^\bot h)\|_{\Ld^2(M^{-{1/2}})}^2
+\max_{i:e_i\le\gamma}\sup_{0\le s\le t}\Big(e^{2\lambda' s}\|J_s^{\alpha+e_i,\gamma-e_i}\|_{\Ld^2(M^{-{1/2}})}^2\Big)\\
+\max_{(\alpha',\gamma')<(\alpha,\gamma)}\int_0^te^{2\lambda' s}\|\nabla_vJ_s^{\alpha',\gamma'}\|_{\Ld^2(M^{-{1/2}})}^2\,\ddr s.
\end{multline*}
A direct induction then yields for all $\alpha,\gamma\ge0$, $\lambda\in[0,\lambda_1)$, and $t\ge0$,
\begin{multline*}
e^{2\lambda t}\|J_t^{\alpha,\gamma}\|_{\Ld^2(M^{-{1/2}})}^2
+\int_0^t e^{2\lambda s}\Big(\|\nabla_vJ_s^{\alpha,\gamma}\|_{\Ld^2(M^{-{1/2}})}^2+\|(\nabla_v+\beta v)J_s^{\alpha,\gamma}\|_{\Ld^2(M^{-{1/2}})}^2\Big)\,\ddr s\\
\,\lesssim_{W,\beta,\lambda,\alpha,\gamma,a}\,
\|\nabla_x^\alpha\nabla_v^\gamma(\pi_M^\bot h)\|_{\Ld^2(M^{-{1/2}})}^2.
\end{multline*}
Recalling $J_t^{\alpha,\gamma}=\nabla_x^\alpha\nabla_v^\gamma\pi_M^\bot e^{tR_M^*}h$ and $\pi_M^\bot h=h-M\int_\Xd h$, this proves the claim~\eqref{eq:res00k} (up to a redefinition of $\lambda_1$). 
\end{proof}

Starting from the above exponential decay for the Fokker--Planck semigroup, we can conclude the proof of Lemma~\ref{lem:estim-Vts-Hk} by means of a perturbation argument.

\begin{proof}[Proof of Lemma~\ref{lem:estim-Vts-Hk}]
Let $\lambda_0,\lambda_1$ be as in Theorem~\ref{thm:ergodic}(i) and in Lemma~\ref{lem:relax-LM0-FP}, respectively.
We split the proof into three steps.

\medskip
\step1 Perturbative semigroup estimate: proof that for all $k\ge0$, $g_s\in C^\infty_c(\Xd)$, and $t\ge s\ge0$,
\begin{equation}\label{eq:lala}
\|V_{t,s}\pi_M^\bot g_s\|_{H^{-k}(M^{-1/2})}
\,\le\,C_{k}e^{-(\lambda_1-\kappa^2C_{k})(t-s)}\|g_s\|_{H^{-k}(M^{-1/2})},
\end{equation}
for some constant $C_{k}$ only depending on $d,\beta,k,a$, $\|W\|_{W^{k+1,\infty}(\R^d)}$.

\medskip\noindent
By Duhamel's formula,
\[V_{t,s}\pi_M^\bot g_s\,=\,e^{(t-s)R_M}\pi_M^\bot g_s+\kappa\int_s^te^{(t-u)R_M}\big((\nabla W\ast(\mu_u-M))\cdot\nabla_v (V_{u,s} \pi_M^\bot g_s)\big)\,\ddr u,\]
and thus, for all $h_t\in C^\infty_c(\Xd)$, integrating by parts and using $\pi_M^\bot e^{tR_M}=e^{tR_M}\pi_M^\bot$,
\begin{multline*}
\int_\Xd  (V_{t,s}\pi_M^\bot g_s) h_t M^{-1}
\,=\,\int_\Xd \big( e^{(t-s)R_M^*}\pi_M^\bot h_t\big)g_sM^{-1}\\
-\kappa\int_s^t\Big(\int_\Xd \big((\nabla_v+\beta v)e^{(t-u)R_M^*} h_t\big)\cdot(\nabla W\ast(\mu_u-M)) (V_{u,s} \pi_M^\bot g_s)M^{-1}\Big)\,\ddr u.
\end{multline*}
Applying the exponential decay estimate~\eqref{eq:res00k} of Lemma~\ref{lem:relax-LM0-FP}, and taking the supremum over~$h_t$ in~$H^k(M^{-{1/2}})$, we deduce for all $k\ge0$ and $t\ge s\ge0$,
\begin{multline}
\label{eq:temp_apriori_34}
\|V_{t,s}\pi_M^\bot g_s\|_{H^{-k}(M^{-1/2})}
\,\lesssim_{W,\beta,k,a}\,e^{-\lambda_1(t-s)}\|g_s\|_{H^{-k}(M^{-1/2})}\\
+\kappa\Big(\int_s^te^{-2\lambda_1(t-u)}\|V_{u,s}\pi_M^\bot g_s\|_{H^{-k}(M^{-1/2})}^2\|\nabla W\ast(\mu_u-M)\|_{W^{k,\infty}(\R^d)}^2\ddr u\Big)^\frac12.
\end{multline}
Simply bounding $\|\nabla W\ast(\mu_u-M)\|_{W^{k,\infty}(\R^d)}\le 2 \|W\|_{W^{k+1,\infty}(\R^d)}$, this yields the claim~\eqref{eq:lala} by Gr\"onwall's inequality.

\medskip 
\step2 Improved estimate through control of moments: proof that for all $k\ge0$, $g_s\in C^\infty_c(\Xd)$, and $t\ge s\ge0$,
\begin{equation}\label{eq:moment_estim_34}
\|V_{t,s}\pi_M^\bot g_s\|_{H^{-k}(M^{-1/2})}
\,\le\,C_{k}e^{-\lambda_1(t-s)} \|g_s\|_{H^{-k}(M^{-1/2})} \Big(1+e^{-\lambda_0s}Q(\mu_\circ) \Big)^{\kappa^2C_{k}},
\end{equation}
for some constant $C_{k}$ only depending on $d,\beta,k,a$, $\|W\|_{W^{k+2,\infty}(\R^d)}$.

\medskip\noindent
Using Theorem~\ref{thm:ergodic}(i), we can estimate
\begin{equation*}
\|\nabla W\ast(\mu_t-M)\|_{W^{k,\infty}(\R^d)}
\,\le\,\|W\|_{W^{k+2,\infty}(\R^d)}\mathcal W_2(\mu_t,M)
\,\lesssim_{W,\beta,a}\,e^{-\lambda_0t}\mathcal W_2(\mu_\circ,M)
\,\lesssim_\beta\,e^{-\lambda_0t}Q(\mu_\circ),
\end{equation*}
and thus, combining this with the trivial bound $\|\nabla W\ast(\mu_t-M)\|_{W^{k,\infty}(\R^d)}\le 2 \|W\|_{W^{k+1,\infty}(\R^d)}$, we get by interpolation, for any $\delta\in[0,1]$,
\begin{equation}\label{eq:csq_thm31}
\|\nabla W\ast(\mu_t-M)\|_{W^{k,\infty}(\R^d)}
\,\lesssim_{W,\beta,k,a}\,\big(e^{-\lambda_0t}Q(\mu_\circ)\big)^\delta,
\end{equation}
where the multiplicative constant only depends on $d,\beta,a$, $\|W\|_{W^{k+2,\infty}(\R^d)}$.
Rather using this in the bound~\eqref{eq:temp_apriori_34} of Step~1, we get by Gr\"onwall's inequality, for any $\delta\in[0,1]$,
\begin{equation*}
\|V_{t,s}\pi_M^\bot g_s\|_{H^{-k}(M^{-1/2})}
\,\le\,C_{k}e^{-\lambda_1(t-s)}\|g_s\|_{H^{-k}(M^{-1/2})}\exp\Big(\kappa^2C_{k} \frac{(e^{-\lambda_0s}Q(\mu_\circ))^{2\delta}}{2\delta}\Big),
\end{equation*}
for some constant $C_{k}$ only depending on $d,\beta,k,a$, $\|W\|_{W^{k+2,\infty}(\R^d)}$. Choosing $2\delta = \log(e + e^{-\lambda_0 s} Q(\mu_\circ))^{-1}$, which indeed belongs to $(0,1)$, this entails
\begin{multline*}
\|V_{t,s}\pi_M^\bot g_s\|_{H^{-k}(M^{-1/2})}
\,\le\,C_{k}e^{-\lambda_1(t-s)}\|g_s\|_{H^{-k}(M^{-1/2})}\exp\Big(\kappa^2C_{k} \log\big(e+e^{-\lambda_0s}Q(\mu_\circ)\big)\Big)\\
\,=\,C_{k}e^{-\lambda_1(t-s)}\|g_s\|_{H^{-k}(M^{-1/2})}\Big(e+e^{-\lambda_0s}Q(\mu_\circ)\Big)^{\kappa^2C_{k}},
\end{multline*}
which proves the claim~\eqref{eq:moment_estim_34}.

\medskip
\step3 Conclusion.\\
First note that the results of Steps~1 and~2 can be unified as follows: for all $k\ge0$, $g_s\in C^\infty_c(\Xd)$, and $t\ge s\ge0$, we have
\begin{equation}\label{eq:estim-Vpibot-final}
\|V_{t,s}\pi_M^\bot g_s\|_{H^{-k}(M^{-1/2})}
\,\le\,C_{k}e^{-\lambda_1(t-s)}\|g_s\|_{H^{-k}(M^{-1/2})}\Big(1+e^{-\lambda_0s}\big(e^{\lambda_0t}\wedge Q(\mu_\circ) \big)\Big)^{\kappa^2C_{k}},
\end{equation}
for some constant $C_{k}$ only depending on $d,\beta,k,a$, $\|W\|_{W^{k+2}(\R^d)}$.
It remains to replace $V_{t,s}\pi_M^\bot$ by~$\pi_M^\bot V_{t,s}$ in this estimate.
For that purpose, let us decompose
\begin{equation}\label{eq:decomposition-Vts0}
\pi_M^\bot V_{t,s} g_s\,=\,V_{t,s}\pi_M^\bot g_s+\Big(\int_\Xd g_s\Big)r_{t,s},\qquad r_{t,s}\,:=\,\pi_M^\bot V_{t,s} M.
\end{equation}
By definition, $r_{t,s}$ satisfies
\begin{eqnarray*}
\partial_t r_{t,s}&=&R_M V_{t,s} M+\kappa(\nabla W\ast(\mu_t-M))\cdot\nabla_vV_{t,s} M\\
&=&R_M r_{t,s}+\kappa(\nabla W\ast(\mu_t-M))\cdot\nabla_vr_{t,s}+\kappa(\nabla W\ast(\mu_t-M))\cdot\nabla_vM,
\end{eqnarray*}
with $r_{t,s}|_{t=s}=0$. By Duhamel's formula, this yields
\begin{eqnarray*}
r_{t,s}
\,=\,\kappa\int_s^tV_{t,u}\big(\nabla W\ast(\mu_u-M)\cdot\nabla_vM\big)\,\ddr u.
\end{eqnarray*}
Now applying~\eqref{eq:estim-Vpibot-final} together with~\eqref{eq:csq_thm31}, we get for all $k\ge0$, $t\ge s\ge0$, and $\delta\in[0,1]$,
\begin{eqnarray*}
\lefteqn{\|r_{t,s}\|_{H^{-k}(M^{-1/2})}}\\
&\lesssim_{W,\beta,k,a}&\kappa \int_s^te^{-\lambda_1(t-u)}\|\nabla W\ast(\mu_u-M)\|_{W^{k,\infty}(\R^d)}\Big(1+e^{-\lambda_0u}\big(e^{\lambda_0t}\wedge Q(\mu_\circ) \big)\Big)^{\kappa^2C_{k}} \ddr u\\
&\lesssim_{W,\beta,k,a}&\kappa Q(\mu_\circ)^\delta e^{-\frac12\delta(\lambda_0\wedge\lambda_1)t}\Big(1+e^{-\lambda_0s}\big(e^{\lambda_0t}\wedge Q(\mu_\circ) \big)\Big)^{\kappa^2C_{k}}.
\end{eqnarray*}
Combined with~\eqref{eq:estim-Vpibot-final} and~\eqref{eq:decomposition-Vts0}, this yields the conclusion up to redefining $\lambda_1$.
\end{proof}

\subsection{Proof of Lemma~\ref{lem:estim-Wts} on $W^{-k,q}(\langle z\rangle^{p})$}\label{sec:pr-lem:estim-Wts}
This section is devoted to the proof of~\eqref{eq:estim-Wts-zp}.
Instead of~$\langle z\rangle^{p}$, we shall consider deformed weights of the form $\omega^p$ in terms of
\begin{equation}\label{eq:defin-omega}
\om(x,v) \,:=\,1+\beta\Big(\tfrac12|v|^2+ a|x|^2+\eta x\cdot v\Big),
\end{equation}
where the parameter $0<\eta \ll 1$ will be properly chosen later on. We naturally restrict to $\eta\le\tfrac12\sqrt{a}$, which ensures
\[\omega(x,v)\simeq_{\beta,a}\langle z\rangle^2.\]
Note that those weights differ from the choice used in~\cite{Mischler_2016} and are critical for the improved result we establish in this work.
We define the weighted negative Sobolev spaces $W^{-k,q}(\omega^p)$ exactly as the spaces $W^{-k,q}(\langle z\rangle^{p})$ in~\eqref{eq:def-W-kqzp}, simply replacing the weight $\langle z\rangle^{p}$ by $\omega^p$ in the definition.
Comparing~$\omega$ and~$\langle z\rangle^2$, the definition of dual norms easily ensures for all $h\in C^\infty_c(\Xd)$,
\begin{equation}\label{eq:comparison-omega-zp}
\|h\|_{W^{-k,q}(\langle z\rangle^{2p})}\,\simeq_{\beta,k,p,a}\,\|h\|_{W^{-k,q}(\omega^p)}.
\end{equation}
For a densely-defined operator $X$ on $\Ld^{q}(\omega^p):=W^{0,q}(\omega^p)$, we denote by $X^{*,p}$ its adjoint on $\Ld^{q'}(\Xd)$ with respect to the weighted duality product $(g,h)\mapsto\int_\Xd gh\omega^p$: more precisely, $X^{*,p}$ stands for the closed operator on $\Ld^{q'}(\Xd)$ defined by the relation
\[\int_\Xd g(Xh)\omega^p\,=\,\int_\Xd h(X^{*,p}g)\omega^p.\]
In particular, for $\mu\in\Pc(\Xd)$, we consider the weighted adjoint $B_\mu^{*,p}$ of $B_\mu$, which takes the explicit form
\begin{multline}\label{eq:def-Bmustarp}
B_\mu^{*,p} h\,=\,
\tfrac12\triangle_v h
+v \cdot \nabla_xh
-\Big(\tfrac12\beta v+\nabla A+\kappa\nabla W\ast\mu-\omega^{-p} \nabla_v\omega^p\Big)\cdot \nabla_v h\\
+\Big(\tfrac12\omega^{-p} \triangle_v\omega^p
+\omega^{-p}v \cdot \nabla_x\omega^p
-\omega^{-p}\big(\tfrac12\beta v+\nabla A+\kappa\nabla W\ast\mu\big)\cdot \nabla_v\omega^p-\Lambda \chi_R\Big) h.
\end{multline}
By the equivalence of norms~\eqref{eq:comparison-omega-zp} and by the definition of dual norms, it suffices to prove that there is some $0<\eta\le \frac12 \sqrt a$ and some $\lambda_2>0$ (only depending on $d,\beta,a$) such that the following result holds:
given $1<q\le2$ and $0<p\le1$ with $pq'\gg_{\beta,a}1$ large enough (only depending on $d,\beta,a$),
choosing $\Lambda,R$ large enough (only depending on $d,\beta,p,a$, $\|W\|_{W^{1,\infty}(\R^d)}$), if $g\in C([0,t];\Ld^{q'}(\Xd))$ satisfies the backward Cauchy problem
\begin{equation}\label{eq:dual-Bmup}
\left\{\begin{array}{l}
\partial_sg_s=-B_{\mu_s}^{*,p}g_s,\quad\text{for $0\le s\le t$},\\
g_s|_{s=t}=g_t,
\end{array}\right.
\end{equation}
for some $t\ge0$ and $g_t\in C^\infty_c(\Xd)$,
then we have for all $k\ge0$ and $0\le s\le t$,
\begin{equation}\label{eq:estim-todo-bmu-expdec}
\|g_s\|_{W^{k,q'}(\Xd)}\,\lesssim_{W,\beta,k,p,q,a}\,e^{-\lambda_2(t-s)}\|g_t\|_{W^{k,q'}(\Xd)}.
\end{equation}
We split the proof into two steps, starting with the case $k=0$ before treating all $k\ge0$ by induction.

\medskip
\step1 Proof of~\eqref{eq:estim-todo-bmu-expdec} for $k=0$.\\
By definition of $B_\mu^{*,p}$ and by integration by parts, we find
\begin{multline*}
\partial_s\|g_s\|_{\Ld^{q'}(\Xd)}^{q'}
\,=\,-q'\int_\Xd |g_s|^{q'-2}g_sB_{\mu_s}^{*,p}g_s\\
\,=\,
\tfrac12q'(q'-1)\int_\Xd |g_s|^{q'-2}|\nabla_vg_s|^2
-\int_\Xd |g_s|^{q'}\Big(
q'\omega^{-p}v \cdot \nabla_x\omega^p
-q'\omega^{-p}\big(\tfrac12\beta v+\nabla A+\kappa\nabla W\ast\mu\big)\cdot \nabla_v\omega^p\\
-q'\Lambda \chi_R
+\tfrac12q'\omega^{-p} \triangle_v\omega^p
+\tfrac{\beta d}2
-\Div_v(\omega^{-p} \nabla_v\omega^p)\Big).
\end{multline*}
Now inserting the form of the weight $\omega$ and explicitly computing its derivatives, we get after straightforward simplifications,
\begin{multline*}
\partial_s\|g_s\|_{\Ld^{q'}(\Xd)}^{q'}
\,\ge\,
\tfrac12q'(q'-1)\int_\Xd |g_s|^{q'-2}|\nabla_vg_s|^2\\
+\int_\Xd |g_s|^{q'}\Big(
pq'\beta\big(\tfrac{\beta}2-\eta(1+\tfrac1{32a}\beta^2)\big)\omega^{-1}|v|^2
+2a\eta\beta pq'\omega^{-1}|x|^2
-\tfrac{\beta d}2
+q'\big(\Lambda \chi_R
-C_{W,\beta,a}\omega^{-\frac12}\big)
\Big).
\end{multline*}
We show that we can choose our parameters in such a way that the last bracket be bounded below by a positive constant, which is the key to the desired exponential decay.
More precisely, choosing
\[\eta:=\min\big\{\tfrac14\beta(1+\tfrac1{32a}\beta^2)^{-1},\tfrac12\sqrt a\big\},\qquad
4\lambda_2:=\min\{\tfrac12\beta,2\eta\},\]
we get
\begin{multline*}
\partial_s\|g_s\|_{\Ld^{q'}(\Xd)}^{q'}
\,\ge\,
\tfrac12q'(q'-1)\int_\Xd |g_s|^{q'-2}|\nabla_vg_s|^2\\
+\int_\Xd |g_s|^{q'}\Big(
4 pq'\beta\lambda_2\omega^{-1}(\tfrac12|v|^2+a|x|^2)
-\tfrac{\beta d}2
+q'\big(\Lambda \chi_R
-C_{W,\beta,a}\omega^{-\frac12}\big)
\Big).
\end{multline*}
As the choice $\eta\le\frac12\sqrt a$ ensures
$\tfrac12|v|^2+a|x|^2\ge \tfrac1{2\beta}(\omega-1)$, this actually means
\begin{equation*}
\partial_s\|g_s\|_{\Ld^{q'}(\Xd)}^{q'}
\,\ge\,
\tfrac12q'(q'-1)\int_\Xd |g_s|^{q'-2}|\nabla_vg_s|^2
+\int_\Xd |g_s|^{q'}\Big(
2pq'\lambda_2
-\tfrac{\beta d}2
+q'\big(\Lambda \chi_R
-C_{W,\beta,a}\omega^{-\frac12}\big)
\Big).
\end{equation*}
Now, recalling the definition of the cut-off function $\chi_R$, we note that we can choose $\Lambda,R>0$ large enough (only depending on~$d,W,\beta,p,a$) such that
\[\Lambda\chi_R-C_{W,\beta,a}\omega^{-\frac12}\,\ge\,-\tfrac12p\lambda_2.\]
Provided that $pq'\ge2\beta d\lambda_2^{-1}$,
we then obtain
\begin{equation}\label{eq:energy-est-Bstarmup}
\partial_s\|g_s\|_{\Ld^{q'}(\Xd)}^{q'}
\,\ge\,
\tfrac12q'(q'-1)\int_\Xd |g_s|^{q'-2}|\nabla_vg_s|^2
+pq'\lambda_2\|g_s\|_{\Ld^{q'}(\Xd)}^{q'},
\end{equation}
hence, by Gr\"onwall's inequality,
\begin{equation}\label{eq:estim-todo-bmu-expdec-0}
\|g_s\|_{\Ld^{q'}(\Xd)}
\,\le\,
e^{-p\lambda_2(t-s)}\|g_t\|_{\Ld^{q'}(\Xd)}.
\end{equation}
that is,~\eqref{eq:estim-todo-bmu-expdec} for $k=0$.

\medskip
\step2 Proof of~\eqref{eq:estim-todo-bmu-expdec} for all $k\ge0$.\\
For multi-indices $\alpha,\gamma\in\N^d$, we set $J_s^{\alpha,\gamma}:=\nabla^\alpha_x\nabla^\gamma_vg_s$.
Differentiating equation~\eqref{eq:dual-Bmup}, we get
\[\left\{\begin{array}{l}
\partial_sJ_s^{\alpha,\gamma}=-B_{\mu_s}^{*,p}J_s^{\alpha,\gamma}-r_s^{\alpha,\gamma},\quad\text{for $0\le s\le t$},\\
J_s^{\alpha,\gamma}|_{s=t}=\nabla_x^\alpha\nabla_v^\gamma g_t,
\end{array}\right.\]
where the remainder is given by
\[r_s^{\alpha,\gamma}\,:=\,[\nabla^\alpha_x\nabla^\gamma_v,B_{\mu_s}^{*,p}]g_s.\]
Repeating the proof of~\eqref{eq:energy-est-Bstarmup}, for the choice of $\eta,\lambda_2,\Lambda,R$ in Step~1, we get
\begin{equation}\label{eq:estim-Jalphgam-energy}
\partial_s\|J_s^{\alpha,\gamma}\|_{\Ld^{q'}(\Xd)}^{q'}\,\ge\,\tfrac12q'(q'-1)\int_\Xd|J_s^{\alpha,\gamma}|^{q'-2}|\nabla_vJ_s^{\alpha,\gamma}|^2+pq'\lambda_2\|J_s^{\alpha,\gamma}\|_{\Ld^{q'}(\Xd)}^{q'}-q'\int_{\Xd}|J_s^{\alpha,\gamma}|^{q'-2}J_s^{\alpha,\gamma}r_s^{\alpha,\gamma},
\end{equation}
and it remains to analyze the last contribution.
By definition of $B_\mu^{*,p}$, cf.~\eqref{eq:def-Bmustarp}, we can compute
\begin{eqnarray}
r_s^{\alpha,\gamma}
&=&
\sum_{i:e_i\le\gamma}\binom{\gamma}{e_i}J^{\alpha+e_i,\gamma-e_i}_s\nonumber\\
&-&\sum_{(\alpha',\gamma')<(\alpha,\gamma)}\binom{\alpha}{\alpha'}\binom{\gamma}{\gamma'}\,\Div_v\Big[J_s^{\alpha',\gamma'}\,\nabla_x^{\alpha-\alpha'}\nabla_v^{\gamma-\gamma'}\Big(\tfrac12\beta v+\nabla A+\kappa\nabla W\ast\mu-\omega^{-p} \nabla_v\omega^p\Big)\Big]\nonumber\\
&+&\sum_{(\alpha',\gamma')<(\alpha,\gamma)}\binom{\alpha}{\alpha'}\binom{\gamma}{\gamma'}J_s^{\alpha',\gamma'}\,\nabla_x^{\alpha-\alpha'}\nabla_v^{\gamma-\gamma'}
\Big(\tfrac12\omega^{-p} \triangle_v\omega^p
+\omega^{-p}v \cdot \nabla_x\omega^p
-\Div_v(\omega^{-p} \nabla_v\omega^p)\nonumber\\
&&\hspace{4cm}-\omega^{-p}\big(\tfrac12\beta v+\nabla A+\kappa\nabla W\ast\mu\big)\cdot \nabla_v\omega^p+\tfrac{\beta d}2-\Lambda \chi_R\Big).
\label{eq:comput-ralphagamma}
\end{eqnarray}
Integrating by parts, this allows us to estimate
\begin{multline*}
\int_{\Xd}|J_s^{\alpha,\gamma}|^{q'-2}J_s^{\alpha,\gamma}r_s^{\alpha,\gamma}
\,\lesssim_{W,\beta,\alpha,\gamma,a}\,
\|J_s^{\alpha,\gamma}\|_{\Ld^{q'}(\Xd)}^{q'-1}\Big(\max_{i:e_i\le\gamma}\|J_s^{\alpha+e_i,\gamma-e_i}\|_{\Ld^{q'}(\Xd)}+\max_{(\alpha',\gamma')<(\alpha,\gamma)}\|J_s^{\alpha',\gamma'}\|_{\Ld^{q'}(\Xd)}\Big)\\
+q'\max_{(\alpha',\gamma')<(\alpha,\gamma)}\int_\Xd|J_s^{\alpha,\gamma}|^{q'-2}|J_s^{\alpha',\gamma'}||\nabla_vJ_s^{\alpha,\gamma}|.
\end{multline*}
Inserting this estimate into~\eqref{eq:estim-Jalphgam-energy} and appealing to Young's inequality to absorb $\nabla_vJ^{\alpha,\gamma}$ into the dissipation term, we are led to
\begin{multline*}
\partial_s\|J_s^{\alpha,\gamma}\|_{\Ld^{q'}(\Xd)}^{q'}\,\ge\,
pq'\lambda_2\|J_s^{\alpha,\gamma}\|_{\Ld^{q'}(\Xd)}^{q'}
-(q')^2C_{W,\beta,\alpha,\gamma,a}\|J_s^{\alpha,\gamma}\|_{\Ld^{q'}(\Xd)}^{q'-2}\Big(\max_{(\alpha',\gamma')<(\alpha,\gamma)}\|J_s^{\alpha',\gamma'}\|_{\Ld^{q'}(\Xd)}^2\Big)\\
-q'C_{W,\beta,\alpha,\gamma,a}\|J_s^{\alpha,\gamma}\|_{\Ld^{q'}(\Xd)}^{q'-1}\Big(\max_{i:e_i\le\gamma}\|J_s^{\alpha+e_i,\gamma-e_i}\|_{\Ld^{q'}(\Xd)}
+\max_{(\alpha',\gamma')<(\alpha,\gamma)}\|J_s^{\alpha',\gamma'}\|_{\Ld^{q'}(\Xd)}\Big).
\end{multline*}
Further appealing to Young's inequality, we get for all $\lambda<\lambda_2$,
\begin{multline*}
\partial_s\|J_s^{\alpha,\gamma}\|_{\Ld^{q'}(\Xd)}^{q'}\,\ge\,
pq'\lambda\|J_s^{\alpha,\gamma}\|_{\Ld^{q'}(\Xd)}^{q'}\\
-C_{W,\beta,\lambda,\alpha,\gamma,p,q,a}\Big(\max_{i:e_i\le\gamma}\|J_s^{\alpha+e_i,\gamma-e_i}\|_{\Ld^{q'}(\Xd)}
^{q'}+\max_{(\alpha',\gamma')<(\alpha,\gamma)}\|J_s^{\alpha',\gamma'}\|_{\Ld^{q'}(\Xd)}^{q'}\Big),
\end{multline*}
and thus, by Gr\"onwall's inequality,
\begin{multline*}
e^{pq'\lambda(t-s)}\|J_s^{\alpha,\gamma}\|_{\Ld^{q'}(\Xd)}^{q'}\,\lesssim_{W,\beta,\lambda,\alpha,\gamma,p,q,a}\,\|\nabla_x^\alpha\nabla_v^\gamma g_t\|_{\Ld^{q'}(\Xd)}^{q'}\\
+\int_s^te^{pq'\lambda(t-u)}\Big(\max_{i:e_i\le\gamma}\|J_u^{\alpha+e_i,\gamma-e_i}\|_{\Ld^{q'}(\Xd)}
^{q'}+\max_{(\alpha',\gamma')<(\alpha,\gamma)}\|J_u^{\alpha',\gamma'}\|_{\Ld^{q'}(\Xd)}^{q'}\Big)\ddr u.
\end{multline*}
Iterating this inequality and starting from the result~\eqref{eq:estim-todo-bmu-expdec-0} of Step~1 for $J^{0,0}_s=g_s$, the conclusion follows.
\qed

\subsection{Proof of Lemma~\ref{lem:estim-Wts} on $H^{-k}(M^{-1/2})$}\label{sec:pr-lem:estim-Wts-2}
This section is devoted to the proof of~\eqref{eq:estim-Wts-M}.
Taking inspiration from the work of Mischler and Mouhot~\cite[Section~4.2]{Mischler_2016}, we consider deformed weights of the form $M^{-1}\zeta$ with the factor $\zeta$ given by
\begin{equation}\label{eq:def-zetaweight}
\zeta(x,v)\,:=\,1+\tfrac12\Big(\frac{x\cdot v}{1+\frac{\eta}{2}|x|^2+\frac1{2\eta}|v|^2}\Big)
\end{equation}
where the parameter $\eta>0$ will be properly chosen later on. Note that for any $\eta>0$ we have
\[\tfrac12\,\le\,\zeta\,\le\,\tfrac32.\]
In these terms, we define the weighted negative Sobolev spaces~$H^{-k}(\sqrt{M^{-1}\zeta})$ exactly as~$H^{-k}(M^{-1/2})$ in~\eqref{eq:def-H-kM}, simply replacing the weight $M^{-1}$ by $M^{-1}\zeta$ in the definition. Comparing~$M^{-1}\zeta$ to~$M^{-1}$, the definition of dual norms easily yields the equivalence, for all $h\in C^\infty_c(\Xd)$,
\begin{equation}\label{eq:comparison-M-zeta}
\|h\|_{H^{-k}(M^{-1/2})}\,\simeq_{k,\eta}\,\|h\|_{H^{-k}(\sqrt{M^{-1}\zeta})}.
\end{equation}
For a densely-defined operator $X$ on $\Ld^2(\sqrt{M^{-1}\zeta})$, we denote by $X^{*,\zeta}$ its adjoint on $\Ld^2(\sqrt{M^{-1}\zeta})$ with respect to the weighted duality product $(g,h)\mapsto \int_\Xd ghM^{-1}\zeta$: more precisely, $X^{*,\zeta}$ stands for the closed operator on $\Ld^2(\sqrt{M^{-1}\zeta})$ defined by the relation
\[\int_\Xd g(Xh)M^{-1}\zeta\,=\,\int_\Xd h(X^{*,\zeta}h)M^{-1}\zeta.\]
By the equivalence of norms~\eqref{eq:comparison-M-zeta} and by definition of dual norms, it suffices to prove that there is some $0<\eta\le1$ and $\lambda_2>0$ (only depending on $d,\beta,a$, and $\|W\|_{W^{1,\infty}(\R^d)}$) such that the following result holds: choosing $\Lambda,R$ large enough (only depending on $d,\beta,a$, and $\|W\|_{W^{1,\infty}(\R^d)}$), if $g\in C([0,t];\Ld^2(M^{-1/2}))$ satisfies the backward Cauchy problem
\begin{equation}\label{eq:dual-Bzeta}
\left\{\begin{array}{l}
\partial_sg_s=-B_{\mu_s}^{*,\zeta}g_s,\quad\text{for $0\le s\le t$},\\
g_s|_{s=t}=g_t,
\end{array}\right.
\end{equation}
for some $t\ge0$ and $g_t\in C^\infty_c(\Xd)$,
then we have for all $\lambda\in[0,\lambda_2)$, $k\ge0$, and $0\le s\le t$,
\begin{equation}\label{eq:estim-todo-bzeta-expdec}
\|g_s\|_{H^k(M^{-{1/2}})}\,\lesssim_{W,\beta,\lambda,k,a}\,e^{-\lambda(t-s)}\|g_t\|_{H^k(M^{-{1/2}})}.
\end{equation}
We split the proof into two steps, starting with the case $k=0$ before treating all $k\ge0$ by induction.

\medskip
\step1 Proof of~\eqref{eq:estim-todo-bzeta-expdec} for $k=0$.\\
By definition of $B_\mu$ and by integration by parts, we find
\begin{multline}\label{eq:estim-energy-gM-1zeta}
\partial_s\|g_s\|_{\Ld^2(\sqrt{M^{-1}\zeta})}^2
\,=\,-2\int_\Xd g_s(B_{\mu_s}^{*,\zeta}g_s)M^{-1}\zeta
\,=\,-2\int_\Xd g_s(B_{\mu_s}g_s)M^{-1}\zeta\\
\,=\,\int_\Xd|\nabla_v(\sqrt{M^{-1}\zeta} g_s)|^2
+\int_\Xd|g_s|^2M^{-1}\Big(
\tfrac14\zeta|\beta v|^2
+(\nabla A+\kappa\nabla W\ast\mu_s)\cdot\nabla_v\zeta
-v\cdot\nabla_x\zeta\\
+2\Lambda\chi_R\zeta
-\tfrac{\beta d}2\zeta
+\kappa\beta\zeta v\cdot(\nabla W\ast(\mu_s-M))
-|\nabla_v\sqrt\zeta|^2
\Big).
\end{multline}
We show that we can choose parameters in such a way that the last bracket be bounded below by a positive constant, which is the key to the desired exponential decay.
By definition of $\zeta$, cf.~\eqref{eq:def-zetaweight}, and by Young's inequality, we find
\begin{eqnarray*}
\lefteqn{\tfrac14\zeta|\beta v|^2
+\nabla A\cdot\nabla_v\zeta
-v\cdot\nabla_x\zeta}\\
&=&
\tfrac12|v|^2\bigg(\tfrac12\zeta\beta^2
-\frac{1}{1+\frac\eta2|x|^2+\frac1{2\eta}|v|^2}\bigg)
+a\frac{|x|^2}{1+\frac\eta2|x|^2+\frac1{2\eta}|v|^2}-(a\eta^{-1}\!-\tfrac12\eta)\frac{(x\cdot v)^2}{(1+\frac\eta2|x|^2+\frac1{2\eta}|v|^2)^2}\\
&\ge&
\tfrac12|v|^2\bigg(\tfrac14\beta^2
-\frac{1+a\eta^{-2}}{1+\frac\eta2|x|^2+\frac1{2\eta}|v|^2}\bigg)
+\tfrac12a\frac{|x|^2}{1+\frac\eta2|x|^2+\frac1{2\eta}|v|^2}\\
&\ge&
\tfrac12|v|^2\bigg(\tfrac14\beta^2
-\frac{1+2a\eta^{-2}}{1+\frac\eta2|x|^2+\frac1{2\eta}|v|^2}\bigg)
+a\eta^{-1}\bigg(1
-\frac{1}{1+\frac\eta2|x|^2+\frac1{2\eta}|v|^2}\bigg).
\end{eqnarray*}
Inserting this into~\eqref{eq:estim-energy-gM-1zeta}, and further using $|\nabla_v\zeta|\lesssim\eta^{-1}\langle z\rangle^{-1}$, we obtain
\begin{multline*}
\partial_s\|g_s\|_{\Ld^2(\sqrt{M^{-1}\zeta})}^2
\,\ge\,\int_\Xd|\nabla_v(\sqrt{M^{-1}\zeta} g_s)|^2\\
+\int_\Xd|g_s|^2M^{-1}\Big(
a\eta^{-1}-C_{W,\beta,a}
+\Lambda\chi_R
+\tfrac1{10}|\beta v|^2\big(1-\eta^{-3}\langle z\rangle^{-2}C_{W,\beta,a}\big)
-C_{W,\beta,a}\eta^{-2}\langle z\rangle^{-1}
\Big).
\end{multline*}
Noting that the dissipation term can be bounded below as
\begin{eqnarray*}
|\nabla_v(\sqrt{M^{-1}\zeta}g)|^2
&\ge&\tfrac12|\zeta|^2|\nabla_v(\sqrt{M^{-1}}g)|^2-2|\nabla_v\zeta|^2|g|^2M^{-1}\\
&\ge&\tfrac18|(\nabla_v+\tfrac\beta2v)g|^2M^{-1}-C\eta^{-2}\langle z\rangle^{-2}|g|^2M^{-1},
\end{eqnarray*}
the above becomes
\begin{multline*}
\partial_s\|g_s\|_{\Ld^2(\sqrt{M^{-1}\zeta})}^2
\,\ge\,\tfrac18\|(\nabla_v+\tfrac{\beta}2v) g_s\|_{\Ld^2(M^{-1/2})}^2\\
+\int_\Xd|g_s|^2M^{-1}\Big(
a\eta^{-1}-C_{W,\beta,a}
+\Lambda\chi_R
+\tfrac1{10}|\beta v|^2\big(1-\eta^{-3}\langle z\rangle^{-2}C_{W,\beta,a}\big)
-C_{W,\beta,a}\eta^{-2}\langle z\rangle^{-1}
\Big).
\end{multline*}
Now let us choose
$\eta:=\tfrac12aC_{W,\beta,a}^{-1}$,
and note that, by definition of the cut-off function $\chi_R$, we may then choose $\Lambda,R>0$ large enough (only depending on $d,W,\beta,a$) such that
\[\Lambda\chi_R
+\tfrac1{10}|\beta v|^2\big(1-\eta^{-3}\langle z\rangle^{-2}C_{W,\beta,a}\big)
-C_{W,\beta,a}\eta^{-2}\langle z\rangle^{-1}\,\ge\,
\tfrac1{20}|\beta v|^2
-\tfrac14a\eta^{-1}.\]
Further setting $\lambda_2:=\tfrac1{12}a\eta^{-1}$, this choice leads us to
\begin{equation}\label{eq:estim-todo-bzeta-expdec+dissip}
\partial_s\|g_s\|_{\Ld^2(\sqrt{M^{-1}\zeta})}^2
\,\ge\,\tfrac18\|(\nabla_v+\tfrac\beta2v)g_s\|_{\Ld^2(M^{-1/2})}^2
+\tfrac1{20}\|\beta vg_s\|_{\Ld^2(M^{-1/2})}^2
+2\lambda_2\|g_s\|_{\Ld^2(\sqrt{M^{-1}\zeta})}^2.
\end{equation}
In particular, by Gr\"onwall's inequality,
\begin{equation}\label{eq:estim-todo-bzeta-expdec-k=0}
\|g_s\|_{\Ld^2(M^{-1/2})}
\,\lesssim\,e^{-\lambda_2(t-s)}\|g_t\|_{\Ld^2(M^{-1/2})},
\end{equation}
that is, \eqref{eq:estim-todo-bzeta-expdec} for $k=0$.

\medskip
\step2 Proof of~\eqref{eq:estim-todo-bzeta-expdec} for all $k$.\\
For multi-indices $\alpha,\gamma\in\N^d$, we set $J_s^{\alpha,\gamma}:=\nabla_x^\alpha\nabla_v^\gamma g_s$. Differentiating equation~\eqref{eq:dual-Bzeta}, we get
\[\left\{\begin{array}{l}
\partial_sJ^{\alpha,\gamma}_s=-B_{\mu_s}^{*,\zeta}J^{\alpha,\gamma}_s-r_s^{\alpha,\gamma},\quad\text{for $0\le s\le t$},\\[1mm]
J^{\alpha,\gamma}_s|_{s=t}=\nabla_x^\alpha\nabla_v^\gamma g_t,
\end{array}\right.\]
where the remainder is given by
\[r_s^{\alpha,\gamma}\,:=\,[\nabla_x^\alpha\nabla_v^\gamma,B_{\mu_s}^{*,\zeta}]g_s.\]
Repeating the proof of~\eqref{eq:estim-todo-bzeta-expdec+dissip}, for the choice of $\eta,\lambda_2,\Lambda,R$ in Step~1, we get
\begin{multline}\label{eq:estim-Jsalphgam-L2M-1}
\partial_s\|J_s^{\alpha,\gamma}\|_{\Ld^2(\sqrt{M^{-1}\zeta})}^2
\,\ge\,
\tfrac18\|(\nabla_v+\tfrac\beta2v)J_s^{\alpha,\gamma}\|_{\Ld^2(M^{-1/2})}^2
+\tfrac1{20}\|\beta vJ_s^{\alpha,\gamma}\|_{\Ld^2(M^{-1/2})}^2
+2\lambda_2\|J_s^{\alpha,\gamma}\|_{\Ld^2(\sqrt{M^{-1}\zeta})}^2\\
-2\int_\Xd J_s^{\alpha,\gamma}r_s^{\alpha,\gamma}M^{-1}\zeta.
\end{multline}
and it remains to analyze the last contribution.
By definition of $B_\mu$, the weighted adjoint $B_\mu^{*,\zeta}$ takes the explicit form
\begin{multline*}
B_\mu^{*,\zeta}h\,=\,
\tfrac12\triangle_vh
+v\cdot\nabla_xh
+\Big(\tfrac{\beta}2v-\nabla A-\kappa\nabla W\ast\mu+\tfrac1{\zeta}\nabla_v\zeta\Big)\cdot\nabla_vh\\
+\Big(
\tfrac{\beta d}2
+\tfrac1{2\zeta}\triangle_v\zeta
-\kappa\beta v\cdot(\nabla W\ast(\mu-M))\\
+\tfrac1{\zeta}v\cdot\nabla_x\zeta
+\tfrac1{\zeta}\big(\tfrac\beta2v-\nabla A-\kappa\nabla W\ast\mu\big)\cdot\nabla_v\zeta
-\Lambda\chi_R
\Big)h,
\end{multline*}
and we may then compute
\begin{multline*}
r_s^{\alpha,\gamma}\,=\,
[\nabla_x^\alpha\nabla_v^\gamma,B_{\mu_s}^{*,\zeta}]g_s
\,=\,
\sum_{i:e_i\le\gamma}\binom{\gamma}{e_i}J_s^{\alpha+e_i,\gamma-e_i}\\
+\sum_{(\alpha',\gamma')<(\alpha,\gamma)}\binom{\alpha}{\alpha'}\binom{\gamma}{\gamma'}\nabla_x^{\alpha-\alpha'}\nabla_v^{\gamma-\gamma'}
\Big(\tfrac{\beta}2v-\nabla A-\kappa\nabla W\ast\mu+\tfrac1{\zeta}\nabla_v\zeta\Big)\cdot\nabla_vJ^{\alpha',\gamma'}_s\\
+\sum_{(\alpha',\gamma')<(\alpha,\gamma)}\binom{\alpha}{\alpha'}\binom{\gamma}{\gamma'}J^{\alpha',\gamma'}_s\,\nabla_x^{\alpha-\alpha'}\nabla_v^{\gamma-\gamma'}
\Big(
\tfrac{\beta d}2
+\tfrac1{2\zeta}\triangle_v\zeta
-\kappa\beta v\cdot(\nabla W\ast(\mu-M))\\
+\tfrac1{\zeta}v\cdot\nabla_x\zeta
+\tfrac1{\zeta}\big(\tfrac\beta2v-\nabla A-\kappa\nabla W\ast\mu\big)\cdot\nabla_v\zeta
-\Lambda\chi_R
\Big),
\end{multline*}
from which we easily estimate
\begin{multline*}
\int_\Xd J_s^{\alpha,\gamma}r_s^{\alpha,\gamma}M^{-1}\zeta
\,\lesssim_{W,\beta,\alpha,\gamma,a}\,
\Big(\|J_s^{\alpha,\gamma}\|_{\Ld^2(\sqrt{M^{-1}\zeta})}+\|\nabla_vJ_s^{\alpha,\gamma}\|_{\Ld^2(M^{-1/2})}+\|\beta vJ_s^{\alpha,\gamma}\|_{\Ld^2(M^{-1/2})}\Big)\\
\times\Big(\max_{i:e_i\le\gamma}\|J_s^{\alpha+e_i,\gamma-e_i}\|_{\Ld^2(M^{-1/2})}
+\max_{(\alpha',\gamma')<(\alpha,\gamma)}\|J_s^{\alpha',\gamma'}\|_{\Ld^2(M^{-1/2})}\Big).
\end{multline*}
Inserting this into~\eqref{eq:estim-Jsalphgam-L2M-1}, and appealing to Young's inequality to absorb $J^{\alpha,\gamma}$, $\nabla_vJ^{\alpha,\gamma}$, and $\beta vJ^{\alpha,\gamma}$ into the dissipation terms, we deduce for all $\lambda<\lambda_2$,
\begin{multline*}
\partial_s\|J_s^{\alpha,\gamma}\|_{\Ld^2(\sqrt{M^{-1}\zeta})}^2
\,\ge\,
2\lambda\|J_s^{\alpha,\gamma}\|_{\Ld^2(\sqrt{M^{-1}\zeta})}^2\\
-C_{W,\beta,\lambda,\alpha,\gamma,a}\Big(\max_{i:e_i\le\gamma}\|J_s^{\alpha+e_i,\gamma-e_i}\|_{\Ld^2(M^{-1/2})}^2
+\max_{(\alpha',\gamma')<(\alpha,\gamma)}\|J_s^{\alpha',\gamma'}\|_{\Ld^2(M^{-1/2})}^2\Big),
\end{multline*}
and thus, by Gr\"onwall's inequality,
\begin{multline*}
e^{2\lambda(t-s)}\|J_s^{\alpha,\gamma}\|_{\Ld^2(M^{-1/2})}^2
\,\lesssim_{W,\beta,\lambda,\alpha,\gamma,a}\,
\|\nabla^\alpha_x\nabla_v^\gamma g_t\|_{\Ld^2(M^{-1/2})}^2\\
+\int_s^te^{2\lambda(t-u)}\Big(\max_{i:e_i\le\gamma}\|J_u^{\alpha+e_i,\gamma-e_i}\|_{\Ld^2(M^{-1/2})}^2
+\max_{(\alpha',\gamma')<(\alpha,\gamma)}\|J_u^{\alpha',\gamma'}\|_{\Ld^2(M^{-1/2})}^2\Big)\ddr u.
\end{multline*}
Iterating this inequality, and starting from the result~\eqref{eq:estim-todo-bzeta-expdec-k=0} of Step~1 for $J^{0,0}_s=g_s$, the conclusion follows.
\qed

\subsection{Proof of Lemma~\ref{lem:estim-reg}}\label{sec:pr-lem:estim-reg}
In this section, we appeal again to duality but we shall use a slightly different notation than in Section~\ref{sec:pr-lem:estim-Wts}:
for a densely-defined operator $X$ on $\Ld^q(\langle z\rangle^p)$ we now denote by $X^{*,p}$ its adjoint on $\Ld^{q'}(\Xd)$ with respect to the weighted duality product $(g,h)\mapsto\int_\Xd gh\langle z\rangle^p$. In other words, we use the same notation as in Section~\ref{sec:pr-lem:estim-Wts} for $X^{*,p}$, but now with the weight $\omega$ replaced by $\langle z\rangle$.
We consider in particular the weighted adjoints $\{W_{t,s}^{*,p}\}_{t\ge s\ge0}$ of the fundamental operators $\{W_{t,s}\}_{t\ge s\ge0}$, and we note that for all $t\ge0$ and $g_t\in C^\infty_c(\Xd)$ the flow $g_s:=W_{t,s}^{*,p}g_t$ is such that $g \in C([0,t];L^{q'}(\Xd))$ satisfies the backward Cauchy problem
\begin{equation}\label{eq:flow-Bmu*p}
\left\{\begin{array}{l}
\partial_sg_s=-B_{\mu_s}^{*,p}g_s,\quad\text{for $0\le s\le t$},\\
g_s|_{s=t}=g_t,
\end{array}\right.
\end{equation}
where $B_\mu^{*,p}$ is the weighted adjoint of $B_\mu$. This operator $B_\mu^{*,p}$ takes the explicit form~\eqref{eq:def-Bmustarp} with $\omega$ now replaced by~$\langle z\rangle$.
The proof of Lemma~\ref{lem:estim-reg} ultimately relies on the following result.

\begin{lem}\label{lem:one_step_gain}
For all $k\ge0$, $0\le p\le1$, $0\le t-s\le 1$, and $g_t\in C^\infty_c(\Xd)$, we have
\begin{equation}\label{eq:regularization_W}
\big\| W^{*,p}_{t,s} g_t \big\|_{H^{k+1}(\Xd)} \,\lesssim_{W,\beta,k,a}\, (t-s)^{-\frac32} \|g_t\|_{H^k(\Xd)},
\end{equation}
where the constant only depends on $d,\beta,k,a$, $\|W\|_{W^{k+2,\infty}(\R^d)}$.
\end{lem}

We postpone the proof of this result for a moment and start by showing that Lemma~\ref{lem:estim-reg} follows as a straightforward consequence.

\begin{proof}[Proof of Lemma~\ref{lem:estim-reg}]
We start by applying the interpolation argument of~\cite[Lemma 2.4]{Mischler_2016}: thanks to the exponential decay estimates of Lemma~\ref{lem:estim-Wts}, it suffices to find some $\theta\ge0$ (only depending on $d$) such that for all $1<q\le2$, $k\ge0$, $0<p\le1$, $0\le t-s\le 1$, and $h_s\in C^\infty_c(\Xd)$ we have
\begin{equation*}
\|\calA W_{t,s}h_s\|_{H^{-k}(M^{-1/2})}\,\lesssim_{W,\beta,k,p,q,a}\, (t-s)^{-\theta}\|h_s\|_{W^{-k,q}(\langle z\rangle^{p})}.
\end{equation*}
In order to prove this, we argue by duality:
more precisely, recalling $A=\Lambda\chi_R$, it suffices to find some $\theta\ge0$ such that for all $1<q\le2$, $k\ge0$, $0<p\le1$, $0\le t-s\le 1$, and $g_t\in C^\infty_c(\Xd)$,
\begin{equation}\label{eq:red-estim-Wts-*mp-0-redred}
\|W_{t,s}^{*,p}(\chi_R M^{-1}\langle z\rangle^{-p}g_t)\|_{W^{k,q'}(\Xd)}\,\lesssim_{W,\beta,k,p,q,a}\, (t-s)^{-\theta}\|g_t\|_{H^k(M^{-{1/2}})}.
\end{equation}
By the Sobolev inequality with $2\le q'<\infty$, the left-hand side can be estimated as follows,
\begin{equation*}
\|W_{t,s}^{*,p}(\chi_R M^{-1}\langle z\rangle^{-p}g_t)\|_{W^{k,q'}(\Xd)}\,\lesssim_{k,q}\,\|W_{t,s}^{*,p}(\chi_R M^{-1}\langle z\rangle^{-p}g_t)\|_{H^{k+d}(\Xd)},
\end{equation*}
and the desired bound~\eqref{eq:red-estim-Wts-*mp-0-redred} then follows with $\theta=\frac{3d}2$ by iterating the result of Lemma~\ref{lem:one_step_gain}.
\end{proof}

The rest of this section is devoted to the proof of Lemma~\ref{lem:one_step_gain}. For $k=0$, this is in fact a standard consequence of the theory of hypoellipticity as in~\cite{Herau_2003,Villani_2009,Mischler_2016}. For $k>0$, we argue by induction, further using parabolic estimates similarly as in Section~\ref{sec:pr-lem:estim-Wts}.

\begin{proof}[Proof of Lemma \ref{lem:one_step_gain}]
Given $t\ge0$ and $g_t\in C^\infty_c(\Xd)$, let $g_s:=W_{t,s}^{*,p}g_t$ be the solution of the backward Cauchy problem~\eqref{eq:flow-Bmu*p}, and recall that $B_\mu^{*,p}$ takes the explicit form~\eqref{eq:def-Bmustarp} with $\omega$ replaced by $\langle z\rangle$,
\begin{equation*}
B_\mu^{*,p} h\,=\,
\tfrac12\triangle_v h
+v \cdot \nabla_xh
+A^0_{\mu,p}h-A^1_{\mu,p}\cdot \nabla_v h,
\end{equation*}
in terms of
\begin{eqnarray*}
A^0_{\mu,p}&:=&\tfrac12\langle z\rangle^{-p} \triangle_v\langle z\rangle^p
+\langle z\rangle^{-p}v \cdot \nabla_x\langle z\rangle^p
-\langle z\rangle^{-p}\big(\tfrac12\beta v+\nabla A+\kappa\nabla W\ast\mu\big)\cdot \nabla_v\langle z\rangle^p-\Lambda \chi_R,\\
A^1_{\mu,p}&:=&\tfrac12\beta v+\nabla A+\kappa\nabla W\ast\mu-\langle z\rangle^{-p} \nabla_v\langle z\rangle^p.
\end{eqnarray*}
We split the proof into two steps. 

\medskip
\step1 Case $k=0$: proof that for all $0\le t-s\le1$ we have
\begin{equation}\label{eq:estim-L2H1}
\|\nabla_xg_s\|_{\Ld^2(\Xd)}\,\lesssim_{W,\beta,a}\,(t-s)^{-\frac32}\|g_s\|_{\Ld^2(\Xd)}.
\end{equation}
Integrating by parts, we can compute
\[\partial_s\int_\Xd|g_s|^2
\,=\,-2\int_\Xd g_s B_{\mu_s}^{*,p}g_s
\,=\,\int_\Xd |\nabla_vg_s|^2
-\int_\Xd |g_s|^2\Big(2 A^0_{\mu_s,p}+\Div_v(A^1_{\mu_s,p})\Big),\]
and thus, by definition of $A^0_{\mu,p},A^1_{\mu,p}$,
\[\partial_s\int_\Xd|g_s|^2
\,\ge\,\int_\Xd |\nabla_vg_s|^2
-C_{W,\beta,a}\int_\Xd |g_s|^2.\]
Similarly, we can easily estimate
\begin{eqnarray*}
\partial_s\int_\Xd|\nabla_xg_s|^2&\ge&
\int |\nabla_{xv}g|^2
-C_{W,\beta,a}\int_\Xd \Big(|\nabla_xg_s|^2+|g_s|^2+ |\nabla_xg_s||\nabla_vg_s|\Big),\\
\partial_s\int_\Xd|\nabla_vg_s|^2&\ge&\int_\Xd |\nabla_v^2g_s|^2
-2\int_\Xd \nabla_vg_s\cdot\nabla_xg_s
-C_{W,\beta,a}\int_\Xd \Big(|\nabla_vg_s|^2+|g_s|^2\Big),\\
-\partial_s\int_\Xd\nabla_xg_s\cdot\nabla_vg_s&\ge&\tfrac12\int_\Xd|\nabla_xg_s|^2-\int_\Xd|\nabla_{xv}g_s||\nabla_v^2g_s|
-C_{W,\beta,a}\int_\Xd\Big(|\nabla_vg_s|^2
+|g_s|^2\Big).
\end{eqnarray*}
Let us consider the functional
\begin{equation}\label{eq:def-Fsgs}
F_s(g)\,:=\,a_0\int_\Xd|g|^2+a_1(t-s)\int_\Xd|\nabla_vg|^2+a_2(t-s)^3\int_\Xd|\nabla_xg|^2-2a_3(t-s)^2\int_\Xd\nabla_xg\cdot\nabla_vg,
\end{equation}
where the constants $a_0,a_1,a_2,a_3>0$ will be suitably chosen in a moment.
In these terms, using Young's inequality,
the above inequalities lead us to deduce for all $0\le t-s\le1$ and $0<\e,\delta\le1$,
\begin{multline*}
\partial_sF_s(g_s)\,\ge\,
\Big(a_0-C_{W,\beta,a}(\delta^{-1}a_1+a_2+a_3)\Big)\int_\Xd|\nabla_vg_s|^2
+\Big(\tfrac12a_3-\delta a_1-C_{W,\beta,a}a_2\Big)(t-s)^2\int_\Xd|\nabla_xg_s|^2\\
+\big(a_1-\e^{-1}a_3\big)(t-s)\int_\Xd|\nabla_v^2g_s|^2
+\big(a_2-\e a_3\big)(t-s)^3\int_\Xd|\nabla_{xv}g_s|^2\\
-C_{W,\beta,a}\big(a_0+a_1+a_2+a_3\big)\int_\Xd|g_s|^2.
\end{multline*}
Choosing for instance $a_0=2C_{W,\beta,a}(\delta^{-1}a_1+a_2+a_3)$, $a_1=4a_3^2$, $a_2=1$, $a_3=4C_{W,\beta,a}$, $\e=(2a_3)^{-1}$, and $\delta=(32a_3)^{-1}$, we obtain
\begin{eqnarray}
\partial_sF_s(g_s)
&\ge&
\tfrac12a_0\int_\Xd|\nabla_vg_s|^2
+\tfrac18a_3(t-s)^2\int_\Xd|\nabla_xg_s|^2\nonumber\\
&&+\tfrac12a_1(t-s)\int_\Xd|\nabla_v^2g_s|^2
+\tfrac12(t-s)^3\int_\Xd|\nabla_{xv}g_s|^2
-2a_0C_{W,\beta,a}\int_\Xd|g_s|^2\nonumber\\
&\ge&-2a_0C_{W,\beta,a}\int_\Xd|g_s|^2.\label{eq:estim-Fsgs-lower}
\end{eqnarray}
By definition of $F_s$, as the choice of $a_1,a_2,a_3$ satisfies $2a_3=\sqrt{a_1a_2}$, we have
\begin{equation}\label{eq:estim-lower-Fsfctal}
F_s(g)\,\ge\,a_0\|g\|_{\Ld^2(\Xd)}^2
+\tfrac12a_1(t-s)\|\nabla_vg\|_{\Ld^2(\Xd)}^2
+\tfrac12(t-s)^3\|\nabla_xg\|_{\Ld^2(\Xd)}^2,
\end{equation}
so that the above estimate~\eqref{eq:estim-Fsgs-lower} entails
\[\partial_sF_s(g_s)
\,\gtrsim_{W,\beta,a}\,-F_s(g_s).\]
By Gr\"onwall's inequality with $F_t(g)=a_0\|g\|_{\Ld^2(\Xd)}^2$, this yields for all $0\le t-s\le1$,
\[F_s(g_s)\,\lesssim_{W,\beta,a}\,\|g_t\|_{\Ld^2(\Xd)},\]
and the claim~\eqref{eq:estim-L2H1} then follows from~\eqref{eq:estim-lower-Fsfctal}.

\medskip
\step2 Conclusion.\\
Given multi-indices $\alpha,\gamma\in\N^d$, we set for abbreviation $J_s^{\alpha,\gamma}:=\nabla_x^\alpha\nabla_v^\gamma g_s$, which satisfies
\[\left\{\begin{array}{l}
\partial_sJ_s^{\alpha,\gamma}=-B_{\mu_s}^{*,p}J_s^{\alpha,\gamma}-r_s^{\alpha,\gamma},\quad\text{for $0\le s\le t$},\\
J_s^{\alpha,\gamma}|_{s=t}=\nabla_x^\alpha\nabla_v^\gamma g_t,
\end{array}\right.\]
where the remainder term is given by
\[r_s^{\alpha,\gamma}\,:=\,[\nabla_x^\alpha\nabla_v^\gamma,B_{\mu_s}^{*,p}]g_s.\]
Repeating the proof of~\eqref{eq:estim-Fsgs-lower}, for $F_s$ defined in~\eqref{eq:def-Fsgs} with the same choice of constants $a_0,a_1,a_2,a_3$ as in Step~1, we get for all $0\le t-s\le1$,
\begin{multline*}
\partial_sF_s(J_s^{\alpha,\gamma})
\,\ge\,
\tfrac12a_0\int_\Xd|\nabla_vJ_s^{\alpha,\gamma}|^2
+\tfrac18a_3(t-s)^2\int_\Xd|\nabla_xJ_s^{\alpha,\gamma}|^2\\
+\tfrac12a_1(t-s)\int_\Xd|\nabla_v^2J_s^{\alpha,\gamma}|^2
+\tfrac12(t-s)^3\int_\Xd|\nabla_{xv}J_s^{\alpha,\gamma}|^2
-2a_0C_{W,\beta,a}\int_\Xd|J_s^{\alpha,\gamma}|^2\\
-2a_0\int_\Xd J_s^{\alpha,\gamma}r_s^{\alpha,\gamma}
-2a_1(t-s)\int_\Xd\nabla_vJ_s^{\alpha,\gamma}\cdot\nabla_vr_s^{\alpha,\gamma}
-a_2(t-s)^3\int_\Xd\nabla_xJ_s^{\alpha,\gamma}\cdot\nabla_xr_s^{\alpha,\gamma}\\
+2a_3(t-s)^2\int_\Xd\nabla_xJ_s^{\alpha,\gamma}\cdot\nabla_vr_s^{\alpha,\gamma}
+2a_3(t-s)^2\int_\Xd\nabla_vJ_s^{\alpha,\gamma}\cdot\nabla_xr_s^{\alpha,\gamma}.
\end{multline*}
Recalling that the remainder term $r_s^{\alpha,\gamma}$ can be written as in~\eqref{eq:comput-ralphagamma} with $\omega$ replaced by $\langle z\rangle$,
integrating by parts, and using Young's inequality to absorb all factors involving $J_s^{\alpha,\gamma}$ into the dissipation terms, we deduce for all~\mbox{$0\le t-s\le1$},
\begin{multline*}
\partial_sF_s(J_s^{\alpha,\gamma})
\,\gtrsim_{W,\beta,\alpha,\gamma,a}\,
-\|J_s^{\alpha,\gamma}\|_{\Ld^2(\Xd)}^2\\
-\max_{(\alpha',\gamma')<(\alpha,\gamma)}\Big(\|J_s^{\alpha',\gamma'}\|_{\Ld^2(\Xd)}^2
+(t-s)\|\nabla_vJ_s^{\alpha',\gamma'}\|_{\Ld^2(\Xd)}^2
+(t-s)^3\|\nabla_xJ_s^{\alpha',\gamma'}\|_{\Ld^2(\Xd)}^2\bigg)\\
-\max_{i:e_i\le\gamma}\Big(\|J_s^{\alpha+e_i,\gamma-e_i}\|_{\Ld^2(\Xd)}^2
+(t-s)\|\nabla_vJ_s^{\alpha+e_i,\gamma-e_i}\|_{\Ld^2(\Xd)}^2
+(t-s)^3\|\nabla_xJ_s^{\alpha+e_i,\gamma-e_i}\|_{\Ld^2(\Xd)}^2\bigg).
\end{multline*}
By~\eqref{eq:estim-lower-Fsfctal}, this entails
\begin{equation*}
\partial_sF_s(J_s^{\alpha,\gamma})
\,\gtrsim_{W,\beta,\alpha,\gamma,a}\,
-F_s(J_s^{\alpha,\gamma})
-\max_{(\alpha',\gamma')<(\alpha,\gamma)} F_s(J_s^{\alpha',\gamma'})
-\max_{i:e_i\le\gamma}F_s(J_s^{\alpha+e_i,\gamma-e_i}),
\end{equation*}
and thus, by Gr\"onwall's inequality with $F_t(g)=a_0\|g\|_{\Ld^2(\Xd)}^2$, we deduce for all $0\le t-s\le1$,
\begin{equation*}
F_s(J_s^{\alpha,\gamma})
\,\lesssim_{W,\beta,\alpha,\gamma,a}\,
\|\nabla^\alpha_x\nabla_v^\gamma g_t\|_{\Ld^2(\Xd)}
+\int_s^t\Big(\max_{(\alpha',\gamma')<(\alpha,\gamma)} F_u(J_u^{\alpha',\gamma'})
+\max_{i:e_i\le\gamma}F_u(J_u^{\alpha+e_i,\gamma-e_i})\Big)\,\ddr u.
\end{equation*}
By a direct iteration, this proves for all $0\le t-s\le1$,
\begin{equation*}
F_s(J_s^{\alpha,\gamma})
\,\lesssim_{W,\beta,\alpha,\gamma,a}\,\|g_t\|_{H^{|\alpha|+|\gamma|}(\Xd)},
\end{equation*}
and the conclusion~\eqref{eq:regularization_W} then follows from~\eqref{eq:estim-lower-Fsfctal}.
\end{proof}

%%%%%%%%%%%%%%%%%%%%%%%%%%%%%%%%%%%%%%%%%%%%
%%%%%%%%%%%%%%%%%%%%%%%%%%%%%%%%%%%%%%%%%%%%
%%%%%%%%%%%%%%%%%%%%%%%%%%%%%%%%%%%%%%%%%%%%

\section*{Acknowledgements}
AB would like to thank St\'ephane Mischler for enlightening discussions which allowed to refine the comparison with~\cite{Mischler_2012, Mischler_2013}. 
AB acknowledges financial support from the European Union's Horizon 2020 research and innovation programme under the Marie Sk\l{}odowska-Curie Grant Agreement n$^\circ$101034324, and from the AAP ``Accueil'' from Universit\'e Lyon 1 Claude Bernard. 
MD acknowledges financial support from the F.R.S.-FNRS, as well as from the European Union (ERC, PASTIS, Grant Agreement n$^\circ$101075879). Views and opinions expressed are however those of the authors only and do not necessarily reflect those of the European Union or the European Research Council Executive Agency. Neither the European Union nor the granting authority can be held responsible for them.

\appendix
\section{Proof of preliminary results}\label{sec:app-pr}
In this appendix, for convenience, we include proofs of several preliminary results stated in Section~\ref{sec:prelim}. Most of these results are well known but cannot be found as such in the literature.

\subsection{Proof of Lemma~\ref{lem:mom-cum}}
We follow~\cite[Proposition 2.2]{Nourdin_2010}, extending it to the present multivariate setting.
Let
\[M_{X_1,\ldots, X_m}(t_1,\ldots,t_m)\,:=\,\E\big[e^{\sum_{j=1}^mt_jX_j}\big]\]
be the multivariate moment generating function of $X_1, \dots, X_m$.  
We can write
\begin{eqnarray*}
\E[X_1 \ldots X_m]
&=& \frac{\ddr^m}{\ddr t_1 \dots \ddr t_m} M_{X_1, \ldots, X_m}(t_1, \ldots, t_m)\bigg|_{t_1 = \ldots = t_m = 0} \\
&=& \frac{\ddr^{m-1}}{\ddr t_2 \ldots\ddr t_{m}} \bigg( \Big(\frac{\ddr}{\ddr t_1} \log M_{X_1, \dots, X_m}(t_1, \ldots,t_m) \Big) M_{X_1, \ldots, X_m}(t_1,\ldots,t_n) \bigg) \bigg|_{t_1 = \ldots = t_m = 0},
\end{eqnarray*}
and thus, by the Leibniz rule,
\begin{multline*}
\E[X_1 \ldots X_m]
\,=\, \sum_{J\subset\llbracket 2,m\rrbracket} \bigg(\frac{\ddr^{\sharp J+1}}{\ddr t_1 \ddr t_{J}} \log M_{X_1 \ldots X_m}(t_1, \ldots, t_m)\bigg)\bigg|_{t_1 = \ldots = t_m = 0} \\
\times\bigg(\frac{\ddr^{m-1-\sharp J}}{\ddr t_{\llbracket2,m\rrbracket\setminus J}} M_{X_1 \ldots X_m}(t_1, \ldots, t_m) \bigg) \bigg|_{t_1 = \ldots = t_m = 0}, 
\end{multline*}
where $\ddr t_J$ stands for $\ddr t_{j_1}\ldots \ddr t_{j_s}$ if $J=\{j_1,\ldots,j_s\}$.
By definition of cumulants and of the moment generating function, this yields the conclusion.\qed

\subsection{Proof of Lemma~\ref{lem:cumtocorrel}}
\label{subsec:appendixLemma26}
We start from the relation between cumulants and moments, cf.~\eqref{eq:cumfrommom-inv0}:
using similar notation as in~\eqref{eq:def-cumGm}, the following general formula holds for all $m\ge1$, for any bounded random variable $X$,
\begin{equation}\label{eq:cumfrommom-inv}
\kappa^m[X]\,=\,\sum_{\pi\vdash \llbracket m \rrbracket}(-1)^{\sharp\pi-1}(\sharp\pi-1)!\prod_{A\in\pi}\E[X^{\sharp A}].
\end{equation}
We apply this to a linear functional of the empirical measure: given $\phi\in C^\infty_c(\Xd)$,
\[\kappa^m\bigg[\int_{\Xd}\phi\,\ddr\mu^N_t\bigg]\,=\,\sum_{\pi\vdash\llbracket m\rrbracket}(-1)^{\sharp\pi-1}(\sharp\pi-1)!\prod_{A\in \pi}\expec{\Big(\int_\Xd\phi\,\ddr\mu^N_t\Big)^{\sharp A}}.\]
In the right-hand side, moments of the empirical measure $\mu^N=\frac1N\sum_{i=1}^N\delta_{Z^{i,N}}$ can be computed as follows,
\begin{eqnarray*}
\expecM{\Big(\int_\Xd\phi\,\ddr\mu^N_t\Big)^n}
&=&\frac1{N^n}\sum_{i_1,\ldots,i_n=1}^N\expecM{\prod_{\ell=1}^n\phi(Z^{i_\ell,N}_t)}\\
&=&\frac1{N^n}\sum_{\pi\vdash\llbracket n\rrbracket} N(N-1)\ldots(N-\sharp\pi+1)\int_{\Xd^{\sharp\pi}}\Big(\bigotimes_{A\in\pi}\phi^{\sharp A}\Big)F^{\sharp \pi,N}_t,
\end{eqnarray*}
while marginals of $F_N$ can be expressed in terms of correlations via the cluster expansion~\eqref{eq:cluster-exp0},
\[F^{n,N}_t(z_{\llbracket n\rrbracket})\,=\,\sum_{\pi\vdash\llbracket n\rrbracket}\prod_{A\in\pi}G^{\sharp A,N}_t(z_A).\]
Combining those different identities, after straightforward simplifications, we obtain the following expression for cumulants of the empirical measure in terms of correlation functions,
\begin{equation}\label{eq:cum-from-correl}
\kappa^{m}\bigg[\int_{\Xd}\phi\,\ddr\mu^N_t\bigg]\,=\,\sum_{\pi\vdash\llbracket m\rrbracket}N^{\sharp\pi-m}\sum_{\rho\vdash\pi}K_N(\rho)\int_{\Xd^{\sharp\pi}}\Big(\bigotimes_{B\in\pi}\phi^{\sharp B}\Big)\Big(\bigotimes_{D\in\rho}G^{\sharp D, N}_t(z_D)\Big)\,\ddr z_\pi,
\end{equation}
where the coefficients are given by
\[K_N(\rho)\,:=\,\sum_{\sigma\vdash\rho}(-1)^{\sharp \sigma-1}(\sharp\sigma-1)!\Big(\prod_{C\in\sigma}(1-\tfrac1N)\ldots\big(1-\tfrac{(\sum_{D\in C}\sharp D)-1}{N}\big)\Big).\]
Isolating $\int_{\Xd^{m}}\phi^{\otimes(m)}G^{m,N}_t$ in the right-hand side of~\eqref{eq:cum-from-correl} (this term is obtained for the choice $\pi=\{\{1\},\ldots,\{m\}\}$ and $\rho=\{\pi\}$), and noting that $|K_N(\rho)|\le C_mN^{1-\sharp\rho}$, the conclusion follows.
\qed

\subsection{Proof of Proposition~\ref{prop:control_Glauber}}
We have shown in~\cite{MD-21} that cumulants can be expressed as polynomials of Glauber derivatives.
Let us first check how this can be extended to the multivariate case, that is, to joint cumulants of families of random variables.
For that purpose, we first introduce some notation and introduce a suitable notion of so-called Stein kernels~$\{\Gamma_n\}_n$ generalizing the one in~\cite{MD-21}. For all $n \ge 1$, given bounded $\sigma((Z_\circ^{j,N})_{j})$-measurable random variables $X_1, \ldots, X_n$, we define for all~$j\in\llbracket N \rrbracket$,
\[\delta_j^n(X_1, \dots, X_n) \,:=\, \E_\circ^{j\prime }\bigg[ \prod_{i=1}^n (X_i - (X_i)^{j\prime}) \bigg],\]
where for all $i,j$ the random variable $(X_i)^{j\prime}$ is obtained from $X_i$ by replacing the underlying variable~$Z^{j,N}_\circ$ by an i.i.d.\@ copy, and where $\E_\circ^{j\prime }$ stands for expectation with respect to this i.i.d.\@ copy.
Note in particular that~$\delta_j^1(X_1)=D_j^\circ X_1$, while $\delta_j^n(X_1,\ldots,X_n)$ should be compared to $\prod_{i=1}^n(D_j^\circ X_i)$.
In these terms, we now define the Stein kernels
\begin{gather*}
\Gamma_0(X_1) \,:=\,\Gamma_0^0(X_1)\,:=\, X_1,\\
\Gamma_1(X_1,X_2) \,:=\,\Gamma_1^1(X_1,X_2) \,:=\, \sum_{j = 1}^N (D_j^\circ X_2)  \calL_\circ^{-1} (D_j^\circ X_1)\,=\, \sum_{j = 1}^N \delta^1_j(X_2)  \calL_\circ^{-1}(D_j^\circ X_1),
\end{gather*}
and iteratively, for all $n \ge 1$, $m\ge0$, and $\sharp J=m$,
\begin{multline*}
\Gamma_n^{n+m}(X_1,\ldots,X_{n+1},X_J)\,:=\,\sum_{j=1}^N\Big(\delta_j^{m+1}(X_{n+1},X_J)\calL_\circ^{-1} D_j\Gamma_{n-1}^{n-1}(X_1,\ldots,X_n)\\[-2mm]
-\tfrac1{m+2}\mathds1_{n>1}\Gamma_{n-1}^{n+m}(X_1,\ldots,X_{n+1},X_J)\Big),
\end{multline*}
where we let $X_J = (X_{j_1}, \ldots, X_{j_s})$ for $J=\{j_1,\ldots,j_s\}$,
and we then set
\[\Gamma_n(X_1,\ldots,X_{n+1})\, :=\, \Gamma^n_n(X_1,\ldots,X_{n+1}).\]
Note that $\Gamma_n(X_1,\ldots,X_{n+1})$ is not symmetric in its arguments $X_1,\ldots,X_{n+1}$ (we could choose to consider instead its symmetrization, but it does not matter here).
In these terms, we can now state the following representation formula for cumulants.

\begin{lem}
\label{lem:cumulant_Y_0}
For all $n\ge0$ and all bounded $\sigma((Z_\circ^{j,N})_{j})$-measurable random variables $X_1,\ldots, X_{n+1}$, we have
\begin{equation*}
\kappa^{n+1}_\circ[X_1,\dots,X_{n+1}] \,=\, \E_\circ \big[\Gamma_n(X_1, \dots, X_{n+1})\big].
\end{equation*}
\end{lem}

\begin{proof} 
We omit the subscript `$\circ$'
for notational simplicity.
By the Hellfer-Sj\"ostrand representation formula~\eqref{eq:HS-rep}, we can write
\begin{eqnarray}
\E \bigg[\prod_{i=1}^m X_i \bigg]
&=&\E[X_1]\, \E \bigg[\prod_{i=2}^m X_i \bigg] + \Cov \bigg[X_1,\prod_{i=2}^m X_i \bigg]\nonumber \\
&=&\E[X_1]\, \E \bigg[\prod_{i=2}^m X_i \bigg]+ \sum_{j=1}^N \E\bigg[ D_j\bigg(\prod_{i=2}^m X_i\bigg)\calL_\circ^{-1} D_jX_1 \bigg].\label{eq:proof_gamma_1}
\end{eqnarray}
Now note that the following formula is easily obtained by induction for differences of products: for all $(a_i)_{2 \le i \le m},(b_i)_{2 \le i \le m}\subset\R$,
\begin{align*}
\prod_{i=2}^m a_i - \prod_{i=2}^m b_i \,=\,  \sum_{\substack{J \subset \llbracket2,m\rrbracket \\ J \ne \varnothing}} (-1)^{\sharp J + 1} \bigg(\prod_{i \not \in J} a_i \bigg)\bigg(\prod_{i  \in J} (a_i- b_i) \bigg),
\end{align*}
and this obviously implies
\begin{equation}\label{eq:decomp-prod-diff}
D_j\bigg(\prod_{i=2}^m X_i\bigg)
\,=\, \E^{j\prime} \bigg[\prod_{i=2}^m X_i - \prod_{i=2}^m (X_i)^{j\prime}\bigg] 
\,=\, \sum_{\substack{J \subset \llbracket2,m\rrbracket \\ J \ne \varnothing}} (-1)^{\sharp J + 1} \Big(\prod_{ i\in\llbracket2,m\rrbracket\setminus J} X_i \Big)\,\delta_j^{\sharp J}(X_J).
\end{equation}
Inserting this into~\eqref{eq:proof_gamma_1},
separating the contributions of singletons in the sum,
and recognizing the definition of~$\Gamma_0,\Gamma_1$,
we get
\begin{eqnarray}
\E \bigg[\prod_{i=1}^m X_i \bigg]
&=& \E[\Gamma_0(X_1)]\,\E\bigg[\prod_{i=2}^m X_i\bigg]
+\sum_{\substack{J \subset \llbracket 2,m\rrbracket \\ J \ne \varnothing}} (-1)^{\sharp J + 1} \sum_{j=1}^N \E\bigg[ \Big(\prod_{i \in \llbracket2,m\rrbracket\setminus J} X_k \Big)\,\delta_j^{\sharp J}(X_J) \calL_\circ^{-1} D_jX_1 \bigg] \nonumber \\
&=&\E[\Gamma_0(X_1)]\, \E\bigg[\prod_{i=2}^m X_i\bigg]
+ \sum_{\ell \in \llbracket 2,m\rrbracket}  \E\bigg[ \Big(\prod_{i\in\llbracket 2,m\rrbracket\setminus \{\ell\}} X_i \Big) \Gamma_1(X_1,X_\ell) \bigg]\nonumber\\
&&
+\sum_{\ell\in \llbracket 2,m\rrbracket}\sum_{\substack{J \subset \llbracket 2,m\rrbracket\setminus\{\ell\} \\ \sharp J \ge 1}} \tfrac{(-1)^{\sharp J}}{\sharp J+1} \sum_{j=1}^N \E\bigg[\Big(\prod_{i \in\llbracket2,m\rrbracket\setminus (\{\ell\}\cup J)} X_i \Big)\Gamma_1^{\sharp J+1}(X_1,X_\ell,X_J)\bigg].\label{eq:temp_gamma}
\end{eqnarray}
Using again the Hellfer--Sj\"ostrand representation formula~\eqref{eq:HS-rep} to handle the second right-hand side term, we can decompose for all $\ell\in \llbracket 2,m\rrbracket$,
\begin{multline*}
\E\bigg[ \Big(\prod_{i\in\llbracket 2,m\rrbracket\setminus \{\ell\}} X_i \Big) \Gamma_1(X_1,X_\ell) \bigg]
\,=\, \E[\Gamma_1(X_1,X_\ell)]\, \E\bigg[ \prod_{i\in\llbracket 2,m\rrbracket\setminus\{\ell\}} X_i\bigg]\\
+ \sum_{j=1}^N \E\bigg[ D_j\Big(\prod_{i\in\llbracket 2,m\rrbracket\setminus\{\ell\}} X_i \Big) \calL_\circ^{-1} D_j\Gamma_1(X_1,X_\ell) \bigg],
\end{multline*}
and thus, appealing again to~\eqref{eq:decomp-prod-diff} to reformulate the last term,
\begin{multline*}
\E\bigg[ \Big(\prod_{i\in\llbracket 2,m\rrbracket\setminus\{\ell\}} X_i \Big) \Gamma_1(X_1,X_\ell) \bigg]
\,=\, \E[\Gamma_1(X_1,X_\ell)]\, \E\bigg[ \prod_{i\in\llbracket 2,m\rrbracket\setminus \{\ell\}} X_i\bigg]\\
+\sum_{\substack{ J \subset \llbracket 2,m\rrbracket \setminus \{\ell\} \\ J \ne \varnothing}} (-1)^{\sharp J+1} \sum_{j=1}^N \E \bigg[ \Big(\prod_{i\in\llbracket 2,m\rrbracket\setminus(\{\ell\}\cup J)} X_i \Big) \delta^{\sharp J}_j(X_{J}) \calL_\circ^{-1} D_j\Gamma_1(X_1,X_\ell) \Big]. 
\end{multline*}
Inserting this into~\eqref{eq:temp_gamma} and recognizing the definition of $\Gamma_2$,
we find
\begin{multline*}
\E\bigg[\prod_{i=1}^m X_i\bigg]
\,=\,  \E[\Gamma_0(X_1)]\, \E\bigg[\prod_{i=2}^m X_i\bigg] + \sum_{\ell\in\llbracket 2,m\rrbracket } \E [\Gamma_1(X_1,X_\ell) ] \,\E\bigg[\prod_{i\in\llbracket 2,m\rrbracket\setminus \{\ell\}} X_i \bigg] \\
+\sum_{\ell,\ell'\in\llbracket 2,m\rrbracket}\E \bigg[ \Big(\prod_{i\in\llbracket 2,m\rrbracket\setminus\{\ell,\ell'\}} X_i \Big)\Gamma_2(X_1,X_\ell,X_{\ell'})\bigg]\\
+ \sum_{\ell,\ell'\in\llbracket 2,m\rrbracket}\sum_{\substack{J \subset [2,m]\setminus\{\ell,\ell'\}\\ J\ne\varnothing}} \tfrac{(-1)^{\sharp J}}{\sharp J+1}\E \bigg[ \Big(\prod_{i\in\llbracket 2,m\rrbracket\setminus(\{\ell,\ell'\}\cup J)} X_i \Big)\Gamma_2^{\sharp J+2}(X_1,X_\ell,X_{\ell'},X_J)\bigg],
\end{multline*}
and the claim follows by iteration and a direct comparison with the formula~\eqref{eq:esp_product}.
\end{proof}

We can now turn to the proof of Proposition~\ref{prop:control_Glauber}: the above representation formula for cumulants implies that cumulants can indeed be controlled as desired in terms of higher-order Glauber derivatives.

\begin{proof}[Proof of Proposition~\ref{prop:control_Glauber}]
First note that Jensen's inequality yields for all $m \ge 0$ and $r \ge 1$,
\begin{equation*}
\| \delta^m_j(X_1, \ldots, X_m) \|_{\Ld^r(\Omega_\circ)}
\,=\, \E_\circ \bigg[ \Big|\E^{j\prime}_\circ\Big[\prod_{i=1}^m (X_i - (X_i)^{j\prime}) \Big]\Big|^r \bigg]^{\frac1{r}}
\,\le\, \E_\circ\E^{j\prime}_\circ \bigg[\prod_{i=1}^m |X_i - (X_i)^{j\prime}|^r\bigg]^{\frac1{r}},
\end{equation*}
and thus, decomposing $X_i - (X_i)^{j\prime}=D_j^\circ X_i -(D_j^\circ X_i)^{j\prime}$ and using H\"older's inequality
\begin{equation*}
\| \delta^m_j(X_1, \ldots, X_m) \|_{\Ld^r(\Omega_\circ)}
\,\le\, \prod_{i=1}^m \E_\circ \E_\circ^{j'} \Big[ \big|D^\circ_j X_i - (D^\circ_j X_i)^{j'} \big|^{rm} \Big]^{\frac1{rm}}
\,\le\, 2^m\prod_{i=1}^m\|D_j^\circ X_i\|_{\Ld^{rm}(\Omega_\circ)}.
\end{equation*}
By induction, using this estimate along with~\eqref{eq:ineq_T}, we find for all $m,n,r$, for all bounded $\sigma((Z_\circ^{j,N})_j)$-measurable random variables $X_1,\ldots,X_{n+m+1}$,
\begin{multline*}
\|\Gamma^{m+n}_{n}(X_1, \ldots, X_{m+n+1})\|_{\Ld^r(\Omega)}\\
\,\lesssim_{m,n,r}\, \sum_{k=0}^{n-1} N^{k+1} \sum_{\substack {a_1, \dots, a_{m+n+1} \ge 1 \\ \sum_j a_j = m+n+k+1}} \prod_{j=1}^{m+n+1}\|(D^\circ)^{a_j} X_j\|_{\ell^\infty_{\ne}\big( \Ld^{\frac{r}{a_j}(m+n+k+1)}(\Omega_\circ)\big)}.
\end{multline*}
Combined with the representation formula of Lemma~\ref{lem:cumulant_Y_0}, this yields the conclusion.
\end{proof}

\subsection{Proof of Proposition~\ref{prop:2ndP}}
By homogeneity, it suffices to consider a bounded random variable~$Y$ with
\[\E_\circ[Y]\,=\,0,\qquad \sigma_Y^2\,=\,\Var_\circ[Y]\,=\,1.\]
Given $g\in C^1_b(\R)$, we define its Stein transform $S_g$ as the solution of Stein's equation
\begin{equation}\label{eq:Stein-tsf}
S_g'(x)- xS_g(x)=g(x)-\E_\Nc[g(\Nc)].
\end{equation}
As shown in~\cite{Stein-86}, the latter can be computed as
\[S_g(x)\,=\,-\int_0^1\frac1{2\sqrt t}\,\E_\Nc\Big[g'\big(\sqrt t \,  x+ \sqrt{1-t} \, \Nc\big)\Big]\,\ddr t,\]
Using this formula and a Gaussian integration by parts, we easily obtain the following bound,
\begin{equation}\label{eq:Stein-bound}
\|S_g'\|_{W^{1,\infty}(\R)}\,\lesssim\,\|g''\|_{\Ld^\infty(\R)}.
\end{equation}
Evaluating equation~\eqref{eq:Stein-tsf} at $Y$ and taking the expectation, we find
\[\E_\circ[g(Y)]-\E_\Nc[g(\Nc)]\,=\,\E_\circ\big[S_g'(Y)-YS_g(Y)\big].\]
Now appealing to the Helffer--Sj\"ostrand representation formula of Lemma~\ref{lem:L0-prop}(iii) for the covariance $\E_\circ[YS_g(Y)]=\Cov_\circ[Y,S_g(Y)]$, this yields
\begin{equation}\label{eq:HS-use-stein}
\E_\circ[g(Y)]-\E_\Nc[g(\Nc)]\,=\,\E_\circ\bigg[S_g'(Y)-\sum_{j=1}^N(D^\circ_jS_g(Y))\Lc_\circ^{-1}(D^\circ_jY)\bigg].
\end{equation}
A Taylor expansion gives for all $p\ge1$,
\[\big\|D^\circ_jS_g(Y)-S_g'(Y)D^\circ_jY\big\|_{\Ld^p(\Omega_\circ)}\,\le\,\|S_g''\|_{\Ld^\infty(\R)}\|D^\circ_jY\|_{\Ld^{2p}(\Omega_\circ)}^2.\]
Using this to replace $D^{\circ}_j S_g(Y)$ in~\eqref{eq:HS-use-stein}, using H\"older's inequality with $p=\frac32$ to bound the error, using the boundedness of $\Lc_\circ^{-1}$ in $\Ld^3(\Omega_\circ)$, cf.~Lemma~\ref{lem:L0-prop}(ii), and recalling the bound~\eqref{eq:Stein-bound} on the Stein transform, we are led to
\[\E_\circ[g(Y)]-\E_\Nc[g(\Nc)]\,\lesssim\,\|g''\|_{\Ld^\infty(\R)}\bigg(\E_\circ\bigg[\Big|1-\sum_{j=1}^N(D^\circ_jY)\Lc_\circ^{-1}(D^\circ_jY)\Big|\bigg]+\sum_{j=1}^N\E_\circ[|D^\circ_jY|^3]\bigg).\]
Now recalling the Helffer--Sj\"ostrand representation formula of Lemma~\ref{lem:L0-prop}(iii) in form of
\[1\,=\,\Var_\circ[Y]\,=\,\E_\circ\Big[\sum_{j=1}^N(D^\circ_jY)\Lc_\circ^{-1}(D^\circ_jY)\Big],\]
we deduce by the Cauchy--Schwarz inequality,
\[\E_\circ[g(Y)]-\E_\Nc[g(\Nc)]\,\lesssim\,\|g''\|_{\Ld^\infty(\R)}\bigg(\Var_\circ\Big[\sum_{j=1}^N(D^\circ_jY)\Lc_\circ^{-1}(D^\circ_jY)\Big]^\frac12+\sum_{j=1}^N\E_\circ[|D^\circ_jY|^3]\bigg).\]
Taking the supremum over $g\in C^2_b(\R)$, the conclusion follows in the second-order Zolotarev distance~$\ddr_2$. The proof in the $1$-Wasserstein distance can be obtained in the same way by noting that on top of~\eqref{eq:Stein-bound} the Stein transform also satisfies $\|S_g'\|_{W^{1,\infty}(\R)}\lesssim\|g'\|_{\Ld^\infty(\R)}$, cf.~\cite{Stein-86}. The proof in Kolmogorov distance is more delicate and we refer to~\cite[Theorem~4.2]{LRP-15}.
\qed

\subsection{Proof of Proposition~\ref{lem:concentration_Glauber}}
We appeal to the following degraded version of a log-Sobolev inequality as obtained in~\cite[Proposition~2.4]{DG20}: for all random variables $Y\in\Ld^2(\Omega_\circ)$, we have
\[\operatorname{Ent}_\circ[Y^2]\,\le\,2\sum_{j=1}^N\E_\circ\Big[\sup_j\,(Y-Y^j)^2\Big],\]
where we recall that $Y^j$ stands for the random variable obtained from $Y$ by replacing the underlying variable $Z_\circ^{j,N}$ by an i.i.d.\@ copy,
and where $\sup_j$ stands for the essential supremum with respect to this i.i.d.\@ copy.
Applying this inequality to $Y=e^{\frac12 X}$, using the bound
\[|e^{\frac12 X}-e^{\frac12 X^j}|\,\le\,\tfrac12|X-X^j|\int_0^1e^{\frac12 (X-t(X-X^j))}\,\ddr t\,\le\,\tfrac12 e^{\frac12 X}|X-X^j|\,e^{\frac12 |X-X^j|},\]
we find
\[\operatorname{Ent}_\circ[e^{ X}]\,\le\,\tfrac12\sum_{j=1}^N\E_\circ\Big[e^{ X}\sup_j\,(X-X^j)^2e^{|X-X^j|}\Big]
\,\le\,\tfrac{N}2M_X^2e^{ M_X}\E_\circ[e^{ X}],\]
in terms of
\[M_X\,:=\,\sup_{1\le j\le N}\supess|X-X^j|\,\le\,2\sup_{1\le j\le N}\supess|D_j^\circ X|\,\le\,L.\]
We are now in position to appeal to Herbst's argument in the form of~\cite[Proposition~2.9 and Corollary~2.12]{Ledoux-99} and the conclusion follows.
\qed

\subsection{Proof of Lemma~\ref{lem:unif-mom-est}}
We focus on the Langevin setting for shortness.
In addition, we only prove the estimates for~$\mu_t^N$, that is, \eqref{eq:estim-mom-Brownian}, \eqref{eq:estim-mom-part-exp1}, and~\eqref{eq:estim-mom-part-exp}, while the corresponding estimates for the mean-field dynamics are skipped: they can be obtained similarly, letting $N=1$ and replacing sums of interaction forces by mean-field force in the stochastic estimates below.
We split the proof into three steps.

\medskip
\step1 Proof of~\eqref{eq:estim-mom-Brownian}.\\
In the spirit of~\cite{Bolley_2010}, we consider the random process
\[G_{t}^N\,:=\,\tfrac1N\sum_{i=1}^N\Big(1+a|X_{t}^{i,N}|^2+|V_{t}^{i,N}|^2+\eta X_{t}^{i,N}\cdot V_{t}^{i,N}\Big),\]
for some $\eta\in(0,2\sqrt{a})$ to be suitable chosen.
Note that this range of $\eta$ ensures
\begin{equation}\label{eq:GtN-equiv}
G_t^N\,\simeq_{a,\eta}\,Q(\mu^N_t)^2.
\end{equation}
By It\^o's formula, the particle dynamics \eqref{eq:Langevin} yields
\begin{multline}\label{eq:mom-ito}
\ddr G_t^N\,=\,-\tfrac1N\sum_{i=1}^N\Big(-1 + a\eta |X_t^{i,N}|^2 +(\beta-\eta)|V_t^{i,N}|^2 +\tfrac{\eta\beta}2 X_t^{i,N}\cdot V_t^{i,N}\Big) \ddr t\\
-\tfrac\kappa{N^2}\sum_{i,j=1}^N(2V_t^{i,N}+\eta X_t^{i,N})\cdot \nabla W(X_t^{i,N}-X_t^{j,N}) \, \ddr t
+\tfrac1N\sum_{i=1}^N(2V_t^{i,N}+\eta X_t^{i,N})\cdot \ddr B_t^i.
\end{multline}
From this equation
and It\^o's formula, we then find for all $k\ge1$,
\begin{multline*}
\partial_t\E_B[(G_t^N)^k]\,\le\,-\tfrac kN\sum_{i=1}^N\E_B\bigg[(G_t^N)^{k-1}\Big(1+\tfrac14a\eta |X_t^{i,N}|^2+\big(\tfrac\beta2-\eta(1+\tfrac{\beta^2}{8a})\big)|V_t^{i,N}|^2\Big)\bigg]\\
+k\big(3 + (\tfrac{\eta}{a}+\tfrac2\beta)\|\nabla W\|_{\Ld^\infty(\R^d)}^2\big)\E_B[(G_t^N)^{k-1}]\\
+\tfrac{k(k-1)}{2N^2}\sum_{i=1}^N\E_B\big[(G_t^N)^{k-2}|2V_t^{i,N}+\eta X_t^{i,N}|^2\big].
\end{multline*}
Provided that $0<\eta\ll_{\beta,a}1$ is small enough (only depending on $\beta,a$), there exist $\lambda,C>0$ (only depending on~$d,W,\beta,a$) such that we get for all $k\ge1$,
\begin{equation*}
\partial_t\E_B[(G_t^N)^k]\,\le\,-\lambda
k\E_B[(G_t^N)^{k}]+Ck(1+\tfrac kN)\E_B[(G_t^N)^{k-1}],
\end{equation*}
and thus, by Gr\"onwall's inequality,
\begin{equation*}
\E_B[(G_t^N)^k]\,\le\,e^{-\lambda kt}\E_B[(G_0^N)^k]+Ck(1+\tfrac kN)\int_0^te^{-\lambda k(t-s)}\E_B[(G_s^N)^{k-1}] \, \ddr s.
\end{equation*}
A direct induction then yields for all $k\ge1$,
\begin{equation*}
\E_B[(G_t^N)^k]\,\le\,\sum_{j=0}^ke^{-\lambda jt}\,\tfrac{k!}{j!}(\tfrac1\lambda C)^{k-j}(1+\tfrac kN)^{k-j}\,\E_B[(G_0^N)^j],
\end{equation*}
and~\eqref{eq:estim-mom-Brownian} follows.

\medskip
\step2 Proof of~\eqref{eq:estim-mom-part-exp1}.\\
From~\eqref{eq:mom-ito} and It\^o's formula, arguing similarly as in Step~1, provided that $0<\eta\ll_{\beta,a}1$ is small enough (only depending on $\beta,a$), there exist $\lambda,C>0$ (only depending on~$d,W,\beta,a$) such that we have for all $t,\alpha,L\ge0$ and $0\le\delta\le1$,
\begin{multline}\label{eq:Lapl-tsf-GtN}
\partial_t\E_B\big[e^{\alpha(G_t^N+L)^\delta}\big]
\,\le\,
-2\lambda\delta\alpha\,\E_B\big[(G_t^N+L)^{\delta}e^{\alpha(G_t^N+L)^\delta}\big]\\
+C(L+1)\delta\alpha\,\E_B\big[(G_t^N+L)^{\delta-1}e^{\alpha(G_t^N+L)^\delta}\big]
+\tfrac{C}N\delta^2\alpha^2\,\E_B\big[(G_t^N+L)^{2\delta-1}e^{\alpha(G_t^N+L)^\delta}\big].
\end{multline}
For $\delta<1$, choosing $L=(\frac{C\alpha}{\lambda N}+1)^{\frac1{1-\delta}}$,
we get $\frac CN\delta^2\alpha^2(G_t^N+L)^{2\delta-1}\le\lambda\delta\alpha(G_t^N+L)^\delta$, and thus
\begin{eqnarray*}
\partial_t\E_B\big[e^{\alpha(G_t^N+L)^\delta}\big]
&\le&
-\lambda\delta\alpha\,\E_B\big[(G_t^N+L)^{\delta}e^{\alpha(G_t^N+L)^\delta}\big]
+C(L+1)\delta\alpha\,\E_B\big[(G_t^N+L)^{\delta-1}e^{\alpha(G_t^N+L)^\delta}\big]\\
&\le&
-\lambda\delta\alpha\,\E_B\big[(G_t^N+L)^{\delta}e^{\alpha(G_t^N+L)^\delta}\big]
+C_\delta\alpha(\tfrac\alpha N+1)^{\frac\delta{1-\delta}}\,\E_B\big[e^{\alpha(G_t^N+L)^\delta}\big].
\end{eqnarray*}
This amounts to the following differential inequality for the Laplace transform $F_\delta(t,\alpha):=\E_B[e^{\alpha (G_t^N+L)^\delta}]$: for all $t,\alpha\ge0$,
\[\partial_tF_\delta(t,\alpha)+\lambda\delta\alpha\partial_\alpha F_\delta(t,\alpha)\,\le\, C_\delta\alpha(\tfrac\alpha N+1)^{\frac\delta{1-\delta}} F_\delta(t,\alpha),\]
which can be rewritten as
\[\partial_t[F_\delta(t,e^{\lambda\delta t}\alpha)]
\,\le\,C_\delta(e^{\lambda\delta t}\alpha)(\tfrac{e^{\lambda\delta t}\alpha} N+1)^{\frac\delta{1-\delta}}F_\delta(t,e^{\lambda\delta t}\alpha).\]
By integration, this yields for all $t,\alpha\ge0$,
\[F_\delta(t,\alpha)
\,\le\,\exp\Big(C_\delta \alpha(\tfrac{\alpha} N+1)^{\frac\delta{1-\delta}}\Big)F_\delta(0,e^{-\lambda\delta t}\alpha),\]
and~\eqref{eq:estim-mom-part-exp1} follows after recalling $G^N_t\simeq_{\beta,a} Q(\mu^N_t)^{2}$.

\medskip
\step3 Proof of~\eqref{eq:estim-mom-part-exp}.\\
Starting point is~\eqref{eq:Lapl-tsf-GtN} with $\delta=1$ and $L=0$: provided that $0<\eta\ll_{\beta,a}1$ is small enough (only depending on~$\beta,a$), there exist $\lambda,C>0$ (only depending on~$d,W,\beta,a$) such that we have for all $t,\alpha>0$,
\begin{equation*}
\partial_t\E[e^{\alpha G_t^N}]\,\le\,
-2\lambda\alpha\E[G_t^Ne^{\alpha G_t^N}]
+C\alpha\E[e^{\alpha G_t^N}]+\tfrac{C}N\alpha^2\E[G_t^Ne^{\alpha G_t^N}].
\end{equation*}
Hence, setting $\alpha_0 :=\frac{\lambda}{C}$, we get for $\alpha\le N\alpha_0$,
\begin{equation*}
\partial_t\E[e^{\alpha G_t^N}]\,\le\,
-\lambda\alpha\E[G_t^Ne^{\alpha G_t^N}]
+C\alpha\E[e^{\alpha G_t^N}].
\end{equation*}
This amounts to the following differential inequality for the Laplace transform $F(t,\alpha):=\E[e^{\alpha G_t^N}]$: for all $t\ge0$ and $0\le\alpha\le N\alpha_0$,
\[\partial_tF(t,\alpha)+\lambda\alpha\partial_\alpha F(t,\alpha)\,\le\, C\alpha F(t,\alpha),\]
which can be rewritten as follows, for all $t\ge0$ and $0\le e^{\lambda t}\alpha\le N\alpha_0$,
\[\partial_t[F(t,e^{\lambda t}\alpha)]\le Ce^{\lambda t}\alpha F(t,e^{\lambda t}\alpha).\]
This yields $F(t,\alpha)\le e^{\alpha C/\lambda}F(0,e^{-\lambda t}\alpha)$ for $t\ge0$ and $0\le \alpha\le N\alpha_0$,
and~\eqref{eq:estim-mom-part-exp} follows.
\qed

\bibliographystyle{plain}
\bibliography{biblio_Chaos}

\end{document}